\documentclass[reqno, 11pt]{amsart}

\usepackage[parfill]{parskip}    
\usepackage{amsrefs, amsmath}
\usepackage{amsfonts, mathrsfs}
\usepackage{amssymb}
\usepackage{amscd, mathtools}
\usepackage{hyperref, eqnarray}
\usepackage{cleveref, multirow, bigdelim} 
\usepackage{fullpage, color}
\usepackage[labelformat=empty]{subfig}
\usepackage{booktabs}
\usepackage{xypic}
\usepackage{enumerate, longtable}
\usepackage{makecell, diagbox}
\usepackage{IEEEtrantools}
\usepackage{pifont}
\usepackage{enumitem}

\numberwithin{equation}{section}

\BibSpec{article}{%
+{}{\sc \PrintAuthors} {author}
+{,}{ \textrm} {title}
+{,}{ \textit} {journal}
+{,}{ } {volume}
+{}{ \IfEmptyBibField{journal}{}{\PrintDate{date}}} {transition}
+{,}{ } {pages}
+{,}{ } {status}
+{.}{} {transition}
+{}{ Preprint available at \tt} {eprint}
+{.}{} {transition}
}

\BibSpec{book}{%
+{}{\sc \PrintAuthors} {author}
+{,}{ \textit} {title}
+{.}{ } {series}
+{,}{ } {volume}
+{.}{ } {publisher}
+{,}{ } {place}
+{,}{ } {date}
+{.}{} {transition}
}

\BibSpec{collection.article}{
+{} {\sc \PrintAuthors} {author}
+{,} { \textrm } {title}
+{,} { in \it \PrintConference} {conference}
+{,} { pp.~} {pages}
+{.} {\PrintBook} {book}
+{,} { \PrintDateB} {date}
+{.} {} {transition}
}

\BibSpec{innerbook}{
+{.} { \emph} {title}
+{.} { } {part}
+{:} { \emph} {subtitle}
+{.} { } {series}
+{,} { } {volume}
+{.} { Edited by \PrintNameList} {editor}
+{.} { Translated by \PrintNameList}{translator}
+{.} { \PrintContributions} {contribution}
+{.} { } {publisher}
+{.} { } {organization}
+{,} { } {address}
+{,} { \PrintEdition} {edition}
+{,} { \PrintDateB} {date}
+{.} { } {note}
}

\allowdisplaybreaks[3]

\crefname{equation}{Eq.}{Eqs.}
\crefname{eqnarray}{Eq.}{Eqs.}
\crefname{algo}{Algorithm}{Algorithms}
\crefname{conj}{Conjecture}{Conjectures}
\crefname{lem}{Lemma}{Lemmas}
\crefname{thm}{Theorem}{Theorems}
\crefname{claim}{Claim}{Claims}
\crefname{rmk}{Remark}{Remarks}
\crefname{prop}{Proposition}{Propositions}
\crefname{section}{Section}{Sections}
\crefname{appendix}{Appendix}{Appendices}
\crefname{cor}{Corollary}{Corollaries}
\crefname{figure}{Figure}{Figures}
\crefname{table}{Table}{Tables}
\crefname{example}{Example}{Examples}
\crefname{prob}{Problem}{Problems}
\crefname{assm}{Assumption}{Assumptions}
\crefname{defn}{Definition}{Definitions}

\newcommand{\de}{{\partial}}

\newcommand{\rd}{\mathrm{d}}
\newcommand{\ri}{\mathrm{i}}
\newcommand{\re}{\mathrm{e}}

\newcommand{\bbA}{\mathbb{A}}

\newcommand{\bbN}{\mathbb{N}}

\newcommand{\bbZ}{\mathbb{Z}}
\newcommand{\bbR}{\mathbb{R}}
\newcommand{\bbC}{\mathbb{C}}
\newcommand{\bbP}{\mathbb{P}}
\newcommand{\bbF}{\mathbb{F}}
\newcommand{\bbQ}{\mathbb{Q}}

\def\bary{\begin{array}} 
\def\eary{\end{array}} 
\def\ben{\begin{enumerate}} 
\def\een{\end{enumerate}}
\def\bit{\begin{itemize}} 
\def\eit{\end{itemize}}
\def\nn{\nonumber} 

\newcommand{\cY}{\mathcal{Y}}
\newcommand{\cZ}{\mathcal{Z}}
\newcommand{\cO}{\mathcal{O}}
\newcommand{\cT}{\mathcal{T}}
\newcommand{\cE}{\mathcal{E}}
\newcommand{\cP}{\mathcal{P}}

\newcommand{\cN}{\mathcal{N}}
\newcommand{\cW}{\mathcal{W}}

\newcommand{\cV}{\mathcal{V}}

\newcommand{\cM}{\mathcal M}

\def\beq{\begin{equation}}                     %
\def\eeq{\end{equation}}                       %
\def\bea{\begin{eqnarray}}                     
\def\eea{\end{eqnarray}}
\def\bary{\begin{array}} 
\def\eary{\end{array}} 
\def\ben{\begin{enumerate}} 
\def\een{\end{enumerate}}
\def\bit{\begin{itemize}} 
\def\eit{\end{itemize}}
\def\nn{\nonumber} 
\def\de {\partial}

\def\IP{{\mathbb P}}

\def\a{\alpha}
\def\b{\beta}

\theoremstyle{plain}

\newtheorem{thm}{Theorem}[section]
\newtheorem{lem}[thm]{Lemma}

\newtheorem{prop}[thm]{Proposition}
\newtheorem{conj}[thm]{Conjecture}

\newtheorem*{conj*}{Conjecture}

\newtheorem*{cor*}{Corollary}

\theoremstyle{definition}

\newtheorem*{rem*}{Remark}

\newtheorem*{rems*}{Remarks}
\newtheorem{rmk}[thm]{Remark}
\newtheorem{example}{Example}[section]

\newcommand{\GITl}[1]{\backslash \!\! \backslash _{\kern-.2em #1 \kern0.1em}}
\newcommand{\GIT}[1]{/\!\!/_{\kern-.2em #1 \kern0.1em}}

\renewcommand{\l}{\left}
\renewcommand{\r}{\right}
\newcommand{\bra}{\left\langle}
\newcommand{\ket}{\right\rangle}

\newcommand{\ev}{\operatorname{ev}}
\newcommand{\qbinom}{\genfrac{[}{]}{0pt}{}}

\def\bred{\begin{color}{red}}
\def\ered{\end{color}}

\def\bes{\begin{subequations}}
\def\ees{\end{subequations}}

\usepackage{tikz}
\usetikzlibrary{decorations.pathreplacing,angles,quotes}

\newcommand\A{\mathbb A}
\newcommand\C{\mathbb C}

\newcommand\F{\mathbb F}

\newcommand\NN{\mathbb N}

\newcommand\PP{\mathbb P}
\newcommand\Q{\mathbb Q}
\newcommand\Z{\mathbb Z}
\newcommand\Aut{\operatorname{Aut}}

\newcommand\virt{\mathrm{virt}}
\newcommand\pt{\mathrm{pt}}

\newtheorem{theorem}{Theorem}[section]

\newtheorem{lemma-definition}[theorem]{Lemma-Definition}

\theoremstyle{definition}

\newtheorem{defn}[theorem]{Definition}
\newtheorem{construction}[theorem]{Construction}

\theoremstyle{remark}

\numberwithin{equation}{section}
\numberwithin{figure}{section}

\newcommand{\ZZ} {\mathbb{Z}}
\newcommand{\QQ} {\mathbb{Q}}
\newcommand{\RR} {\mathbb{R}}

\renewcommand{\AA} {\mathbb{A}}

\newcommand {\shD}  {\mathcal{D}}

\newcommand {\shO}  {\mathcal{O}}

\newcommand {\loc} {\mathrm{loc}}

\newcommand {\ol} {\overline}

\newcommand {\Tot}  {\operatorname{Tot}}

\DeclareMathOperator {\hhh} {H}

\newcommand\bS{\mathbb S}
\newcommand\ccL{\mathcal L}
\newcommand\fd{\mathfrak d}

\newcommand{\sstyle}{\scriptstyle}

\def\mydate{\ifcase\month \or January\or February\or March\or
April\or May\or June\or July\or August\or September\or October\or 
November\or December\fi \space\number\day,\space\number\year}

\newcommand{\vir}{{\rm vir}}

\newcommand{\ptwo}{\IP^2}
\newcommand{\pone}{\IP^1}

\DeclareMathOperator{\pic}{Pic}
\DeclareMathOperator{\tot}{Tot}

\DeclareMathOperator{\mmm}{M}

\newcommand{\wt}{\widetilde}

\DeclareMathOperator{\fzero}{\mathbb{F}_0}
\DeclareMathOperator{\fone}{\mathbb{F}_1}
\DeclareMathOperator{\ftwo}{\mathbb{F}_2}
\DeclareMathOperator{\fn}{\mathbb{F}_n}

\DeclareMathOperator{\dpr}{dP_r}
\DeclareMathOperator{\delp}{dP}
\DeclareMathOperator{\obstr}{Ob}

\DeclareMathOperator{\scatt}{Scatt}
\DeclareMathOperator{\fend}{\scriptscriptstyle{end}}

\newcommand*\circleed[1]{\tikz[baseline=(char.base)]{
            \node[shape=circle,draw,inner sep=1pt] (char) {#1};}}

\begin{document}

\title{
  Stable maps to Looijenga pairs
  }

\author{Pierrick Bousseau}

\address{\tiny
Universit\'e Paris-Saclay, CNRS, Laboratoire de mathématiques d'Orsay, 91405, Orsay, France}


\email{pierrick.bousseau@u-psud.fr}


\author{Andrea Brini}
\address{\tiny 
University of Sheffield, School of Mathematics and Statistics, S11 9DW, Sheffield, United Kingdom\\
On leave from CNRS, DR 13, Montpellier, France}
\email{a.brini@sheffield.ac.uk}

\author{Michel van Garrel}

\address{\tiny University of Birmingham, School of Mathematics, B15 2TT, Birmingham, United Kingdom}
\email{m.vangarrel@bham.ac.uk}

\thanks{This project has been supported by the European Union's Horizon
  2020 research and innovation programme under the Marie Sklodowska-Curie
  grant agreement No 746554 (M.~vG.), the Engineering and Physical Sciences
  Research Council under grant agreement ref.~EP/S003657/2 (A.~B.) and by Dr.\ Max R\"ossler, the Walter Haefner Foundation and the ETH Z\"urich Foundation (P.~B. and M.~vG.).}

\begin{abstract}
A log Calabi--Yau surface with maximal boundary, or Looijenga pair, is a pair $(Y,D)$ with $Y$ a smooth rational projective complex surface and $D=D_1+\dots + D_l \in |-K_Y|$ an anticanonical singular nodal curve. Under some natural conditions on the pair, we propose a series of correspondences relating five different classes of enumerative invariants attached to $(Y,D)$:
\ben
\item the log Gromov--Witten theory of the pair $(Y,D)$, 
\item the Gromov--Witten theory of the total space of $\bigoplus_i \cO_Y(-D_i)$,
\item the open Gromov--Witten theory of special Lagrangians in a Calabi--Yau 3-fold determined by $(Y,D)$,  
\item the Donaldson--Thomas theory of a symmetric quiver specified by $(Y,D)$, and 
\item a class of BPS invariants considered in different contexts by Klemm--Pandharipande, Ionel--Parker, and Labastida--Mari\~no--Ooguri--Vafa. \een We furthermore provide a complete closed-form solution to the calculation of all these invariants.
\end{abstract}

\maketitle
\setcounter{tocdepth}{1}
\tableofcontents

\section{Introduction}

\subsection{Looijenga pairs}
A log Calabi--Yau surface with maximal boundary,  or Looijenga pair, is a pair $Y(D) \coloneqq (Y,D)$ consisting of a smooth rational projective complex surface $Y$ and an anticanonical singular nodal curve $D=D_1+\dots + D_l \in |-K_Y|$.
A prototypical example of Looijenga pair is given by $
(Y,D)=(\ptwo,D_1+D_2)$ for $D_1$ a line and $D_2$ a conic not tangent to $D_1$. 

Looijenga pairs were first 
systematically studied in relation with resolutions and deformations of elliptic surface singularities \cite{MR632841} and with degenerations of K3 surfaces \cite{MR813580}. 
More recently, Looijenga pairs have played  an important role as two-dimensional examples for 
mirror symmetry \cites{GHKlog, MR3518552, MR3563248, yu2016enumeration,barrott2020explicit,MR4048291, hacking2020homological} and for the theory of cluster varieties \cites{GHKclu, MR3954363, zhou2019weyl}. These new developments have had in return non-trivial applications to the classical geometry of Looijenga pairs \cites{GHKlog,GHKmod,MR3825608,Fri}.

\subsection{Summary of the main results}

\label{intro:overview}

In this paper we develop a series of correspondences relating different enumerative invariants associated to a given Looijenga pair. We start off by giving a very succinct summary of the main objects we will consider, and the main statements we shall prove.

\subsubsection{Geometries}

Let $(Y,D=D_1+\cdots+D_l)$ be a Looijenga pair with $l\geq2$. In this paper we will construct four different geometries out of $(Y,D)$:\medskip

\begin{itemize}
    \item
    the log Calabi--Yau surface geometry $Y(D)$;
    \item
    the local Calabi--Yau $(l+2)$-fold geometry $E_{Y(D)}\coloneqq {\rm Tot}\begin{pmatrix}\shO_Y(-D_1)\oplus\cdots\oplus\shO_Y(-D_l)\end{pmatrix}$;
    \item
    a non-compact Calabi--Yau threefold geometry canonically equipped with a disjoint union of $l-1$ Lagrangians,
    \[
    Y^{\rm op}(D) \coloneqq  \begin{pmatrix}{\rm Tot}\begin{pmatrix}\shO(-D_l)\to Y\setminus\left(D_1\cup\cdots\cup D_{l-1}\right)\end{pmatrix},L_1\sqcup\cdots\sqcup L_{l-1}\end{pmatrix},
    \]
    where $L_i$ are fibred over real curves in $D_i$;
    \item
    for $l=2$, a non-commutative geometry given by a symmetric quiver $\mathsf{Q}(Y(D))$ made from the combinatorial data of the divisors $D_i$ and their intersections. 
\end{itemize}
\medskip

\subsubsection{Enumerative theories}

Our main focus will be on the enumerative geometry of curves in these geometries.
More precisely, to a Looijenga pair $Y(D)$ 
satisfying some natural positivity conditions, we shall associate several classes of \emph{a priori} different enumerative invariants:\medskip

\begin{enumerate}[leftmargin=2.5cm]
    \item[(\tt log GW)] all genus log GW invariants of $Y(D)$, counting curves in the surface $Y$ with maximal tangency conditions along the divisors $D_i$;
    \item[(\tt local GW)] genus zero local GW invariants of the ${\rm CY}(l+2)$-fold $E_{Y(D)}$;
    \item[(\tt open GW)] all genus open GW invariants counting open Riemann surfaces in the CY3-fold $Y^{\rm op}(D)$ with $l-1$ boundary components mapping to $L_1\sqcup\cdots \sqcup L_{l-1}$; 
\item [(\tt local BPS)] genus zero local BPS invariants of $E_{Y(D)}$, in the form of Gopakumar--Vafa/Klemm--Pan\-dha\-ri\-pan\-de/Ionel--Parker (GV/KP/IP) BPS invariants;
    \item [(\tt open BPS)] all genus open BPS invariants of $Y^{\rm op}(D)$, in the form of Labastida--Mari\~no--Ooguri--Vafa (LMOV) BPS invariants;
    \item [(\tt quiver DT)] if $l=2$, Donaldson--Thomas (DT) invariants of $\mathsf{Q}(Y(D))$.
        
\end{enumerate}
\medskip

\subsubsection{Correspondences}
Under some positivity conditions on $(Y,D)$, we will prove that the invariants above essentially coincide with one another. In particular, we shall show:

\begin{itemize}
    \item[(i)] an equality between ({\tt log~GW}) and ({\tt local~GW}) in genus zero ({\bf \cref{thm:log_local_intro}});
    \item[(ii)] an equality between ({\tt log~GW}) and ({\tt open~GW}) in all genera ({\bf \cref{thm:logopen_intro}});
    \item[(iii)]  an equality
    between ({\tt local BPS}) and ({\tt open BPS}) in genus zero for all $l$;
    \item[(iv)] an equality between ({\tt local BPS}) and ({\tt quiver DT}) for $l=2$, i.e. when the local geometry $E_{Y(D)}$ is CY$_4$ ({\bf \cref{thm:kpdt_intro}}).
\end{itemize}
The equality (i) establishes for log~CY surfaces a version of a conjecture of van~Garrel--Graber--Ruddat about log and local GW invariants \cite{vGGR}, while (ii) and (iv) are new. Equality (iii) follows from (i)-(ii) after a BPS-type change of variables.

\subsubsection{Integrality}
Furthermore, we shall prove that the enumerative invariants of Looijenga pairs considered in this paper obey strong integrality constraints ({\bf \cref{thm:openbps_intro}}), reflecting the conjectured integrality of the ({\tt open BPS}) and ({\tt local BPS}) counts. This shows the existence of novel integral structures underlying the higher genus log~GW theory of $Y(D)$. Restricting to genus zero, we will obtain as a corollary an algebro-geometric proof of the conjectured integrality  of the genus zero Gopakumar--Vafa invariants of the CY-$(l+2)$-fold $E_{Y(D)}$. In particular, for $l=2$, this proves for CY4 local surfaces an integrality conjecture of Klemm--Pandharipande \cite[Conjecture~0]{Klemm:2007in}.

\subsubsection{Solutions}
Moreover, we will completely solve the enumerative counts for these geometries ({\bf \cref{thm:log_local_intro,thm:logopen_intro}}), by finding explicit closed-form, non-recursive expressions for the generating series of the invariants associated to our Looijenga pairs. 
\\

The rest of the introduction is organised as follows. \medskip

\bit
\item \cref{intro:enum} sets the stage by giving a self-contained account of the enumerative theories we shall consider. 
\item \cref{sec:inter-intro} illustrates the geometric picture underpinning the web of correspondences explored in the paper. We spell out the enumerative relations (i)-(iv) in the broadest generality where we believe them to hold, and describe in detail the geometric heuristics which led us to (i) in \cref{sec:introloglocal} ({\bf \cref{conj:loglocgen}}), to (ii) in \cref{sec:intrologopen} ({\bf \cref{conj:ocd,conj:refocd}}), and to (iii)-(iv) in \cref{sec:introquiverbps}. 
\item \cref{sec:introres} we put these ideas on a rigorous footing. We first place a natural positivity condition on the irreducible components $D_i$ by requiring them to be all smooth and nef; depending on the context, we often supplement this with a mild condition of ``quasi-tameness'', whose rationale is justified in \cref{sec:loggw_intro_res,sec:logopen_intro_res}. The statements of the proof of the correspondences, the integrality results, and the full solutions for our enumerative counts are spelled out in {\bf  \cref{thm:openbps_intro,thm:kpdt_intro,thm:log_local_intro,thm:logopen_intro}}.
\item \cref{sec:intro_implications} surveys the implications of our results for related work, with emphasis on the possible sheaf-theoretic interpretations of the BPS invariants we consider.
\eit

\subsection{Enumerative problems}\label{intro:enum}
\subsubsection{Higher genus log Gromov--Witten invariants}
\label{sec:loggw_intro}
Log Gromov--Witten theory, developed by Abramovich--Chen \cites{Chen14,AbramChen14}
and Gross--Siebert \cite{GS13}, provides a deformation-invariant way to count curves with prescribed tangency conditions along a normal crossings divisor, by virtual intersection theory on moduli spaces of stable log maps. For 
$Y(D)$ a Looijenga pair where $D$ has $l \geq 2$ irreducible components, we consider rational curves in $Y$ with given degree $d\in\hhh_2(Y,\ZZ)$ that meet each component $D_j$ in
one point of maximal tangency $d \cdot D_j$ and pass through $l-1$ given points in $Y$.
Counting such curves is an enumerative problem of expected dimension $0$ and we denote by $N_{0,d}^{\rm log}(Y(D))$ the corresponding log Gromov--Witten invariants. By \cite[Proposition 6.1]{Man19}, these log Gromov--Witten invariants are enumerative: they are simply equal to the naive count of rational curves in $Y$ when the $l-1$ given points are in sufficiently general position in $Y$.

For $g \geq 0$, the expected dimension of the moduli space of genus $g$ curves in $Y$ with given degree $d\in\hhh_2(Y,\ZZ)$ that meet each component $D_j$ in
one point of maximal tangency $d \cdot D_j$ and pass through $l-1$ given points in $Y$, is $g$. On the other hand, assigning to every stable log map $f \colon C \rightarrow Y(D)$ the vector space $\hhh^0(C,\omega_C)$ of sections of the dualising sheaf of the domain curve defines a rank $g$ vector bundle over the moduli space, called the Hodge bundle, and we denote by 
$\lambda_g$ its top Chern class. 
We define log Gromov--Witten invariants $N^{\rm log}_{g,d}(Y(D))$ by integration
of $(-1)^g \lambda_g$
over the virtual fundamental class of the moduli space. For $g=0$, $N^{\rm log}_{0,d}(Y(D))$ recovers the naive count of rational curves but for $g>0$, the log Gromov--Witten invariants $N^{\rm log}_{g,d}(Y(D))$ no longer have an obvious interpretation in terms of naive enumeration of curves.
Fixing the degree $d$ and summing over all genera, we define  generating series 
\beq
\mathsf{N}^{\rm log}_d(Y(D))(\hbar) \coloneqq \frac{1}{\left( 
2 \sin \left( \frac{\hbar}{2} \right) \right)^{l-2}}
\sum_{g \geqslant 0} N^{\rm log}_{g,d} (Y(D)) \hbar^{2g-2+l}\,.
\label{eq:GWexp_intro}
\eeq
The term $\left( 2 \sin \left( \frac{\hbar}{2} \right) \right)^{2-l}$ is natural from the point of view of the $q$-refined scattering diagrams of \cref{sec:loggw}. It is accounted for in the correspondence with the open invariants.

\subsubsection{Local Gromov--Witten invariants}
\label{sec:locgw_intro}
To a Looijenga pair $Y(D)=(Y,D=D_1+\cdots+D_l)$, we associate the $(l+2)$-dimensional 
non-compact 
Calabi--Yau variety 
$E_{Y(D)} \coloneqq \mathrm{Tot}(\oplus_{i=1}^l (\cO_Y(-D_i)))$. 
We view $Y$ in $E_{Y(D)}$ via the inclusion given by the zero-section.
We refer to 
$E_{Y(D)}$ as the local geometry attached to $Y(D)$. 
If each component $D_i$ is nef, then
for every $d\in\hhh_2(Y,\ZZ)$ 
intersecting $D_i$ generically, the moduli space of genus $0$ stable maps to 
$E_{Y(D)}$ of degree $d$ is compact: every stable map to $E_{Y(D)}$ of class 
$d$ factors through the zero-section $Y$. Thus, it makes sense to consider the local
genus $0$ Gromov--Witten invariants
$N^{\rm loc}_{0,d}(Y(D))$, which define virtual counts of rational curves in $E_{Y(D)}$ passing through $l-1$ given points in $Y$.

\subsubsection{Higher genus open Gromov--Witten invariants}
\label{sec:opengw_intro}
Let $X$ be a semi-projective toric Calabi--Yau 3-fold, i.e.\ a toric Calabi--Yau 3-fold which admits a presentation as the GIT quotient of a vector space by a torus action \cite{MR2015052}.
We will be concerned with a class of Lagrangian submanifolds of $X$ considered by Aganagic--Vafa in \cite{Aganagic:2000gs}, that we simply refer to as toric Lagrangians: symplectically, these are singular fibres of the Harvey--Lawson fibration associated to $X$. A toric Lagrangian is diffeomorphic to 
$\mathbb{R}^2 \times S^1$, and so its first homology group is isomorphic to 
$\Z$.

We fix $ L=L_1 \cup \dots \cup L_s$ a disjoint union of  toric Lagrangians $L_i$ in $X$.  In informal terms, the open Gromov--Witten theory of $(X, L=L_1 \cup \dots \cup L_s)$ should be a virtual count of maps to $X$ from open Riemann surfaces of fixed genus, relative homology degree, and boundary winding data around $S^1 \hookrightarrow L$.
A precise definition of such counts in the algebraic category has been given by
Li--Liu--Liu--Zhou \cites{Li:2001sg,Li:2004uf} using relative Gromov--Witten theory and virtual localisation.
These invariants depend on the choice of a framing $\mathsf{f}$ of $L$, which is a choice of integer $f_i$ for each connected component $L_i$ of $L$.
Given partitions $\mu_1,\cdots, \mu_s$ of length
$\ell(\mu_1),\cdots, \ell(\mu_s)$, we denote by $O_{g; \beta; (\mu_1, \dots, \mu_s)}(X,L,\mathsf{f})$ the invariants defined in  \cites{Li:2001sg,Li:2004uf}, which are informally open Gromov--Witten invariants counting connected genus $g$ Riemann surfaces of class $\beta \in \hhh_2(X,L,\ZZ)$ with, for every $1 \leq i \leq s$, $\ell(\mu_i)$ boundary components wrapping $L_i$
with winding numbers given by the parts of $\mu_i$.
We package the open Gromov--Witten invariants $O_{g,\beta,\mu_1, \dots, \mu_s}(X,L,\mathsf{f})$ into formal generating functions
%
\bea
\mathsf{O}_{\beta;\vec \mu}(X,L,\mathsf{f})(\hbar) &\coloneqq & \sum_{g\geq0} \hbar^{2g-2+\ell(\vec\mu)}  O_{g; \beta; \vec \mu}(X,L,\mathsf{f}) \,,
\label{eq:openfreeen_intro}
\eea
where $\ell(\vec\mu)=\sum_{i=1}^s \ell(\mu_i)$. We simply denote by $O_{g; \beta}(X,L,\mathsf{f})$ and $\mathsf{O}_{\beta}(X,L,\mathsf{f})(\hbar)$ 
the $s$-holed open Gromov--Witten invariants obtained when each partition 
$\mu_i$ consists of a single part (whose value is then determined by 
the class $\beta  \in \hhh_2(X,L,\ZZ)$).

\subsubsection{Quiver DT invariants}
\label{sec:quiverdt_intro}
Let $\mathsf{Q}$ be a quiver with an ordered set $\mathsf{Q}_0$ of $n$ vertices $v_1, \dots v_n \in \mathsf{Q}_0$ and a set of oriented edges $\mathsf{Q}_1 = \{\a:v_i \to v_j  \}$. We let $\bbN \mathsf{Q}_0$ be the free abelian semi-group generated by $\mathsf{Q}_0$, and for $\mathsf{d}=\sum d_i v_i$, $\mathsf{e}=\sum e_i v_i \in \bbN \mathsf{Q}_0$ we write $E_\mathsf{Q}(\mathsf{d},\mathsf{e})$ for the Euler form
\beq
E_\mathsf{Q}(\mathsf{d},\mathsf{e}) \coloneqq \sum_{i=1}^n d_i e_i -\sum_{\a:v_i \to v_j} d_i e_j\,.
\eeq
Assume that $\mathsf{Q}$ is symmetric, that is, for every $i$ and $j$, the number of oriented edges from $v_i$ to $v_j$ is equal to the number of oriented edges from 
$v_j$ to $v_i$. The Euler form is then a symmetric bilinear form.
The motivic DT invariants  $\mathrm{DT}_{\mathsf{d}; i}(\mathsf{Q})$ of $\mathsf{Q}$ are defined by the equality \cites{MR2956038, MR2851153, MR2889742}:
\beq
\sum_{\mathsf{d} \in \bbN^n} \frac{\big(-q^{1/2}\big)^{E_Q(\mathsf{d},\mathsf{d})} \mathsf{x}^\mathsf{d}}{\prod_{i=1}^n (q;q)_{d_i}}
= \prod_{\mathsf{d} \neq 0} \prod_{i \in \bbZ} \prod_{k \geq 0} \l(1-(-1)^i \mathsf{x}^\mathsf{d} q^{-k-(i+1)/2} \r)^{-\mathrm{DT}_{\mathsf{d}; i}(\mathsf{Q})}\,,
\label{eq:DTmot_intro}
\eeq
where $\mathsf{x}^\mathsf{d} =\prod_{i=1}^n x_i^{d_i}$. In other words, the motivic DT invariants are defined by taking the plethystic logarithm of the generating series of Poincar\'e rational functions of the stacks of representations of $\mathsf{Q}$.
The numerical DT invariants $\mathrm{DT}^{\rm num}_{\mathsf{d}}(\mathsf{Q})$
are defined by
\beq
\mathrm{DT}^{\rm num}_{\mathsf{d}}(\mathsf{Q}) \coloneqq \sum_{i \in \bbZ} (-1)^i \mathrm{DT}_{\mathsf{d},i}(\mathsf{Q}) \,.
\label{eq:DTnum_intro}
\eeq
According to Efimov \cite{MR2956038}, the numerical DT invariants 
 $\mathrm{DT}^{\rm num}_{\mathsf{d}}(\mathsf{Q})$ are non-negative integers.

\subsubsection{Open/closed BPS invariants}
\label{sec:bps_intro}

Gromov--Witten invariants define virtual counts of curves and are in general rational numbers, but they are well-known to exhibit hidden integrality properties in terms of underlying BPS counts. The original physics definition due to Gopakumar--Vafa in the classical context of closed Gromov--Witten invariants of 
Calabi--Yau 3-folds  \cites{Gopakumar:1998jq,Gopakumar:1998ii} predicted the form of these counts in terms of degeneracies of BPS particles in four/five dimensions arising from type IIA/M-theory as D2/M2-branes wrapping 2-cycles in the compactification. A long-standing effort has been made on multiple fronts to make the physics definition rigorous either using the associated cohomologies of sheaves \cites{Katz:2006gn,MR3842061}, stable pairs \cite{MR2552254}, and direct symplectic methods \cite{MR3739228}. In this paper, we will consider BPS invariants for genus $0$ Gromov--Witten invariants of Calabi--Yau 4-folds and higher genus open Gromov--Witten invariants of toric Calabi--Yau 3-folds. As an immediate corollary we obtain a new definition of all genus BPS invariants of Looijenga pairs \eqref{eq:Omegad_intro}.

Let $Y(D)=(Y,D=D_1+D_2)$ be a $2$-component Looijenga pair. The corresponding local geometry 
$E_{Y(D)}$ is a non-compact Calabi--Yau 4-fold. Following Greene--Morrison--Plesser \cite[App.~B]{Greene:1993vm} and Klemm--Pandharipande in \cite[Sec.~1.1]{Klemm:2007in}, we define BPS invariants 
$\mathrm{KP}_d(E_{Y(D)})$ in terms of the local
genus $0$ Gromov--Witten invariants
$N^{\rm loc}_{0,d}(Y(D))$ by the formula
\beq
\mathrm{KP}_d(E_{Y(D)}) = \sum_{k \mid d} \frac{\mu(k)}{k^2}N^{\rm loc}_{d/k}(Y(D)) \,.
\label{eq:kp_intro}
\eeq

Let $X$ be a toric Calabi--Yau 3-fold, $L=L_1 \cup \cdots \cup L_s$ a disjoint union of toric Lagrangian branes and $\mathsf{f}$ a choice of framing.
Following \cites{Labastida:2000yw, Labastida:2000zp, Ooguri:1999bv,Marino:2001re}, we define the Labastida--Mari\~no--Ooguri--Vafa (LMOV) generating function of BPS invariants $\Omega_d(X,L,f)(q)\in \bbQ(q^{1/2})$ in terms of the $s$-holed higher genus open Gromov--Witten generating series  $\mathsf{O}_{\beta}(X,L,\mathsf{f})(\hbar)$ by the formula
\bea
\Omega_{\beta}(X,L,\mathsf{f})(q) 
=
[1]_q^{2} \l(\prod_{i=1}^{s} \frac{w_i}{[w_i]_q}\r)  \sum_{k | \beta}  \frac{\mu(k)}{k}
\mathsf{O}_{\beta/k}(X,L,\mathsf{f})(-\ri k \log q) \,,
\label{eq:Omegad_intro_0}
\eea 
where $w_1,\cdots, w_{s}$ are the winding numbers around the Lagrangians 
$L_1, \cdots, L_s$ of the boundary components of a $s$-holed Riemann surface with relative homology class $\beta$, and where
$
 [n]_q:=q^{\frac{n}{2}}-q^{-\frac{n}{2}}
$
are the $q$-integers, defined for all integers $n$.

\subsection{The web of correspondences: geometric motivation}
\label{sec:inter-intro}

\begin{figure}[t]
\includegraphics[scale=1]{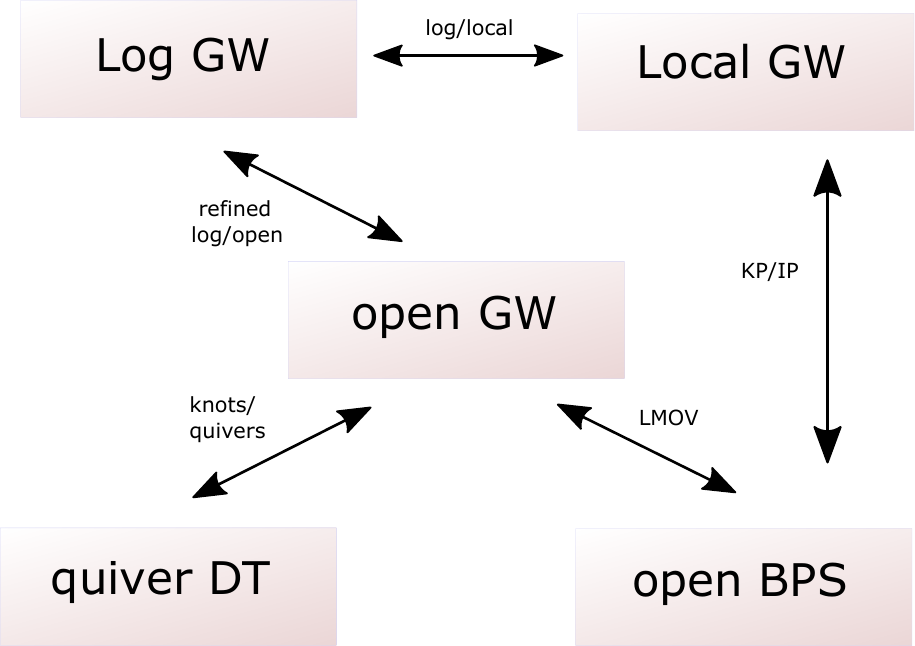}
\caption{Enumerative invariants of $Y(D)$ and their mutual relations.}
\label{fig:web}
\end{figure}

The enumerative theories of the previous section have superficially distant flavours, but they will turn out to be in close and often surprising relation to each other (\cref{fig:web}). We start by explaining the general geometric motivation behind the web of relations below, deferring rigorous statements for the case of Looijenga pairs to \cref{sec:introres}.


\subsubsection{From log to local invariants}
\label{sec:introloglocal}

Let $(Y,D=D_1+\cdots +D_l)$ be a log smooth pair of maximal boundary; unless specified at this stage we do not restrict to $Y$ being a surface, and we neither impose $(Y,D)$ to be log Calabi--Yau nor any positivity conditions on $D_j$. We will say that a curve class $d\in \hhh_2(Y,\bbZ)$ is \emph{\textbf{$D$-convex}} if $d\cdot D_i>0$ for all $i$, and 
for every decomposition $d=[C_1]+\cdots +[C_m]\in\hhh_2(Y,\ZZ)$, with each $C_j$ an effective curve, we have $C_j\cdot D_i\geq0$ for all $i,j$.

We begin by introducing some intermediate geometries built from $Y(D)$: for $\mathsf{m}=1, \dots, l+1$, let
\beq 
Y^{(\mathsf{m})}\coloneqq \tot\left(\bigoplus_{k\geq \mathsf{m}}\shO_Y(-D_{k})\right),
\label{eq:Ym}
\eeq
and $D^{(\mathsf{m})}$ be the preimage $\pi^{-1}\left(\bigcup_{k<\mathsf{m}} D_k\right)$ by the projection $\pi : Y^{(\mathsf{m})}\to Y$.
Note that, by definition, $Y^{(1)}(D^{(1)}) = E_{Y(D)}$ and $Y^{(l+1)}(D^{(l+1)}) = Y(D)$: the  geometries $Y^{(\mathsf{m})}(D^{(\mathsf{m})})$ for $1<\mathsf{m}\leq l$ consist of intermediate setups where a log condition is imposed on  $\{D_k\}_{k<\mathsf{m}}$, and a local one on $\{D_k\}_{k\geq \mathsf{m}}$. 
For $d$ a $D^{(\mathsf{m})}$-convex curve class, we denote by $N^{\rm log}_{0,d}(Y^{(\mathsf{m})}(D^{(\mathsf{m})}))$ a genus 0 maximal tangency log GW invariant of class $d$ of $Y^{(\mathsf{m})}(D^{(\mathsf{m})})$ with a choice of point and $\psi$-class insertions, see \cref{sec:loggwmaxtang}. $D^{(\mathsf{m})}$-convexity ensures that this is well-defined, despite $Y^{(\mathsf{m})}(D^{(\mathsf{m})})$ not being proper for $\mathsf{m}\leq l$.\medskip

%
Assume first that $l=1$, i.e.~that $D$ is a smooth divisor. In \cite{vGGR}, the genus 0 local Gromov--Witten invariants of $E_{Y(D)}$ were related to the genus 0 maximal tangency Gromov--Witten theory of $(Y,D)$ by the {\it stationary log/local correspondence},
\begin{equation}\label{eq:vGGR}
N^{\rm loc}_{0,d}(Y(D))=\frac{(-1)^{d\cdot D-1}}{d\cdot D} N_{0,d}^{\rm log}(Y(D)).
\end{equation}
The argument of \cite{vGGR} is geometric, and it gives a stronger statement at the level of virtual fundamental classes: $E_{Y(D)}$ is degenerated to $Y\times\AA^1$ glued along $D\times\AA^1$ to a line bundle over the projective bundle $\PP(\shO_D\oplus \shO_D(-D))$. This degeneration moves genus~0 stable maps in $E_{Y(D)}$ to genus~0 stable 
maps splitting along both components of the central fibre: the degeneration formula then states that $N^{\rm loc}_{0,d}(Y(D))$ equals the weighted sum over splitting type of the product of invariants associated to each component, and a careful analysis shows that only one term is non-zero, leading to \eqref{eq:vGGR}. In \cite[Conjecture 6.4]{vGGR}, a conjectural cycle-level 
log-local correspondence was also proposed for simple normal crossing pairs: we  propose here a slight variation of its restriction to stationary invariants and anti-canonical $D$ in the following Conjecture.
\begin{conj}[The stationary log/local correspondence for maximal log~CY  pairs]
Let $(Y,D=D_1 +\dots +D_l)$ be a log~smooth log~Calabi--Yau pair of maximal boundary, $d$ a $D$-convex curve class, and $1 \leq \mathsf{n} < \mathsf{m} \leq l+1$. Then,
\beq
N^{\rm log}_{0,d}(Y^{(\mathsf{m})}(D^{(\mathsf{m})}))= \l(\prod_{i=\mathsf{n}}^{\mathsf{m}-1} (-1)^{d \cdot D_i +1} d \cdot D_i\r) N^{\rm log}_{0,d}(Y^{(\mathsf{n})}(D^{(\mathsf{n})})).
\label{eq:loglocgen1}
\eeq 
In particular, when $(\mathsf{n},\mathsf{m})=(1,l+1)$,
\beq
N^{\rm log}_{0,d}(Y(D))= \l(\prod_{i=1}^l (-1)^{d \cdot D_i +1} d \cdot D_i\r) N^{\rm loc}_{0,d}(Y(D)).
\label{eq:loglocgen2}
\eeq
\label{conj:loglocgen}
\end{conj}
When all $D_j$ are nef and $(\mathsf{n},\mathsf{m})=(1,l+1)$, this gives the numerical version of \cite[Conjecture 6.4]{vGGR} for point insertions and anti-canonical $D$. When $\mathsf{m}-\mathsf{n}=1$, \eqref{eq:loglocgen1} is an extension of the main result of \cite{vGGR} to the non-compact case. 
\medskip

The extent to which the argument of \cite{vGGR} generalises to the case of simple normal crossings pairs of \cref{conj:loglocgen} is a somewhat thorny issue. In particular, the cycle-level conjecture of \cite[Conjecture 6.4]{vGGR} is known to fail in the non-stationary sector
for general $l$, as
recently observed in a non-log~Calabi--Yau example
by Nabijou--Ranganathan \cite{NR_new}.
At the same time, there is a non-trivial body of evidence that a generalisation of the stationary sector equality \eqref{eq:vGGR} (i.e., with descendent point insertions only) might hold for simple normal crossings log~Calabi--Yau pairs $Y(D)$ -- see \cite{BBvG1} for a proof for toric orbifold pairs.
%
%
It  is therefore an open question to find the exact boundaries of validity of the stationary log-local correspondence, and in this paper we chart a conceptual pathway to delineate them for the (special, but central) case of log~Calabi--Yau pairs of \cref{conj:loglocgen}, as follows.

At a geometric level, the degeneration of \cite{vGGR} can be generalised to a birational modification of one where
the generic fibre is $E_{Y(D)}$,
and the special fibre is obtained by gluing, for each $j=1,\dots,l$, $Y\times(\A^1)^l$ along $D_j\times(\A^1)^l$ to a rank $l$ vector bundle over $\PP(\shO_{D_j}\oplus \shO_{D_j}(-D_j))$.
After an (explicit) birational modification this gives a log smooth family: we describe the details of the degeneration for the case of surfaces in \cref{sec:logloc}.
When $l>1$, instead of the degeneration formula the decomposition formula \cite{abramovich2017decomposition} applies, expressing $N^{\rm loc}_{0,d}(Y(D))$ as a weighted sum of terms, indexed by tropical curves $h \colon \Gamma \to \Delta$,
where $\Delta$ is the dual intersection complex of the central fibre:
\beq N^{\rm loc}_{0,d}(Y(D))
=\sum_{h \colon \Gamma \rightarrow \Delta} 
\frac{m_h}{|\Aut(h)|} 
N^{{\rm loc}, h}_{0,d}(Y(D))\,.
\label{eq:nlocdec}
\eeq
The geometric picture above, and the ensuing decomposition formula \eqref{eq:nlocdec}, provides a rather  general and geometrically motivated blueprint to measure the deviation, or lack thereof, of the local invariants from their expected relation to maximal tangency log invariants in \eqref{eq:loglocgen2}. As a proof-of-concept step, and as we shall describe in detail in \cref{sec:logloc}, in this paper we show how this framework bears fruit in the context of Looijenga pairs\footnote{It is an intriguing question, and one well beyond the scope of this paper, to test how this philosophy generalises to log Calabi--Yau varieties of any dimension, and to revisit the non-log Calabi--Yau, non-stationary negative result of \cite{NR_new} in this light.}: here  correction terms indexed by non-maximal tangency tropical curves turn out, remarkably, to {\it all} individually vanish, whilst the maximal tangency tropical contribution exactly returns the r.h.s. of \eqref{eq:loglocgen2}.

\subsubsection{From log to open invariants}
\label{sec:intrologopen}
Let $Y(D)$ be a log Calabi--Yau surface. 
By \eqref{eq:Ym}, the complement $Y^{(l)}\setminus D^{(l)}$ is isomorphic to the total space of $\cO(-D_l) \to Y\setminus \left( D_1\cup \cdots \cup D_{l-1} \right)$; since $D$ is anti-canonical, this is a non-compact Calabi--Yau threefold. 
We propose that the log invariants $N^{\log}_{0,d}(Y(D))$ can be precisely related to open Gromov--Witten invariants of $Y^{(l)}\setminus D^{(l)}$
with boundary in fixed disjoint Lagrangians $L_k$, $k<l$ near the divisor $D^{(l)}$. 
These Lagrangians should have a specific structure as described in \cite[Section 7]{MR2386535}, namely they should be fibred over Lagrangians $L_k'$ in $\pi^{-1}\left(D_k\right)$ 
with fibres Lagrangians in the normal bundle $\big(N_{\pi^{-1}(D_k)/Y^{(\mathsf{m})}}\big){\big \vert}_{L_k'}$. 
Denoting $L\coloneqq  \cup_{k<l} L_k$ and
 $Y^{\rm op}(D) \coloneqq (Y^{(l)}\setminus D^{(l)}, L)$,  there is a natural isomorphism
$\iota: \hhh^{\rm rel}_2(Y^{\rm op}(D), \bbZ) \rightarrow  \hhh_2(Y, \bbZ)$
induced by the embedding $Y^{(l)} \setminus D^{(l)} \hookrightarrow Y^{(l)}$ and the identification of winding degrees along $L_k$ with contact orders along $D_k$ 
(see \cref{prop:iota} for details).\\

\begin{figure}[h]
    \includegraphics{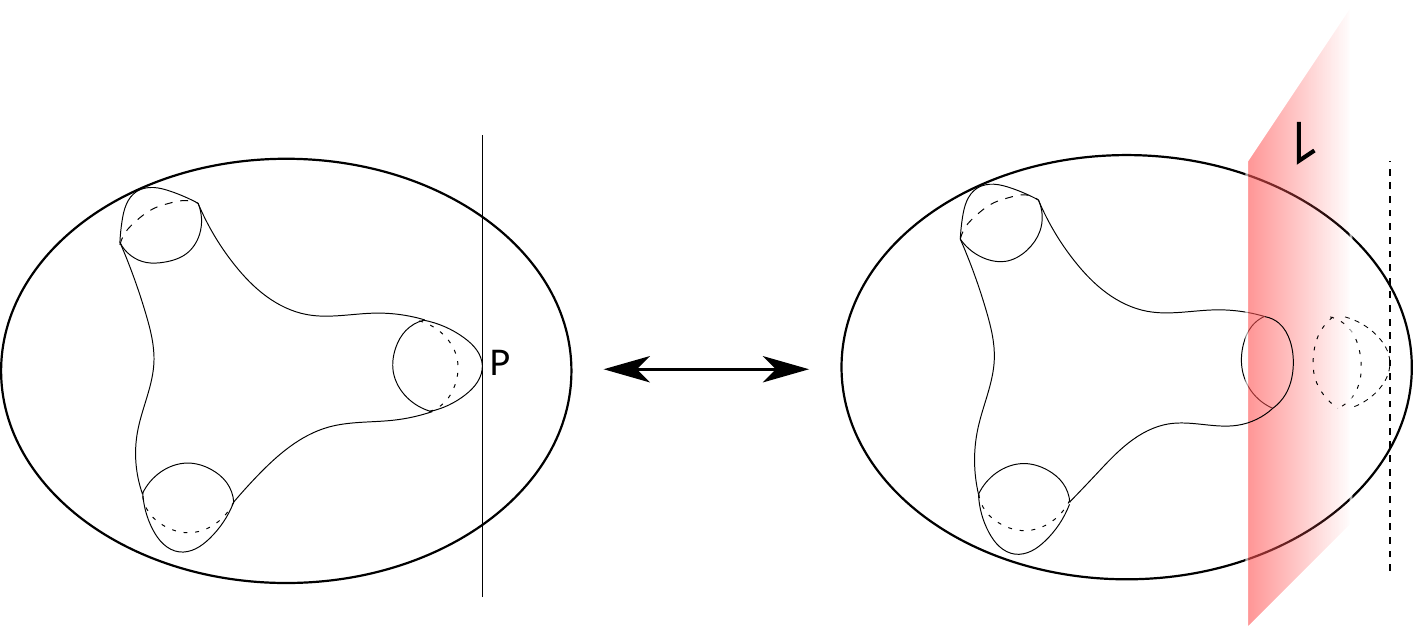}
\caption{Exchanging log and open conditions.} 
\label{fig:logopen}
\end{figure}

Suppose 
now that there is a well-posed definition\footnote{An example of this situation (see  Construction~\ref{constr:YopD}) is when up to deformation both $Y$ and the divisors $D_k$ ($k<l$) are toric, implying  that
$Y^{\rm op}(D)$ is a toric Calabi--Yau threefold geometry equipped with framed toric Lagrangians $L_{k}$: in this case the open GW invariants were introduced in \cref{sec:opengw_intro}.} of genus zero open GW counts $O_{0;d}(Y^{\rm op}(D))$ as in \cite{Solomon:2016vey,MR3539371}.  In such a scenario, we expect a close relationship between these
and the log invariant $N^{\log}_{0,d}(Y(D))$. 


\begin{conj}[Log-open correspondence for surfaces]\label{conj:ocd}
Let $Y(D)$ be a log~Calabi--Yau surface with maximal boundary and $d$ a $D$-convex curve class. Then,

\beq \label{eq_log_open_duality}
O_{0;\iota^{-1}(d)}(Y^{\rm op}(D)) = \left( \prod_{k=1}^l \dfrac{(-1)^{d\cdot D_k-1}}{d\cdot D_k} \right) N^{\log}_{0,d}(Y(D)). 
\eeq
\end{conj}

There is an intuitive symplectic heuristics behind \cref{conj:ocd}: removing a tubular neighbourhood of $D^{(l)}$ turns pseudo-holomorphic log curves in $Y^{(l)}$ with prescribed tangencies along $D^{(l)}$ into pseudo-holomorphic open Riemann surfaces with boundaries in $L$, with winding numbers determined by the tangencies (see \cref{fig:logopen}). 
The relative factor 
$\prod_{k<l} (-1)^{d\cdot D_k-1}(d\cdot D_k)^{-1}$ at the level of GW counts
in \cref{conj:ocd} can be understood by looking at the simplest example where 
$Y=\PP^1 \times \A^1$, $D_1= \{0\} \times \A^1$ and $D_2=\{\infty\} \times \A^1$, where $0, \infty \in \PP^1$.
For the curve class 
$d$ times the class of $\PP^1$ we have $N_{0,d}^{\log}((D))=1$, as there exists a unique degree $d$ cover of $\PP^1$ fully ramified over two points, and the order $d$ automorphism group of this cover is killed by the point condition. 
By \cref{conj:loglocgen}, and in particular \eqref{eq:loglocgen1} with $\mathsf{m}=2$, we deduce that 
$N_{0,d}^{\log}(Y^{(2)}(D^{(2)}))=\frac{(-1)^{d-1}}{d}$; 
on the other hand, the open geometry $Y^{\rm op}(D)$ is 
$\C^3$ with a singular Harvey--Lawson Lagrangian $L$ of framing zero (see Construction~\ref{constr:YopD}): the degree $d$ multicovers of the unique embedded disk  \cite[Theorem~7.2]{Katz:2001vm} contribute $O_{0;\iota^{-1}(d)}(Y^{\rm op}(D))= \frac{1}{d^2}$,  
from which the relative factor in \eqref{eq_log_open_duality} is recovered.
\medskip

Much as in 
\cref{conj:loglocgen}, the invariants in \cref{conj:ocd} live in different dimensions: \eqref{eq_log_open_duality} relates log invariants of the log~CY surface $Y(D)$ to open invariants of special Lagrangians in a Calabi--Yau threefold. Note that combining \cref{conj:loglocgen,conj:ocd} further gives a surprising conjectural relation
\beq
O_{0;\iota^{-1}(d)}(Y^{\rm op}(D)) = N^{\rm loc}_{0,d}(Y(D)),
\label{eq:openlogloc}
\eeq
which equates the GW invariants of the CY3 open geometry $Y^{\rm op}(D)$ with the local GW invariants of the CY-($l+2$) variety $E_{Y(D)}$.\footnote{The relation \eqref{eq:openlogloc} is in tune with physics expectations from type IIA string theory compactification on $\bbR^{1,1} \times X$
where $X$ is a Calabi--Yau fourfold: the low energy effective theory is a $\cN=(2,2)$ QFT,  whose effective holomorphic superpotential is computed by the genus-0 Gromov--Witten invariants of $X$. Now precisely the same type of holomorphic F-terms can be engineered by considering D4-branes wrapping special Lagrangians on a Calabi--Yau 3-fold  \cite{Ooguri:1999bv}: the superpotential here is a generating function of holomorphic disk counts with boundary on the Lagrangian three-cycle. It was suggested by Mayr \cite{Mayr:2001xk} (see also \cite{Aganagic:2009jq}) that there exist cases where 2d superpotentials can be engineered in both ways, resulting in an identity between local genus 0 invariants of CY 4-folds and disk invariants of CY 3-folds: the equality in \eqref{eq:openlogloc} asserts just that.}
\medskip

We also expect a precise uplift of this picture to higher genus invariants.
For a single irreducible divisor, an all-genus version of the log-local correspondence of \cite{vGGR} was described in \cite[Thm~ 1.1-1.2]{Bousseau:2020ckw}. Its generalisation to a log-open correspondence in higher genus for a completely general pair is likely to take an unwieldy form, but we expect it to be particularly simple for a maximal boundary log Calabi--Yau surface.
Indeed, in the degeneration to the normal cone along $D_l$, only multiple covers of a 
$\PP^1$-fiber in $\PP(\cO_{D_l}\oplus \cO_{D_l}(-D_l))$ will contribute. The resulting combination of the multiplicity 
$d \cdot D_l$ in the degeneration formula with the higher genus multiple cover contribution 
$\frac{(-1)^{d \cdot D_l +1}}{(d \cdot D_l)[d \cdot D_l]_q}$ leads us to predict a precise, and tantalisingly simple $q$-analogue of \cref{conj:ocd}.

\begin{conj}[The all-genus log-open correspondence for surfaces]\label{conj:refocd}
Let $Y(D)$ be a log~CY surface with maximal boundary and $d$ a $D$-convex curve class. With notation as in \cref{sec:loggw_intro,sec:opengw_intro}, we have
\beq
 \mathsf{O}_{\iota^{-1}(d)}(Y^{\rm op}(D))(-\ri \log q)
 = [1]_q^{l-2} \,
  \frac{(-1)^{d \cdot D_l+1}}{[d \cdot D_l]_q} \prod_{k=1}^{l-1} \frac{(-1)^{d \cdot D_k+1}}{d \cdot D_k} \,
 \mathsf{N}_{d}^{\rm log}(Y(D))(-\ri \log q)\,.
 \label{eq:refocd}
\eeq
\end{conj}
The factor $[1]_q^{l-2}$ corresponds to the relative normalisation of the higher genus generating functions in \eqref{eq:GWexp_intro} and \eqref{eq:openfreeen_intro}. The allusive hints of this Section will be put on a rigorous footing in \cref{sec:logopen_intro_res}.

\subsubsection{Quivers and BPS invariants}
\label{sec:introquiverbps}

Given $Y(D=D_1+D_2)$ a 2-component Looijenga pair, the virtual count of curves in the non-compact Calabi-Yau 4-fold 
$E_{Y(D)}$ \cite{Klemm:2007in} is expected to be expressible in terms of sheaf counting 
\cite{MR3861701, cao2019stable}. More precisely, it is expected that that the BPS invariants of $E_{Y(D)}$ are extracted from a ${\rm DT}_4$ virtual fundamental class associated to the moduli space of one-dimensional coherent sheaves on $E_{Y(D)}$. As coherent sheaves are often very closely related to modules over quivers, it might be tempting to ask if 
curve counting in $E_{Y(D)}$ (and, via the arguments of the previous section, the log/open GW theory of $Y(D)$) can be described in terms of some quiver DT theory.
\medskip

This is more than a suggestive speculation. Consider for example $Y=\bbP^2$ and $D=D_1+D_2$ the union of a line $D_1$ and a conic $D_2$, so that $E_{Y(D)}$ is the total space of $\cO_{\bbP^2}(-1) \oplus \cO_{\bbP^2}(-2)$. Let $\cM_{\bbP^1}^{\rm Higgs}(d,n)$ be the moduli space  of rank-$d$, degree-$n$ $\cO(1)$-twisted Higgs bundles $\cO_{\PP^1}^{\oplus d} \rightarrow \cO_{\PP^1}^{\oplus d} \otimes \cO_{\PP^1}(1)$ on $\PP^1$. The total space of $\cO_{\pone}(1)$ is the complement of a point in $\PP^2$, and as $\PP^1$ has normal bundle $\cO(1)$ in $\PP^2$, $\cM_{\bbP^1}^{\rm Higgs}(d,n)$ sits as an open part of the moduli space of one-dimensional coherent sheaves on $E_{Y(D)}$. At the same time, as $\cO(1)$ has two sections on 
$\PP^1$, $\cM_{\bbP^1}^{\rm Higgs}(d,n)$ is isomorphic to the moduli space of representations of the quiver 
with one vertex and two loops. Strikingly, we remark here that this is reflected into a completely unexpected identity for the corresponding invariants:  the Klemm--Pandharipande BPS invariants of $E_{Y(D)}$ computed in \cite[Sec.~3.2]{Klemm:2007in} simultaneously coincide (up to sign) with the DT invariants of the 2-loop quiver computed in \cite[Thm.~4.2]{MR2889742}, as well as with the top Betti numbers\footnote{The degree-independence of these Betti numbers, at least for $(d,n)=1$, is explained in \cite[Sec.~5]{Rayan:2018fpb}. 
} $\mathfrak{B}^{\rm Higgs}_d(\bbP^1) \coloneqq \dim \hhh^{\rm top}(\cM_{\bbP^1}^{\rm Higgs}(d,n),\bbQ)$  of the moduli spaces of $\cO(1)$-twisted Higgs bundles on the line considered in \cite[Sec.~5]{Rayan:2018fpb}:
\bea
& 
\Big|\mathrm{KP}_d\big(\cO_{\bbP^2}(-1) \oplus \cO_{\bbP^2}(-2)\big)\Big|  = \mathfrak{B}^{\rm Higgs}_d(\bbP^1)= \mathrm{DT}^{\rm num}_d(\hbox{2-loop quiver})& \nn \\
& = 
(1, 1, 1, 2, 5, 13, 35, 100, 300, 925, 2915, 9386, \dots )_d. &
\label{eq:surprise}
\eea

From a sheafy point of view, this raises the question how the definition of Calabi-Yau 4-fold invariants from the moduli space of coherent sheaves \cite{MR3861701, cao2019stable} interacts with the quiver description, and whether such a startling coincidence is an isolated example -- or not. 
\medskip

An upshot of \cref{conj:loglocgen,conj:ocd} is a  surprising Gromov--Witten theoretic take on this question: for $l=2$ and when $Y^{\rm op}(D)$ is an open geometry given by toric Lagrangians in a toric CY3, the quiver can be reconstructed systematically from the geometry of $Y(D)$ via a version of the ``branes--quivers'' correspondence introduced in \cite{Kucharski:2017ogk,Panfil:2018sis,Ekholm:2018eee,Ekholm:2019lmb}.
According to the open GW/quiver dictionary of \cite{Ekholm:2019lmb},
the quiver nodes are identified with basic (in the sense of \cite{Ekholm:2018eee,Ekholm:2019lmb}) embedded holomorphic disks with boundary on $L$, edges and self-edges correspond to linking and self-linking numbers of the latter, and the DT invariants of the quiver return (up to signs) the genus zero LMOV count of holomorphic disks obtained as ``boundstates'' of the basic ones \cite[Section~4]{Ekholm:2018eee}.

Now, by the $q \to 1$ limit of \eqref{eq:Omegad_intro_0}, the genus zero LMOV and GW invariants of $Y^{\rm op}(D)$ are related to each other by the {\it same} BPS change of variables relating KP invariants and local GW invariants of $E_{Y(D)}$ in \eqref{eq:kp_intro}. Then a direct consequence of the conjectural open=local GW equality \eqref{eq:openlogloc} is that the KP invariants of the local CY4-fold $E_{Y(D)}$ coincide with the LMOV invariants of the open CY3 geometry $Y^{\rm op}(D)$-- which by the branes-quivers correspondence above are in turn DT invariants of a symmetric quiver!  In particular, for the example above of $Y=\bbP^2$ and $D=D_1+D_2$ the union of a line and a conic, we shall find the open geometry $Y^{\rm op}(D)$ to be three-dimensional affine space with a single toric Lagrangian at framing one (see Construction~\ref{constr:YopD}) -- and as expected, in this case the quiver construction in  \cite[Sec.~5.1]{Panfil:2018sis} returns exactly the quiver with one loop and two arrows we had found in \eqref{eq:surprise}. In general, this connection leads to some non-trivial implications for the Gopakumar--Vafa/Donaldson--Thomas theory of CY4 local surfaces from log Gromov--Witten theory, which we describe precisely in \cref{sec:bpsloglocal,sec:cmt}.

\subsection{The web of correspondences: results}\label{sec:introres}
In order to state our results, we introduce some notions of positivity for Looijenga pairs. A Looijenga pair $Y(D)=(Y,D=D_1+\cdots+D_l)$ is \emph{nef}
if each irreducible component $D_i$ of $D$ is smooth and nef: note that the condition that the components $D_i$ are smooth implies in particular that $l \geq 2$, and nefness entails that a generic stable map to $Y$ is $D$-convex, which implies that the corresponding 
local Gromov--Witten invariants are well-defined.

A nef Looijenga pair $Y(D)$ is \emph{tame} if either $l>2$ or 
$D_i^2>0$ for all $i$, and \emph{quasi-tame} if the associated local geometry 
$E_{Y(D)}$ is deformation equivalent to the local geometry 
$E_{Y'(D')}$ associated to a tame Looijenga pair 
$Y'(D')$: we explain the relevance of these two properties in \cref{sec:loggw_intro_res}. As we will show in \cref{sec:logCYsurf},
there are 18 smooth deformation types of nef Looijenga pairs in total, 11 of which are tame and 15 are quasi-tame. In particular, a nef Looijenga pair $Y(D)$ is uniquely determined by $Y$ and the self-intersection numbers $D_i^2$, and we sometimes use the notation $Y(D_1^2,\dots,D_l^2)$ for $Y(D)$,  
see \cref{tab:classif}. We state our results in a slightly discursive form below, including pointers to their precise versions in the main body of the text.

\subsubsection{The stationary log-local correspondence}
\label{sec:loggw_intro_res}

Our first result establishes the stationary log/local correspondence of  \cref{conj:loglocgen} in the form given by \eqref{eq:loglocgen2}.

\begin{thm}[=\cref{thm_log_local,lem:localgw,thm:dP5,thm:dP33comp,thm:F04comp}] 
\label{thm:log_local_intro}
For every nef Looijenga pair $Y(D)$, the genus $0$ log invariants $N_{0,d}^{\rm log}(Y(D))$
and the genus $0$ local invariants $N^{\rm loc}_{0,d}(Y(D))$ are related by
\begin{equation} N^{\rm loc}_{0,d}(Y(D))= \left( \prod_{j=1}^l \frac{(-1)^{d \cdot D_j -1}}{
d \cdot D_j}\right) 
N_{0,d}^{\rm log}(Y(D)) \,.
\label{eq:loglocthm}
\end{equation}
Moreover, we provide a closed-form solution to the calculation of both sets of invariants in \eqref{eq:loglocthm}.
\end{thm}
As explained in \cref{sec:introloglocal}, the key idea to prove \cref{thm:log_local_intro} is by a degeneration argument, illustrated in \cref{sec:logloc} for $l=2$: we follow the general strategy  of \cite{vGGR} to deduce 
log-local relations from a degeneration to the normal cone, and we solve in our case of interest the difficulties of the normal-crossings situation through a detailed
study of the tropical curves contributing in the decomposition formula of Abramovich--Chen--Gross--Siebert \cite{abramovich2017decomposition}
for log Gromov--Witten invariants. For $l>2$, and more generally when $Y(D)$ is tame, it turns out to be more convenient to structure the proof so that an uplift to the all-genus story, absent in other approaches, is immediate.
The notion of tameness is first shown to be synonymous of finite scattering, and for tame pairs we compute closed-form solutions for the log Gromov--Witten invariants using tropical geometry, more precisely two-dimensional scattering diagrams \cites{GPS,GHKlog,Gro11,Man19}. The statement of the Theorem for tame cases follows by subsequently comparing with a closed-form solution of the local theory via  Givental-style mirror theorems: the proof  follows from a general statement valid for local invariants of toric Fano varieties in any dimension twisted by a sum of concave line bundles (\cref{lem:localgw}), and the notion of tameness is shown to coincide here with the vanishing of quantum corrections to the mirror map. For non-quasi-tame cases, we use a blow-up formula which allows to restrict to the case of highest Picard number;
the proof of the equality \eqref{eq:loglocthm} in this case, in \cref{thm:dP5}, requires a highly non-trivial mirror map calculation.

\subsubsection{The all-genus log-open correspondence}
\label{sec:logopen_intro_res}

A notable property of the scattering approach to \cref{thm:log_local_intro} for $l>2$ (and, in general, for tame Looijenga pairs) is that it can be bootstrapped to obtain all-genus results for the log invariants through the $q$-deformed version of the two-dimensional scattering diagrams of \cites{GPS,GHKlog,Gro11,Man19} and 
the general connection between higher genus log invariants of surfaces with $\lambda_g$-insertion and $q$-refined tropical geometry
studied in \cites{MR3904449,bousseau2018quantum}. This is key to establishing the following version of \cref{conj:ocd,conj:refocd}.

\begin{thm}[=\cref{prop:dp311,thm:log_dp3_0_0_0,thm:logf0_0000,thm:logopen}] 
For every quasi-tame Looijenga pair $Y(D)$ distinct from 
$\delp_3(0,0,0)$, there exists a triple $Y^{\rm op}(D)=(X,L,\mathsf{f})$, geometrically related to $Y(D)$ by Construction~\ref{constr:YopD}, where 
$X$ is a semi-projective toric Calabi--Yau 3-fold, 
$L=L_1 \cup \cdots \cup L_{l-1}$ is a disjoint union of 
$l-1$ toric Lagrangians in $X$, $\mathsf{f}$ is a framing for $L$, and an 
isomorphism $\iota \colon \hhh_2(X,L,\Z) \stackrel{\sim}{\rightarrow} \hhh_2(Y,\Z)$ such that
\beq 
O_{0;\iota^{-1}(d)}(Y^{\rm op}(D)) = N_{0,d}^{\rm loc}(Y(D)) =  \prod_{i=1}^{l}  \frac{(-1)^{d \cdot D_i+1}}{d \cdot D_i}  N_{0,d}^{\rm log}(Y(D))\,.
\label{eq:logopen0_intro}
\eeq
Furthemore, if $Y(D)$ is tame, 
\beq \label{eq:logopen_intro}
 \mathsf{O}_{\iota^{-1}(d)}(Y^{\rm op}(D))(-\ri \log q)
 = [1]_q^{l-2} \,
  \frac{(-1)^{d \cdot D_l+1}}{[d \cdot D_l]_q} \, \prod_{i=1}^{l-1} \, \frac{(-1)^{d \cdot D_i+1}}{d \cdot D_i} \,
 \mathsf{N}_{d}^{\rm log}(Y(D))(-\ri \log q)\,.
\eeq
Moreover, we provide a closed-form solution to the calculation of the invariants in \eqref{eq:logopen0_intro}--\eqref{eq:logopen_intro}.
\label{thm:logopen_intro}
\end{thm}
The open geometry $Y^{\rm op}(D)$ is constructed following the ideas of \cref{sec:intrologopen} -- see \cref{sec:logopenpr} for full details. Key to the proof of \cref{thm:logopen_intro} is the fact that quasi-tame Looijenga pairs can always be deformed to pairs for which the both surface $Y$ and the divisors $D_i$ with $i<l$ are toric: as we shall explain in \cref{sec:logopenpr}, the corresponding open geometry $Y^{\rm op}(D)$ is given by suitable Aganagic--Vafa (singular Harvey--Lawson) Lagrangian branes in a toric Calabi--Yau threefold, whose open Gromov--Witten theory can be compactly encoded through the topological vertex.\footnote{A conceptual explanation for the exclusion of $\delp_3(0,0,0)$ from the statement of \cref{thm:logopen_intro} is given by the notion of Property~O, which we introduce in  \cref{def:propO}.} \cref{conj:refocd} then predicts a completely unexpected relation between the $q$-scattering and topological vertex formalisms, which \cref{thm:logopen_intro} establishes for tame pairs. The combinatorics underlying the resulting comparison of invariants is in general extremely non-trivial: for $l=2$, it can be shown to be equivalent to Jackson's $q$-analogue of the Pfaff--Saalsch\"utz summation for  the $_3 \phi_2$ generalised $q$-hypergeometric function.

We furthermore conjecture that the higher genus log-open correspondence of \cref{thm:logopen_intro} extends to all quasi-tame pairs. The scattering diagrams become substantially more complicated in the non-tame cases, and \eqref{eq:logopen0_intro} translates into an intricate novel set of $q$-binomial identities:
see \cref{conj:dp1binom} for explicit examples.\footnote{After the first version of this paper appeared on the arXiv, we received a combinatorial proof of \cref{conj:dp1binom} from C.~Krattenthaler \cite{CKpriv}.} The log-local correspondence of \cref{thm:log_local_intro} establishes their limit for $q\to 1$.

\subsubsection{BPS invariants and quiver DT invariants}\label{sec:bpsquiver}

As anticipated in \cref{sec:introquiverbps}, the log/open correspondence of \cref{thm:logopen_intro} can be leveraged to produce a novel correspondence between log/local Gromov--Witten invariants and quiver DT theory.
%

\begin{thm}[=\cref{thm:kpdt}]
Let $Y(D)=(Y,D_1+D_2)$ be a 2-component quasi-tame Looijenga pair. Then there exists a symmetric quiver $\mathsf{Q}(Y(D))$ with $\chi(Y)-1$ vertices and a lattice isomorphism $\kappa: \bbZ (\mathsf{Q}(Y(D)))_0 \stackrel{\sim}{\rightarrow}  \hhh_2(Y,\bbZ) $ such that
\beq
\mathrm{DT}_{d}^{\rm num}(\mathsf{Q}(Y(D))) =
\Big|\mathrm{KP}_{\kappa(d)}(E_{Y(D)})+ \sum_i \a_i \delta_{d,v_i}\Big| \,,
\label{eq:kpdt_intro}
\eeq
with $\a_i \in \{-1,0,1\}$. In particular, $\mathrm{KP}_{d}(E_{Y(D)})\in \bbZ$.
\label{thm:kpdt_intro}
\end{thm}

A symplectic proof of the integrality of genus $0$ BPS invariants for projective Calabi--Yau 4-folds, although likely adaptable to the non-compact setting, was given by Ionel--Parker in \cite{MR3739228}. In \cref{thm:kpdt_intro}, the integrality for the local Calabi--Yau 4-folds $E_{Y(D)}$ follows from the identification of the BPS invariants with DT invariants of a symmetric quiver\footnote{The equality modulo the integral shift by $\sum_i \a_i \delta_{d,v_i}$ in \eqref{eq:kpdt_intro} can be traded to an actual equality of absolute values at the price of considering a larger disconnected quiver $\widetilde{\mathsf{Q}}$, and a corresponding epimorphism $\widetilde{\kappa}:\bbZ(\widetilde{\mathsf{Q}}(Y(D)))_0 \to \hhh_2(Y,\bbZ)$; see \cite{Panfil:2018faz}.}. We construct the symmetric quiver $\mathsf{Q}(Y(D))$ by combining the 
log-open correspondence given by \cref{thm:logopen_intro} with a correspondence previously established by Panfil--Sulkowski between toric Calabi--Yau 3-folds with ``strip geometries'' and symmetric quivers \cites{Panfil:2018faz,Kimura:2020qns}.

\cref{thm:logopen_intro} 
associates to a Looijenga pair $Y(D)$ satisfying Property~O 
the toric Calabi--Yau 3-fold geometry $Y^{\rm op}(D)$.
Denote 
$\Omega_d(Y(D))(q) \coloneqq \Omega_{\iota^{-1}(d)}(Y^{\rm op}(D))(q)$ the open BPS 
invariants defined in \eqref{eq:Omegad_intro_0}. In general, for any Looijenga pair we can define
\beq
\Omega_{d}(Y(D))(q) 
 \coloneqq 
 [1]^2_q
 \left( \prod_{i=1}^{l} \frac{ 1 }{[ d \cdot D_i]_{q}} \right)
 \sum_{k | d}
\frac{(-1)^{ d/k \cdot D + l}\mu(k)}{[k]_{q}^{2-l} \,  k^{2-l} \,} \,
\mathsf{N}_{d/k}^{\rm log}(Y(D))(-\ri k \log q)\, .
\label{eq:Omegad_intro}
\eeq
When $Y(D)$ is tame and satisfies Property O, the equivalence of the definitions \eqref{eq:Omegad_intro_0} and \eqref{eq:Omegad_intro} is a rephrasing of the log-open correspondence of \cref{thm:logopen_intro} at the level of BPS invariants.

A priori, 
$\Omega_d(Y(D))(q) \in \bbQ(q^{1/2})$.
By a direct arithmetic argument, we prove the following integrality result, which in particular 
establishes the existence of an integral BPS structure underlying the 
higher genus log Gromov--Witten theory of $Y(D)$.

\begin{thm}[=\cref{thm:openbps}]
Let $Y(D)$ be a quasi-tame Looijenga pair. Then $\Omega_{d}(Y(D))(q) \in  q^{-\frac{\mathsf{g}_{Y(D)}(d)}{2}} \bbZ[q] $
for an integral quadratic polynomial $\mathsf{g}_{Y(D)}(d)$.
\label{thm:openbps_intro}
\end{thm}

\subsubsection{Orbifolds}

In the present paper, we mainly focus on the study of the finitely many deformation families
of nef Looijenga pairs $(Y,D)$ with $Y$ smooth. Nevertheless, most of our techniques and results should extend to the more general setting where we allow $Y$ to have orbifold singularities 
at the intersection of the divisors: the log Gromov--Witten theory is then well-defined since $Y(D)$ is log smooth, and the local Gromov--Witten theory makes sense by viewing $Y$ and $E_{Y(D)}$ as smooth Deligne--Mumford stacks. There are infinitely many examples of nef/tame/quasi-tame Looijenga pairs in the orbifold sense. Deferring a treatment of more general examples to the companion note \cite{Bousseau:2020ryp}, we content ourselves here to show in \cref{sec:orbi} that the log-local, log-open and Gromov--Witten/quiver correspondences still hold for the infinite family of examples
obtained by taking $Y=\bbP_{(1,1,n)}$, the weighted projective plane with weights $(1,1,n)$, and $D=D_1+D_2$ with 
$D_1$ a line passing through the orbifold point and $D_2$ a smooth member of the linear system given by the sum of the two other toric divisors.

\subsection{The web of correspondences: implications} 
\label{sec:intro_implications}

The results of the previous section subsume and were motivated by several disconnected strands of development in the study of the enumerative invariants in \cref{sec:loggw_intro,sec:locgw_intro,sec:opengw_intro,sec:quiverdt_intro,sec:bps_intro}. We briefly describe here how they relate to and impact ongoing progress in some allied contexts.

\subsubsection{BPS structures in log/local GW theory}
\label{sec:bpsloglocal}
The relation of log GW invariants to BPS invariants in \cref{thm:openbps_intro} echoes very similar\footnote{A non-trivial difference is that here the log Gromov--Witten invariants are {\it not} interpreted as BPS invariants themselves, unlike in \cites{Bou18,bousseau2018quantum}, but rather are related to them via \eqref{eq:Omegad_intro}.}
statements relating log GW theory to DT and LMOV invariants in \cites{Bou18,bousseau2018quantum}, and in particular it partly demystifies the interpretation of log GW partition functions as related to some putative open curve counting theory on a Calabi--Yau 3-fold in \cite[\S 9]{bousseau2018quantum} by realising the open BPS count in terms of actual, explicit special Lagrangians in a toric Calabi--Yau threefold. Aside from its conceptual appeal, its power is revealed by some of its immediate consequences: the Klemm--Pandharipande conjectural integrality \cite[Conjecture~0]{Klemm:2007in} for local CY4 surfaces follows as a zero-effort corollary of the log-open correspondence of \cref{thm:logopen_intro} by constructing the associated quiver in \cref{thm:kpdt_intro}, identifying the KP invariants of the local surface with its DT invariants, and applying Efimov's theorem \cite{MR2956038}.

We note that this chain of connections opens the way to a proof of the Calabi-Yau 4-fold Gromov-Witten/Donaldson-Thomas correspondence \cite{MR3861701, cao2019stable} which is an open conjecture even for the simplest local surfaces.
The analysis of the underlying integrality of the $q$-scattering calculation in \cref{thm:openbps_intro} furthermore gives, in the limit $q\to 1$, an algebro-geometric version of symplectic results of Ionel--Parker \cite{MR3739228} for Calabi--Yau vector bundles on toric surfaces; and away from this limit, it provides a  refined integrality statement whose enumerative salience for the local theory is hitherto unknown, and worthy of further study: see \cref{sec:cmt}.

\subsubsection{The general log-open correspondence for surfaces}
Throughout the heuristic description of the motivation for \cref{conj:loglocgen,conj:ocd,conj:refocd}, we have been mindful not to impose any nefness condition on the divisors $D_i$: the only request we made was that the genus zero obstruction theory for the moduli problem of the local theory was encoded by a genuine obstruction bundle over the untwisted moduli space. This was taken into account by the condition of $D$-convexity for the stable maps: restricting to $D$-convex maps\footnote{Note that for $Y$ Fano, this contains the interior of a full-dimensional sub-cone of $\mathrm{NE}(Y)$.}  widens the horizon of the log/local correspondence of \cite{vGGR} to a vast spectrum of cases which were not accounted for in previous studies of the correspondence. And indeed, inspection of examples where $Y$ is a blow-up of the plane and $D_1$ an exceptional curve reveals that  \cref{conj:loglocgen,conj:ocd,conj:refocd} hold with flying colours in these cases as well.\footnote{Their detailed study will appear in \cite{BvGS}.}

The discussion of \cref{sec:intrologopen} also opens the door to pushing the log/open correspondence beyond the maximal contact setting: it is tempting to see how the maximal tangency condition could be removed from \cref{conj:ocd}, with the splitting of contact orders amongst multiple points on the same divisor being translated to windings of multiple boundary disks ending on the same Lagrangian. The multi-covering factor of \eqref{eq:logopen_intro} would then be naturally given by a product of individual contact orders/disk windings -- an expectation that the reader can verify to be fulfilled in the basic example presented there of $(Y,D)=(\bbP^1\times \AA^1, \AA^1 \cup \AA^1)$. More generally, the link to the topological vertex and open GW invariants of arbitrary topology raises a fascinating question how much the topological vertex knows of the log theory of the surface -- and how it can be effectively used in the construction of (quantum) SYZ mirrors.


\subsubsection{Relation to the Cao--Maulik--Toda conjecture}
\label{sec:cmt}
Another direction towards a geometric understanding of the integrality of KP invariants is provided by sheaf-counting theories for Calabi--Yau $4$-folds, which were originally introduced by Borisov--Joyce \cite{MR3692967} (see also \cite{cao2014donaldson}) and have recently been given an algebraic construction by Oh--Thomas \cite{OT20}. 
More precisely, 
Cao--Maulik--Toda have conjectured in 
\cite{cao2019stable} (resp.~\cite{MR3861701})
explicit relations between genus $0$ KP invariants 
and stable pair invariants (resp.~counts of one-dimensional coherent sheaves)
on Calabi--Yau 4-folds. Recently, Cao--Kool--Monavari \cite{cao2020stable} have checked the conjecture of \cite{cao2019stable} for low degree classes on local toric surfaces; their proof hinges on the solution of the Gromov--Witten/Klemm--Pandharipande side given by \cref{thm:log_local_intro,thm:kpdt_intro} in this paper.

The results of \cref{thm:kpdt_intro,thm:openbps_intro} also raise a host of new questions. First and foremost, it would be extremely interesting to find for local toric surfaces a direct connection between the symmetric quivers appearing in \cref{thm:kpdt}
and the moduli spaces of coherent sheaves appearing in the conjectures of 
\cites{cao2019stable,MR3861701}.
Furthermore, since for $l=2$ we have $\mathrm{KP}_d(E_{Y(D)}) = \Omega_d(Y(D))$,  a fascinating direction would be to find an interpretation of the $q$-refined invariants $\Omega_d(Y(D))(q)$ in terms of the Calabi--Yau 4-fold $E_{Y(D)}$. 
A natural suggestion is that
$\Omega_d(Y(D))(q)$ should take the form of some appropriately refined Donaldson--Thomas invariants of $E_{Y(D)}$.
As the topic of refined DT theory of Calabi--Yau 4-folds is still in its infancy, we leave the question open for now.

\addtocontents{toc}{\protect\setcounter{tocdepth}{0}}
\section*{Acknowledgements}
\addtocontents{toc}{\protect\setcounter{tocdepth}{1}}

Discussions with T.~Bridgeland, M.~Gross, A.~Klemm, M.~Kool, P.~Kucharski, P.~Longhi, C.~Manola\-che, N.~Nabijou, D.~Ranganathan and P.~Sulkowski are gratefully acknowledged. We  thank A.~Afandi, Y.~Sch\"uler, and especially R.~Thomas for very helpful comments on previous versions of the manuscript, and C.~Krattenthaler for providing us with a copy of \cite{CKpriv} where \cref{conj:dp1binom} is proved for $q \neq 1$.



\section{Nef Looijenga pairs}

\label{sec:logCYsurf}

We start off by establishing some general facts about the classical geometry of nef log~Calabi--Yau (CY) surfaces. We first proceed to classify them in the smooth case, recall some basics of their birational geometry and the construction of toric models, and describe the structure of their pseudo-effective cone in preparation for the study of curve counts in them. We then end by introducing the notions of (quasi-) tameness.

\subsection{Classification}

We start by giving the following

\begin{defn}
An \emph{$l$-component log~CY surface with maximal boundary}, or \emph{$l$-component Looijenga pair}, is a pair $Y(D) \coloneqq (Y,D=D_1+\dots + D_l)$ consisting of a smooth rational projective surface $Y$ and a singular nodal anti-canonical divisor $D$ that admits a decomposition $D=D_1+\cdots+D_l$. We say that an $l$-component log~CY surface is \emph{nef} if $l \geq 2$ and each $D_i$ is a smooth, irreducible, and nef rational curve. 
\label{def:lp}
\end{defn}

Examples of log~CY surfaces arise when $Y$ is a projective toric surface and $D$ is the complement of the maximal torus orbit in $Y$; we call these pairs {\it toric}. 
By definition, if $Y(D)$ is nef, $Y$ is a weak Fano surface together with a choice of distribution of the anti-canonical degree amongst components $D_i$ preserving the condition that $D_i \cdot C \geq 0$ for any effective curve $C$ and all $i=1, \dots, l$ with $l\geq 2$. We classify these by recalling some results of di~Rocco \cite{Rocco} (see also \cite[Section 2]{CGKT1} and \cites{CGKT2,CGKT3}). 

Let $\dpr$ be the surface obtained from blowing up $r \geq 1$ general points in $\ptwo$. 
The Picard group of $\dpr$ is generated by the hyperplane class $H$ and the classes $E_i$ of the exceptional divisors. 
The anticanonical class is $-K_{\dpr}=3H-\sum_{i=1}^r E_i$. Recall that a \emph{line class} on $\dpr$ is $l\in\pic(\dpr)$ such that $l^2=-1$ and $-K_{\dpr}\cdot l=1$; for $r\leq 5$ and up to permutation of the $E_i$, they are given by $E_i$, $H-E_1-E_2$  or $2H-\sum_{i=1}^5 E_i$. Furthermore, for $n\geq 0$, denote by $\fn$ the $n$th Hirzebruch surface. Its Picard group is of rank 2 generated by the sections $C_{-n}$, resp. $C_n$, with self-intersections $-n$, resp. $n$, and by the fibre class $f$, subject to the relation that $C_{-n}+nf=C_n$. Note that $\fzero\simeq \pone\times\pone$ and $\fone\simeq \mathrm{dP}_1$ is the blowup of $\ptwo$ in one point. 

\begin{lem}[\cite{Rocco}]\label{lem:nef}
Assume that $1\leq r \leq 5$ and let $D\in\pic(\dpr)$. Then $D$ is nef if and only if
\begin{enumerate}
\item[(i)] for $r=1$, $D \cdot l\geq 0$ for all line classes $l$ and $D \cdot (H-E_1)\geq 0$,
\item[(ii)] for $5 \geq r\geq 2$, $D \cdot l\geq 0$ for all line classes $l$.
\end{enumerate}
\end{lem}

\subsubsection{$l=2$}
Let's start by setting $l=2$. With the sole exception of $\delp_4(1,0)$ versus $\delp_4(0,1)$, the next proposition shows that up to deformation and permutation of the factors, and assuming that $D_1$ and $D_2$ are nef, $Y(D)$ is completely determined by the self-intersections $(D_1^2, D_2^2)$, and we will employ the short-hand notation $Y(D) \leftrightarrow Y(D_1^2,D_2^2)$ to indicate this.

\begin{prop}\label{prop:class}
Let $(Y(D=D_1+D_2)$ be a 2-component nef log~CY surface. Then up to deformation and interchange of $D_1$ and $D_2$, $Y(D_1^2,D_2^2)$ is one of the following types
\begin{enumerate}
\item $\ptwo(1,4)$,
\item $\dpr(1,4-r)$ for $1\leq r \leq 4$,
\item $\dpr(0,5-r)$ for $1\leq r \leq 5$,
\item $\fzero(0,4)$,
\item $\fzero(2,2)$.
\end{enumerate}
\end{prop}

\begin{proof}

A minimal model of $Y$ is given by $\ptwo$, $\fzero$ or $\mathbb{F}_n$ for $n\geq2$. By assumption $-K_Y=D_1+D_2$ is nef, ruling out $\mathbb{F}_n$ for $n>2$. 
If $Y=\fzero$, then the stipulated decompositions of $-K_{\fzero}$ are immediate.
If $\ftwo$ is a minimal model of $Y$, then $Y=\ftwo$. In this case, the only possible decomposition of $-K_{\ftwo}$ into nef divisors is as $D_1=C_{-2}+2f=C_2$ and $D_2=C_2$.
The resulting pair $\ftwo(2,2)$ is deformation equivalent to $\fzero(2,2)$, see the proof of Proposition \ref{prop:local-deform}.

Assume now that $\ptwo$ is a minimal model of $Y$. If $Y=\ptwo$, we are done. Otherwise, up to deformation, we may assume that $Y=\delp_r$. Since $-K_Y$ is nef, $r\leq 9$. As $D_1$ and $D_2$ are nef, they are of the form $dH-\sum_{i=1}^r a_i E_i$ for $d\geq 1$ and $a_i\geq 0$.
Applying Lemma \ref{lem:nef}, we find that the only nef decompositions are as follows:
\begin{itemize}
\item Either $D_1=H$, $D_2=2H-\sum_{i=1}^r E_i$ for $r \leq 4$;
\item Or $D_1=H-E_j$, $D_2=2H-\sum^r_{i\neq j} E_i$ for $r \leq 5$.
\end{itemize}
They are all basepoint-free by \cite{Rocco} (see \cite[Lemma 2.7]{CGKT1}) and hence a general member will be smooth by Bertini.
\end{proof}


\subsubsection{$l=3$}

Next, we classify the surfaces with $l=3$ nef components. The short-hand notation $Y(D_1^2, D_2^2, D_3^2)$ is employed as in the previous section.

\begin{prop}\label{prop:class2}
Let $Y(D=D_1+D_2+D_3)$ be a 3-component log~CY surface with $Y$ smooth and $D_1$, $D_2$ and $D_3$ nef. Then up to deformation and permutation of $D_1$, $D_2$ and $D_3$, $Y(D_1^2, D_2^2, D_3^2)$ is one of the following:
\begin{enumerate}
\item $\ptwo(1,1,1)$,
\item $\delp_1(1,1,0)$,
\item $\delp_2(1,0,0)$,
\item $\delp_3(0,0,0)$,
\item $\fzero(2,0,0)$.
\end{enumerate}
\end{prop}

\begin{proof}
A minimal model of $Y$ is given by $\ptwo$, $\fzero$ or $\mathbb{F}_n$ for $n\geq2$. By assumption $-K_Y=D_1+D_2+D_3$ is nef, ruling out $\mathbb{F}_n$ for $n\geq2$. 
For $\ptwo$, the only possibility is to choose $D_1,D_2,D_3$ in class $H$. For $\fzero$, it is to choose $D_1=H_1+H_2$ the diagonal and $D_2=H_1$, $D_3=H_2$. 
Necessarily, all other surfaces are given by iterated blow-ups of the minimal models, keeping the divisors nef, leading to the list. As in the previous proposition, they are all base-point free and thus a general member will be smooth.
\end{proof}


\subsubsection{$l \geq 4$}

For $l=4$, a minimal model for $Y$ is $\fzero$, for which the only possibility is given by $D$ being its toric boundary. There are no other cases preserving nefness of the divisors. For 5 components or more, there are no surfaces keeping each divisor nef.

\subsection{Toric models}

We consider two basic operations on log~CY surfaces $Y(D)$.
\begin{itemize}
\item Let $\wt{Y}$ be the blow-up of $Y$ at a node of $D$ and let $\wt{D}$ be the preimage of $D$ in $\wt{Y}$. Then the log~CY surface $(\wt{Y},\wt{D})$ is said to be a \emph{corner blow-up} of $Y(D)$.
\item Let $\wt{Y}$ be the blow-up of $Y$ at a smooth point of $D$. Let $\wt{D}$ be the strict transform of $D$ in $\wt{Y}$. Then the log~CY surface $(\wt{Y},\wt{D})$ is said to be an \emph{interior blow-up} of $Y(D)$.
\end{itemize}

A corner blow-up does not change the complement $Y\setminus D$, whereas an interior blow-up does; accordingly corner blowups do not change log Gromov--Witten invariants \cite{AW}.

\begin{defn}
Let
$\pi \colon Y(D) \longrightarrow \ol{Y}(\ol{D})$
be a sequence of interior blow-ups between log~CY surfaces such that $\ol{Y}(\ol{D})$ is toric. Then $\pi$ is said to be a \emph{toric model} of $Y(D)$.

\end{defn}

We will describe toric models by giving the fan of $(\ol{Y},\ol{D})$ with {\it focus-focus singularities} on its rays. A focus-focus singularity on the ray corresponding to a toric divisor $F$ encodes that we blow up $F$ at a smooth point. Each focus-focus singularity produces a wall and interactions of them create a scattering diagram $\scatt(Y(D))$, as we discuss in Section \ref{sec:theta}.

\begin{prop}[Proposition 1.3 of \cite{GHKlog}]\label{prop:logCY}
Let $Y(D)$ be a log~CY surface. Then there exist log~CY surfaces $\wt{Y}(\wt{D})$ and $\ol{Y}(\ol{D})$, with the latter toric, and a diagram
\begin{equation} \label{diag:tm}
\xymatrix{
 & \wt{Y}(\wt{D}) \ar[ld]_\varphi \ar[rd]^\pi & \\
Y(D) & &  \ol{Y}(\ol{D})
}
\end{equation}
such that $\varphi$ is a sequence of corner blow-ups and $\pi$ is a toric model.
\end{prop}

The diagrams as in \eqref{diag:tm} are far from unique, and they are related by cluster mutations \cite{GHKclu}. Because of the invariance of log Gromov--Witten invariants by corner blow-ups, we can calculate the log Gromov--Witten invariants of $Y(D)$ on the scattering diagram $\scatt(Y(D))$ associated to the toric model $\pi$.

\subsection{The effective cone of curves}
\label{sec:eff}

Given $Y(D)$ a nef log~CY surface and $d\in A_1(Y)$, it will be convenient for the discussion in the foregoing sections to determine numerical conditions for $d$ to be an element of the pseudo-effective cone. If $Y = \bbF_n$, $\mathrm{NE}(Y)$ is just the monoid generated by $C_{-n}$ and $f$, so let us assume that $Y=\dpr$. 
We will write a curve class $d$ as $d_0(H-\sum_{i=1}^r E_i)+\sum_{i=1}^r d_iE_i$. If $\rho(Y)\geq 2$, then the extremal rays of the effective cone $\mathrm{NE}(Y)$ of $Y$ are generated by extremal classes $D$ with $D^2\leq 0$, and in the case of del Pezzo surfaces these are the line and fibre classes described above. Using the classification \cite[Examples 2.3 and 2.11]{CGKT1}, up to permutation of the $E_i$ and $E_j$, we find the following lists of generators of extremal rays of $\mathrm{NE}(Y)$. 
\begin{itemize}
\item If $r=1$,
\beq
E_1, \quad H-E_1 \,.
\eeq
\item If $2\leq r \leq 4$,
\beq
E_i, \quad H-E_i, \quad H-E_i-E_j \; (i\neq j) \,.
\eeq
\item If $r=5$,
\beq
E_i, \quad H-E_i, \quad H-E_i-E_j \; (i\neq j), \quad 2H-\sum_{i=1}^5 E_i \,.
\eeq
\end{itemize}
Note that the effective cone is closed since it is generated by finitely many elements. The following Proposition can be specialised to the del Pezzo surfaces $\delp_r$ for $r\leq5$ by setting the corresponding $d_i$ to 0 and removing the superfluous equations such as the last one that only holds for $r=5$.

\begin{prop}\label{prop:eff}
A class $d=d_0(H - \sum_{i=1}^5 E_i) + \sum_{i=1}^5 d_i E_i$ of $\delp_5$ is effective if and only if 
\beq
d_0\geq0, \quad d_i\geq0, \quad d_i+d_j+d_k\geq d_0, \quad d_i+d_j+d_k+d_l\geq 2d_0, \quad 2d_i + \sum_{j\neq i}d_j\geq 3d_0,
\eeq
where the $i,j,k,l$ are always pairwise distinct.

\end{prop}

The statement follows from the explicit description of the effective cone as generated by extremal rays. A direct calculation using the Polymake package in Macaulay2 computes the half-spaces defining the cone, yielding the above inequalities for the effective curves.

\subsection{Tame and quasi-tame Looijenga pairs}
\label{sec:quasitame}
The computation of curve-counting invariants of nef Looijenga pairs is strongly affected by the number $l$ of smooth irreducible components of $D$ and the positivity of $D_i$, $i=1,\dots, l$. We spell this out with the following definition, whose significance will be worked out in \cref{sec:mirrorthm,sec:theta}. 

Let $Y(D)$ be a nef Looijenga pair and let
\beq
E_{Y(D)} \coloneqq \mathrm{Tot}(\oplus_{i=1}^l \cO_Y(-D_i))
\label{eq:EYD}
\eeq
be the total space of the direct sum of the dual line bundles to $D_i$, $i=1, \dots, l$.

\begin{defn}
We call a nef log~CY surface $(Y, D=D_1 + \dots + D_l)$ \emph{tame} if $Y$ either $l >2$ or $D_i^2>0$ for all $i$. A nef log~CY surface $Y(D)$ is \emph{quasi-tame} if $E_{Y(D)}$ is deformation equivalent to $E_{Y'(D')}$, with $Y'(D')$ tame.
\end{defn}

We will use the abbreviated notation $E_{Y(D_1^2,D_2^2)}$ for the local Calabi--Yau fourfold $E_{Y(D_1+D_2)}$ associated by \eqref{eq:EYD} to a $2$-component log~CY surface $Y(D_1^2, D_2^2)$ in the classification of \cref{prop:class}. Quasi-tame pairs are classified by the following Proposition.

\begin{prop}
\label{prop:local-deform}
The following varieties are deformation-equivalent:
\ben
\item $E_{\mathbb{F}_0(0,4)}$, $E_{\mathbb{F}_0(2,2)}$ and $E_{\mathbb{F}_2(2,2)}$;
\item $E_{\delp_r(1,4-r)}$ and $E_{\delp_r(0,5-r)}$, $1 \leq r \leq 4$.
\een
\end{prop}

\begin{proof}

For the first part of the Proposition, denote by $H_1$ and $H_2$ the two generators of the Picard group of $\fzero$ corresponding to the pullbacks of a point in $\pone$ along $\rm{proj_{1,2}} : \fzero \to \pone$. The Euler sequence on $\pone$, pulled back to $\fzero$ along $\rm{proj}_1$ and tensored by $\shO(-H_2)$ yields
\beq \xymatrix{
0 \ar[r] & \shO(-2H_1-H_2)  \ar[r] & \shO(-H_1-H_2)\oplus\shO(-H_1-H_2) \ar[r] & \shO(-H_2) \ar[r] & 0 \,.
}\eeq
This determines a family with general fibre the total space of $\shO(-H_1-H_2)\oplus\shO(-H_1-H_2)$ and special fibre the total space of $\shO(-H_2)\oplus\shO(-2H_1-H_2)$ hence a deformation between $E_{\mathbb{F}_0(0,4)}$ and $E_{\mathbb{F}_0(2,2)}$. 

Next, consider again the Euler sequence over $\pone$,
\beq \xymatrix{
0 \ar[r] & \shO(-2) \ar[r] & \shO(-1)\oplus\shO(-1) \ar[r] & \shO \ar[r] & 0 \,,
}\eeq
and the associated deformation of the total space of $\shO(-1)\oplus\shO(-1)$ into the total space of $\shO\oplus\shO(-2)$. Taking the projectivisation of this family yields a deformation between $\fzero$ and $\ftwo$. In this deformation, $-H_1-H_2$ specialises to $-C_2$. Taking twice the associated line bundles yields the deformation between $E_{\mathbb{F}_0(2,2)}$ and $E_{\mathbb{F}_2(2,2)}$. 



To prove the second part, assume first that $r=1$. We start with the relative (dual) Euler sequence for the fibration $\delp_1\to\pone$ with distinct sections with image $H$ and $E_1$:
\beq \xymatrix{
0 \ar[r] & \shO \ar[r] & \shO(H)\oplus\shO(E_1) \ar[r] & \shO(H+E_1) \ar[r] & 0\,.
}\eeq
We tensor it with $\shO(-2H)$ to obtain
\beq \xymatrix{
0 \ar[r] & \shO(-2H) \ar[r] & \shO(-H)\oplus\shO(-2H+E_1) \ar[r] & \shO(-H+E_1) \ar[r] & 0\,.
}\eeq
This determines a family with general fibre the total space of $\shO(-H)\oplus\shO(-2H+E_1)$ and special fibre the total space of $\shO(-2H)\oplus\shO(-H+E_1)$ hence a deformation between $E_{\delp_1(1,3)}$ and $E_{\delp_1(0,4)}$.

Dually, we have
\beq \xymatrix{
0 \ar[r] & \hhh^0\left(\shO(H-E_1)\right)  \ar[r] & \hhh^0\left(\shO(H)\right)\oplus\hhh^0\left(\shO(2H-E_1)\right) \ar[r] & \hhh^0\left(\shO(2H)\right) \,,
}\eeq 
and a section of $\shO(2H-E_1)$ in the general fibre gives a section of $\shO(2H)$ in the special fibre. Hence we have a divisor $\shD$ in the family in class $2H-E_1$ for the general fibre and of class $2H$ for the special fibre. Blowing up a general point of $\shD$ in the family gives a deformation between $E_{\delp_2(1,2)}$ and $E_{\delp_2(0,3)}$. Iterating the process we obtain the desired deformations.
\end{proof}

We summarise the discussion of this Section in \cref{tab:classif}. There are 18 smooth deformation types of nef Looijenga pairs in total, 11 of which are tame and 15 are quasi-tame. The three non-quasi tame cases occur when $Y$ is a del Pezzo surface of degree 5 or less.

\begin{table}[t]
    \centering
    \begin{tabular}{|c|c|c|c|c|c|c|c|c|}
    \hline
    $Y(D)$ & $l$ &  $K_Y^2$  & $D_1$ & $D_2$ & $D_3$ & $D_4$ & Tame & Quasi-tame \\
        \Xhline{3\arrayrulewidth}
         $\bbP^2(1,4)$ & $2$ & $9$  & $H$ & $2H$ & -- & -- & \checkmark & \checkmark \\ 
         \hline
          $\bbF_0(2,2)$ & $2$ & $8$  & $H_1+H_2$ & $H_1+H_2$ & -- & -- & \checkmark & \checkmark \\ \hline
          $\bbF_0(0,4)$ & $2$ & $8$  & $H_1$ & $H_1+2 H_2$ & -- & -- & \ding{55} & \checkmark \\ \hline
          $\delp_1(1,3)$ & $2$ & $8$  & $H$ & $2H-E_1$ & -- & -- & \checkmark & \checkmark \\ \hline
        $\delp_1(0,4)$ & $2$ & $8$ &  $H-E_1$ & $2H$ & -- & -- & \ding{55} & \checkmark \\ \hline
          $\delp_2(1,2)$ & $2$ & $7$  & $H$ & $2H-E_1-E_2$ & -- & -- & \checkmark & \checkmark \\ \hline
        $\delp_2(0,3)$ & $2$ & $7$ &  $H-E_1$ & $2H-E_2$ & -- & -- & \ding{55} & \checkmark \\ \hline
          $\delp_3(1,1)$ & $2$ & $6$  & $H$ & $2H-E_1-E_2-E_3$ & -- & -- & \checkmark & \checkmark \\ \hline
        $\delp_3(0,2)$ & $2$ & $6$ &  $H-E_1$ & $2H-E_2-E_3$ & -- & -- & \ding{55} & \checkmark \\ \hline
          $\delp_4(1,0)$ & $2$ & $5$  & $H$ & $2H-E_1-E_2-E_3-E_4$ & -- & -- & \ding{55} & \ding{55} \\ \hline
        $\delp_4(0,1)$ & $2$ & $5$ &  $H-E_1$ & $2H-E_2-E_3-E_4$ & -- & -- & \ding{55} & \ding{55} \\ \hline
        $\delp_5(0,0)$ & $2$ & $4$ &  $H-E_1$ & $2H-E_2-E_3-E_4-E_5$ & -- & -- & \ding{55} & \ding{55} \\ 
        \Xhline{3\arrayrulewidth}
         $\bbP^2(1,1,1)$ & $3$ & $9$  & $H$ & $H$ & $H$ & -- & \checkmark & \checkmark \\ \hline
          $\bbF_0(2,0,0)$ & $3$ & $8$  & $H_1+H_2$ & $H_1$ & $H_2$ & -- & \checkmark & \checkmark \\ \hline
          $\delp_1(1,1,0)$ & $3$ & $8$  & $H$ & $H$ & $H-E_1$ & -- & \checkmark & \checkmark \\ \hline
          $\delp_2(1,0,0)$ & $3$ & $7$  & $H$ & $H-E_1$  & $H-E_2$ & -- & \checkmark & \checkmark \\ \hline
          $\delp_3(0,0,0)$ & $3$ & $6$  & $H-E_1$ & $H-E_2$ & $H-E_3$ & -- & \checkmark & \checkmark \\ \hline
          \Xhline{3\arrayrulewidth}
           $\bbF_0(0,0,0,0)$ & $4$ & $8$  & $H_1$ & $H_2$ & $H_1$ & $H_2$ & \checkmark & \checkmark \\ \hline
    \end{tabular}
    \caption{Classification of smooth nef Looijenga pairs.}
    \label{tab:classif}
\end{table}

\section{Local Gromov--Witten theory}

\label{sec:localGW}

\subsection{$1$-pointed local Gromov--Witten invariants}
\label{sec:mirrorthm}
In this section, we provide general formulas for the Gromov--Witten invariants with point insertions of toric Fano varieties in any dimension twisted by a sum of concave line bundles. For the remainder of this section, let $Y$ be an $n$-dimensional smooth projective variety of Picard rank $r$, let $D=D_1+\dots +D_{l} \in A_{n-1}(Y)$ with $D \in |-K_Y|$ and each $D_i$ smooth and irreducible, and let $d$ be a $D$-convex curve class.

Let $E_{Y(D)} \coloneqq \mathrm{Tot}(\oplus_{i=1}^l (\cO_Y(-D_i)))$ be as in \eqref{eq:EYD} and let $\pi_Y:E_{Y(D)} \to Y$ be the natural projection.  
Since $d$ is $D$-convex, the moduli space $\ol{\mmm}_{0,m}(E_{Y(D)},d)$ of genus 0 $m$-marked stable maps $[f:C\to E_{Y(D)}]$ with $f_{*}([C])=d \in \hhh_2(Y,\bbZ)$ is scheme-theoretically the moduli stack $\ol{\mmm}_{0,m}(Y,d)$ of stable maps to the base $Y$, as every stable map to the total space  factors through the zero section $Y \hookrightarrow E_{Y(D)}$. In particular, $\ol{\mmm}_{0,m}(E_{Y(D)},d)$ is proper.
Consider the universal curve $\pi: \mathcal{C} \to \ol{\mmm}_{0,m}(Y,d)$, and denote by
$f : \mathcal{C} \to Y$ 
the universal stable map. 
Then $\hhh^0(C,f^{*}\shO_{Y}(-D_i))=0$ 
and we have 
obstruction bundles
$
\obstr_{D_j}:=R^{1}\pi_{*}f^{*}\shO_{Y}(-D_j),
$
of rank $d \cdot D_j-1$ with fibre $\hhh^1(C,f^{*}\shO_{Y}(-D_j))$ over a stable map $[f:C\to Y]$. The virtual fundamental class on $\ol{\mmm}_{0,m}(E_{Y(D)},d)$ is defined by intersecting the virtual fundamental class on $\ol{\mmm}_{0,m}(Y,d)$ with the top Chern class of $\oplus_j \obstr_{D_j}$:
\begin{align}
\label{eq:dtvfc}
  [\ol{\mmm}_{0,m}(E_{Y(D)},d)]^{\rm vir}:=& c_{\rm top}\left(\obstr_{D_1}\right) \cap \dots \cap
c_{\rm top}\left(\obstr_{D_{l}}\right) \cap  [\ol{\mmm}_{0,m}(Y,d)]^{\rm
  vir} \nn \\ 
\in &\hhh_{m+1-l}(\ol{\mmm}_{0,m}(Y,d),\QQ).
\end{align}
There are tautological classes $\psi_i:=c_1(L_i)$, where $L_i$ is the $i^{\rm th}$ tautological line
bundle on $\ol{\mmm}_{0,m}(Y,d)$ whose fibre at $[f:(C,x_1,\dots,x_m)\to Y]$
is the cotangent line of $C$ at $x_i$, and we denote by $\ev_i$ the
evaluation maps at the $i^{\rm th}$ marked point.  For an effective $D$-convex curve class $d\in\hhh_2(Y,\ZZ)$, 
genus 0 local Gromov--Witten invariants of $E_{Y(D)}$ with point insertions on the base are defined as
\bea
\label{eq:Nd}
N^{\rm loc}_{0,d}(Y(D)) &\coloneqq& \int_{[\ol{\mmm}_{0,l-1}(E_{Y(D)},d)]^{\rm vir}} \, \prod_{j=1}^{l-1} \ev_j^{*}(\pi_Y^{*} [\pt_Y]),\\
N^{\rm loc, \psi}_{0,d}(Y(D))&:=&\int_{[\ol{\mmm}_{0,1}(E_{Y(D)},d)]^{\rm vir}} \ev_1^*(\pi_Y^{*} [\pt_Y])\cup\psi_1^{l-2},
\label{eq:Npsid}
\eea
which we think of as the virtual counts of curves through $l-1$ points (resp. 1-point with a $\psi$-condition) on the zero section of the vector bundle $E_{Y(D)}$.  

Since $D$ is anticanonical, $E_{Y(D)}$ is a non-compact Calabi--Yau $(n+l)$-fold. The case $n+l=3$ has been the main focus in the study of local mirror symmetry, and as such it has been abundantly studied in the literature \cite{Chiang:1999tz}. It turns out that the lesser studied situation when $n+l>3$ has a host of simplifications, often leading to closed-form expressions for \eqref{eq:Nd}--\eqref{eq:Npsid}. We start by fixing some notation which will be of further use throughout this Section. Let $T\simeq (\bbC^\star)^{l}\circlearrowright E_{Y(D)}$ be the fibrewise torus action and denote by 
$\lambda_i \in \hhh(BT)$, $i=1, \dots, l$ its equivariant parameters.  Let $\{\phi_\a\}_{\a}$ be a graded $\bbC$-basis for the non-equivariant cohomology of the image of the zero section $Y\hookrightarrow E_{Y(D)}$ with $\deg \phi_\alpha \leq \deg \phi_{\alpha+1}$; in particular $\phi_1=\mathbf{1}_{\hhh(Y)}$. Its elements have canonical lifts $\phi_\a \to \varphi_\a$ to $T$-equivariant cohomology forming a $\bbC(\lambda_1, \dots, \lambda_{l})$ basis
for $\hhh_T(Y(D))$. The latter is furthermore endowed with a perfect pairing
\beq
\eta_{E_{Y(D)}}(\varphi_\a, \varphi_\b) := \int_{Y} \frac{\phi_\a \cup
  \phi_\b}{\cup_i \mathrm{e}_T(\cO_Y(-D_i))},
\label{eq:grampd}
\eeq
with $\mathrm{e}_T$ denoting the $T$-equivariant Euler class. In what follows, we will indicate by $\eta^{-1}_{E_{Y(D)}}(\varphi_\a, \varphi_\b)$ the inverse of the Gram matrix \eqref{eq:grampd}.

Let now $\tau \in H_T(Y(D))$. The  $J$-function of $E_{Y(D)}$ is the formal power series  
\beq
J_{\rm big}^{E_{Y(D)}}(\tau, z):= z+\tau+\sum_{d\in
  \mathrm{NE}(Y)}\sum_{n\in \bbZ^+}\sum_{\a,\b} \frac{\eta_{E_{Y(D)}}^{-1}(\varphi_\a, \varphi_\b)}{n!}\bra \tau, \dots, \tau, \frac{\varphi_\a}{z-\psi}
\ket_{0,n+1,d}^{E_{Y(D)}} \varphi_\b 
\eeq
where we employed the usual correlator notation for GW invariants,
\beq
\bra \tau_1 \psi_1^{k_1}, \dots, \tau_n \psi^{k_n}_n\ket_{0,n,d}^{E_{Y(D)}}:=\int_{[\ol{\mmm}_{0,m}(E_{Y(D)},d)]^{\rm vir}} \prod_i \ev^*_i(\tau_i) \psi_i^{k_i}.
\eeq
Restriction to $t=\sum_{i=1}^{r+1} t_i \varphi_i$ and use of the Divisor
Axiom gives the small $J$-function
\beq
J_{\rm small}^{E_{Y(D)}}(t, z):= z  \re^{\sum t_i \varphi_i/z}\l(1+\sum_{d\in
  \mathrm{NE}(Y)}\sum_{\a,\b} \eta^{-1}(\varphi_\a, \varphi_\b)_{E_{Y(D)}} \re^{t(d)} \bra \frac{\varphi_\a}{z(z-\psi_1)}
\ket_{0,1,d}^{E_{Y(D)}} \varphi_\b\r).
\label{eq:Jsmall}
\eeq

\begin{lem}
Suppose that $Y$ is a toric Fano variety and either $n+r=4$ and $D_i$ is ample $\forall~i$, or $n+r>4$ and $D_i$ is nef  $\forall~i$. Let $\cT:=\{T_i \in A_{\dim Y-1}(Y)\}_{i=1}^{n+r}$ be the collection of its prime toric divisors, and $\sqcup_{i=1}^{m} S_i = \{1, \dots, n+r\}$ a length-$m$ partition of $n+r$ such that $D_i := \sum_{j \in S_i} T_{j}$. For an effective curve class $d \in \mathrm{NE}(Y)$, write $\mathsf{d}_i := d \cdot D_i$ and $\mathsf{t}_i := d \cdot T_i$ for its intersection multiplicities with respectively the nef divisors $D_i$ and the toric divisors $T_i$. Then,
\beq
N^{\rm loc, \psi}_{0,d}(Y(D_1+\dots+D_l)) = \frac{(-1)^{\sum_{i=1}^l (\mathsf{d}_i-1)}}{\prod_{i=1}^l \mathsf{d}_i} \prod_{i=1}^l \binom{\mathsf{d}_i}{\{ \mathsf{t}_{j}\}_{j\in S_i}},
\label{eq:localgw}
\eeq
where $\binom{k}{\{i_j\}_{j=1}^m}=\frac{k!}{\prod_{j=1}^m i_j!}$ is the multinomial coefficient.
\label{lem:localgw}
\end{lem}

\begin{proof}
By \eqref{eq:Npsid} and \eqref{eq:Jsmall}, we have
\beq
N^{\rm loc, \psi}_{0,d}(Y(D))) =\sum_\beta \eta(\varphi_{\bar \a}, \varphi_\b)_{E_{Y(D)}}  [z^{1-l} \re^{t(d)+\sum t_i \varphi_i/z}  \varphi^\beta]  J_{\rm small}^{E_{Y(D)}}(t, z)  
\eeq
where $\bar \a$ is defined by $\varphi_{\bar \a} = [\mathrm{pt}]$. From \eqref{eq:grampd}, we have $\eta(\varphi_{\bar\a}, \varphi_\b)_{E_{Y(D)}}=\delta_{\bar \a 1} \prod_{i=1}^{l} \lambda_i^{-1}$, hence
\beq
N^{\rm loc, \psi}_{0,d}(Y(D)) =\frac{1}{\prod_{i=1}^{l} \lambda_i}  \Big[z^{1-l} \re^{t(d)+\sum t_i \varphi_i/z}  \mathbf{1}_{H_T(Y)}\Big]  J_{\rm small}^{E_{Y(D)}}(t, z)  
\label{eq:NpsidJfun}
\eeq
The right hand side can be computed by Givental-style toric mirror theorems. Let $\theta_a:=T_a^\vee \in \hhh^2(Y)$ be the Poincar\'e dual class of the $a^{\rm th}$
toric divisor of $Y$, $\kappa_i:=c_1(\cO(-D_i)$ the $T$-equivariant Chern class of $D_i$,
and $y_i$, $i=1, \dots, r+1$ be variables in a formal disk around the origin. Writing $(x)_n:=\Gamma(x+n)/\Gamma(x)$ for the Pochhammer symbol of $(x,n)$ with $n\in \bbZ$, the $I$-function of $E_{Y(D)}$ is the $\hhh_T(Y)$-valued Laurent series
\begin{align}
I^{E_{Y(D)}}(y,z) := & z \mathbf{1}_{\hhh(Y)} + \prod_i y_i^{\varphi_i/z}
\sum_{0\neq d\in \mathrm{NE}(Y)}
  \prod_i y_i^{d_i} z^{1-l} \frac{\prod_i \kappa_i
    \l(\frac{\kappa_i}{z}+1\r)_{\mathsf{d}_i-1}}{\prod_a \l(\frac{\theta_a}{z}
    +1\r)_{\mathsf{t}_a}}
  \label{eq:Ifunloc}
\end{align}
and its mirror map is their formal $\cO(z^0)$ coefficient,
\begin{align}
  \tilde{t}^i_{E_{Y(D)}}(y) := & \big[z^0 \varphi_i \big] I^{E_{Y(D)}}(y,z)\,.
  \label{eq:mm}
\end{align}
Then \cites{MR1408320, MR2276766, MR2510741}
\beq
 J_{\rm small}^{E_{Y(D)}}\l(\tilde{t}_{E_{Y(D)}(D)}(y),
    z\r) =  I^{E_{Y(D)}}(y,z)\,.
\label{eq:I=J}
\eeq
Inspecting \eqref{eq:Ifunloc} shows that if either $n+l>4$,  or $n+l=4$ and $D_i$ is ample, the mirror map does not receive quantum corrections:
\beq
  \tilde{t}^i_{E_{Y(D)}}(y) = \log y_i \,. 
\eeq
Therefore, under the assumptions of the Lemma,
\bea
N^{\rm loc, \psi}_{0,d}(Y(D))  &=& \frac{1}{\prod_{i=1}^{l} \lambda_i}  \l[z^{1-l}  \prod_i y_i^{d_i} \prod_i y_i^{\varphi_i/z}  \mathbf{1}_{H_T(Y(D))}\r]  I^{E_{Y(D)}}(y, z)  \nn \\
&=& \frac{1}{\prod_{i=1}^{l} \lambda_i}  \l[z^{1-l}  \prod_i y_i^{d_i}\r]  I^{E_{Y(D)}}(y, z)\bigg|_{\varphi_\a \to 0}.
\eea
The claim then follows by replacing $\theta_a|_{\varphi_\a \to 0} = 0$, $\kappa_i|_{\varphi_\a \to 0} = -\lambda_i$ into \eqref{eq:Ifunloc}.
\end{proof}

\subsubsection{Quasi-tame Looijenga pairs}
\label{sec:quasitameloc}
Let us now go back to the case of log~CY surfaces and specialise the discussion in the previous Section to $Y(D)$ a tame Looijenga pair. The key observation in the proof of \cref{lem:localgw} was that no contributions to the $\cO(z^0)$ Laurent coefficient of the $I$-function could possibly come from any stable maps in any degrees, which is automatic for $n+l>4$, and requires that 
$\mathsf{d}_i>0$ when $n+l=4$. We can in fact partly relax the condition that $D_i$ is ample by just requiring by {\it fiat} that no curves with $\mathsf{d}_i = d \cdot D_i = 0$ contribute to the mirror map. A direct calculation from \eqref{eq:Ifunloc} shows that in the case of nef log~CY surfaces with $Y$ a Fano surface,
this relaxed assumption coincides with $Y(D)$ being tame as in \cref{def:lp}. Since by \cref{prop:class}, $Y$ is toric for all tame cases, \cref{lem:localgw} computes \eqref{eq:Npsid} for all of them. 

\begin{example}
Let $Y(D)=\bbP^2(1,4)$. Then \cref{lem:localgw} gives for the degree-$d$ local invariants of the projective plane
\beq 
N^{\rm loc, \psi}_{0,d}(Y(D))=N^{\rm loc}_{0,d}(Y(D)) = \frac{(-1)^d}{2d^2} \binom{2d}{d}\,.
\label{eq:NlocP2}
\eeq 
This recovers a direct localisation calculation by Klemm--Pandharipande in \cite[Proposition~2]{Klemm:2007in}.
\end{example}

\begin{example}
Let $Y(D)=\bbF_0(2,2)$ and write $d=d_1 H_1 + d_2 H_2$. \cref{lem:localgw} yields
\beq 
N^{\rm loc, \psi}_{0,d}(Y(D))=N^{\rm loc}_{0,d}(Y(D)) = \frac{1}{(d_1+d_2)^2} \binom{d_1+d_2}{d_1}^2
\label{eq:NlocF_0}
\eeq 
as in \cite[Proposition~3]{Klemm:2007in}.
\end{example}

Moreover, if $Y(D)$ is a quasi-tame Looijenga pair, the Calabi--Yau vector bundle $E_{Y(D)}$ is deformation equivalent to $E_{Y'(D')}$ for some tame Looijenga pair by definition. It therefore carries the same local Gromov--Witten theory, and the calculation of $N^{\rm loc, \psi}_{0,d}(Y(D))=N^{\rm loc, \psi}_{0,d}(Y'(D'))$ from \cref{lem:localgw} extends immediately to these cases as well. 

\subsubsection{Non-quasi-tame Looijenga pairs}

\cref{lem:localgw} cannot be immediately extended to non-quasi-tame pairs $Y(D)$, as $Y$ is not toric and $E_{Y(D)}$ does not deform to $E_{Y'(D')}$ for tame $Y'(D')$. We will proceed by exhibiting a closed-form solution for the case of lowest anticanonical degree $Y(D)=\mathrm{dP}_5(0,0)$. This recovers all other cases with $l=2$ by blowing-down, as per the following
\begin{prop}[Blow-up formula for local GW invariants]
Let $Y(D)$ be an $l$-component log CY surface. Let $\pi : Y'(D') \to Y(D)$ be the $l$-component log CY surface obtained by an interior blow-up at a general point of $D$ with exceptional divisor $E$. Let $d$ be a curve class of $Y(D)$ and let $d'\coloneqq \pi^* d$. Then
\bea
N^{\rm loc}_{0,d}(Y(D)) &=& N^{\rm loc}_{d'}(Y'(D'))\,, \nn \\
N^{\rm loc, \psi}_{0,d}(Y(D)) &=& N^{\rm loc, \psi}_{d'}(Y'(D'))\,.
\label{eq:locblup}
\eea
\label{prop:locblup}
\end{prop}

\begin{proof}
 
By \cite[Proposition 5.14]{Man12},
\beq
\overline{\pi}_*[\ol{\mmm}_{0,m}(Y',d')]^{\rm
  vir} = [\ol{\mmm}_{0,m}(Y,d)]^{\rm
  vir},
\eeq
where $\overline{\pi}$ is the morphism between the moduli spaces induced by $\pi$.
Since $E\cdot d'=0$,
\beq
\overline{\pi}_*[\ol{\mmm}_{0,m}(E_{Y'(D')},d')]^{\rm
  vir} = [\ol{\mmm}_{0,m}(E_{Y(D)},d)]^{\rm
  vir}.
\eeq
\end{proof}
\begin{thm}
With notation as in \cref{prop:eff}, we have
\bea
& &
N^{\rm loc}_{0,d}(\mathrm{dP}_5(0,0)) = \nn \\
& &
\sum_{j_1, \dots, j_4=0}^\infty \Bigg[
\frac{(-1)^{d_1+d_2+d_3+d_4+d_5} \left(d_1+d_2+d_3+d_4+d_5-3 d_0+j_1+j_2-1\right)!}
{j_1! j_2! j_3! j_4! \left(d_1+d_2+d_4-2 d_0+j_1\right)!  \left(-d_1+d_0-j_1-j_2\right)! \left(-d_3+d_5+j_4\right)!}\nn \\
& & \times 
\frac{ \left(d_1+d_4-d_0+j_1+j_3-1\right)!
   \left(d_1+d_5-d_0+j_2+j_4-1\right)! \left(d_4+d_5-d_0+j_3+j_4-1\right)!}{ \left(d_1+d_3+d_5-2 d_0+j_2\right)! \left(-d_4+d_0-j_1-j_3\right)! \left(-d_2+d_4+j_3\right)!
   \left(-d_5+d_0-j_2-j_4\right)!} \nn \\
& & \times \frac{1}{ \left(d_2+d_3-d_0-j_3-j_4\right)!  \left(\left(d_1+d_4+d_5-2
   d_0+j_1+j_2+j_3+j_4-1\right)!\right){}^2}\Bigg] \,.
\label{eq:dP5}
\eea
\label{thm:dP5}
\end{thm}

\begin{proof}[Sketch of the proof]
The strategy of the proof runs by deforming $\mathrm{dP}_5$  to the blow-up of $\bbF_0$ at four toric points, which is only weak Fano but allows to work torically along the lines of \cref{lem:localgw} at the price of extending the pseudo-effective cone by four generators of self-intersection $-2$. These contribute non-trivially to the mirror map, alongside curves with zero intersections with the boundary divisors $D_1=H-E_1$ and $D_2=2H-\sum_{i\neq 1} E_i $. However the mirror map turns out to be algebraic, and furthermore it remarkably has a closed-form rational inverse, leading to the final result \eqref{eq:dP5}. Full details are given in \cref{app:PdP5}.
\end{proof}


\begin{rmk}
The final expression \eqref{eq:dP5} is significantly more involved than \eqref{eq:localgw}, to which it reduces when blowing down to the quasi-tame del~Pezzo cases $\delp_k$, $k\leq 3$ by setting $d_{i}=d_0$ for all $i\geq k+1$ using \eqref{eq:locblup}, since then only the summand with $j_i=0~\forall~i$ survives. It is also noteworthy that, while the summands in \eqref{eq:dP5} are not symmetric under permutation of the degrees $\{d_2, d_3, d_4, d_5\}$, the final sum is highly non-obviously warranted to be $S_4$-invariant since the left hand side is\footnote{There is an obvious $S_5$ symmetry under permutation of the  exceptional classes $E_i$ in $Y$, which is reduced to an $S_4$ symmetry in the degrees $(d_2,d_3,d_4,d_5)$ in $E_{Y(D)}$ by the splitting $D_1=H-E_1$, $D_2=2H-E_2-E_3-E_4-E_5$.}, and we verified this explicitly in low degrees. The BPS invariants arising from \eqref{eq:dP5} should also be integers, and we checked this is indeed the case for a large sample of non-primitive classes with multi-covers of order up to 11.
\end{rmk}

\subsection{Multi-pointed local GW invariants}

The primary multi-point invariants \eqref{eq:Nd} of nef Looijenga pairs with $l>2$ can be reconstructed from the descendent single-insertion invariants \eqref{eq:Npsid}. We shall show how this arises by combining the associativity of the quantum product with the vanishing of quantum corrections for particular classes. 

\subsubsection{$l=3$}
It suffices to compute the invariants for the case of maximal Picard rank, $\mathrm{dP}_3(0,0,0)$, from which the other $l=3$ cases can be recovered by blowing down.
\begin{thm}
With notation as in \cref{prop:eff}, we have
\bea
N^{\rm loc}_{0,d}(\mathrm{dP}_3(0,0,0)) &=& \l(d_0^2-d_1 d_0-d_2 d_0-d_3 d_0+d_1 d_2+d_1 d_3+d_2 d_3 \r) N^{\rm loc, \psi}_{0,d}(\mathrm{dP}_3(0,0,0)), \\
N^{\rm loc, \psi}_{0,d}(\mathrm{dP}_3(0,0,0)) &=& \frac{(-1)^{d_1+d_2+d_3+1} \left(d_1-1\right)! \left(d_2-1\right)! \left(d_3-1\right)!}{\left(d_1+d_2-d_0\right)! \left(d_1+d_3-d_0\right)!
   \left(d_2+d_3-d_0\right)! \left(d_0-d_1\right)! \left(d_0-d_2\right)! \left(d_0-d_3\right)!}.\nn \\
\label{eq:dP33comp}
\eea
\label{thm:dP33comp}
\end{thm}
\begin{proof}
In the notation of the proof of \cref{lem:localgw}, for $i,j=2, \dots, 5$ the components of the small $J$-function of $E_{\mathrm{dP}_3(0,0,0))}$ satisfy the quantum differential equations
\beq
z \nabla_{\varphi_i}\nabla_{\varphi_j} J^{E_{\mathrm{dP}_3(0,0,0))}}_{\rm small}(t,z) = \nabla_{\varphi_i \star_t \varphi_j} J^{E_{\mathrm{dP}_3(0,0,0))}}_{\rm small}(t,z)	\label{eq:qdedP3}
\eeq
where $\a \star_t \b$ denotes the small quantum cohomology product, and the cohomology classes $\{ \varphi_1=\mathbf{1}_{H_T(E_{\mathrm{dP}_3(0,0,0))}}, \dots, \varphi_5\}$ are denoted as in the proof of \cref{lem:localgw}.  We take $\{\varphi_i\}_{i=2}^{5}$ to be the basis elements of $\hhh_T(E_{\mathrm{dP}_3(0,0,0))})$ given by lifts to $T$-equivariant cohomology of the integral K\"ahler classes dual to $\{C_i \in \hhh_2(\mathrm{dP}_3, \bbZ)\}_i$ with $C_{i+1}=E_i$, $i=1,2,3$, and $C_5=H-E_1-E_2-E_3$, and an effective curve will be written $d=d_0 C_5 +\sum_{i=1}^3 d_i C_{i+2}$.
From the proof of \cref{lem:localgw}, the small $J$-function in the tame setting equates the $I$-function,
\bea
& & J_{\rm small}^{E_{\mathrm{dP}_3(0,0,0)}}(t, z)  = \sum_{d_i>0} \re^{\sum_{i=0}^3 t_{i+2} d_i}\Bigg[
\frac{(-1)^{d_1+d_2+d_3}  \left(\phi _2-\lambda _1\right) \left(\phi _3-\lambda _2\right)}{z^2 \left(\frac{z+\phi _2+\phi _3-\phi _5}{z}\right)_{-d_0+d_1+d_2} \left(\frac{z+\phi _2+\phi _4-\phi
   _5}{z}\right)_{-d_0+d_1+d_3} } \nn \\
& &
\frac{\left(\phi _4-\lambda
   _3\right) \left(\frac{z-\lambda _1+\phi _2}{z}\right)_{d_1-1} \left(\frac{z-\lambda _2+\phi _3}{z}\right)_{d_2-1} \left(\frac{z-\lambda _3+\phi
   _4}{z}\right)_{d_3-1}}{\left(\frac{z+\phi _3+\phi _4-\phi _5}{z}\right)_{-d_0+d_2+d_3} \left(\frac{z-\phi _2+\phi _5}{z}\right)_{d_0-d_1}
   \left(\frac{z-\phi _3+\phi _5}{z}\right)_{d_0-d_2} \left(\frac{z-\phi _4+\phi _5}{z}\right)_{d_0-d_3}}
\Bigg].
\label{eq:JfundP3}
\eea
By \eqref{eq:Jsmall}, the small quantum product can be computed from the $\cO(z^{-1})$ formal Taylor coefficient of \eqref{eq:JfundP3} as
\beq
\varphi_i \star_t \varphi_j = \sum_\a \varphi_\a \Big[z^{-1} \varphi_\a\Big]  \de^2_{t_i t_j} J^{E_{\mathrm{dP}_3(0,0,0))}}_{\rm small}(t,z)\,.
\eeq
Inspection of \eqref{eq:JfundP3} shows that the r.h.s. receives quantum corrections of the form $1/n^2$ from curves with either $d_i =\delta_{ij} n$ or $d_i=(1-\delta_{ij}) n$ and $j \neq 0$, $n \in \bbN^+$, with vanishing contributions in all other degrees. This implies that
\beq
\l(\de^2_{t_5}-\de_{t_2}\de_{t_5}-\de_{t_3}\de_{t_5}-\de_{t_4}\de_{t_5}+\de_{t_2}\de_{t_3}+\de_{t_3}\de_{t_4}+\de_{t_2}\de_{t_4}\r)  \Big[z^{-1} \Big] J^{E_{\mathrm{dP}_3(0,0,0))}}_{\rm small}(t,z) = 0\,,
\eeq
which amounts to
\beq
\varphi_5 \star_t \varphi_5 - \sum_{j=2}^4 \varphi_5 \star_t \varphi_i+ \sum_{j>i=2}^5 \varphi_i \star_t \varphi_j=\varphi_5 \cup \varphi_5 - \sum_{j=2}^4 \varphi_5 \cup \varphi_i+ \sum_{j>i=2}^4 \varphi_i \cup \varphi_j\,.
\eeq
It is immediate to verify that the r.h.s. is the Poincar\'e dual of the point class. Therefore, from \eqref{eq:qdedP3},
\beq
N^{\rm loc}_{0,d}(\mathrm{dP}_3(0,0,0))=\l(d_0^2-d_1 d_0-d_2 d_0-d_3 d_0+d_1 d_2+d_1 d_3+d_2 d_3 \r) N^{\rm loc, \psi}_{0,d}(\mathrm{dP}_3(0,0,0))\,,
\eeq
and the second equation in the statement follows by \cref{lem:localgw}.
\end{proof}

\subsubsection{$l=4$}
In this case $D$ is the toric boundary, and the invariants
$N^{\rm loc}_{0,d}(\bbF_0(0,0,0,0))$ were computed in \cite{BBvG1} by a strategy similar to that of \cref{thm:dP33comp}. The final result is the following Proposition.
\begin{prop}[\cite{BBvG1}, Theorem~3.1, Corollary~6.4]
\beq \label{eq:local_4_comp}
N^{\rm loc}_{0,d}(\bbF_0(0,0,0,0)) = d_1^2 d_2^2 N^{\rm loc, \psi}_{0,d}(\bbF_0(0,0,0,0)) = 1\,,
\eeq
\label{thm:F04comp}
\end{prop}
This concludes the calculation of local invariants with point insertions for nef Looijenga pairs. We collate the results in \cref{tab:localgw}.

\begin{table}[t]
    \centering
    \begin{tabular}{|c|c|c|}
    \hline
    $Y(D)$ & $N^{\rm loc, \psi}_{0,d}$ & $N^{\rm loc}_{0,d}/N^{\rm loc, \psi}_{0,d}$ \\
    \Xhline{2\arrayrulewidth}
    $\bbP^2(1,4)$ & $\frac{1}{2 d^2} \binom{2d}{d}$ & 1 \\ \hline
        $\bbF_0(2,2)$ &  \multirow{2}{*}{$\l(\frac{1}{d_1+d_2}\binom{d_1+d_2}{d_1} \r)^2$} & \multirow{2}{*}{1} \\ \cline{1-1}
$\bbF_0(0,4)$ & & \\ \hline
        $\delp_1(1,1)$ &  \multirow{2}{*}{$\frac{(-1)^{d_1}}{d_0 (d_1+d_0)}\binom{d_0}{d_1}  \binom{d_1+d_0}{d_0}$} & \multirow{2}{*}{1} \\ \cline{1-1}
$\delp_2(0,4)$ & & \\ \hline
        $\delp_2(1,2)$ &  \multirow{2}{*}{$\frac{(-1)^{d_0+d_1+d_2}}{d_0 (d_1+d_2)}\binom{d_0}{d_1} \binom{d_0}{d_2} \binom{d_1+d_2}{d_0}$} & \multirow{2}{*}{1} \\ \cline{1-1}
$\delp_2(0,3)$ & & \\ \hline
    $\delp_3(1,1)$ &  \multirow{2}{*}{$\frac{(-1)^{d_1+d_2+d_3} (d_0-1)!(d_1+d_2+d_3-d_0-1)!}{(d_0-d_1)!(d_0-d_2)!
(d_0-d_3)!(d_1+d_2-d_0)!(d_1+d_3-d_0)!(d_2+d_3-d_0)!}$} & \multirow{2}{*}{1} \\ \cline{1-1}
$\delp_3(0,2)$ & & \\ \hline
$\delp_4(1,0)$ & \multirow{2}{*}{\eqref{eq:dP5}$\Big|_{d_5 \to d_0}$} & \multirow{2}{*}{1}\\ \cline{1-1}
$\delp_4(0,1)$ & & \\ \hline
$\delp_5(0,0)$ & \eqref{eq:dP5} & 1 \\ \hline
\Xhline{2\arrayrulewidth}
$\bbP^2(1,1,1)$ & $\frac{(-1)^{d+1}}{d^3} $ & $d^2$ \\     \hline
$\bbF_0(2,0,0)$ & $-\frac{1}{d_1 d_2 (d_1+d_2)}\binom{d_1+d_2}{d_2}$ & $d_1 d_2$ \\ \hline
$\delp_1(1,1,0)$ & $\frac{(-1)^{d_1+1}}{d_0^2 d_1} \binom{d_0}{d_1}$ & $d_1 d_0$ \\     \hline
 $\delp_2(1,0,0)$ & $\frac{(-1)^{d_0+d_1+d_2+1}}{d_0 d_1 d_2} \binom{d_0}{d_1}\binom{d_1}{d_0-d_2}$ & $d_1 d_2$ \\ 
    \hline
 \multirow{2}{*}{$\delp_3(0,0,0)$} & \multirow{2}{*}{$\frac{(-1)^{d_1+d_2+d_3+1}}{d_1 d_2 d_3} \binom{d_1}{d_0-d_2}\binom{d_2}{d_0-d_3}\binom{d_3}{d_0-d_1}$} & $d_0^2-(d_1+d_2 +d_3)d_0$ \\
   & & $+d_1 d_2+d_1 d_3+d_2 d_3$ \\ 
    \Xhline{2\arrayrulewidth}
    $\bbF_0(0,0,0,0)$ & 1 & $d_1^2 d_2^2$ \\ \hline
    \end{tabular}
    \caption{Local Gromov--Witten invariants of nef Looijenga pairs.}
    \label{tab:localgw}
\end{table}

\section{Log Gromov--Witten theory}
\label{sec:loggw}

\subsection{Log Gromov--Witten invariants of maximal tangency}
\label{sec:loggwmaxtang}

Let $Y(D)$ be an $l$-component log~CY surface with maximal boundary. 
We endow $Y$ with the divisorial log structure coming from
$D$.  The log structure is used to
impose tangency conditions along the components $D_j$ of $D$.
In this paper we will be looking
at genus $g$ stable maps into $Y$ of class $d\in\hhh_2(Y,\ZZ)$ that meet each component $D_j$ in
one point of maximal tangency $d \cdot D_j$. The appropriate moduli space $\ol{\mmm}^{\log}_{g,m}(Y(D),d)$ of maximally
tangent basic stable log maps was constructed in all generality in
\cites{GS13,Chen14,AbramChen14}.

There are tautological classes $\psi_i:=c_1(L_i)$ for $L_i$ the $i^{\rm ith}$ tautological line
bundle on $\ol{\mmm}^{\log}_{g,m}(Y(D),d)$ whose fibre at $[f:(C,x_1,\dots,x_m)\to Y]$
is the cotangent line of $C$ at $x_i$. Let $\ev_i$ be the
evaluation map at the $i^{\rm th}$ marked point, and for $\pi : \mathcal{C} \to \ol{\mmm}^{\log}_{g,m}(Y(D),d)$ the universal curve with relative dualising sheaf $\omega_\pi$, denote by $\mathbb{E}\coloneqq\pi_*\omega_\pi$ the Hodge bundle, which is a rank $g$ vector bundle on $\ol{\mmm}^{\log}_{g,m}(Y(D),d)$. The $g^{\rm th}$ {\it lambda class} is its top Chern class $\lambda_g\coloneqq c_g(\mathbb{E})$.

We will be concerned with the virtual log~GW count count of genus $g$ curves in $Y$ of degree $d$ meeting $D_j$ in one point of maximal tangency $d\cdot D_j$, passing through $l-1$ general points of $Y$ and with insertion $\lambda_g$:
\begin{equation}\label{eq:def-Nlog-g-d}
N^{\rm log}_{g,d}(Y(D)):=\int_{[\ol{\mmm}^{\log}_{g,l-1}(Y(D),d)]^{\rm vir}} (-1)^g \, \lambda_g \, \prod_{j=1}^{l-1} \ev_j^*([{\rm pt}]).
\end{equation}
%
Furthermore, we will denote by $N^{\rm log, \psi}_{0,d}(Y(D))$ the genus~0 log~GW invariants of maximal tangency passing through 1 general point of $Y$ with psi class to the power $l-2$:
\begin{equation}\label{eq:def-Nlog-psi-g-d}
N^{\rm log, \psi}_{0,d}(Y(D)):=\int_{[\ol{\mmm}^{\log}_{0,1}(Y(D),d)]^{\rm vir}} \ev_1^*([{\rm pt}])\cup\psi_1^{l-2}.
\end{equation}
It will be useful in the following to define all-genus generating functions for the logarithmic invariants of $Y(D)$ at fixed degree,
\beq
\mathsf{N}^{\rm log}_d(Y(D))(\hbar) \coloneqq \frac{1}{\left( 
2 \sin \left( \frac{\hbar}{2} \right) \right)^{l-2}}
\sum_{g \geqslant 0} N^{\rm log}_{g,d} \hbar^{2g-2+l} \,.
\label{eq:GWexp}
\eeq
%


In the setting of Proposition \ref{prop:logCY}, it follows from \cite{AW} that
$N^{\rm log}_{g,d}(Y(D))$, resp.\ $N^{\rm log, \psi}_{0,d}(Y(D))$, equals the log GW invariant of $(\wt{Y},\wt{D})$, of class
$\varphi^*d$, with maximal tangency along each of the strict transforms of
$D_i$, not meeting the other boundary components and meeting $l-1$ general points
of $\wt{Y}$, resp.\ 1 point with psi class to the power $l-2$. The above numbers are deformation invariant in log~smooth families \cite{ManRu}.
The genus $0$ invariants 
$N^{\rm log}_{0,d}(Y(D))$
and 
$N^{\rm log, \psi}_{0,d}(Y(D))$
are moreover enumerative 
by \cite[Proposition 6.1]{Man19} since $Y$ has dimension 2 (and in particular is a cluster variety).

\subsection{Scattering diagrams}\label{sec:theta}

Our main tool for the calculation of \eqref{eq:def-Nlog-g-d}--\eqref{eq:def-Nlog-psi-g-d} will be their associated quantum scattering diagrams and quantum broken lines \cites{mandel2015scattering, bousseau2018quantum, MR4048291, davison2019strong, bousseau2020strong}. In the classical limit, in dimension 2 this is treated in \cites{GPS,GHKlog,Gro11} and in full generality in \cites{GS11,GHS16}. The quantum scattering diagram consists of an affine integral manifold $B$ and a collection of walls $\mathfrak{d}$ with wall-crossing funtions $f_{\mathfrak{d}}$. The latter are functions on open subsets of the mirror.

Let $\pi : (\wt{Y},\wt{D}) \longrightarrow (\ol{Y},\ol{D})$  
be a toric model as in Proposition \ref{prop:logCY} with $s$ interior blowups. Up to deformation, we may assume that the blowup points are disjoint. Note that 
$s=\chi_{\rm top}\left(Y\setminus D\right)=\chi_{\rm top}\left(\wt{Y}\setminus\wt{D}\right)$ is an invariant of the interior.
We construct an affine integral manifold $B$ from $\pi$ as follows. First, we start with the fan of $(\ol{Y},\ol{D})$. 
Then, for every interior blowup, we add a focus-focus singularity in the direction of the corresponding ray. In practice, we introduce cuts connecting the singularities to infinity and we use charts to identify the complements of the cuts with an open subset of $\bbR^2$.

Denote by $\delta_1,\dots,\delta_s$ the focus-focus singularities and let $B(\ZZ)$ be the set of integral points of $B\setminus\{\delta_1,\dots,\delta_s\}$. In the limit where the singularities are sent to infinity, $B(\ZZ)$ can be identified with the integral points of the fan of $(\ol{Y},\ol{D})$. The singularity $\delta_j$ corresponds to an interior blowup on a toric divisor $D(\delta_j)$ of $(\ol{Y},\ol{D})$ with exceptional divisor $\cE_j$. Viewing the ray of the fan of $(\ol{Y},\ol{D})$ corresponding to $D(\delta_j)$ as going from $(0,0)$ to infinity, denote by $\rho_j$ its primitive direction.

Each $\delta_j$ creates a quantum wall $\mathfrak{d}_j$ propagating into the direction $-\rho_j$ and decorated with the wall-crossing function $f_{\mathfrak{d}_j}:=1+t_j z^{\rho_j}$, where $t_j=t^{[\cE_j]}$ is a formal variable keeping track of the exceptional divisor and $z^{\rho_j}$ is the tangent monomial $x^ay^b$ if $\rho_j=(a,b)$. Note that the wall also propagates into the $\rho_j$ direction (decorated with $1+t_j z^{-\rho_j}$), but that part of the scattering diagram is not relevant to us. 

When two walls meet, this creates scattering: up to perturbation, we may assume that at most two walls $\mathfrak{d}_j$ and $\mathfrak{d}_k$ come together at one point, which in the following is taken to be the origin for simplicity. We refer to \cite{GPS} for the general case and only describe the explicit result in the two cases relevant to us:

\bit
\item $\det (\rho_j, \rho_k)=\pm 1$ (\emph{simple scattering}): the scattering algorithm draws an additional quantum wall $\mathfrak{d}$ in the direction $-\rho_j-\rho_k$ decorated with the function $1+t_jt_kz^{\rho_j+\rho_k}$.
\item $\det (\rho_j, \rho_k)=\pm 2$ (\emph{infinite scattering}): the algorithm creates a central quantum wall $\mathfrak{d}$ in the direction $-\rho_j-\rho_k$ decorated with the function
\beq
\prod_{\ell=-\frac{1}{2}({\rm ind}(\rho_j+\rho_k)-1)}^{\frac{1}{2}({\rm ind}(\rho_j+\rho_k)-1)}
(1-q^{-\frac{1}{2}+\ell} t_jt_kz^{\rho_j+\rho_k})^{-1}
(1-q^{\frac{1}{2}+\ell} t_jt_kz^{\rho_j+\rho_k})^{-1},
\eeq
where ${\rm ind}(\rho_j+\rho_k)$ is the index of $\rho_j+\rho_k$. 
We then add quantum walls $\mathfrak{d}_1,\dots,\mathfrak{d}_n,\dots$ in the directions $(n+1)\rho_j+n\rho_k$ decorated with functions
\beq
1+t_j^{n+1}t_k^nz^{(n+1)\rho_j+n\rho_k},
\eeq
for $n\geq0$, as well as quantum walls $\prescript{}{1}{\mathfrak{d}},\dots,\prescript{}{n}{\mathfrak{d}},\dots$ in the directions $n\rho_j+(n+1)\rho_k$ decorated with functions
\beq
1+t_j^{n}t_k^{n+1}z^{n\rho_j+(n+1)\rho_k},
\eeq
for $n\geq0$.
\eit
The classical scattering algorithm is recovered in the classical limit 
$q^{\frac{1}{2}}=1$. Only the central quantum wall in the case $\det (\rho_j, \rho_k)=\pm 2$ is different from its classical version, for which the wall-crossing function specialises to $(1-t_jt_kz^{\rho_j+\rho_k})^{-2 \, {\rm ind}(\rho_j+\rho_k)}$.

If $\mathfrak{u}$ and $\mathfrak{u}'$ are adjacent chambers of $B$ separated by the quantum wall $\mathfrak{d}$ decorated with $f_{\mathfrak{d}}$, we can define a {\it quantum wall-crossing transformation} $\theta_{\mathfrak{d}}$ from $\mathfrak{u}$ to $\mathfrak{u}'$ as follows.  Denote by $n_{\mathfrak{d}/\mathfrak{u}}$ the primitive orthogonal vector pointing from $\mathfrak{d}$ into $\mathfrak{u}$. 
Let $m$ be such that $\langle n_{\mathfrak{d}/\mathfrak{u}}, m \rangle \geq 0$.
For a polynomial $a$ in the variables $t_j$, consider an expression $az^m$, which we think of as a function on $\mathfrak{u}$. Then, writing
\beq f_{\mathfrak{d}} = \sum_{r \geq 0} c_r z^{r \rho_{\mathfrak{d}}} \,,\eeq
where $-\rho_{\mathfrak{d}}$ is the primitive direction of 
$\mathfrak{d}$, 
\beq
\theta_{\mathfrak{d}} \, : \, az^m \longmapsto a  z^m
\prod_{\ell=-\frac{1}{2}(\langle n_{\mathfrak{d}/\mathfrak{u}},m \rangle -1)}^{\frac{1}{2}(\langle n_{\mathfrak{d}/\mathfrak{u}},m \rangle -1)}
\left(
\sum_{r \geq 0} c_r q^{r \ell} z^{r \rho_{\mathfrak{d}}} \right) \,.
\eeq
Note that in the classical limit $q^{\frac{1}{2}}=1$, we recover the formula for the classical wall-crossing transformation, which is
%
$\theta^{cl}_{\mathfrak{d}}  :  az^m \longmapsto f_{\mathfrak{d}}^{\langle n_{\mathfrak{d}/\mathfrak{u}},m \rangle} a z^m.$
%
Writing $\theta_{\mathfrak{d}}(az^m) = \sum_i a_i z^{m_i}$, any summand $a_i z^{m_i}$ is called \emph{a result of quantum transport of $az^m$} from $\mathfrak{u}$ to $\mathfrak{u}'$. 

The final object we will need is the algebra of quantum broken lines associated to the scattering diagram, which we describe in the generality needed here (see \cite{GHS16} for full details in the classical limit).
Let $B_0:=B\setminus\{\delta_1,\dots,\delta_s,\mathfrak{d}_j\cap\mathfrak{d}_k \; \big{|} \; \forall j,k\}$.
Let $z^m$ be an \emph{asymptotic monomial}, in our case this means that $m=(a,b)\neq(0,0)$, and let $p\in B$. Then a \emph{quantum broken line} $\beta$ with asymptotic monomial $z^m$ and endpoint $p$ consists of:
\begin{enumerate}
\item a directed piece-wise straight path in $B_0$ of rational slopes, coming from infinity in the direction $-m$, bending only at quantum walls and ending at $p$;
\item a labelling of the initial ray by $L_1$ and the successive line segments in order by $L_2,\dots,L_s$, where $p$ is the endpoint of $L_s$;
\item if $L_i\cap L_{i+1}\in \mathfrak{d}_{i}$, then, iteratively defined from $1$ to $s$, the assignment of a monomial $a_iz^{m_i}$, where
\begin{itemize}
\item $a_{1} z^{m_{1}} = z^m$,
\item $a_{i+1}z^{m_{i+1}}$ is a result of the quantum transport of $a_{i}z^{m_{i}}$ across $\mathfrak{d}_{i}$,
\item $L_i$ is directed in the direction $-m_i$.
\end{itemize}
\end{enumerate}
Note that if $n_{\mathfrak{\fd}_i/L_i}$ is the primitive orthogonal vector to $\mathfrak{\fd}_i$ pointing into the half-plane containing $L_i$, then, as $L_i$ is directed in the direction $-m_i$, we have $\langle n_{\mathfrak{\fd}_i/L_i},m_i \rangle \geq 0$, and so the quantum transport of $a_{i}z^{m_{i}}$ across $\mathfrak{d}_{i}$ is indeed well-defined.
We call $a_{\fend}z^{m_{\fend}}=a_sz^{m_s}$ the \emph{end monomial} of $\beta$ and $a_{\fend}$ the \emph{end coefficient} of $\beta$.

If $z^m$ is an asymptotic monomial, the \emph{theta function} $\vartheta_m$ is the sum of the end monomials of all broken lines with asymptotic monomial $z^m$ and ending at $p$. Note that a priori $\vartheta_m$ depends on $p$, but it is one of the main results of \cite{GHS16} that it is constant in chambers and transforms from chamber to chamber according to the wall-crossing transformations.

We first describe the classical algebra of theta functions, i.e.\ we set $q^{\frac{1}{2}}=1$. For $A$ an element in the algebra of theta functions, we denote by $\langle A, \vartheta_m \rangle$ the coefficient of $\vartheta_m$ in $A$; note that $\langle A, \vartheta_{m} \rangle$ is a polynomial in the $t_j$. 
Then the identity component $\langle \vartheta_{m^1} \cdot \vartheta_{m^2}, \vartheta_{0} \rangle$ is given as the sum of products of end coefficients $a^1_{\fend}a^2_{\fend}$ over all broken lines $\beta_1$ with asymptotic monomial $z^{m^1}$ and $\beta_2$ with asymptotic monomial $z^{m^2}$ such that $m^1_{\fend}=-m^2_{\fend}$. The identity component $\langle \vartheta_{m^1} \cdot \vartheta_{m^2} \cdot \vartheta_{m^3}, \vartheta_{0} \rangle$ is given as the sum of products of end coefficients $a^1_{\fend}a^2_{\fend}a^3_{\fend}$ over all broken lines $\beta_1$, $\beta_2$, $\beta_3$, with asymptotic monomials $z^{m^1}$, $z^{m^2}$, $z^{m^3}$ and such that $m^1_{\fend}+m^2_{\fend}+m^3_{\fend}=0$.

For $(Y(D=D_1+\cdots + D_l))$, consider the scattering diagram associated to a toric model $\pi$ coming from a diagram as in Proposition \ref{prop:logCY}:
\beq\label{eq:prop24}
\xymatrix{
 & \wt{Y}(\wt{D}) \ar[ld]_\varphi \ar[rd]^\pi & \\
Y(D) & &  \ol{Y}(\ol{D})
}
\eeq
Then the proper transform and pushforward of $D_j$ is a toric divisor in $\ol{Y}$ corresponding to a ray in $B$. Up to reordering the indices, we assume that the ray corresponding to $D_j$ is directed by $\rho_j$.

\begin{prop}[\cite{Man19}]\label{prop:calc-log-gw}

Let $Y(D)$ be an $l$-component log Calabi--Yau surface of maximal boundary.  Let $d\in\hhh_2(Y,\ZZ)$ be an effective curve class and write $e_j:=d\cdot \varphi_* \cE_j$ for $j=1,\dots,s$, where $\varphi$ is as in \eqref{eq:prop24}.
\begin{itemize}
\item Assume that $l=2$. Set $m^1=(d\cdot D_1)\rho_1$ and $m^2=(d\cdot D_2)\rho_2$. Then $N_{0,d}^{\rm log}(Y(D))$ is the coefficient of $\prod_{j=1}^{s} t_j^{e_j}$ in $\langle \vartheta_{m^1} \cdot \vartheta_{m^2}, \vartheta_{0} \rangle$.
\item Assume that $l=3$. Set $m^1=(d\cdot D_1)\rho_1$, $m^2=(d\cdot D_2)\rho_2$ and $m^3=(d\cdot D_3)\rho_3$. Then $N^{\rm log,\psi}_{0,d}(Y(D))$ is the coefficient of $\prod_{j=1}^{s} t_j^{e_j}$ in $\langle \vartheta_{m^1} \cdot \vartheta_{m^2} \cdot \vartheta_{m^3}, \vartheta_{0} \rangle$. 
\end{itemize}
\end{prop}


We return to the algebra of quantum theta functions.
For every $m^1, m^2$ and $p \in B_0$, denote by $C_{m^1,m^2}$ the polynomial in the variables 
$t_j$ with coefficients in $\Z[q^{\pm \frac{1}{2}}]$ given as the sum of products of end coefficients $a^1_{\fend}a^2_{\fend}$ over all quantum broken lines $\beta_1$ with asymptotic monomial $z^{m^1}$ and $\beta_2$ with asymptotic monomial $z^{m^2}$, with common endpoint $p$ and such that $m^1_{\fend}=-m^2_{\fend}$. The polynomial $C_{m^1,m^2}$ is independent of the choice of 
$p \in B_0$. 


\begin{prop}\label{prop:calc-log-gw-q}
Let $Y(D)$ be an $l$-component log Calabi--Yau surface of maximal boundary. Let $d\in\hhh_2(Y,\ZZ)$ be an effective curve class and write $e_j:=d\cdot \varphi_* \cE_j$ for $j=1,\dots,s$, where $\varphi$ is as in \eqref{eq:prop24}.
\begin{itemize}
    \item Assume that $l=2$.
    Set $m^1=(d\cdot D_1)\rho_1$ and $m^2=(d\cdot D_2)\rho_2$. Then after the change of variables 
$q=e^{i \hbar}$, the series 
\beq \mathsf{N}^{\rm log}_d(Y(D))(\hbar) = \sum_{g \geqslant 0} N^{\rm log}_{g,d} \hbar^{2g}\eeq
is the $\hbar$-expansion of the $q$-polynomial which is the coefficient of  $\prod_{j=1}^{s} t_j^{e_j}$ in $C_{m^1,m^2}$.
\item Assume that $l=3$. Set $m^1=(d\cdot D_1)\rho_1$, $m^2=(d\cdot D_2)\rho_2$ and $m^3=(d\cdot D_3)\rho_3$.
Then after the change of variables 
$q=e^{i \hbar}$, the series 
\beq \mathsf{N}^{\rm log}_d(Y(D))(\hbar) = \frac{1}{ 
2 \sin \left( \frac{\hbar}{2} \right)} \sum_{g \geqslant 0} N^{\rm log}_{g,d} \hbar^{2g+1}\eeq
is the $\hbar$-expansion of the $q$-polynomial obtained as the sum over all quantum broken lines 
$\beta_1$ with asymptotic monomial $z^{m_1}$,
$\beta_2$ with asymptotic monomial $z^{m_2}$,
and $\beta_3$ with asymptotic monomial $z^{m_3}$, with common endpoint and such that 
$m^1_{\fend}+m^2_{\fend}+m^3_{\fend}=0$,
of 
\begin{equation}
\frac{[|\det (m^1_{\fend}, m^2_{\fend})|]_q}{[1]_q}a^1_{\fend}a^2_{\fend}a^3_{\fend}
\end{equation}
\end{itemize}
where $a^i_{\fend}z^{m^i_{\fend}}$ are the end monomials of the broken lines $\beta_i$ and where the $q$-integers $[\cdot]_q$ are defined in \eqref{eq:q-int} below. 

\end{prop}

\begin{proof}
We only give a sketch of the proof.
The proof in  \cite{Man19} of the genus $0$ case, Proposition \ref{prop:calc-log-gw}, 
relies on studying the degeneration to a toric situation considered in 
\cite{GPS} and then on using 
a toric tropical correspondence theorem \cites{NiSi, ManRu}. In the higher genus case, the study of the degeneration of \cite{GPS} is done using the techniques introduced in 
\cite{bousseau2018quantum}, and then the result follows from the toric 
tropical correspondence theorem for higher genus log Gromov--Witten invariants with 
$\lambda_g$-insertion proved in \cite{MR3904449}.
\end{proof}

\subsubsection{Binomials and $q$-binomial coefficients}
In our applications of Propositions \ref{prop:calc-log-gw}
and \ref{prop:calc-log-gw-q}, we will mostly consider (quantum) broken lines bending along 
(quantum) walls $\mathfrak{f}_{\fd}$ decorated by a function of the form 
\beq f_{\mathfrak{\fd}} = 1+tz^{\rho_{\mathfrak{d}}} \,,\eeq
where $-\rho_{\mathfrak{d}}$ is the primitive direction of 
$\mathfrak{d}$.
By the binomial theorem, we have 
\beq f_{\mathfrak{d}}^{\langle n,m \rangle}
= (1+tz^{\rho_{\mathfrak{d}}})^{\langle n,m \rangle}
= \sum_{k=0}^{\langle n,m \rangle} \binom{\langle n,m \rangle}{k}
t^k z^{k \rho_{\mathfrak{d}}}\,. \eeq
Therefore, each application of transport across such a wall will produce a binomial coefficient,  
and
so our genus $0$ log Gromov--Witten invariants will be product of binomial coefficients.
By the $q$-binomial theorem, we have 
\beq
\prod_{\ell=-\frac{1}{2}(\langle n,m \rangle -1)}^{\frac{1}{2}(\langle n,m \rangle -1)}
\left(1+tq^{\ell} z^{ \rho_{\mathfrak{d}}} \right) 
= \sum_{k=0}^{\langle n,m \rangle} \qbinom{\langle n,m \rangle}{k}_q
t^k z^{k \rho_{\mathfrak{d}}}\,,
\eeq
where the $q$-binomial coefficients 
\beq \qbinom{N}{k}_q \coloneqq \frac{[N]_q!}{[k]_q![(N-k)]_q!} \eeq
are defined in terms of the $q$-factorials 
\beq [n]_q! \coloneqq \prod_{j=1}^n [j]_q \,,\eeq
where the $q$-integers are 
\begin{equation}\label{eq:q-int}
[n]_q:=q^{\frac{n}{2}}-q^{-\frac{n}{2}}\, .
\end{equation}
It follows that the formulas for the higher genus log Gromov--Witten invariants
$\mathsf{N}^{\rm log}_d(Y(D))(\hbar)$ will be obtained by replacing 
binomial coefficients by $q$-binomial coefficients in the
formulas for the genus $0$ invariant $N_{0,d}^{\rm log}(Y(D))$.

\subsection{Log Gromov--Witten invariants under interior blow-up}

\begin{prop}[Blow-up formula for log GW invariants]\label{prop:blowuplog}
Let $Y(D)$ be an $l$-component log CY surface with maximal boundary. 
Let $\pi : Y'(D') \to Y(D)$ be the $l$-component log CY surface with maximal boundary
obtained by an interior blow-up at a general point of $D$ with exceptional divisor $E$. Let $d$ be a curve class of $Y(D)$ and let $d'\coloneqq \pi^* d$. Then
\beq
N^{\rm log}_{g,d}(Y(D)) = N^{\rm log}_{g,d'}(Y'(D'))
\label{eq:logblup}
\eeq
and
\beq
N^{\rm log, \psi}_{0,d}(Y(D)) = N^{\rm log, \psi}_{0,d'}(Y'(D'))
\eeq
\label{prop:logblup}
\end{prop}

\begin{proof}
Let $D_j$ be the irreducible component of $D$ containing the point that we blow-up.
We consider the degeneration of $Y(D)$ to the normal cone of $D_j$: the fibre over any point of $\A^1-\{0\}$  is 
$Y(D)$ and the special fibre over $\{0\}$ has two irreducible components, which are isomorphic to 
$Y(D)$ and a $\PP^1$-bundle $\PP_j$ over $D_j$, and are glued together along a copy of $D_j$. We blow-up a section of the closure of $D_j \times (\A^1-\{0\})$ in the total space of the degeneration. The resulting family has fibre $Y'(D')$ over any point of
$\A^1-\{0\}$, and a special fibre over $\{0\}$ given by the union of two irreducible components, which are isomorphic to $Y(D)$ and to the blow-up 
$\tilde{\PP}_j$ of $\PP_j$ at one point. 
We compare the invariants $N_{g,d}^{\rm log}$ and $N_{0,d}^{\rm log,\psi}$ of 
$Y(D)$ and $Y'(D')$ using this degeneration. Following the 
general strategy of Section 5 of \cite{bousseau2018quantum}, we obtain that the invariants 
of $Y(D)$ and $Y'(D')$ only differ by a multiplicative factor coming from multiple covers of a fibre of $\tilde{\PP}_j \rightarrow D_j$. By deformation invariance, we can assume that this fibre is a smooth $\PP^1$-fibre, with trivial normal bundle in 
$\tilde{\PP_j}$. Therefore, the correction factor is an integral over a moduli space of 
stable log maps to $\PP^1$ with extra insertion of the class 
$e(\hhh^1(C,\cO_C))=(-1)^g \lambda_g$. Because our genus $g$ invariants already contain an insertion of $\lambda_g$ and $\lambda_g^2=0$ for $g>0$ by Mumford's relation \cite{MR717614}, the correction factor only receives contributions from genus $0$. The genus $0$ corrections involves 
degree $d \cdot D_j$ stable log maps to $(\PP^1, \{0\} \cup \{ \infty \})$, fully ramified over $0$ and $\infty$. The corresponding moduli space is a point with an automorphism group of order $d \cdot D_j$ and so contributes $1/(d \cdot D_j)$. Because of the extra $(d \cdot D_j)$ multiplicity factor in the degeneration formula, the total multiplicative correction factor is $1$.
\end{proof}

\begin{rmk} A more visual way to see that \cref{prop:logblup} holds, at least in genus 0 and with no $\psi$-classes, is to notice that the invariants $N^{\rm log}_{0,d}(Y(D))$ are enumerative: blowing up a point away from the curves does not change the local geometry, and the counts remain the same. 
\end{rmk}

As a consequence of \cref{prop:blowuplog}, if we calculate $N^{\rm log}_{g,d}(Y(D))$, resp.\ $N^{\rm log, \psi}_{0,d}(Y(D))$, for all $g$, $d$, then we will know the invariants for all interior blow-downs of $Y(D)$. Therefore it is enough to calculate the invariant for the cases of highest Picard rank in \cref{prop:class,prop:class2}. In the following Section, we calculate the higher genus log invariants $\mathsf{N}^{\rm log}_d(Y(D))(\hbar)$ for all tame Looijenga pairs: using \cref{prop:blowuplog}, it is enough to consider the pairs
$\delp_3(1,1)$, $\delp_3(0,0,0)$ and $\bbF_0(0,0,0,0)$, which are treated in
\cref{prop:dp311,thm:log_dp3_0_0_0,thm:logf0_0000}. 

For non-tame pairs, the genus 0 invariants can be obtained by combining the log-local correspondence of \cref{thm_log_local} and \eqref{eq:dP5} in \cref{thm:dP5} giving the local invariants. For quasi-tame pairs we furthermore make the following general conjecture for the higher genus invariants $\mathsf{N}^{\rm log}_d(Y(D))(\hbar)$.

\begin{conj} \label{conj:log_defo}
Let $Y(D)$ and $Y'(D')$ be nef $2$-components log CY surfaces with maximal boundary such that the corresponding local geometries $E_{Y(D)}$ and 
$E_{Y'(D')}$ are deformation equivalent. Then, under suitable identification of $d$, we have 
\beq
\left( \prod_{j=1}^{l=2} [d \cdot D_j']_q \right) 
\mathsf{N}^{\rm log}_d(Y(D))(\hbar)=
\left( \prod_{j=1}^{l=2} [d \cdot D_j]_q \right)
\mathsf{N}^{\rm log}_d(Y'(D'))(\hbar) \,.
\eeq
\end{conj}

\cref{conj:log_defo} holds in the genus $0$, i.e.\ $q^{\frac{1}{2}}=1$ limit, as a corollary of 
the log-local correspondence given by \cref{thm_log_local} and of the deformation invariance of local Gromov--Witten invariants. In higher genus, \cref{conj:log_defo} translates to conjectural, new non-trivial $q$-binomial identities: see e.g. \cref{conj:dp1binom} for the cases of $\delp_1(0,4)$ and $\fzero(0,4)$.


\subsection{Toric models: $l=2$}

Extending \cite[Section 5]{Bou18} we find toric models for all $l=2$ nef log Calabi--Yau surfaces except for $\fzero(2,2)$, which we leave as an exercise to the reader. For each toric model, we draw the corresponding fans with focus-focus singularities.
Note that by \cite[Lemma 2.10]{Fri}, a log Calabi--Yau surface with maximal boundary $(\ol{Y},\ol{D})$ is toric if the sequence of self-intersection numbers of irreducible components of $\ol{D}$ is realised as the sequence of self-intersection numbers of toric divisors on a toric surface. Once we have the toric models, we calculate the part of the scattering diagram relevant to us, and by \cref{prop:calc-log-gw} the relevant structural coefficients for the multiplication of theta functions yield the maximal tangency log Gromov--Witten invariants. 

\subsubsection{Tame pairs: simple scattering} 

By \cref{prop:blowuplog}, it suffices to consider the case $Y(D)=\delp_3(1,1)$.
Start with $\ptwo(1,4)$. The anticanonical decomposition of $D$ is given by $D_1$ a line and $D_2$ a smooth conic not tangent to $D_1$. For notational convenience in what follows we will identify $D_1$, $D_2$, $F_1$ and $F_2$ with their strict transforms, resp.\ push-forwards, under blow-ups, resp.\ blow-downs.

Denote by $\mathrm{pt}$ one of the intersection points of $D_1$ and $D_2$ and by $L$ the line tangent to $D_2$ at $\mathrm{pt}$. We blow up $\mathrm{pt}$ leading to the exceptional divisor $F_1$. We further blow up the intersection of $F_1$ with $D_2$ and write $F_2$ for the exceptional divisor. Denote the resulting log Calabi--Yau surface with maximal boundary $\left(\wt{\ptwo(1,4)},\wt{D}\right)$, where $\wt{D}$ is the strict transform of $D$.

The toric model $\left(\ol{\ptwo(1,4)},\ol{D}\right)$ is given by blowing down the strict transform of $L$, resulting in $\ol{\ptwo(1,4)}=\ftwo$ and $\ol{D}=D_1\cup F_1 \cup F_2 \cup D_2$, with $F_1$ the $(-2)$-curve of $\ftwo$, $D_2$ a section of self-intersection $2$, and $D_1, F_2$ linearly equivalent to fibre classes. Labeling the toric boundary divisors with their self-intersections, we obtain the diagram in \cref{fig:p2scatt}.

\smallskip

\begin{figure}[h]
\begin{center}
\begin{tikzpicture}[smooth, scale=1.2]
\draw[step=1cm,gray,very thin] (-2.5,-2.5) grid (2.5,2.5);
\draw[thick] (0,0) to (-2.5,0);
\draw[thick] (0,-2.5) to (0,2.5);
\draw[thick] (0,0) to (1.25,2.5);
\node at (-0.55,-1.7) {$D_2(2)$};
\node at (1.5,1.7) {$D_1(0)$};
\node at (-0.55,1.7) {$F_1(-2)$};
\node at (-1.5,0.3) {$F_2(0)$};
\end{tikzpicture}
\end{center}
\caption{The toric model of $\bbP^2(1,4)$.}
\label{fig:p2scatt}
\end{figure}

To obtain the toric model for $\delp_3(1,1)$, we need to blow up a non-toric point on $F_2$ (and thus reproducing $L$), and three non-toric points on $D_2$. Tropically, this amounts to introducing a focus-focus singularity on the ray of $F_2$ and three on on the ray of $D_2$ as in \cref{fig:dpr3}. Walls emanate out of these focus-focus singularities. While they propagate into two directions, for our calculations only one direction matters (the other ray being close to infinity and thus non-interacting). We perturb the focus-focus singularities on $D_2$ horizontally.
\begin{figure}[!h]
\begin{tikzpicture}[smooth, scale=1.2]
\draw[step=1cm,gray,very thin] (-2.5,-2.5) grid (2.5,2.5);
\draw[thick] (0,0) to (-2.5,0);
\draw[thick] (0,-2.5) to (0,2.5);
\draw[thick] (0,0) to (1.25,2.5);
\node at (-0.3,-1.7) {$D_2$};
\node at (1.3,1.7) {$D_1$};
\node at (-0.3,1.7) {$F_1$};
\node at (-1.5,0.2) {$F_2$};
\node at (-2.3,0) {$\times$};
%
\node at (0,-1) {$\times$};
\node at (0,-1.5) {$\times$};
\node at (0,-2) {$\times$};
%
\end{tikzpicture}
\caption{The toric model of $\delp_3(1,1)$}
\label{fig:dpr3}
\end{figure}

The cone of curves is generated by $H-E_i-E_j$ for $1\leq i < j \leq r $ and the $E_i$. In particular, any curve class $d\in\hhh_2(\delp_3,\bbZ)$ can be written as $d=d_0(H-E_1-E_2-E_3)+d_1E_1+d_2E_2 +d_3E_3$.

\begin{thm}\label{prop:dp311}
We have
\begin{equation}
\mathsf{N}^{\rm log}_d(\delp_3(1,1))(\hbar) = \qbinom{d_3}{d_0-d_1}_q 
\qbinom{d_3}{d_0-d_2}_q
\qbinom{d_0}{d_3}_q
\qbinom{d_1+d_2+d_3-d_0}{d_3}_q \,,
\label{eq:dP311log}
\end{equation}
where $q=\re^{\ri \hbar}$.
\end{thm}



\begin{proof}
Write $t=z^{[L]}$ and let $t_i=z^{[E_i]}$. Since $D_1=H$, $D_2=2H-E_1-E_2-E_3$, we have the following intersection multiplicities:
$d \cdot D_1= d_0$, 
$d \cdot D_2=d_1+d_2+d_3-d_0$, 
$d \cdot E_i=d_0-d_i$.
All of the scattering is simple. The initial wall-crossing functions are drawn in Figure \ref{fig:dp3scatt}, and all successive functions are easily obtained. We have two broken lines, one coming from the $D_1$-direction with attaching monomial $(xy^2)^{d\cdot D_1}$ and one coming from the $D_2$-direction with attaching monomial $(y^{-1})^{d\cdot D_2}$. Provided we choose our endpoint $p$ to be sufficiently far into the $x$-direction, Figure \ref{fig:dp3scatt} contains all the relevant walls. We start from the broken line coming from the $D_2$ direction and summarise the wall-crossing functions attached to the walls it meets:

\begin{center}
\begin{tabular}{cccccc}
$\circleed{1}$ & $1+tx^{-1}$ & $\circleed{2}$ & $1+tt_3x^{-1}y^{-1}$ & $\circleed{3}$ & $1+tt_2x^{-1}y^{-1}$  \\ $\circleed{4}$ & $1+tt_1x^{-1}y^{-1}$ & 
 $\circleed{5}$ & $1+t^2t_1t_2t_3x^{-2}y^{-3}$  & &
\end{tabular}
\end{center}

Crossing these walls leads to $y^{d_0-d_1-d_2-d_3}$ mapping to
\begin{align*}
&\sstyle{ \qbinom{d_1+d_2+d_3-d_0}{k}_q \qbinom{d_1+d_2+d_3-d_0-k}{k_1}_q \qbinom{d_1+d_2+d_3-d_0-k}{k_2}_q \qbinom{d_1+d_2+d_3-d_0-k}{k_3}_q \qbinom{2d_1+2d_2+2d_3-2d_0-3k-k_1-k_2-k_3}{k_4}_q  } \\
&  \cdot \, \sstyle{ t^{k+k_1+k_2+k_3+2k_4} \, t_3^{k_1+k_4} \, t_2^{k_2+k_4} \, t_1^{k_3+k_4} \, x^{-k-k_1-k_2-k_3-2k_4} \, y^{d_0-d_1-d_2-d_3-k_1-k_2-k_3-3k_4} }.
\end{align*}
The intersection multiplicities with the divisors impose the following conditions:
\beq
k+k_1+k_2+k_3+2k_4=d_0, \quad k_1+k_4 = d_0-d_3, \quad k_2+k_4 = d_0-d_2, \quad k_3+k_4 = d_0-d_1.
\eeq
Choose as indeterminate $k$. For the coefficient to be non-zero, $0\leq k \leq d_1+d_2+d_3-d_0$. Then
\begin{center}
\begin{tabular}{cc}
$k_4=k+2d_0-d_1-d_2-d_3$, & $k_1 = d_1+d_2-d_0-k$, \\
$k_2 = d_1+d_3-d_0-k$, & $k_3 = d_2+d_3-d_0-k$.
\end{tabular}
\end{center}
Hence the sum of the coefficients of the broken lines is
\begin{align*}
&\sum_{k=0}^{\sstyle{d_1+d_2+d_3-d_0}} \sstyle{ \qbinom{d_1+d_2+d_3-d_0}{k}_q \qbinom{d_1+d_2+d_3-d_0-k}{d_3}_q \qbinom{d_1+d_2+d_3-d_0-k}{d_2}_q \qbinom{d_1+d_2+d_3-d_0-k}{d_1}_q \qbinom{2d_1+2d_2+2d_3-2d_0-3k-k_1-k_2-k_3}{k_4}_q} \\
=&\sum_{k=0}^{k(d_0,d_1,d_2,d_3)} \sstyle{ \qbinom{d_1+d_2+d_3-d_0}{k}_q \qbinom{d_1+d_2+d_3-d_0-k}{d_3}_q \qbinom{d_1+d_2+d_3-d_0-k}{d_2}_q \qbinom{d_1+d_2+d_3-d_0-k}{d_1}_q \qbinom{d_0}{k+2d_0-d_1-d_2-d_3}_q},
\end{align*}
where $k(d_0,d_1,d_2,d_3):=\min\{d_0,d_1+d_2-d_0,d_1+d_3-d_0,d_2+d_3-d_0\}$.
\begin{figure}[t]
\begin{tikzpicture}[smooth, scale=1.2]
\draw[step=1cm,gray,very thin] (-2.5,-2.5) grid (6.5,7.5);
\draw[<->] (-2.5,2) to (-1.5,2);
\draw[<->] (-2,1.5) to (-2,2.5);
\node at (-1.67,2.15) {$\scriptstyle{x}$};
\node at (-2.15,2.3) {$\scriptstyle{y}$};
\draw(-2.5,0) to (3.5,0);
\draw (0,-2.5) to (0,2.5);
\draw (0,0) to (1.25,2.5);
%
\node at (-2.3,0) {$\times$};
\node at (0,-2.3) {$\times$};
\node at (1,-2.3) {$\times$};
\node at (2,-2.3) {$\times$};
\draw (0,0) to (5.3,5.3);
\draw (1,-2.5) to (1,1.75);
\draw (2,-2.5) to (2,3.75);
\draw (1,0) to (4.4,3.4);
\draw (2,0) to (3.4,1.4);
\draw (1,1) to (2.35,3.7);
\draw (2,3) to (2.2,3.9);
\draw (3,3) to (5.8,7.2);
\draw (2,2) to (3.1,4.2);
\draw (2,1) to (3.4,3.8);
\node at (-0.5,0.2) {$\sstyle{1+tx^{-1}}$};
\node[rotate=90] at (-0.15,-0.65) {$\sstyle{1+t_1y^{-1}}$};
\node[rotate=90] at (0.85,-0.65) {$\sstyle{1+t_2y^{-1}}$};
\node[rotate=90] at (1.85,-0.65) {$\sstyle{1+t_3y^{-1}}$};
\draw[<->,thick] (2.5,-2.5) to (2.5,0) to (2.6,0.6) to (3,2) to (3.6,3.6) to (4,4.5) to (5.5,7.5);
\node at (3.5,-2.2) {$\sstyle{y^{d_0-d_1-d_2-d_3}}$};
\node at (4.9,7.2) {$\sstyle{x^{d_0}y^{2d_0}}$};
%
%
\node[shape=circle,draw,inner sep=1pt] at  (2.7,-0.2) {$\sstyle{1}$};
\node[shape=circle,draw,inner sep=1pt] at  (2.8,0.5) {$\sstyle{2}$};
\node[shape=circle,draw,inner sep=1pt] at  (3.15,1.85) {$\sstyle{3}$};
\node[shape=circle,draw,inner sep=1pt] at  (3.75,3.5) {$\sstyle{4}$};
\node[shape=circle,draw,inner sep=1pt] at  (4.2,4.5) {$\sstyle{5}$};
\node at (5,6.5) {$\bullet$};
\node at (4.9,6.6) {$\sstyle{p}$};
\end{tikzpicture}
\caption{$\scatt\delp_3(1,1)$}
\label{fig:dp3scatt}
\end{figure}
Therefore, we obtain
\( \mathsf{N}^{\rm log}_d(\delp_3(1,1))(\hbar)
= \)
\beq
\sum_{k \geqslant 0}
\sstyle{\qbinom{d_1+d_2+d_3-d_0}{k}_q
\qbinom{d_1+d_2+d_3-d_0-k}{d_3}_q
\qbinom{d_1+d_2+d_3-d_0-k}{d_2}_q
\qbinom{d_1+d_2+d_3-d_0-k}{d_1}_q
\qbinom{d_0}{k+2d_0-d_1-d_2-d_3}_q}\,.
\eeq

Writing the $q$-binomial coefficients in terms of $q$-factorials, 
and changing the indexing variable 
$k \mapsto k-d_0+\frac{d_1+d_2+d_3}{2}$, we have
\beq
\frac{[d_0]_q![d_1+d_2+d_3-d_0]_q!}{[d_1]_q![d_2]_q!
[d_3]_q!} 
\sum_{k}
\sstyle{\frac{[\frac{d_1+d_2+d_3}{2}-k]_q!}{
[\frac{d_1+d_2-d_3}{2}-k]_q!
[\frac{d_1+d_3-d_2}{2}-k]_q!
[\frac{d_2+d_3-d_1}{2}-k]_q!
[k-d_0+\frac{d_1+d_2+d_3}{2}]_q!
[k+d_0-\frac{d_1+d_2+d_3}{2}]_q!}}\,.
\eeq
We can resum this explicitly using the $q$-Pfaff--Saalsch\"utz identity\footnote{Note that unlike \cite{MR2128719,ZeilSaal}, we are using $q$-factorials and 
$q$-binomial coefficients symmetric under 
$q \mapsto q^{-1}$. This explains the absence in the above expression of the power 
$q^{n^2-k^2}$ which is present in \cite[Eq.~(1q)]{ZeilSaal}.
}, under its form given in \cite[Eq.~(1q)]{ZeilSaal}:
\beq
\sum_{k }
\frac{[a+b+c-k]_q!}{
[a-k]_q!
[b-k]_q!
[c-k]_q!
[k-m]_q!
[k+m]_q!}
= \qbinom{a+b}{a+m}_q
\qbinom{a+c}{c+m}_q
\qbinom{b+c}{b+m}_q \,.
\label{eq:qPS}
\eeq
Therefore, specialising \eqref{eq:qPS} to
$a+b=d_1$, $b+c=d_3$, $a+c=d_2$, 
$a+m=d_0-d_3$, 
$b+m=d_0-d_2$, $c+m=d_0-d_1$, we have 
\beq \mathsf{N}^{\rm log}_d(\delp_3(1,1))(\hbar)
=
\frac{[d_0]_q![d_1+d_2+d_3-d_0]_q!}{[d_1]_q![d_2]_q!
[d_3]_q!} 
\qbinom{d_1}{d_0-d_3}_q
\qbinom{d_2}{d_0-d_1}_q
\qbinom{d_3}{d_0-d_2}_q \,,
\eeq
which after elementary simplifications gives \eqref{eq:dP311log}.
%
%

\end{proof}

\begin{rmk}
It follows from the above proof that 
Theorem \ref{prop:dp311} is in fact equivalent to the $q$-Pfaff--Saalsch\"utz identity.
In genus 0, Theorem \ref{thm_log_local} applied to $\delp_3(1,1)$ gives a geometric proof of Theorem \ref{prop:dp311}. Thus, we obtain a new geometric, albeit quite indirect, proof of the classical ($q=1$) Pfaff--Saalsch\"utz
identity.
\end{rmk}

\subsubsection{Non-tame pairs: infinite scattering} \cref{fig:fzero0} gives the toric model of $\fzero(0,4)$. For the other non-tame pairs, let $1 \leq r \leq 5$. Then $\dpr(0,5-r)$ is obtained from $\ptwo(1,4)$ by blowing up the first point on the line $D_1$ and the remaining $r-1$ points on the conic $D_2$. Hence we obtain the toric model of $\dpr(0,5-r)$ by adding 1 focus-focus singularity on the ray $D_1$ and $r-1$ focus-focus singularities on the ray $D_2$, as in Figure \ref{fig:dpr0}. The singularities on the ray of $D_2$ can be perturbed horizontally.

\begin{figure}[h]
\caption{$\dpr(0,5-r)$}
\label{fig:dpr0}
\begin{tikzpicture}[smooth, scale=1.2]
\draw[step=1cm,gray,very thin] (-2.5,-2.5) grid (2.5,2.5);
\draw[thick] (0,0) to (-2.5,0);
\draw[thick] (0,-2.5) to (0,2.5);
\draw[thick] (0,0) to (1.25,2.5);
\node at (-0.25,-1.7) {$D_2$};
\node at (1.2,1.7) {$D_1$};
\node at (-0.25,1.7) {$F_1$};
\node at (-1.5,0.2) {$F_2$};
\node at (-2.3,0) {$\times$};
\node at (0.75,1.5) {$\times$};
\node at (0,-0.5) {$\times$};
\node at (0,-1) {$\times$};
\node at (0,-1.5) {$\times$};
\node at (0,-2) {$\times$};
\draw[decoration={brace,mirror,raise=5pt},decorate]
  (0.1,-2.2) -- node[right=6pt] {$r\text{-}1$} (0.1,-0.3);
\end{tikzpicture}
\end{figure}

Write a curve class $d\in\hhh_2(\delp_3(0,2),\bbZ)$ as $d=d_0(H-E_1-E_2-E_3)+d_1E_1+d_2E_2+d_3E_3$. As $D_1=H-E_1$ and $D_2=2H-E_2-E_3$, $d \cdot D_1=d_1$, $d \cdot D_2=d_2+d_3$.
As $E_{\delp_3(0,2)}$ is deformation equivalent to 
$E_{\delp_3(1,1)}$ by \cref{prop:local-deform}, \cref{conj:log_defo,prop:dp311} give the following conjecture.

\begin{conj}\label{conj:dp302}
The generating function $\mathsf{N}^{\rm log}_d(\delp_3(0,2))(\hbar)$ equals
\beq
\frac{[d_1]_q [d_2+d_3]_q}{[d_0]_q [d_1+d_2+d_3-d_0]_q}\qbinom{d_3}{d_0-d_1}_q 
\qbinom{d_3}{d_0-d_2}_q
\qbinom{d_0}{d_3}_q
\qbinom{d_1+d_2+d_3-d_0}{d_3}_q \,,
\eeq
where $q=\re^{\ri \hbar}$.
\end{conj}

\cref{thm_log_local} in \cref{sec:logloc} implies that  \cref{conj:dp302} holds in the classical limit $q^{\frac{1}{2}}=1$. 
Direct scattering computation for $\delp_r(0,5-r)$
with $r>1$ are particularly daunting owing to the present of infinite scattering, and in particular the final formulas take the shape of somewhat intricate multiple $q$-sums, which \cref{conj:dp302} predicts should take a remarkably simple $q$-binomial form. We exemplify this for the blow-down geometries $\delp_1(0,4)$ and $\bbF_0(0,4)$ in \cref{app:scatt}. For these cases, the specialisation of \cref{conj:dp302}  reduces to non-trivial, and apparently novel, conjectural $q$-binomial identities, see e.g. \cref{conj:dp1binom}.


\subsection{Toric models: $l=3$}

For $l=3$, recall from \eqref{eq:def-Nlog-psi-g-d}, resp.\ \eqref{eq:def-Nlog-g-d}, that $N^{\rm log, \psi}_{0,d}(Y(D))$, resp.\ $N^{\rm log}_{0,d}(Y(D))$, is the genus zero log Gromov--Witten invariant of maximal tangency passing through one point with psi-class, resp.\ passing through 2 points.
By Proposition \ref{prop:blowuplog}, it is enough to treat $\delp_3(0,0,0)$ as the other cases are obtained from it by interior blow-downs. 
We leave the description of the other toric models as an exercise to the reader.

Via \cref{prop:calc-log-gw}, the invariant $N^{\rm log, \psi}_{0,d}(Y(D))$ is calculated from the scattering diagram as a structural coefficient of the product of 3 theta functions. For each constellation of three broken lines, the union of these corresponds to a tropical curve in the degeneration encoded by the scattering diagram. It is counted with multiplicity the product of the coefficients of the final monomials of the broken lines. Using \cref{prop:calc-log-gw-q}, one can compute the generating series 
$\mathsf{N}^{\rm log}_d(Y(D))(\hbar)$ of higher genus $2$-point log Gromov-Witten invariants. The 
relevant tropical curves are identical to those entering the 
computation of $N^{\rm log, \psi}_{0,d}(Y(D))$. The difference is in the weighting of the tropical curves. For the $\psi$ class, the trivalent vertex at the endpoint of the broken lines carry weight $1$. For the $2$-point invariants, we consider quantum broken lines and the trivalent vertex is counted with Block--G\"ottsche multiplicity.

Let $Y(D)=\delp_3(0,0,0)$. We take $D_1$ in class $H-E_3$, $D_2$ in class $H-E_2$, $D_3$ in class $H-E_1$ and $d=d_0(H-E_1-E_2-E_3)+d_1E_1+d_2E_2+d_3E_3$. Then $d \cdot D_1=d_3$, $d \cdot D_2=d_2$, $d \cdot D_3=d_1$, $d \cdot E_1=d_0-d_1$, $d \cdot E_2=d_0-d_2$ and $d \cdot E_3=d_0-d_3$.
The calculations of Figure \ref{fig:delp3000} give for the broken line $\circleed{1}$ the contribution
\beq
\binom{d_1}{d_0-d_2}t_2^{d_0-d_2}x^{-d_1}y^{d_2-d_0},
\eeq
for the broken line $\circleed{2}$
the contribution
\beq
\binom{d_2}{d_0-d_3}t_3^{d_0-d_3}x^{d_0-d_3}y^{d_0-d_2-d_3},
\eeq
and from the broken line $\circleed{3}$ the contribution
\beq
\binom{d_3}{d_0-d_1}t_1^{d_0-d_1}x^{d_3+d_1-d_0}y^{d_3}.
\eeq
Taken together, we obtain the following result.

\begin{thm} \label{thm:log_dp3_0_0_0_psi}
\begin{equation}
N^{\rm log, \psi}_{0,d}(\delp_3(0,0,0))=\binom{d_1}{d_0-d_2}\binom{d_2}{d_0-d_3}\binom{d_3}{d_0-d_1}.
\end{equation}
\end{thm}
For the 2-point invariant, the tropical multiplicity at $p$ is
\beq
\Big{|} \det \begin{pmatrix}
d_1 & d_1+d_3-d_0 \\
-d_2+d_0 & d_3
\end{pmatrix}
 \Big{|} = \big{|} \,d_1d_3 + d_1d_2 + d_2d_3 - d_0d_2 - d_0d_1 - d_0d_3 + d_0^2 \, \big{|}. \label{eqn:quadform}
\eeq
For the invariant to be non-zero, the curve class needs to lie in the effective cone determined in Proposition \ref{prop:eff} by 
\beq
d_0\geq0, \; d_i\geq0, \; d_1+d_2+d_3\geq d_0.
\eeq
Also, for the binomial coefficients to be non-zero, the curve class needs to satisfy the equations
\beq
0\leq d_0-d_2 \leq d_1, \, 0\leq d_0-d_3 \leq d_2, \, 0\leq d_0-d_1 \leq d_3.
\eeq
These inequalities determine a cone. Using the Polyhedra Package of Macaulay2, in the basis $(H-E_1-E_2-E_3,E_1,E_2,E_3)$ we find extremal rays generated by
\beq
(1,1,1,0), \, (1,1,0,1), \, (1,0,1,1), \, (2,1,1,1).
\eeq
Using this as a new basis, we find that the quadratic form in \eqref{eqn:quadform} is given by
\beq
xy+xz+yz+w(x+y+z)+w^2,
\eeq
which is always positive in the cone.
Therefore, we have proven the following result.

\begin{thm} \label{thm:log_dp3_0_0_0}
The generating function $\mathsf{N}^{\rm log}_d(\delp_3(0,0,0))(\hbar)$ equals
\beq   
\frac{[d_0^2 - d_1(d_0-d_2) - d_2(d_0-d_3) - d_3(d_0-d_1)]_q}{[1]_q}
\qbinom{d_1}{d_0-d_2}_q \qbinom{d_2}{d_0-d_3}_q \qbinom{d_3}{d_0-d_1}_q\,,
 \label{eq:log_dp3_0_0_0}
\eeq
where $q=\re^{\ri \hbar}$.
\end{thm}

\begin{figure}[h]
\caption{$\delp_3(0,0,0)$}
\label{fig:delp3000}
\begin{tikzpicture}[smooth, scale=1.2]
\draw[step=1cm,gray,very thin] (-2.5,-4.5) grid (5.5,2.5);
\draw[<->] (-2.5,2) to (-1.5,2);
\draw[<->] (-2,1.5) to (-2,2.5);
\node at (-1.67,2.15) {$\scriptstyle{x}$};
\node at (-2.15,2.3) {$\scriptstyle{y}$};
%
%
\draw[<->,thick] (-2.5,-2) to (0,-2) to (1,-1) to (1,0) to (3.5,2.5);
\draw[->,thick] (1,-1) to (1.5,-1.5) to (1.5,-4.5);
\draw (-2.5,0) to (5.5,0);
\draw (0,-4.5) to (0,2.5);
\draw (0,0) to (2.4,2.4);
\draw (-1.4,-4.4) to (5.4,2.4);
\draw (0,-3) to (-2.5,-3);
\draw (3,0) to (3,-4.5);
\node at (0.25,-3.4) {$D_2$};
\node at (1.8,2.2) {$D_1$};
\node at (-2.2,0.2) {$D_3$};
\node at (-1.8,0) {$\times$};
\node at (0,-4.2) {$\times$};
\node at (4.8,1.8) {$+$};
\node at (-1,0.2) {$\sstyle{1+t_1x^{-1}}$};
\node[rotate=90] at (-0.2,-3.8) {$\sstyle{1+t_2y^{-1}}$};
\node[rotate=45] at (4.2,1.4) {$\sstyle{1+t_3xy}$};
%
\node at (-2,-2.2) {$\sstyle{x^{-d_1}}$};
\node at (1.85,-3.8) {$\sstyle{y^{-d_2}}$};
\node at (3.4,1.85) {$\sstyle{x^{d_3}y^{d_3}}$};
\node at (1.3,-0.3) {$\circleed{3}$};
\node at (1.5,-1.2) {$\circleed{2}$};
\node at (0.5,-1.2) {$\circleed{1}$};
\node at (1,-1) {$\bullet$};
\node at (1.2,-1) {$\sstyle{p}$};
\end{tikzpicture}
\end{figure}

\subsection{Toric models: $l=4$}

There is only one 4-component log Calabi--Yau surface with maximal boundary, namely the toric surface $\fzero(0,0,0,0)$.
For $d=d_1H_1+d_2H_2$, through tropical correspondence \cites{Mikh05,NiSi,ManRu,ManRu19}, we calculated in \cite{BBvG1} that
\beq
N^{\rm log, \psi}_{0,d}(\fzero(0,0,0,0))=1, \quad N^{\rm log}_{0,d}(\fzero(0,0,0,0))=d_1^2d_2^2.
\eeq
To obtain the higher genus invariant, we replace the tropical multiplicities by the Block-G\"ottsche multiplicities \cite{MR3453390}. Applying \cite{Bou19a} we obtain the following result.

\begin{thm}
\begin{equation}
\label{eq:fzerolog}
\mathsf{N}^{\rm log}_d(\fzero(0,0,0,0))(\hbar)=\frac{[d_1d_2]_q^2}{[1]_q^2}.
\end{equation}
\label{thm:logf0_0000}
\end{thm}

\section{Log-local correspondence}
\label{sec:logloc_gen}

In this section, we prove the following log-local correspondence theorem.

\begin{thm} \label{thm_log_local}
For every nef Looijenga pair $Y(D)$, the genus $0$ log invariants $N^{\rm log}_{0,d}(Y(D))$
and the genus $0$ local invariants $N^{\rm loc}_{0,d}(Y(D))$ are related by
\begin{equation} N^{\rm loc}_{0,d}(Y(D))= \left( \prod_{j=1}^l \frac{(-1)^{d \cdot D_j -1}}{
(d \cdot D_j)}\right) 
N^{\rm log}_{0,d}(Y(D)) \,.
\label{eq:loglocal}
\end{equation}
\end{thm}

The proof will be divided in two parts. In Section \ref{sec:logloc}, we prove the 
result for $l=2$ by a degeneration to the normal cone argument.
In Section \ref{sec:loglocal_3_4}, we prove the result for 
$l=3$ and $l=4$ by direct comparison of the local 
results of Section \ref{sec:localGW} with the log results of 
Section \ref{sec:loggw}.

\subsection[Log-local for 2 components]{Log-local for 2 components}

\label{sec:logloc}

For convenience in the following proof, we state separately the case $l=2$ 
of Theorem \ref{thm_log_local}.

\begin{thm} \label{thm_log_local_2}
For every $2$-component nef Looijenga pair $Y(D)$, we have 
\begin{equation} N^{\rm loc}_{0,d}(Y(D))= \left( \prod_{j=1}^l \frac{(-1)^{d \cdot D_j -1}}{
(d \cdot D_j)}\right) 
N^{\rm log}_{0,d}(Y(D)) \,.\end{equation}
\end{thm}

The proof of Theorem \ref{thm_log_local_2} takes the remainder of Section 
\ref{sec:logloc} and is a degeneration argument in log Gromov--Witten theory.

\subsubsection{Construction of the degeneration}
We first construct the relevant degeneration for a general 
$l$-component nef Looijenga pair 
$Y(D)=(Y,D_1+\cdots+D_l)$.

Let
$\bar{\nu}_{\bar{\cY}} \colon \bar{\mathcal{Y}} \rightarrow \A^1$ be the degeneration of $Y$
to the normal cone of $D$, obtained by
blowing up $D \times \{0\}$ in $Y \times \A^1$.
Irreducible components of the special fibre 
$\bar{\cY}_0 \coloneqq \bar{\nu}_{\bar{\cY}}^{-1}(0)$ are 
$Y$ and, for every 
$1 \leq j \leq l$, $\bar{\PP}_j \coloneqq \PP(\cO \oplus N_{D_j|Y})$, where 
$N_{D_j|Y}$ is the normal bundle to $D_j$ in $Y$.
For every double point $p \in D_j \cap D_{j'}$ of $D$, a local description of $\bar{\cY}_0$ is given by 
Figure \ref{figure_special_fibre}. 
In particular, we have a point $p^{\partial}$ in $\bar{\cY}_0$ 
where the total space $\bar{\cY}$ is singular. 
This can be seen as follows. Locally near a double point
$p \in D_{j'} \cap D_j$, the degeneration to the normal cone admits a 
toric description, whose fan is given by the closure of the cone over the 
polyhedral decomposition of $\bbR_{\geq 0}^2$ in Figure \ref{figure_fan}.
The point $p^{\partial}$ corresponds to the 3-dimensional cone obtained by taking the closure of the cone over the unbounded region of
$\bbR_{\geq 0}^2$ in \ref{figure_fan}. This cone is generated by four rays, so is not simplicial and so $p^{\partial}$ is a singular point.
More precisely, $p^{\partial}$
is an ordinary double point in $\bar{\cY}$. Every singular point of $\bar{\cY}$ is of the form $p^{\partial}$ for $p$ a double point of $D$.

\begin{figure}
\caption{Local description of $\bar{\cY}_0$}
\begin{center}
\setlength{\unitlength}{1cm}
\begin{picture}(6,7) 
\thicklines
\put(3,3){\circle*{0.1}}
\put(3,3){\line(-1,0){3}}
\put(3,3){\line(0,-1){3}}
\put(3,3){\line(1,1){2}}
\put(5,5){\circle*{0.1}}
\put(5,5){\line(-1,0){5}}
\put(5,5){\line(0,-1){5}}
\put(2.7,2.7){$p$}
\put(5.2,5.2){$p^{\partial}$}
\put(1,1){$Y$}
\put(2.4,1){$D_j$}
\put(5,1){$D_{j}^{\partial}$}
\put(4,1.5){$\bar{\PP}_j$}
\put(1,2.5){$D_{j'}$}
\put(1,5.3){$D_{j'}^{\partial}$}
\put(1.5,4){$\bar{\PP}_{j'}$}
\end{picture} 
\end{center}
\label{figure_special_fibre}
\end{figure}

\begin{figure}
\caption{Toric polyhedral decomposition of $\bbR_{\geq 0}^2$ describing locally 
$\bar{\cY}_0$ (fan picture)}
\begin{center}
\setlength{\unitlength}{1cm}
\begin{picture}(6,5)
\thicklines
\put(2,1){\circle*{0.1}}
\put(2,1){\line(1,0){4}}
\put(2,1){\line(0,1){4}}
\put(4,1){\circle*{0.1}}
\put(2,3){\circle*{0.1}}
\put(2,3){\line(1,-1){2}}
\end{picture} 
\end{center}
\label{figure_fan}
\end{figure}

We resolve the singularities of $\bar{\cY}$ by blowing up the
ordinary double points $p^{\partial}$, and we obtain a new degeneration 
$\nu_{\cY} \colon \cY \rightarrow \A^1$. The total space 
$\cY$ is now smooth and the special fibre $\cY_0 \coloneqq \nu_{\cY}^{-1}(0)$
is a normal crossings divisor on $\cY$. We view $\cY$ as a log scheme for the 
divisorial log structure defined by $\cY_0 \subset \cY$. Viewing $\A^1$
as a log scheme for the divisorial log stucture defined by $\{ 0\} \subset 
\A^1$, the morphism $\nu_{\cY} \colon \cY \rightarrow \A^1$ can naturally be viewed as a log smooth log morphism.

Irreducible components of $\cY_0$ consist of $Y$, for every 
$1 \leq j \leq l$ the strict transform 
$\PP_j$ of the $\bar{\PP}_j$, and for every double point $p$
of $D$ the exceptional divisor 
$\bS_{p} \simeq \PP^1 \times \PP^1$ created by the blow-up of 
$p^{\partial}$.
Locally near a double point $p \in D_j \cap D_{j'}$,  
irreducible components of $\cY_0$ are glued together as in 
Figure \ref{figure_special_fibre_2}. Locally near $p$, the total space $\cY$ 
admits a toric description whose fan is the closure of the cone over the 
polyhedral decomposition of $\bbR_{\geq 0}^2$ given in \cref{figure_fan_2}.
Remark that the log structure that we consider on $\cY$ is only partially compatible with this local toric description: one needs to remove from the toric boundary the horizontal toric divisors in order to obtain the divisorial log structure defined by the special fibre.

\begin{figure}
\caption{Local description of $\cY_0$}
\begin{center}
\setlength{\unitlength}{1cm}
\begin{picture}(6,7) 
\thicklines
\put(3,3){\circle*{0.1}}
\put(3,3){\line(-1,0){3}}
\put(3,3){\line(0,-1){3}}
\put(3,3){\line(1,1){1}}
\put(4,4){\circle*{0.1}}
\put(4,4){\line(1,0){1}}
\put(4,4){\line(0,1){1}}
\put(5,5){\circle*{0.1}}
\put(5,4){\circle*{0.1}}
\put(4,5){\circle*{0.1}}
\put(5,5){\line(-1,0){5}}
\put(5,5){\line(0,-1){5}}
\put(1,1){$Y$}
\put(2.4,1){$D_j$}
\put(5,1){$D_{j}^{\partial}$}
\put(5.1,4.2){$D_{j,p}^{\partial}$}
\put(4,1.5){$\PP_j$}
\put(1,2.5){$D_{j'}$}
\put(1,5.3){$D_{j'}^{\partial}$}
\put(4,5.3){$D_{j',p}^{\partial}$}
\put(1.5,4){$\PP_{j'}$}
\put(4.3,4.4){$\bS_p$}
\end{picture} 
\end{center}
\label{figure_special_fibre_2}
\end{figure}

\begin{figure}
\caption{Toric polyhedral decomposition of $\bbR_{\geq 0}^2$ describing locally 
$\cY_0$ (fan picture)}
\begin{center}
\setlength{\unitlength}{1cm}
\begin{picture}(6,5)
\thicklines
\put(2,1){\circle*{0.1}}
\put(2,1){\line(1,0){4}}
\put(2,1){\line(0,1){4}}
\put(4,1){\circle*{0.1}}
\put(2,3){\circle*{0.1}}
\put(2,3){\line(1,-1){2}}
\put(4,3){\circle*{0.1}}
\put(2,3){\line(1,0){4}}
\put(4,1){\line(0,1){4}}
\end{picture} 
\end{center}
\label{figure_fan_2}
\end{figure}

For every $1 \leq j \leq l$, let 
$\ccL_j$ be the line bundle on $\cY$ defined by minus the divisor obtained by taking the closure in $\cY$ of the divisor $D_j \times (\A^1-\{0\}) 
\subset Y \times (\A^1 -\{0\})$.
We have 
\beq \ccL_j|_{\cY_0} 
= \cO_{\cY_0}\left(-(D_{j}^{\partial} \cup \bigcup_{p \in D_j} 
D_{j,p}^{\partial})\right) \,,\eeq
where the union is taken over the double points $p$
of $D$ contained in $D_j$.

We define $\cV \coloneq \Tot(\bigoplus_{j=1}^l \ccL_j)$
and denote by $\pi_{\cV} \colon \cV \rightarrow \cY$
and $\nu_{\cV} \colon \cV \rightarrow \A^1$ the natural projections. 
We also denote by $\cV_0 \coloneqq \nu_{\cV}^{-1}(0)$ the special fibre and by 
$\pi_{\cV_0} \colon \cV_0 \rightarrow \cY_0$ the restriction of $\pi_{\cV}$
to the special fibre.

The irreducible components of $\cV_0$ are
\beq\cV_{0,Y}\coloneqq \Tot(\cO_Y^{\oplus l})\,,\eeq 
for every $1 \leq j \leq l$ 
\beq\cV_{0,j} \coloneqq \Tot(\cO_{\PP_j}(-D_j^{\partial}) \oplus \cO_{\PP_j}^{\oplus (l-1)})\,\eeq
and, for every double point 
$p \in D_j \cap D_{j'}$ of $D$, 
\beq \cV_{0,p} \coloneqq \Tot(\cO_{\bS_p}(-D_{j,p}^{\partial}) 
\oplus \cO_{\bS_p}(- D_{j',p}^{\partial}) \oplus \cO_{\bS_p}^{\oplus (l-2)})\,.\eeq

We view $\cV$ as a log scheme for the divisorial log structure defined by 
$\cV_0 \subset \cV$, and then $\nu_{\cV} \colon \cV \rightarrow \A^1$ is naturally a log smooth log morphism. Remark that the log structure on $\cV$
is the pullback of the log structure on $\cY$, i.e.\ the log morphism $\pi_{\cV}
\colon \cV \rightarrow \cY$ is strict. In particular, $\cV$ and $\cY$
have identical tropicalisations.

For every $1 \leq j \leq l$, we consider the projectivisation
$\PP(\ccL_j \oplus \cO_{\cY})$ of $\ccL_j$ and the corresponding 
fibrewise compactification 
\beq \PP(\cV) \coloneqq \PP(\ccL_1 \oplus \cO_{\cY}) \times_{\cY} 
\dots \times_{\cY} \PP(\ccL_l \oplus \cO_{\cY}) \eeq
of $\cV$. We denote by 
$\pi_{\PP(\cV)} \colon \PP(\cV) \rightarrow \cY$ and 
$\nu_{\PP(\cV)} \colon \PP(\cV) \rightarrow \A^1$
the natural projections. We also denote by 
$\PP(\cV_0) \coloneqq \nu_{\PP(\cV)}^{-1}(0)$ the special fibre
and by 
$\pi_{\PP(\cV_0)} \colon \PP(\cV_0) \rightarrow \cY_0$
the restriction of $\pi_{\PP(\cV)}$ to the special fibre.
We denote by 
$\PP(\cV_{0,Y})$, $\PP(\cV_{0,j})$, $\PP(\cV_{0,p})$
the irreducible components of $\PP(\cV_0)$
obtained by compactification of $\cV_{0,Y}$, 
$\cV_{0,j}$, $\cV_{0,p}$.

We view $\PP(\cV)$ as a log scheme for the divisorial log 
structure defined by $\PP(\cV_0) \subset \PP(\cV)$, and then 
$\nu_{\PP(\cV)} \colon \PP(\cV) \rightarrow \A^1$ is naturally a log 
smooth log morphism. Remark that the log structure on 
$\PP(\cV)$ is the pullback of the log structure on 
$\cY$, i.e.\ the log morphism 
$\pi_{\PP(\cV)} \colon \PP(\cV) \rightarrow \cY$ is strict.
In particular, $\PP(\cV)$ and $\cY$ have identical tropicalisations.

Let $\Delta$ be the polyhedral complex obtained by taking the fibre over
$1$ of the tropicalisation of $\nu_{\PP(\cV)} \colon \PP(\cV) \rightarrow \A^1$. 
Combinatorially, $\Delta$ is the dual intersection complex of the 
special fibre $\cV_0$, see Figure \ref{figure_fan_3}. Vertices of $\Delta$ are:
\begin{itemize}
    \item 
$v_Y$, corresponding to the irreducible component $\PP(\cV_{0,\cY})$, 
\item for every 
$1 \leq j \leq l$, $v_j$ corresponding to the irreducible component 
$\PP(\cV_{0,j})$, 
\item for every double point $p$ of $D$, $v_p$ corresponding to the
irreducible component $\PP(\cV_{0,p})$.
\end{itemize}
Edges of $\Delta$ are:
\begin{itemize}
    \item for every $1 \leq j \leq l$, $e_{Y,j}$ connecting $v_Y$ and $v_j$, corresponding to the divisor 
    $\PP(\cV_{0,Y}) \cap \PP(\cV_{0,j})$,
    \item for every double point $p \in D_j \cap D_{j'}$ of $D$, $e_{j,j'}^p$
    connecting $v_j$ and $v_{j'}$, corresponding to the component of the  divisor 
    $\PP(\cV_{0,j}) \cap \PP(\cV_{0,j'})$ containing $p$, 
    \item for every double point $p \in D_j \cap D_{j'}$ of $D$, 
    $e_{p,j}$ connecting $v_p$ and $v_j$, corresponding to the divisor 
    $\PP(\cV_{0,j}) \cap \PP(\cV_{0,p})$, and $e_{p,j'}$ connecting $v_p$ and $v_{j'}$,
    corresponding to the divisor $\PP(\cV_{0,j'}) \cap \PP(\cV_{0,p})$.
\end{itemize}
Faces of $\Delta$ are:
\begin{itemize}
    \item for every double point $p \in D_j \cap D_{j'}$ of $D$, a triangle 
    $f_p$ of sides $e_{Y,j}$, $e_{Y,j'}$, $e_{j,j'}^p$, corresponding to the triple intersection $\PP(\cV_{0,Y}) \cap \PP(\cV_{0,j}) \cap \PP(\cV_{0,j'})$,
    \item for every double point $p \in D_j \cap D_{j'}$ of $D$, a triangle 
    $g_p$ of sides $e_{j,j'}^p$, $e_{p,j}$, $e_{p,j'}$, corresponding to the triple intersection $\PP(\cV_{0,p}) \cap 
    \PP(\cV_{0,j}) \cap \PP(\cV_{0,j'})$.
\end{itemize}

As we are assuming that the components of $D$ 
form a cycle, the boundary $\partial \Delta$  
of $\Delta$ 
can be described as 
\beq \partial \Delta = \bigcup_{1 \leq j \leq l} (\partial \Delta)_j \,,\eeq
where for every $1 \leq j \leq l$,
\beq (\partial \Delta)_j \coloneqq \bigcup_{p \in D_j \cap D_{j'}} e_{p,j}\,.\eeq

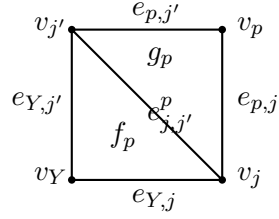
\begin{figure}
\caption{Local description of $\Delta$}
\begin{center}
\setlength{\unitlength}{1cm}
\begin{picture}(6,4)
\thicklines
\put(2,1){\circle*{0.1}}
\put(2,1){\line(1,0){2}}
\put(2,1){\line(0,1){2}}
\put(4,1){\circle*{0.1}}
\put(2,3){\circle*{0.1}}
\put(2,3){\line(1,-1){2}}
\put(4,3){\circle*{0.1}}
\put(2,3){\line(1,0){2}}
\put(4,1){\line(0,1){2}}
\put(3,2.6){$g_p$}
\put(2.5,1.5){$f_p$}
\put(3,1.8){$e_{j,j'}^p$}
\put(1.5,1){$v_Y$}
\put(4.2,1){$v_j$}
\put(4.2,3){$v_p$}
\put(1.5,3){$v_{j'}$}
\put(1.2,2){$e_{Y,j'}$}
\put(4.2,2){$e_{p,j}$}
\put(2.8,3.2){$e_{p,j'}$}
\put(2.8,0.7){$e_{Y,j}$}
\end{picture} 
\end{center}
\label{figure_fan_3}
\end{figure}

We view $\PP(\cV_0)$ as a log scheme by restriction of 
the log structure on 
$\PP(\cV)$. 
We denote by $\nu_{\PP(\cV_0)} \colon \PP(\cV_0) \rightarrow \pt_{\NN}$ the corresponding log smooth log morphism to the standard log point.
We view the curve class $d$ as a class on $\PP(\cV_0)$ via the embedding 
$\cY_0 \rightarrow \PP(\cV_0)$ induced by the zero section of $\cV_0$.
Let 
$\ol{\mmm}_{0,m}(Y^{\loc} (D),d)$  
be the moduli space of genus $0$ class $d$
 stable log maps to 
 $\nu_{\PP(\cV_0)} \colon \PP(\cV_0) \rightarrow \pt_{\NN}$ 
with $m$ marked points with contact order $0$ with $\PP(\cV_{0,Y})$.
Let $[\ol{\mmm}_{0,m}( \PP(\cV_0),d)]^{\vir}$ 
be the corresponding virtual fundamental class, of dimension $l-1+m$. 
Using the nefness of the divisors $D_j$, the condition 
$d \cdot D_j >0$ for every $1 \leq j \leq l$, and the
deformation invariance of log Gromov--Witten invariants, we have 
\beq N^{\rm loc}_{0,d}(Y(D)) = 
\int_{[\ol{\mmm}_{0,l-1}( \PP(\cV_0),d)]^{\vir}} \prod_{k=1}^{l-1} \ev_k^{*}(\pi_{\PP(\cV_0)}^* [\pt_Y]) \,,\eeq
where $\ev_k$ is the evaluation at the $k$-th interior 
marked point and $[\pt_Y]$ is the class of a point on 
$Y \subset \cY_0$.

\subsubsection{Degeneration formula}

According to the decomposition formula of Abramovich--Chen--Gross--Siebert \cite{abramovich2017decomposition}, we have
\beq [\ol{\mmm}_{0,l-1}( \PP(\cV_0),d)]^{\vir} = \sum_{h \colon \Gamma \rightarrow \Delta} \frac{m_h}{|\Aut(h)|}
[\ol{\mmm}^h_{0,l-1}( \PP(\cV_0),d)]^{\vir} \,.\eeq
The sum is over the genus $0$ rigid decorated parametrised tropical curves $h \colon \Gamma \rightarrow \Delta$,
where $\Gamma$ has $l-1$ unbounded edges,
all contracted by $h$ to $v_Y$, and the sum of classes attached to the vertices of $\Gamma$ is $d$.
The moduli space $\ol{\mmm}^h_{0,l-1}( \PP(\cV_0),d)$ 
parametrises genus $0$ class $d$
stable log maps to 
$\nu_{\PP(\cV_0)} \colon \PP(\cV_0) \rightarrow 
\pt_{\NN}$ marked by $h$.

Therefore, we have 
\beq N^{\rm loc}_{0,d}(Y(D))
=\sum_{h \colon \Gamma \rightarrow \Delta} 
\frac{m_h}{|\Aut(h)|} 
N^{{\rm loc}, h}_{0,d}(Y(D))\,,
\eeq
where 
\beq N^{{\rm loc}, h}_{0,d}(Y(D)) \coloneqq
\int_{[\ol{\mmm}^h_{0,l-1}( \PP(\cV_0),d)]^{\vir} }
\prod_{k=1}^{l-1} \ev_k^{*}(\pi_{\PP(\cV_0)}^{*} [\pt_Y] )\,,\eeq
and 
$\ev_k$ is the evaluation at the $k$-th marked point.
Thus, for every $h \colon \Gamma \rightarrow \Delta$, 
we have to compute $N^{{\rm loc},h}_d(Y(D))$.

Let $\Delta^h$ be a polyhedral complex obtained by refining the polyhedral decomposition of $\Delta$ and containing the $h(\Gamma)$ in its one-skeleton, i.e.\ such that, for every 
vertex $V$ of $\Gamma$, $h(V)$ is a vertex of $\Delta^h$, and for every edge $E$ of $\Gamma$, $h(E)$ is an edge of
$\Delta^h$. 
We denote by $\cY_0^h$, $\cV_0^h$, $\PP(\cV_0^h)$ 
the corresponding log modifications of 
$\cY_0$, $\cV_0$, $\PP(\cV_0)$.
Let $\ol{\mmm}^h_{0,l-1}( \PP(\cV_0^h),d)$ 
the moduli space of stable log maps to $\PP(\cV_0^h)$ marked by $h$.
By the invariance of log Gromov--Witten invariants under log 
modification \cite{AW}, we have 
\beq N^{{\rm loc}, h}_{0,d}(Y(D)) \coloneqq
\int_{[\ol{\mmm}^h_{0,l-1}( \PP(\cV_0^h),d)]^{\vir} }
\prod_{k=1}^{l-1} \ev_k^{*}(\pi_{\PP(\cV_0^h)}^{*} [\pt_Y] )\,.\eeq

For every vertex $V$ of $\Gamma$, let $\PP(\cV_{0,V}^h)$ be the irreducible 
component of $\PP(\cV_0^h)$ corresponding to the vertex $h(V)$ of $\Delta^h$. 
We view $\PP(\cV_{0,V}^h)$ as a log scheme for the divisorial log structure 
defined by the divisor $\partial \PP(\cV_{0,V}^h)$, which is the union of intersection divisors with the other irreducible components of $\PP(\cV_0^h)$.
Similarly, we define the component $\cY_{0,V}^h$
of $\cY_0^h$ and $\partial \cY_{0,V}^h$.

If $h(V) \in \Delta^h - \partial \Delta^h$, then 
\beq \cV_{0,V}^h = \Tot( \cO_{\cY_{0,V}^h}^{\oplus l}) \,.\eeq
If furthermore, $h(V) \notin \bigcup_{j=1}^l e_{Y,j}$, then $(\cV_{0,V}^h, \partial \cV_{0,V}^h)$
is a toric variety with its toric boundary.

If $h(V) \in e_{p,j} -v_p$ for some $p \in D_j \cap D_{j'}$, let 
$D_{j,V}^{\partial}$ be the irreducible component contained in 
$\cY_{0,V}^h$ of the intersection with $\cY_0^h$ of the closure of 
$D_j\times (\A^1-\{0\})$ in $\cY^h$. Then, we have
\beq \cV_{0,V}^h= \Tot (\cO_{\cY_{0,V}^h}(-D_{j,V}^{\partial}) \oplus \cO_{\cY_{0,V}^h}^{\oplus l-1)} )\,,\eeq
and $(\cV_{0,V}^h, \partial \cV_{0,V}^h \cup D_{j,V}^{\partial})$ is a toric variety with its toric boundary.

If $h(V)=v_p$ for some $p \in D_j \cap D_{j'}$, then, still denoting by 
$\bS_p$, $D_{j,p}^{\partial}$, and $D_{j,p}^{\partial}$ the strict tranforms in $\cY_0^h$ of $\bS_p$, $D_{j,p}^{\partial}$, and $D_{j,p}^{\partial}$, we have 
\beq \cV_{0,V}^h = \cV_{0,p} ^h=\Tot(\cO_{\bS_p}(-D_{j,p}^{\partial}) 
\oplus \cO_{\bS_p}(- D_{j',p}^{\partial}) \oplus \cO_{\bS_p}^{\oplus (l-2)})\,,\eeq
and 
$(\cV_{0,V}^h, \partial \cV_{0,V}^h \cup D_{j,p}^{\partial} \cup D_{j',p}^{\partial})$
is a toric variety with its toric boundary.

For every vertex $V$ of $\Gamma$, let $M_V$ be the moduli space of genus $0$ stable log maps to $\PP(\cV_{0,V}^h)$, with class given by the class decoration of $V$, and contact orders specified by the local behavior of $h$ around $V$.
Our goal is to compute the invariant 
$N^{{\rm loc},h}_d(Y(D))$ in terms of the virtual classes
$[M_V]^{\virt}$. We are in a particularly favorable 
situation: we consider curves of genus $0$ and the dual
intersection complex $\Delta^h$ has dimension $2$. In such case, the degeneration formula in log Gromov--Witten theory has a particular simple form, as described in Section 6.5.2 
of \cite{ranganathan2019logarithmic}
(see also Section 4 of \cite{parker2015gluing}
for the corresponding discussion in the language of exploded manifolds).

We choose a flow on $\Gamma$ such that unbounded edges are ingoing and such that every vertex has at most one outgoing edge. 
Such flow exists as $\Gamma$ has genus $0$ and then there is exactly one vertex, that we denote by $V_0$, without outgoing incident edge, that we call the sink.
All vertices distinct from $V_0$ have exactly one outgoing edge.
In fact, for every vertex $V$ of $\Gamma$, we can find such flow
with sink $V_0=V$.

For every edge $E$ of $\Gamma$, we denote by $X_E$
the stratum of $\PP(\cV_0^h)$ dual to $E$. The stratum $X_E$
is a divisor if $E$ is bounded and 
is the irreducible
component
$\cV_{0,Y}^h$ if $E$ is unbounded. 
For every $E$, $X_E$ is 
a projective bundle over a stratum $Y_E$ of $\cY_0^h$.
We denote by 
$\pi_E \colon X_E \rightarrow Y_E$ the corresponding projection.
More precisely, $X_E$ is naturally a fibre product over $Y_E$
of $l$ bundles in projective lines.

For every edge $E$ incident to a vertex $V$, we have the evaluation map 
\beq \ev_{V,E} \colon M_V \rightarrow X_E \,.\eeq

For every vertex $V$ distinct from $V_0$, let $\cE_{in}(V)$
be the set of ingoing incident edges to $V$, and let $E_V$ be the outgoing incident edge to $V$.
The virtual class 
$[M_V]^{virt}$ defines a map
\beq \eta_V \colon \prod_{E \in \cE_{in}(V)} \hhh^{*}(X_E) 
\rightarrow  \hhh^{*}(X_{E_V}) \eeq
by
\beq \eta_V
\left(
 \prod_{E \in \cE_{in}(V_0)} \alpha_E \right)
\coloneqq (\ev_{V,E_V})_* 
\left( 
\left(
\prod_{E \in \cE_{in}(V)}
\ev_{V,E}^* \alpha_E \right) \cap [M_V]^{virt} \right) \,.\eeq
Note that if $\cE_{in}(V)$ is empty, then $\eta_V$ is a map of the form 
\beq \eta_V \colon \Q \rightarrow \hhh^{*}(X_{E_V}) \,.\eeq

Denote by $\cE_{in}(V_0)$ the set of ingoing incident edges to $V_0$.
The virtual class $[M_{V_0}]^{virt}$
defines a map 
\beq \eta_{V_0} \colon \prod_{E \in \cE_{in}(V_0)} \hhh^{*}(X_E) 
\rightarrow \Q \eeq
by
\beq \eta_{V_0} \left(
 \prod_{E \in \cE_{in}(V_0)} \alpha_E \right)
\coloneqq \int_{[M_{V_0}]^{virt}} 
\prod_{E \in \cE_{in}(V_0)}
\ev_{V_0,E}^* \alpha_E \,.\eeq

Denote by $\cE_{\infty}(\Gamma)$ the set of unbounded edges of 
$\Gamma$.
Composing the maps $\eta_V$ and $\eta_{V_0}$, we obtain a map
\beq \eta_h \colon \prod_{E \in \cE_\infty(\Gamma)}
\hhh^{*}(X_E) \rightarrow \Q \,.\eeq

For every edge $E$ of $\Gamma$, let $\pt_E \in \hhh^2(Y_E)$ be the class of a point 
on $Y_E$. We consider the class $\pi_E^{*} \pt_E \in \hhh^2(X_E)$.
The degeneration formula is then
\beq N^{{\rm loc},h}_d(Y(D))
= \eta_h \left (\prod_{E \in \cE_{\infty}(\Gamma)} \pi_E^{*} \pt_E \right) \,.\eeq

We define a rigid genus $0$ rigid parametrised tropical curve 
$\bar{h} \colon \bar{\Gamma} \rightarrow \Delta$ as follows.
Let $\bar{\Gamma}$ be the star-shaped graph consisting of vertices $V_j$ for 
$0 \leq j \leq l$, and edges $E_j$ connecting $V_0$ and $V_j$ for 
$1 \leq j \leq l$. We assign the length $1/(d \cdot D_j)$ to the edge $E_j$.
Let $\bar{h} \colon \bar{\Gamma} \rightarrow \Delta$ be the piecewise linear map
such that $\bar{h}(V_0)=v_Y$ and $\bar{h}(V_j)=v_j$ for 
$1 \leq j \leq l$. In particular, we have 
$\bar{h}(E_j)=e_{Y,j}$ for $1 \leq j \leq j$. As $e_{Y,j}$ has integral length
$1$, we deduce that $E_j$ has weight $d \cdot D_j$.
Finally, curve classes decoration of the vertices are given by: 
$d_{V_0}=d$ and $d_{V_j}$ is equal to 
$(d \cdot D_j)$ times the class of a 
$\PP^1$-fibre of $\PP_j$ for $1 \leq j \leq l$.

\begin{lem} \label{lem_computation}
\beq N^{{\rm loc}, \bar{h}}_{0,d} (Y(D)) = \left( \prod_{j=1}^l \frac{(-1)^{d \cdot D_j -1}}{(d \cdot D_j)^2}\right) N_{0,d}^{\log}(Y(D)) \,.\eeq
\end{lem}

\begin{proof}
We choose the flow on $\Gamma$ with sink $V_0$. Applying the degeneration formula gives immediately the result, using the fact that the normal bundle in $\PP(\cV_0)$ to a $\PP^1$-fibre of $\PP_j$ is 
$\cO(-1) \oplus \cO^{\oplus (l+1)}$ and so the corresponding multicover contribution is 
\beq \frac{(-1)^{d \cdot D_j -1}}{(d \cdot D_j)^2}\eeq
by \cite{MR2115262} (see the proof of Theorem 5.1).
\end{proof}

\begin{thm} \label{thm_key}
Assume $l=2$.
Let $h \colon \Gamma \rightarrow \Delta^h$ be a rigid decorated parametrised tropical curve as above with $N^{{\rm loc}, h}_{0,d} (Y(D)) \neq 0$. Then $h = \bar{h}$.
\end{thm}

\cref{thm_key} follows from a judicious analysis of the possible topologies of contributing tropical curves, which we perform in \cref{sec:constr_trop}. Theorem \ref{thm_log_local_2} then follows from the combination of 
Theorem \ref{thm_key}, Lemma \ref{lem_computation}, and the decomposition formula using that 
$|\Aut(\bar{h})|=1$ and $m_{\bar{h}} = \prod_{j=1}^l (d \cdot D_j)$.

\subsection{The log-local correspondence for $3$ and $4$ components}
\label{sec:loglocal_3_4}

We end the proof of Theorem \ref{thm_log_local} for $l=3$ and $l=4$.
For $l=3$, it is enough to treat the case of 
$\delp_3(0,0,0)$ as all the other $3$-component cases are obtained from it by blow-up.
The result for 
$\delp_3(0,0,0)$ follows by comparing the local result given by Theorem
\ref{thm:dP33comp} with the log result given by Theorem 
\ref{thm:log_dp3_0_0_0}.

For $l=4$, the result follows by comparing the local result given by 
\eqref{eq:local_4_comp} and the log result given by \eqref{eq:fzerolog}.

\section{Open Gromov--Witten theory}
\label{sec:opengw}

In this section we relate the quantised scattering calculations of \cref{sec:loggw}
to the higher genus open Gromov--Witten theory of Aganagic--Vafa A-branes. We first give in \cref{sec:ogw} an overview of the framework of \cite{Li:2004uf} to cast open toric Gromov--Witten theory within the realm of formal relative invariants, and recall the topological vertex formalism of  Aganagic--Klemm--Mari\~no--Vafa. Our treatment throughout this section, while self-contained, will keep the level of detail to the necessary minimum, and we refer the reader to \cites{Li:2004uf,Fang:2011qd} for further details.  The reader who is familiar with this material may wish to skip to \cref{sec:logopenpr} where the stable log counts of \cref{sec:loggw} are related to open Gromov--Witten theory, with the main statement condensed in \cref{thm:logopen}, and proved in \cref{sec:logopen}.

In the following, for a partition $\lambda \vdash d$ of $d\in \bbN$ we write $|\lambda|=d$ for the order of $\lambda$, $\ell_\lambda =r$ for the cardinality of the partitioning set, $\kappa_\lambda := \sum_{i=1}^{\ell_\lambda} \lambda_i(\lambda_i-2i+1)$ for its second Casimir invariant, and let $m_j(\lambda):=\#\{\lambda_i | \lambda_i=j \}_{i=1}^{\ell_\lambda}$, $z_\lambda \coloneqq \prod_j m_j(\lambda)! j^{m_j(\lambda)}$.  We furthermore denote $\cP$ the set of partitions, and $\cP^d$ the set of partitions of order~$d$. We will extensively need, particularly in the proof of \cref{thm:logopen}, some classical results on principally-specialised shifted symmetric functions, for which notation and necessary basic results are collected in \cref{sec:prelim}.

\subsection{Toric special Lagrangians}
\label{sec:ogw}

 Let $X$ be a smooth complex toric threefold with $K_X \simeq \cO_X$. If the affinisation morphism to $\mathrm{Spec}(\Gamma(X, \cO_X))$ is projective, $X$ can be realised as a symplectic quotient $\bbC^{r+3}\GIT{}{G}$ \cite{MR2015052} where $G\simeq U(1)^r$ acts on the affine co-ordinates $\{z_i\}_{i=1}^{r+3}$ of $\bbC^{r+3} = \mathrm{Spec}\bbC[z_1, \dots, z_{r+3}]$ by 
$$(t_1,\ldots, t_r) \cdot (z_1,\ldots, z_{r+3}) =\l(\prod_{i=1}^r t_i^{w_1^{(i)}}\cdot z_1,\ldots,  \prod_{i=1}^r t_i^{w_{r+3}^{(i)}} \cdot z_{r+3}\r),$$
where $w_{j}^{(i)}\in \bbZ$, $i=1, \dots, r$, $j=1,\dots, r+3$ are the weights of the $G$-action. This is a Hamiltonian action with respect to the canonical K\"ahler form on $\bbC^{r+3}$,
\beq
\omega := \frac{\ri}{2}\sum_{i=1}^{r+3}\rd z_i \wedge \rd \bar{z}_i\,,
\label{eq:symform}
\eeq
with moment map
$$
\widetilde{\mu}(z_1,\ldots, z_{r+3}) =
\l(\sum_{i=1}^{r+3} w^{(1)}_i|z_i|^2,\ldots, \sum_{i=1}^{r+3} w^{(r)}_i |z_i|^2\r)\,.
$$
If $(t_1,\ldots, t_k) \in \hhh^{1,1}(X;\bbR) \simeq (\mathfrak{u}(1)^r)^\star$  is a K\"{a}hler class, then $X$ is the geometric quotient
\beq
X = \widetilde{\mu}^{-1}(t_1,\ldots, t_r)/G\,.
\label{eq:quotX}
\eeq
with symplectic structure given by the Marsden--Weinstein reduction $\omega_{t}$ of \eqref{eq:symform} onto the quotient \eqref{eq:quotX}, where $[\omega_{t}]=(t_1, \ldots, t_r)\in \hhh^{1,1}(X;\bbR)$. 

We will be concerned with a class of special Lagrangian submanifolds $L=L_{\widehat{w},c}$ of $(X,\omega_t)$
constructed by Aganagic--Vafa \cite{Aganagic:2000gs} which are invariant under the natural Hamiltonian torus action on $X$. They are defined by 
\beq
\sum_{i=1}^{r+3} \widehat{w}_i^1|z_i|^2 = c\,,\quad
\sum_{i=1}^{r+3} \widehat{w}_i^2 |z_i|^2 = 0\,,\quad
\sum_{i=1}^{r+3} \arg z_i  = 0\,,
\eeq
with $\widehat{w}^a_i\in \bbZ$, $\sum_{i=1}^{r+3} \widehat{w}^a_i =0$, and $c \in \bbR$. These Lagrangians have the topology of $\bbR^2 \times S^1$, and they intersect a unique torus fixed curve $C_{L}$ along an $S^1$: we say that $L$ is an {\it inner} (resp. {\it outer}) brane if $C \simeq \bbP^1$ (resp. $\bbC$). Throughout the foregoing discussion we will assume that $L$ is always an outer brane.

Let $T \simeq (\bbC^\star)^2$ be the algebraic subtorus of $(\bbC^\star)^3 \subset X$ acting trivially on $K_X$, and $T_\bbR \simeq U(1)^2$ be its maximal compact subgroup. Then by construction any toric Lagrangian $L$ is preserved by $T_\bbR$, which acts on $\bbC \times S^1$ by scaling $(\lambda_1, \lambda_2) \cdot (w, \theta) \to (\lambda_1 w, \lambda_2 \theta)$. Writing $\mu_T:X\to \bbR^2 \simeq (\mathfrak{u}(1)^2)^*$ for the moment map of the $T_\bbR$-action, the union of the 1-dimensional $(\bbC^\star)^3$ orbit closures of $X$ is mapped by $\mu_T$ to a planar trivalent metric graph $\Gamma_X$ whose sets of vertices $(\Gamma_X)_0$, compact edges $(\Gamma_X)_1^{\rm cp}$ and non-compact edges $(\Gamma_X)_1^{\rm nc}$ correspond respectively to $T$-fixed points, $T$-invariant proper curves, and $T$-invariant affine lines in $X$ respectively. Since the moment map is an integral quadratic form, the tangent directions of the edges have rational slopes in $\bbR^2$: we can explicitly keep track of this information by regarding $\Gamma_X$ as a topological oriented graph\footnote{In doing so we forget the metric information about $\Gamma_X$ which stems from a choice of a K\"ahler structure on $X$: this is inconsequential for the definition of the invariants in the next section. We thus make a slight abuse of notation by indicating the decorated topological graph obtained by forgetting the information about the lengths of the edges by the same symbol $\Gamma_X$.} decorated by the assignment to each vertex $v \in (\Gamma_X)_0$ of primitive integral lattice vectors $\mathfrak{p}^{e}_v \in \bbZ^2$, representing the directions of the edges $e$ emanating from $v \in (\Gamma_X)_0$. The graph $\Gamma_X$ is determined bijectively by the weights $w^{(i)}_j$, and its knowledge suffices to reconstruct 
$X$. 
\begin{rmk} Let $\Sigma(X)$ be the fan of $X$. As $K_X \simeq \mathcal{O}_X$, 
$\Sigma(X)$ is a cone in $\mathbb{R}^3$ over an integral polygon $P$ in 
$\mathbb{R}^2 \times \{ 1 \} \subset \mathbb{R}^3$. The graph 
$\Gamma_X$ can be obtained as the dual graph of $P$ and taking orientations to be 
outgoing at every vertex. Conversely, one can recover 
(the $\mathrm{SL}(2,\mathbb{Z}$)-equivalence class of) $P \subset \mathbb{R}^2$
as the dual polygon of 
$\Gamma_X$, and then $\Sigma(X)$ as the cone in $\mathbb{R}^3$ over 
$\mathbb{R}^2 \times \{ 1\}$.
\label{rmk:fangraph}
\end{rmk}
If $L$ is a toric outer Lagrangian, its image under $\mu_T$ is a point $\mu_T(L)$ lying on the non-compact edge $\mu_T(C)$ representing the curve it is incident to. Write $e_L:=\mu_T(C)$, $v$ for its adjacent vertex, and $e_L'$ for the first edge met by moving clockwise from $e_L$.
\begin{defn}\label{def:framing}
A {\it framing} of $L$ is the choice of an integral vector $\mathsf{f}$ such that $\mathsf{p}_v^{e_L}\wedge \mathsf{p}_v^{e_L'}=\mathsf{p}_v^{e_L}\wedge \mathsf{f}$; equivalently, $\mathsf{f}=\mathsf{p}_v^{e_L'}-f \mathsf{p}_v^{e_L}$ for some $f \in \bbZ$. We say that $L$ is canonically framed if $f=0$, i.e. $\mathsf{f}=\mathsf{p}_v^{e_L'}$. 
\end{defn}

\begin{rmk}\label{rmk:clockwise}
By construction, since $\mathsf{f}\wedge \mathsf{p}_v^{e_L}>0$, a framing at an outer vertex is always pointing in the clockwise direction 
\end{rmk}

\begin{defn}
We call $(X,L,\mathsf{f})$ a \emph{toric Lagrangian triple} if
\bit
\item $X$ is a semi-projective toric CY3 variety;
\item $L=\sqcup_i L_{\widehat{w}_i,c_i}$ is a disjoint union of Aganagic--Vafa special Lagrangian submanifolds of $X$;
\item $\mathsf{f}$ is the datum of a framing choice for each connected component of $L$.
\eit
\end{defn}

We will write $\Gamma_{(X,L,\mathsf{f})}$ for the graph obtained from $\Gamma_X$ by the extra decoration of an integral vector incident to the edge $e_L$ representing the toric outer Lagrangian $L$ at framing $\mathsf{f}$ (see \cref{fig:toricgraph}). 

\begin{example}
Let $w^{(1)}=(1,1,-1,-1)$, $\widehat{w}^{(1)}=(1,0,-1,0)$, $\widehat{w}^{(1)}=(0,1,-1,0)$. For any $t\neq 0$, the corresponding toric variety $X$ is the resolved conifold $\mathrm{Tot}(\cO_{\bbP^1}(-1)\oplus\cO_{\bbP^1}(-1))$, with $\int_{0_* \bbP^1}\omega_{t}=t$. 
\begin{figure}[!h]
\centering
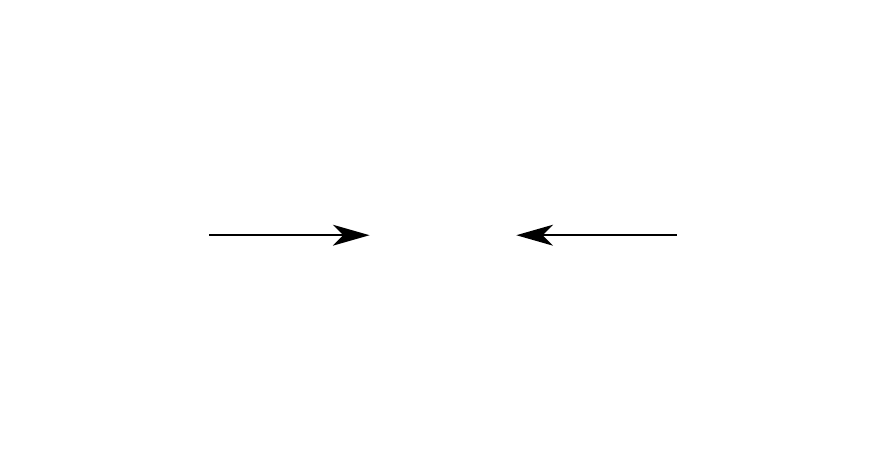
\caption{The toric Calabi--Yau graph $\Gamma_{(X,L,\mathsf{f})}$ of the resolved conifold with an outer Lagrangian at framing $\mathsf{f}=\mathsf{f}_{\rm can}-\mathsf{p}_{v_1}^{e_1}$ (i.e. $f=1$).}
\label{fig:toricgraph}
\end{figure}
The compact edge $e_5$ corresponds to the $\bbP^1$ given by the zero section of $X$. The edges $e_i$, $i=1,2,3,4$ correspond to the $T$-invariant $\bbA^1$-fibres above the points $[1:0]$ and $[0:1]$ of the $\bbP^1$ base. The weights $\widehat{w}^{(i)}$ furthermore determine a toric Lagrangian, whose image in the toric graph lies in $e_1$, and is depicted in \cref{fig:toricgraph} at framing $\mathsf{f}=\mathsf{p}^{e_5}_{v_1}-\mathsf{p}^{e_1}_{v_1}$. 
\end{example}

\subsubsection{Open Gromov--Witten invariants}

In informal terms, the open Gromov--Witten theory of $(X, L=L_1 \cup \dots \cup L_s)$ of toric Lagrangians $L_i$, $i=1, \dots, s$ is a virtual count of maps to $X$ from open Riemann surfaces of fixed genus, relative homology degree, and boundary winding data around $S^1 \hookrightarrow L$. This raises two orders of problems when trying to define these counts in the algebraic category, as the boundary conditions for the curve counts are imposed in odd real dimension, and the target geometry is non-compact. A strategy to address both issues simultaneously for framed outer toric Lagrangians, and which we will follow for the purposes for the paper, was put forward by Li--Liu--Liu--Zhou \cite{Li:2004uf}, which we briefly review below. The main idea in \cite{Li:2004uf} is to replace the toric Lagrangian triple $(X,L,\mathsf{f})$ by a formal relative Calabi--Yau pair $(\widehat{X}, \widehat{D})$, where $\widehat{X}$ is obtained as the formal neighbourhood along a partial compactification, specified by $L$ and the framing $\mathsf{f}$, of the toric 1-skeleton of $X$, and $\widehat{D}=\widehat{D}_1+ \dots + \widehat{D}_s$ is a formal divisor\footnote{See \cite[Section 5]{Li:2004uf} for the details of the relevant construction.} in the partial compactification $\widehat{X}$ with $K_{\widehat{X}}+\widehat{D}=0$, the aim being to trade the theory of open stable maps with prescribed windings along the boundary circles on $L$ by a theory of relative stable maps with prescribed ramification profile above torus fixed points in $\widehat{X}$, as previously suggested in \cite{Li:2001sg}. The resulting moduli space $\overline{\cM}_{g;\beta;\mu_1, \dots, \mu_s}^{\rm rel}(\widehat{X},\widehat{D})$ of  degree $\beta$ stable maps from $\ell(\mu_1)+\ldots + \ell(\mu_s)$-pointed, arithmetic genus $g$ nodal curves with ramification profile $\mu_i$ above $\widehat{D}_i$ at the punctures is a formal Deligne--Mumford stack carrying a perfect obstruction theory $[\cT^1 \to \cT^2]$ of virtual dimension $\ell(\mu_1)+\ldots + \ell(\mu_s)$. While the moduli space is not itself proper, it inherits a $T\simeq (\bbC^\star)^2$ action from $\widehat{X}$ with compact fixed loci, and open Gromov--Witten invariants 
\beq  
O_{g; \beta; (\mu_1, \dots, \mu_s)}(X,L,\mathsf{f})
\coloneqq \frac{1}{|\mathrm{Aut}(\vec\mu)|}\int_{[\overline{\cM}_{g;\beta;\mu_1, \dots, \mu_s}^{\rm rel}(\widehat{X},\widehat{D})]^{{\rm vir},T}}
\frac{\re_T (\cT^{1,m})}{\re_T (\cT^{2,m})}
\eeq
where $\cT^{i,m}$, $i=1,2$ denote the moving parts of the obstruction theory, are defined in a standard manner by $T$-virtual localisation \cite{MR1666787}. It is a central result of \cite{Li:2004uf} that the Calabi--Yau condition on $T$ entails that  $O_{g, \beta, (\mu_1, \dots, \mu_s)}(X,L,\mathsf{f})$ are non-equivariantly well-defined rational numbers: the invariants however do depend on the framings $\mathsf{f}_i$ specified to construct the formal relative Calabi--Yau $(\widehat{X}, \widehat{D})$, in keeping with expectations from large $N$ duality \cite{Ooguri:1999bv}.

It will be helpful to package the open Gromov--Witten invariants $O_{g,\beta,\mu_1, \dots, \mu_s}(X,L,\mathsf{f})$ into formal generating functions. Let $\mathsf{x}^{(i)}=(\mathsf{x}^{(i)}_1,\mathsf{x}^{(i)}_2, \ldots)$, $i=1, \dots, s$ be formal variables and for a partition $\mu$ define $\mathsf{x}^{(i)}_\mu := \prod_{j=1}^{\ell(\mu)}  \mathsf{x}^{(i)}_{\mu_j}$. We furthermore abbreviate $\vec{\mathsf{x}}_{\vec\mu} = (\mathsf{x}_{\mu_1}^{(1)}, \dots, \mathsf{x}_{\mu_s}^{(s)})$, $\vec\mu=(\mu_1, \dots, \mu_s)$, $|\vec\mu|=\sum_i |\mu_i|$ and $\ell(\vec\mu)=\sum_i \ell(\mu_i)$, and define the connected generating functions
\bea
\mathsf{O}_{\beta;\vec \mu}(X,L,\mathsf{f})(\hbar) &\coloneqq & \sum_{g} \hbar^{2g-2+\ell(\vec\mu)}  O_{g; \beta; \vec \mu}(X,L,\mathsf{f}) \nn \\
\cO_{\vec \mu}(X,L,\mathsf{f})(Q, \hbar) &\coloneqq & \sum_\beta \mathsf{O}_{\beta,\vec \mu}(X,L,\mathsf{f})(\hbar) Q^\beta \nn \\
\mathfrak{O}(X,L,\mathsf{f})(Q,\hbar, \mathsf{x})
& \coloneqq & \sum_{\vec\mu \in (\cP)^s}\cO_{\vec \mu}(X,L,\mathsf{f})(Q, \hbar) \vec{\mathsf{x}}_{\vec \mu}\,,
\label{eq:openfreeen}
\eea
as well as generating functions of disconnected invariants in the winding number and representation bases
\bea
\mathfrak{Z}(X,L,\mathsf{f})(Q, \hbar, \mathsf{x}) &:= & \exp\l(\mathfrak{O}(X,L,\mathsf{f})(Q, \hbar, \mathsf{x}) \r) \nn \\ &=:& \sum_{\vec\mu \in \cP^s}  \cZ_{\vec\mu}(X,L,\mathsf{f})(Q,\hbar)  \vec{\mathsf{x}}_{\vec \mu}, \nn \\
&=:& \sum_{\vec\mu \in \cP^s} \sum_{\vec \nu \in (\cP)^s} \prod_{i=1}^s \frac{\chi_{\nu_i}(\mu_i)}{z_{\mu_i}} \cW_{\vec\nu}(X,L,\mathsf{f})(Q, \hbar)  \vec{\mathsf{x}}_{\vec \mu}\,. \nn \\
\label{eq:opengf}
\eea
Here $\chi_\nu(\mu)$ denotes the irreducible character of $S_{|\nu|}$ evaluated on the conjugacy class labelled by $\mu$. 
When $\mathsf{x}=0$, \eqref{eq:opengf} reduces to the ordinary generating function of disconnected Gromov--Witten invariants of $X$.

\subsubsection{The topological vertex}

The invariants \eqref{eq:opengf} can be computed algorithmically to all genera using the topological vertex of Aganagic--Klemm--Mari\~no--Vafa \cite{Aganagic:2003db}. We can succinctly condense this into the following three statements.
\ben[leftmargin=*]
\item Let $X= \bbC^3$, $L=\cup_{i=1}^3 L_i$, and $L_i=L_{\widehat{w}^{(i)},c}$ with $\widehat{w}_j^{(i)}= \delta_{i,j}-\delta_{i ,j+1~\mathrm{mod}~3}$, $i=1, \dots, 3$ be the outer Lagrangians of $\bbC^3$ as in \cref{fig:3vertex}, and fix framing vectors $\mathsf{f}_i$ for each of them. Then
\beq
\cW_{\vec \mu}(\bbC^3,L,\mathsf{f}) = \prod_{i=1}^3 q^{f_i \kappa(\mu_i)/2}(-1)^{f_i |\mu_i|} \cW_{\vec \mu}(\bbC^3,L,\mathsf{f}_{\rm can})
\label{eq:framvert}
\eeq
where $q=\re^{\ri \hbar}$.
\begin{figure}
\centering
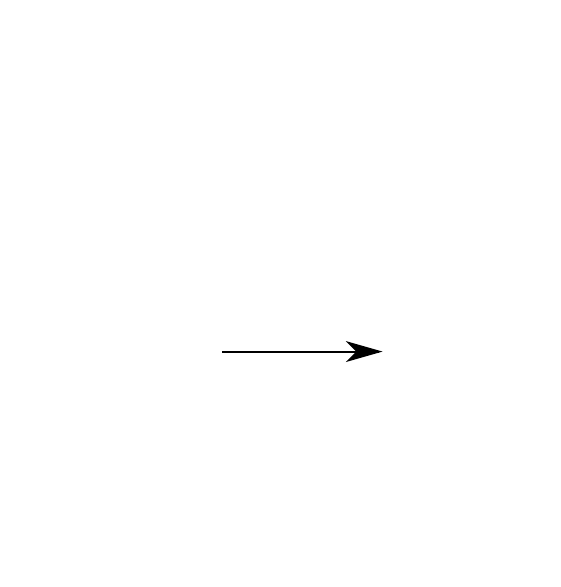
\caption{The framed vertex $(\bbC^3, L_1 \cup L_2 \cup L_3)$, depicted with framings $\mathsf{f}_1=\mathsf{p}_v^{e_2}$, $\mathsf{f}_2=\mathsf{p}_v^{e_2}+\mathsf{p}_v^{e_3}$,  $\mathsf{f}_3=\mathsf{p}_v^{e_1}$.}
\label{fig:3vertex}
\end{figure}
\item Let $(X^{(1)},L^{(1)},\mathsf{f}^{(1)})$, $(X^{(2)}, L^{(2)},\mathsf{f}^{(2)})$ be smooth toric Calabi--Yau 3-folds with framed outer toric Lagrangians $L^{(i)}=\cup_{j=1}^{s_i} L^{(i)}_j$. Suppose that there exist non-compact edges $\widetilde{e}_i \in (\Gamma_{(X^{(i)},L^{(i)},\mathsf{f}^{(i)})})_1^{\rm nc}$  emanating from vertices $\widetilde{v}_i \in (\Gamma_{(X^{(i)},L^{(i)},\mathsf{f}^{(i)})})_0$  such that $\mu_T(L^{(i)}_{s_i})\cap \widetilde{e}_i \neq \emptyset$, and that moreover $\mathsf{p}_{\widetilde{v}_1}^{\widetilde{e}_1} = -\mathsf{p}_{\widetilde{v}_2}^{\widetilde{e}_2}$, $\mathsf{f}^{(1)}_{s_1} = \mathsf{f}^{(2)}_{s_2}$ (see \cref{fig:glueing}). We can construct a planar trivalent graph $\Gamma_{X_1 \cup_{e_{12}} X_2}$ decorated with triples of primitive integer vectors at every vertex by considering the disconnected union of $\Gamma_{X^{(1)}}$ and $\Gamma_{X^{(2)}}$, deleting $\widetilde{e}_1$ and $\widetilde{e}_2$, and adding a compact edge $e_{12}$ connecting $\widetilde{v}_1$ to $\widetilde{v}_2$. A toric Calabi--Yau 3-graph reconstructs uniquely a smooth toric CY3 with a $T$ action isomorphic to the $T$-equivariant formal neighbourhood of the configuration of rational curves specified by the edges, and we call $X$ the threefold determined by the glueing procedure such that $\Gamma_X=\Gamma_{X_1 \cup_{e_{12}} X_2}$. In the same vein, the collection of framed Lagrangians $L^{(i)}$ on $X_i$ determine framed outer Lagrangians $L=\cup_{i=1}^{s_1+s_2-2} L_i$ on $X$: we have canonical projection maps $\pi_{i} : \Gamma_X \to \Gamma_{X^{(i)}}$, and we place an outer Lagrangian brane at framing $f_j$ on all non-compact edges $e$ such that $\pi_i(e) \cap \mu_T(L^{(i)}_{j})\neq \emptyset$ for some $j$. Write $\vec{\mu}=(\mu_1^{(1)}, \dots, \mu_{s_1-1}^{(1)},\mu_1^{(2)}, \dots, \mu_{s_2-1}^{(2)})$, $\vec{\mu}_{12}^{(1)}=(\mu_1^{(1)}, \dots, \mu_{s_1-1}^{(1)}, \nu_{12})$, $\vec{\mu_{12}}^{(2)} = (\mu_1^{(2)}, \dots, \mu_{s_2-1}^{(2)}, \nu_{12}^T)$. Then the following glueing formula holds
\bea
\cW_{\vec \mu}(X,L,\mathsf{f})(Q, \hbar) &=& \sum_{\nu_{12} \in \cP} (-Q_{\beta_{12}})^{|\nu_{12}|} q^{f_{12}\kappa(\nu_{12})/2}(-1)^{f_{12}} \cW_{\vec \mu_{12}^{(1)}}(X^{(1)},L^{(1)},\mathsf{f}^{(1)})(Q, \hbar) \nn \\ & & \cW_{\vec \mu_{12}^{(2)}}(X^{(2)},L^{(2)},\mathsf{f}^{(2)})(Q, \hbar)\,.
\label{eq:glueing}
\eea
Here $f_{12}=\det(\mathsf{p}_{\widetilde{e}_1'}, \mathsf{p}_{\widetilde{e}_2'})$, where $\widetilde{e}_i'\in (\Gamma_{(X,L,\mathsf{f})})_1$ is the first edge met when moving counterclockwise from $\widetilde{e}_i$, and $Q_{\beta_{12}}$ is the exponentiated K\"ahler parameter associated to the homology class $\beta_{12}=[\mu_T^{-1}(e_{12})] \in \hhh_2(X, \bbZ)$. The glueing formula \eqref{eq:glueing} originally proposed by \cite{Aganagic:2003db} is derived in \cite{Li:2004uf} as a consequence of Li's degeneration formula for relative Gromov--Witten theory \cite{MR1938113}. 
\item The glueing formula \eqref{eq:glueing} allows to recursively compute open Gromov--Witten invariants of any  toric Lagrangian triple $(X, L, \mathsf{f})$ starting from those of the framed vertex, i.e. affine 3-space with framed toric Lagrangians incident to each co-ordinate line. The framing transformation \eqref{eq:framvert} further reduces the problem to the knowledge of the open Gromov--Witten invariants of $(\bbC^3, L=L_1\cup L_2 \cup L_3, \mathsf{f}^{\rm can})$ in canonical framing $\mathsf{f}_i=\mathsf{f}^{\rm can}_i \coloneqq \mathsf{p}_{i+1}~\mathrm{mod}~3$. This is given by
\beq
\cW_{\mu_1,\mu_2,\mu_3}(\bbC^3,L,\mathsf{f}^{\rm can})(\hbar) = q^{\kappa(\mu_1)/2} \sum_{\delta \in \cP} s_{\frac{\mu_1^t}{\delta}}(q^{\rho+\mu_3})s_{\frac{\mu_2}{\delta}}(q^{\rho+\mu_3^t}) s_{\mu_3}(q^{\rho})\,.
\label{eq:vertex}
\eeq
where the shifted skew Schur function $s_\frac{\a}{\b}(q^{\rho+\gamma})$ is defined in \eqref{eq:shiftskewschur}.
The formula \eqref{eq:vertex} follows from an explicit evaluation of formal relative Gromov--Witten invariants in terms of descendent triple Hodge integrals.  It was first proved in \cites{MR2257399,Li:2004uf} when $\mu_3 = \emptyset$, and
the general case was established in \cite{moop}.
\een
An immediate consequence of \eqref{eq:glueing} and \eqref{eq:vertex} is that if $i: (X,L,\mathsf{f}) \hookrightarrow (X', L', \mathsf{f}')$ is an embedding of toric Lagrangian triples corresponding to an embedding of graphs $i_\# : \Gamma_{(X,L,\mathsf{f})} \hookrightarrow \Gamma_{(X', L', \mathsf{f}')}$, where $\Gamma_{(X', L', \mathsf{f}')}$ is obtained from $\Gamma_{(X,L,\mathsf{f})}$ by addition of a single vertex $v_2$ and glueing along a compact edge $e_{12}$ to a vertex $v_1 \in (\Gamma_{(X,L,\mathsf{f})})_0$ by the above procedure, then 
\beq
\cW_{\vec \mu}(X,L,\mathsf{f})(Q, \hbar) = \cW_{\vec \mu}(X',L',\mathsf{f}')(Q, \hbar)\big|_{Q_{\beta_{12}}=0}\,.
\label{eq:forget}
\eeq
\begin{figure}[t]
\centering
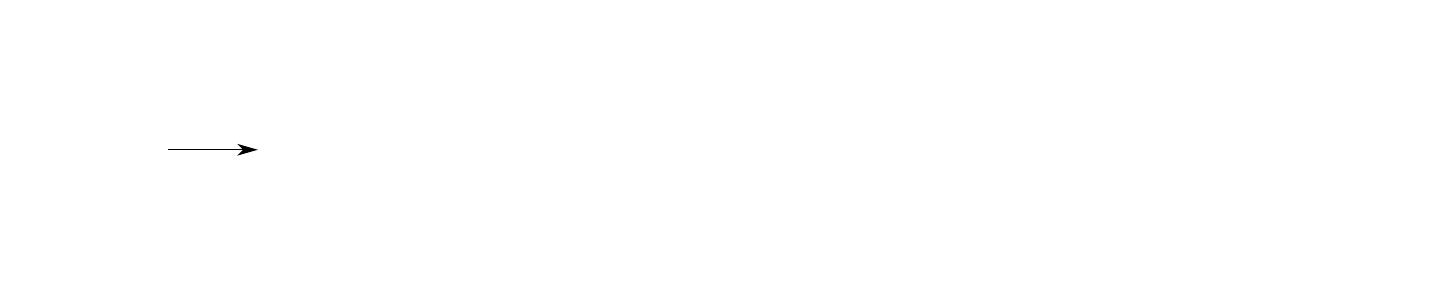
\caption{The glueing procedure for the topological vertex. In the notation of the text, we have $s_1=s_2=3$,  $\widetilde{e}_1=e_2^{(1)}$, $\widetilde{e}_2=e_3^{(2)}$, $\widetilde{e}'_1=e_3^{(1)}$, $\widetilde{e}'_2=e_1^{(2)}$.}
\label{fig:glueing}
\end{figure}

\subsection{The higher genus log-open principle}
\label{sec:logopenpr}

In this Section we associate certain toric Lagrangian triples to the geometry of Looijenga pair, under an additional condition given by the following Definition.

\begin{defn}
Let $Y(D=D_1+\dots +D_l)$ be a nef Looijenga. We say that it satisfies \textbf{Property O} if $E_{Y(D)}$ deforms to $E_{Y'(D')}$ for a Looijenga pair $Y'(D'=D_1'+\dots+D_l')$ such that
\bit
\item $Y'$ is a toric surface,
\item $D_i'$ is a prime toric divisor  $\forall~i=1, \dots, l-1$,
\item any non-trivial effective curve in $Y'$ is $D'_l$-convex. 
\eit
\label{def:propO}
\end{defn}

\begin{example}\label{ex:propO}
Denote by $Y'(D'_1,D'_2)$ the toric surface whose fan is given by  \cref{fig:dp3open}, with $D'_1=H-E_3$ and the class of $D'_2$ corresponding to the sum of the other rays. $Y'(D'_1,D'_2)$ is obtained from $\ptwo(1,4)=\ptwo(D_1,D_2)$ by blowing up a smooth point on $D_1$ and two infinitesimally close points on $D_2$. Moving the latter two apart (while staying on $D_2$) determines a deformation to $\delp_3(0,2)$. Given nefness of $D'_2$, it follows that $\delp_3(0,2)$ satisfies Property O. By Proposition \ref{prop:local-deform}, $\delp_3(1,1)$ also satisfies Property O. The property holds after blowing down $(-1)$-curves, including for $\fzero(0,4)$. Applying \cref{prop:local-deform} it thus also holds for $\fzero(2,2)$ and $\ftwo(2,2)$.

\end{example}

\begin{example} Consider now $\delp_4(D_1,D_2)$ with $D_1^2=0$. Deforming $\delp_4(D_1,D_2)$ to a smooth  toric surface with $D_1$ a toric divisor leads the fan of \cref{fig:dp3open} with an additional ray in the lower half-plane. Up to deformation, there are two ways of doing so: by either adding a ray between $H-E_1-E_2$ and $E_2$, or between $E_2$ and $E_1-E_2$. Either way, this creates a curve $C$ with $C\cdot (-K-D_1)<0$ and therefore $\delp_4(0,1)$ does not satisfy Property~O. The same argument applies to $\delp_5(0,0)$.

\end{example}

\begin{example} When $l>2$, Property~O is always  satisfied for all surfaces except for $\delp_3(0,0,0)$. The only way of deforming $\delp_3$ to a toric surface with $D_1$ and $D_2$ toric is to take the fan of \cref{fig:annulidp2} and add a ray in the lower left quadrant. But this creates a curve $C$ with $C\cdot (-K-D_1-D_2)<0$ and hence $\delp_3(0,0,0)$ does not satisfy Property~O.
\end{example}

From \cref{tab:classif},  Property~O coincides with quasi-tameness of $Y(D)$, with the sole exception of $\delp_3(0,0,0)$. \\


We make some informal comments about the geometric transition from stable log maps to open maps, which inform the construction of the open geometries below. This discussion is motivated by \cite[Section 7]{MR2386535} and in particular a natural generalisation of \cite[Conjecture 7.3]{MR2386535}. That description applied to our setting makes clear the structure of the toric Lagrangians.
Denote by $(Y,D=D_1+\cdots+D_l)$ a possibly non-compact log Calabi--Yau variety.
For a maximally tangent stable log map to $(Y,D)$, the expectation is that maximal tangency $d_j$ with $D_j$ can be replaced by an open boundary condition of winding number $d_j$ with a special Lagrangian $L_j$ near $D_j$.
The special Lagrangian needs to have the property that it bounds a holomorphic disk $\mathcal{D}$ in the normal bundle to $D_j$, see \cite[Section 7]{MR2386535}.
This property dictates how to compactify $Y\setminus D_j$: in a toric limit $\mathcal{D}$ is simply the disk used to compactify the edge the framing lies on.
If $d$ is a $D_j$-convex curve class, then we can alternatively remove the maximal tangency condition by twisting the geometry by $\shO_Y(-D_j)$.
$D_j$-convexity then guarantees that no maps move into the fibre direction. To obtain the Calabi--Yau threefold geometry from a surface, we adopt the convention of twisting by the last divisor $D_l$.

In the toric limits of Construction \ref{constr:YopD}, the choice of framing corresponds to a choice of compactification. If an outer edge $e$ has framing $f$, then, see \cite[\S3.2]{Li:2004uf}, the normal bundle of the compactification $C$ of $e$ is $\shO(f)\oplus\shO(-1-f)$. 
In our setting, one line bundle is the normal bundle $\shO_C(C^2)$ of the curve $C$ in the surface and the other is the normal bundle $\shO_Y(-D_l)\big{|}_C$ of the curve in the fibre direction. In Construction \ref{constr:YopD}, it follows from our conventions that if the framing points to the interior of the polytope, then the normal bundle of $C$ in the surface is $\shO(f)$, and if the framing points to the outside of the polytope, then the normal bundle of $C$ in the surface is $\shO(-1-f)$. In particular, from a Looijenga pair $Y(D_1,\dots,D_l)$ satisfying Property~O, we construct a dual Aganagic--Vafa open Gromov--Witten geometry via the following construction.

\begin{construction}\label{constr:YopD}
Let $Y(D_1,\dots,D_l)$ be a Looijenga pair satisfying Property O for $Y'(D_1',\dots,D'_l)$. Denote by $\Delta_{Y'}$ the polytope of $Y'$ polarised by $-K_{Y'}$.
We assume that $\Delta_{Y'\setminus (\cup_{j\neq l}D'_j)}$ is 2-dimensional or equivalently that $D'_l$ is not toric, implying $l<4$.
Denote by $e_j$ the edge of $\Delta_{Y'}$ corresponding to $D'_j$ for $1\leq j\leq l-1$ and denote by $e_l,\dots,e_{l+r}$ the remaining edges. Up to reordering, we may assume that the $e_i$ are oriented clockwise. We construct a toric Lagrangian triple $Y^{\rm op}(D) \coloneqq (X,L,\mathsf{f})$ as follows.
In $\Delta_{Y'}$ remove the edge $e_1$ and replace it by a framing $\mathsf{f}_1$ on $e_{l+r}$ parallel to $e_1$. By \cref{def:framing} and \cref{rmk:clockwise}, there is a unique way to do so and $\mathsf{f}_1$ is pointing into the interior of $\Delta_{Y'}$. Denote the resulting graph by $\Delta^1$. If $l=2$, add outer edges to $\Delta^1$ so that each vertex satisfies the balancing condition and denote the resulting toric Calabi--Yau graph by $\Gamma$. If $l=3$, in $\Delta^1$ remove the edge $e_2$ and replace it by a framing on $e_3$ parallel to $e_2$. Denote the resulting graph by $\Delta^2$. Add edges to $\Delta^2$ so that each vertex satisfies the balancing condition and denote the resulting toric Calabi--Yau graph by $\Gamma$.
\end{construction}

The graph $\Gamma$ in Construction \ref{constr:YopD} gives the discriminant locus of the SYZ fibration of the toric Calabi--Yau threefold $X=\mathrm{Tot}(K_{Y'\setminus (\cup_{j\neq l} D'_j)})$. The base of the fibration is an $\RR$-bundle over the polyhedron $\Delta_{Y'\setminus (\cup_{j\neq l} D'_j)}$. The framings determines toric special Lagrangians $L_j$, and the added outer edges correspond to the toric fibres of $\shO(-D'_l)$. As is readily seen from the fan, $f_1$, resp.\ $-1-f_2$, is the degree of the normal bundle of the divisor in $Y'$ corresponding to $e_{l+r}$, resp.\ to $e_3$. The framing keeps track of the compactification of $Y'\setminus (\cup_{j\neq l} D'_j)$.

\begin{rmk}
Tangency with more than one point can be incorporated by having parallel framings on different outer edges.
\end{rmk}

\begin{rmk}
If $\Delta_{Y'\setminus (D'_1\cup D_2')}$ is not 2-dimensional, then we blow up $Y$ in a smooth point of $D$ such that the resulting $\widetilde{Y}(\widetilde{D})$ satisfies Property O. We construct $\widetilde{Y}^{\rm op}(\widetilde{D})$, and recover the open invariants of $Y(D)$ by considering the curve classes that do not meet the exceptional divisor. In particular, for $l>3$ we  stipulate that Construction~\ref{constr:YopD} can be extended through suitable flopping of $(-1,-1)$-curves in the toric Calabi--Yau 3-fold geometry. We leave a precise formulation to future work, and develop the sole example relevant to our paper to illustrate this.
\end{rmk}

\begin{example}\label{ex:pairofpants}
We adapt the construction to the only nef Looijenga pair with 4 boundary components $\fzero(H_1,H_2,H'_1,H_2')$. 
Since $\Delta_{\fzero\setminus(H_1\cup H_2\cup H'_1)}$ is 1-dimensional, we start by blowing up a smooth point of $H'_1$. In a toric deformation, we obtain $\delp_2(0,0,-1,0)$ with divisors $D'_1=H_1$, $D'_2=H_2$, $D'_3=H'_1-E$ and $D'_4=H'_2$. 
We assume that corresponding $e_1,\dots,e_5$ are ordered clockwise in $\Delta_{\delp_1}$. Start with the graph $\Delta^2$ from Construction \ref{constr:YopD}. Balance the vertices and flop the inner edge. On the inner edge, add a framing parallel to $e_3$. The result is the graph of \cref{fig:delp2threeholes}, with the notational shift $\mathsf{f}_2\leftrightarrow D'_1$, $\mathsf{f}_3\leftrightarrow D'_2$, $\mathsf{f}_1\leftrightarrow D'_3$. To obtain the graph for $\fzero(0,0,0,0)$, we remove the two outer edges that have no framing. The result is Figure \ref{fig:f0threeholes}.

\end{example}

\begin{figure}[t]
\centering
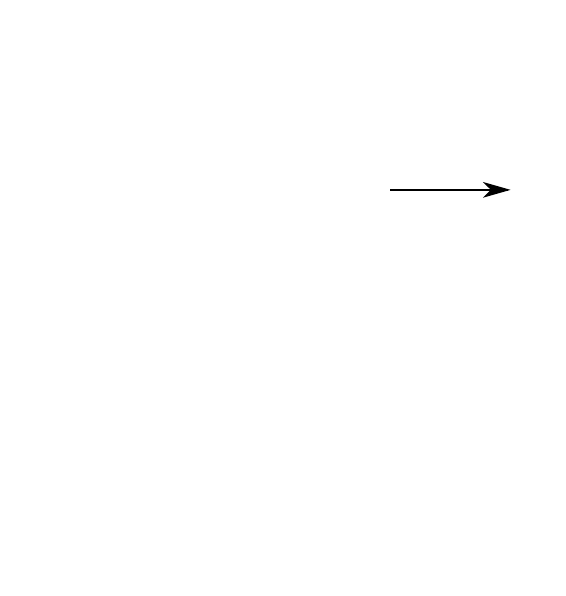
\caption{The toric CY3 graph of $\delp_2^{\rm op}(0,0,-1,0)$.}
\label{fig:delp2threeholes}
\end{figure}

\begin{lem}\label{lem:h2Y}

Let $Y(D_1,\dots,D_l)$ and $Y'(D_1',\dots,D'_l)$ be as in Construction \ref{constr:YopD}. Then $\hhh_2(Y, \bbZ)=\hhh_2(Y', \bbZ)$ is generated by the divisors corresponding to $e_3,\dots,e_{l+r}$.

\end{lem}

\begin{proof}
In the fan of $Y'$, define an ordering of the 2-dimensional cones by letting $\sigma_i$ be the cone corresponding to $e_i\cap e_{i+1}$ when $1\leq i < l+r$ and let $\sigma_{l+r}$ be the cone corresponding to $e_{l+r}\cap e_1$. Define cones $\tau_i:=\sigma_i\cap \bigcap_{j\in J_i} \sigma_j$, where $J_i$ is the set of $j>i$ such that $\sigma_i\cap\sigma_j$ is 1-dimensional. Then $\tau_1=\{0\}$, $\tau_i$ is the ray corresponding to $e_{i+1}$ for $2\leq i < l+r$ and $\tau_{l+r}=\sigma_{l+r}$.
By \cite[\S 5.2, Theorem]{MR1234037}, these cones generate $\hhh_*(Y', \bbZ)$ and hence the divisors corresponding to $e_3,\dots,e_{l+r}$ generate $\hhh_2(Y', \bbZ)$.
\end{proof}

Note that $\hhh^{\rm rel}_2(Y^{\rm op}(D),\bbZ)$ is generated by the curve classes $[e]$ corresponding to inner edges $e$ and by the relative disk classes $[\mathcal{D}_e]$ corresponding to outer edges $e$ with framings. By the corresponding short exact sequence, the latter can be identified with $[S^1]\in\hhh_1(S^1,\bbZ)$, where $L\supset S^1=\partial \mathcal{D}_e$ and the degrees in the $[S^1]$ keep track of the winding numbers. By construction, the $e$ thus described are edges of $\Delta_{Y'}$.

\begin{defn}\label{def:iota}
Let $Y(D_1,\dots,D_l)$ and $Y'(D_1',\dots,D'_l)$ be as in Construction \ref{constr:YopD}.
Define
\beq
\iota: \hhh^{\rm rel}_2(Y^{\rm op}(D), \bbZ) \rightarrow  \hhh_2(Y, \bbZ)
\eeq
by sending $[e]$ to the divisor corresponding to $e$ in $Y$.
\end{defn}

\begin{prop}\label{prop:iota}
The morphism $\iota$ is an isomorphism. 
\end{prop}

\begin{proof}
This is a direct consequence of Lemma \ref{lem:h2Y}.
\end{proof}

\begin{example}\label{ex:pairofpants2}
We continue with \cref{ex:pairofpants}. Following \cref{fig:delp2threeholes}, denote by $e_i$ the edge with framing $\mathsf{f}_i$ for $i=1,2,3$. Generalising \cref{def:iota}, we define $\iota : \hhh^{\rm rel}_2(\delp_2^{\rm op}(0,0,-1,0), \bbZ) \stackrel{\sim}{\rightarrow}  \hhh_2(\delp_2, \bbZ)$ by
\beq
\iota[e_1]=[D_2']=[H_2]\,,  \quad  \iota[e_2]=[D_4'-E]=[H_2-E]\,, \quad \iota [e_3]=[D'_3]=[H_1-E]\,,
\eeq
which yields an isomorphism.
\end{example}

\begin{thm}[The higher genus log-open principle]
Suppose $Y(D)$ satisfies Property~O. 
Then,
\beq 
O_{0;\iota^{-1}(d)}(Y^{\rm op}(D)) = N_{0,d}^{\rm loc}(Y(D)) =  \prod_{i=1}^{l}  \frac{(-1)^{d \cdot D_i+1}}{d \cdot D_i}  N_{0,d}^{\rm log}(Y(D))\,.
\label{eq:logopen0}
\eeq
Moreover, if $Y(D)$ is tame, 
\beq
 \mathsf{O}_{\iota^{-1}(d)}(Y^{\rm op}(D))(-\ri \log q)
 = [1]_q^{l-2} \,
  \frac{(-1)^{d \cdot D_l+1}}{[d \cdot D_l]_q} \, \prod_{i=1}^{l-1} \, \frac{(-1)^{d \cdot D_i+1}}{d \cdot D_i} \,
 \mathsf{N}_{d}^{\rm log}(Y(D))(-\ri \log q)\,.
\label{eq:logopen}
\eeq

\label{thm:logopen}
\end{thm}

\begin{rmk}
As is evident from Example \ref{ex:propO}, $Y^{\rm op}(D)$ depends on the toric model and hence is not unique.  However, it can be checked directly for the examples of \cref{tab:classif}  that if  $(X^{(1)}, L^{(1)}, \mathsf{f}^{(1)})$ and $(X^{(2)}, L^{(2)}, \mathsf{f}^{(2)})$ correspond to two such choices, then $\exists~\varpi: \hhh_2^{\rm rel}(X^{(1)},L^{(1)},\Z) \stackrel{\sim}{\rightarrow} \hhh_2^{\rm rel}(X^{(2)},L^{(2)},\Z)$ such that $\iota^{(1)} = \iota^{(2)} \circ \varpi$.
\end{rmk}

\subsection{Proof of \cref{thm:logopen}}
\label{sec:logopen}

In order to work our way to a general $Y(D)$ satisfying Property~O, we  
first show that if $\pi: Y'(D') \to Y(D)$ is an interior blow-up, Construction \ref{constr:YopD} implies that the higher genus open GW invariants $\mathsf{O}_{\iota^{-1}(d)}(Y^{\rm op}(D))$ 
satisfy the same blow-up formula \eqref{eq:logblup} of the log invariants on the r.h.s. of \eqref{eq:logopen}.

\begin{prop}[Blow-up formula for open GW invariants]
Let $\pi : \tilde Y(\tilde D) \to Y(D)$ be an interior blow-up of Looijenga pairs with both $\tilde Y(\tilde D)$ and $Y(D)$ satisfying Property~O, and denote $\pi^*_{\rm op}$ 
the monomorphism of Abelian groups defined by
\beq 
\xymatrix{ 
\hhh_2(Y(D),\bbZ) \ar[d]^{\iota^{-1}}  \ar@{^{(}->}[r]^{\pi^*} &
  \hhh_2(\tilde Y(\tilde D),\bbZ)\ar[d]^{\iota^{-1}}    \\
                     \hhh_2^{\rm rel}(Y^{\rm op}(D),\bbZ)  \ar@{^{(}->}[r]^{\pi^*_{\rm op}} 
 &  
\hhh_2^{\rm rel}(\tilde Y^{\rm op}(\tilde D),\bbZ)  & 
}
\eeq 
Then $\mathsf{O}_{j}(Y^{\rm op}(D)) = \mathsf{O}_{\pi^*_{\rm op} j}(\tilde Y^{\rm op}(\tilde D))$ for all $j\in\hhh_2^{\rm rel}(Y^{\rm op}(D),\bbZ)$.
\label{prop:openblup}
\end{prop}

\begin{proof}[Sketch of the proof] We provide an overview here and leave the details to the reader.
The claim is proved by noting that Construction \ref{constr:YopD} implies the following: if $Y(D)$ is obtained from $\tilde Y(\tilde D)$ by contraction of a $(-1)$ curve, then $Y^{\rm op}(D)$ is an open embedding into a flop of $\tilde Y^{\rm op}(\tilde D)$ along a $(-1,-1)$ curve. The resulting non-trivial equality of open Gromov--Witten invariants under restriction on the image of $\pi^*_{\rm op}$ is then a combination of the invariance of open Gromov--Witten invariants under `forgetting an edge' in \eqref{eq:forget} and the flop invariance of the topological vertex \cite{MR2380471}. 
\end{proof}

By the previous Proposition it then suffices to prove \cref{thm:logopen} for the pairs $Y(D)$ of highest Picard number for each value of $l=2,3,4$, as all other pairs are recovered from these by blowing-down. We show this from a direct use of the topological vertex to determine the l.h.s. of \eqref{eq:logopen}. The reader is referred to \cref{sec:prelim} for notation and basic results for shifted power sums $p_\a(q^{\rho+\gamma})$ and shifted skew Schur functions $s_\frac{\a}{\b}(q^{\rho+\gamma})$ in the principal stable specialisation. The notation $\{\a,\b\}_Q$ indicates the symmetric pairing on $\cP$ of \eqref{eq:cauchyprodpss}.

\subsubsection{$l=2$: holomorphic disks}
\label{sec:disks}
The classification of \cref{prop:class,prop:class2}, the deformation equivalences in \cref{prop:local-deform} and the definition of Property~O in \cref{def:propO} together imply that if $Y(D=D_1+\dots +D_l)$
and $Y(D'=D'_1+\dots +D'_l)$ are $l$-component Looijenga pairs both satisfying Property~O, then there is a toric model for both with resulting $Y^{\rm op}(D)=Y^{\rm op}(D')$: in other words a model $Y^{\rm op}(D)$ for the open geometry only depends on $Y$ and the number of irreducible components of $D$. Since $2$-component log CY surfaces with maximal boundary come in pairs $Y(D)$ and $Y(D')$ from \cref{tab:classif}, throughout this section we will simplify notation and write $\Upsilon(Y)\coloneqq Y^{\rm op}(D)=Y^{\rm op}(D')$ for the toric Lagrangian triple they share.

By \cref{prop:openblup}, it suffices to consider the case of highest Picard rank $Y=\delp_3$. For $Y(D)$ either $\delp_3(1,1)$ or $\delp_3(0,2)$, a toric model for $Y$ 
is given by the toric surface $Y'$ described by the fan of 
\cref{fig:dp3open}, and in particular $D'=H-E_3$ is a toric divisor. Therefore $Y(D)$ satisfies Property~O and, by \cref{rmk:fangraph}, $\Upsilon(\delp_3)$ is described by the toric CY3 graph 
of \cref{fig:dp3open}. With conventions as in \cref{fig:dp3open}, let $C_1 = \mu_T^{-1}(e_2)$, $C_2 = \mu_T^{-1}(e_5)$, $C_3 = \mu_T^{-1}(e_7)$ and for a relative 2-homology class $j\in \hhh_2(\Upsilon(\delp_3), \bbZ)$ write $j =  j_0 [S^1] + \sum_{i=1}^3 j_i [C_i]$. \\
We will compute generating functions of higher genus 1-holed open Gromov--Witten invariants of $\Upsilon(\delp_3)$ in class $j$ using the theory of the topological vertex. For simplicity, we'll employ the shorthand notation $\mathsf{O}_{j_1,j_2,j_3; j_0}(\Upsilon(\delp_3))$, resp. $\cO_{j_0}(\Upsilon(\delp_3))$, to denote the generating function $\mathsf{O}_{\beta; \mu}(\Upsilon(\delp_3))$, resp. $\cO_{\beta}(\Upsilon(\delp_3))$ in \eqref{eq:opengf} with $\beta=\sum_{i=1}^3 j_i [C_i]$ and $\mu=(j_0)$ a 1-row partition of length $j_0$. From \eqref{eq:openfreeen} and \eqref{eq:opengf}, we have
\bea
\cO_{j_0}(\Upsilon(\delp_3)) &=& \frac{\cZ_{(j_0)}(\Upsilon(\delp_3))}{\cZ_{\emptyset}(\Upsilon(\delp_3))} = \sum_{\nu \in \cP}  \frac{\chi_{\nu}( (j_0) )}{z_{(j_0)}} \frac{\cW_{\nu}(\Upsilon(\delp_3))}{\cW_{\emptyset}(\Upsilon(\delp_3))} \nn \\
&=&  \sum_{s=0}^{j_0-1}  \frac{(-1)^s}{j_0} \frac{\cW_{(j_0-s,1^s)}(\Upsilon(\delp_3))}{\cW_{\emptyset}(\Upsilon(\delp_3))}\,,
\label{eq:diskgf}
\eea
where we have used the Murnaghan--Nakayama rule \cite[Corollary~7.17.5]{MR1676282}
\beq
\chi_{\nu}( (j_0) ) =
\l\{
\bary{cc}
(-1)^s, & \nu = (j_0-s, 1^s)\,, \\
0 & \mathrm{else}.
\eary
\r.
\label{eq:murnak}
\eeq
The framing $\mathsf{f}$ in \cref{fig:dp3open} is shifted by one unit $f=-1$ from the canonical choice $\mathsf{f_{\rm can}}=\mathsf{p}_{v_1}^{e_2}$. From \eqref{eq:framvert}, \eqref{eq:glueing} and \eqref{eq:vertex}, we then have, for any $\a \in \cP$, that
\bea
& & \cW_\a(\Upsilon(\delp_3))(Q,\hbar) = \nn \\
& & (-1)^{|\a|}q^{-\kappa(\a)/2} \sum_{\lambda, \mu,\nu, \delta, \epsilon \in \cP} s_{\lambda^t}(-Q_1 q^{\rho+\a}) s_\a(q^\rho) s_\frac{\lambda}{\delta}(q^{\rho})s_\frac{\mu}{\delta}(q^{\rho}) Q_2^{|\mu|} s_\frac{\mu}{\epsilon}(q^{\rho})s_\frac{\nu}{\epsilon}(q^{\rho}) s_{\nu^t}(-Q_3 q^\rho) \nn \\
&=&
 \frac{(-1)^{|\a|} s_{\alpha^t}(q^\rho) \{\a,\emptyset\}_{Q_1} \{\a,\emptyset\}_{Q_1 Q_2} \{\emptyset,\emptyset\}_{Q_3} \{\emptyset,\emptyset\}_{Q_2 Q_3}}{\{\emptyset, \emptyset \}_{Q_2} \{\a,\emptyset\}_{Q_1 Q_2 Q_3}} \,,
 \nn \\
\eea
where we have used \eqref{eq:schurt} and, repeatedly, \eqref{eq:cauchyprodpss} to express the sums over partitions in terms of Cauchy products. Then, specialising to $\a=(j_0-s,1^s)$ a hook partition with $j_0$ boxes and $s+1$ rows, and using \eqref{eq:stanley} and \eqref{eq:hookpairing}, we have
\bea
\frac{\cW(\Upsilon(\delp_3))_{(j_0-s,1^s)}}{\cW(\Upsilon(\delp_3))_{\emptyset}}
&=&  \frac{(-1)^{j_0} s_{(s+1,1^{j_0-s-1})}(q^\rho) 
\{(j_0-s, 1^s),\emptyset\}_{Q_1} \{(j_0-s, 1^s),\emptyset\}_{Q_1 Q_2}}{\{(j_0-s, 1^s),\emptyset\}_{Q_1 Q_2 Q_3}}
\nn \\
&=&  \frac{(-1)^{j_0} q^{-\frac{1}{2}\binom{j_0}{2} + \frac{j_0 s}{2}}}{[j_0]_q [j_0-s-1]_q! [s]_q!} \frac{\prod_{k=0}^{j_0-1} (1-q^{k} Q_1 q^{-s}) \prod_{l=0}^{j_0-1} (1-q^{l} Q_1 Q_2 q^{-s})}{\prod_{m=0}^{j_0-1}  (1-q^{m} Q_1 Q_2 Q_3 q^{-s})}\,.\nn \\
\label{eq:Wdp3hooks}
\eea
\begin{figure}[t]
\begingroup%
  \makeatletter%
  \providecommand\color[2][]{%
    \errmessage{(Inkscape) Color is used for the text in Inkscape, but the package 'color.sty' is not loaded}%
    \renewcommand\color[2][]{}%
  }%
  \providecommand\transparent[1]{%
    \errmessage{(Inkscape) Transparency is used (non-zero) for the text in Inkscape, but the package 'transparent.sty' is not loaded}%
    \renewcommand\transparent[1]{}%
  }%
  \providecommand\rotatebox[2]{#2}%
  \newcommand*\fsize{\dimexpr\f@size pt\relax}%
  \newcommand*\lineheight[1]{\fontsize{\fsize}{#1\fsize}\selectfont}%
  \ifx\svgwidth\undefined%
    \setlength{\unitlength}{120.13718492bp}%
    \ifx\svgscale\undefined%
      \relax%
    \else%
      \setlength{\unitlength}{\unitlength * \real{\svgscale}}%
    \fi%
  \else%
    \setlength{\unitlength}{\svgwidth}%
  \fi%
  \global\let\svgwidth\undefined%
  \global\let\svgscale\undefined%
  \makeatother%
  \begin{picture}(1,0.98297525)%
    \lineheight{1}%
    \setlength\tabcolsep{0pt}%
    \put(0,0){\includegraphics[width=\unitlength,page=1]{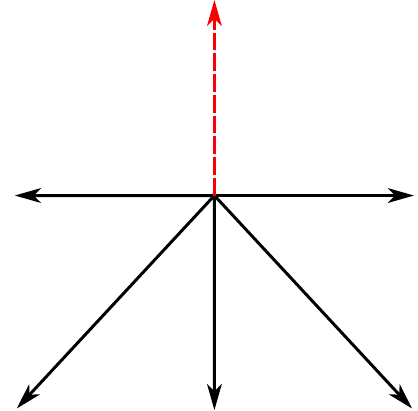}}%
    \put(0.21320338,0.87110191){\color[rgb]{0.97647059,0,0}\makebox(0,0)[lt]{\lineheight{1.25}\smash{\begin{tabular}[t]{l}$H-E_3$\end{tabular}}}}%
    \put(-0.00509817,0.55469234){\color[rgb]{0,0,0}\makebox(0,0)[lt]{\lineheight{1.25}\smash{\begin{tabular}[t]{l}$H-E_1-E_2$\end{tabular}}}}%
    \put(0.107125,0.2369378){\color[rgb]{0,0,0}\makebox(0,0)[lt]{\lineheight{1.25}\smash{\begin{tabular}[t]{l}$E_2$\end{tabular}}}}%
    \put(0.5287105,0.06018823){\color[rgb]{0,0,0}\makebox(0,0)[lt]{\lineheight{1.25}\smash{\begin{tabular}[t]{l}$E_1-E_2$\end{tabular}}}}%
    \put(0.78722804,0.23577819){\color[rgb]{0,0,0}\makebox(0,0)[lt]{\lineheight{1.25}\smash{\begin{tabular}[t]{l}$H-E_1-E_3$\end{tabular}}}}%
    \put(0.79054748,0.54476408){\color[rgb]{0,0,0}\makebox(0,0)[lt]{\lineheight{1.25}\smash{\begin{tabular}[t]{l}$E_3$\end{tabular}}}}%
  \end{picture}%
\endgroup%

\hspace{2cm}
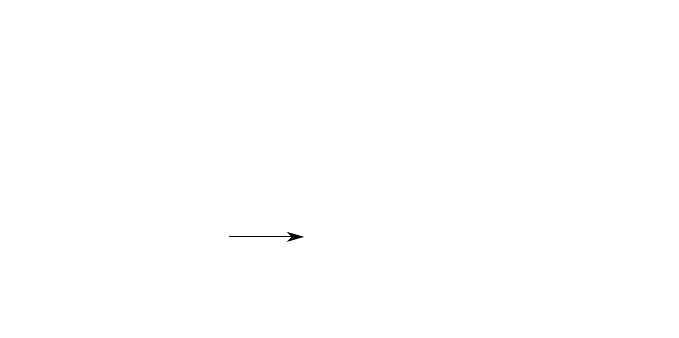
\caption{$\Upsilon(\delp_3) = \delp_3^{\rm op}(0,2)=\delp_3^{\rm op}(1,1)$ from the blow-up of the plane at three non-generic toric points.}
\label{fig:dp3open}
\end{figure}
Replacing this into \eqref{eq:diskgf} we get
\bea
& & \cO_{j_0}(\Upsilon(\delp_3))(Q,\hbar)  
 =    \sum_{s=0}^{j_0-1} \frac{(-1)^s}{j_0} \frac{\cW_{(j_0-s,1^s)}(\Upsilon(\delp_3))(Q,\hbar)}{\cW_{\emptyset}(\Upsilon(\delp_3))(Q,\hbar)} \nn \\
&=& \frac{(-1)^{j_0} q^{-\frac{1}{2}\binom{j_0}{2}}}{j_0 [j_0]_q!} \sum_{j_1,j_2,j_3=0}^\infty  q^{\frac{j_1(j_0-1)}{2}} \qbinom{j_0}{j_1-j_2}_q  \qbinom{j_0}{j_2-j_3}_q \qbinom{j_0+j_3-1}{j_3}_q (-1)^{j_1+j_3} Q_1^{j_1} Q_2^{j_2} Q_3^{j_3} \nn
\\ & &
\times \sum_{s=0}^{j_0-1} \qbinom{j_0-1}{s}_q (-q^{-j_1})^{s}  q^{\frac{1}{2}j_0s }   \nn \\
&=& \frac{(-1)^{j_0}}{j_0 [j_0]_q!} \sum_{j_1,j_2,j_3}^\infty   \qbinom{j_0}{j_1-j_2}_q  \qbinom{j_0}{j_2-j_3}_q  \qbinom{j_0+j_3-1}{j_3}_q (-1)^{j_1+j_3}Q_1^{j_1} Q_2^{j_2} Q_3^{j_3} \frac{[j_1-1]_q!}{[j_1-j_0]_q!} \,, 
\eea 
where the $q$-binomial theorem has been used to expand the products in \eqref{eq:Wdp3hooks} and to perform the summation over $s$ in the last line. Isolating the $\cO(Q_1^{j_1}Q_2^{j_2} Q_2^{j_3} )$ coefficient yields
\beq
\mathsf{O}_{j_1,j_2,j_3; j_0}(\Upsilon(\delp_3))(\hbar) = \frac{(-1)^{j_1+j_0+j_3} [j_0]_q}{j_0 [j_1]_q [j_0+j_3]_q}  \qbinom{j_0}{j_1-j_2}_q  \qbinom{j_0}{j_2-j_3}_q  \qbinom{j_0+j_3}{j_3}_q  \qbinom{j_1}{j_0}_q  \,.
\label{eq:opendP3}
\eeq
From \cref{fig:dp3open}, the lattice isomorphism $\iota: \hhh^{\rm rel}_2(\Upsilon(\delp_3),\bbZ) \to  \hhh_2(\delp_3, \bbZ)$ in this case reads
\bea
\iota[S^1]=[H-E_1-E_2]\,, & \quad & \iota[C_1]=[E_2]\,, \nn \\
\quad \iota [C_2]=[E_1-E_2]\,, & \quad & \iota[C_3] =[H-E_1-E_3]\,,
\eea
and the change-of-variables relating the curve degrees $(d_0, d_1, d_2,d_3)$ in $\hhh_2(\delp_3, \bbZ)$ and the relative homology variables $(j_0; j_1, j_2,j_3)$ in $\hhh^{\rm rel}_2(\Upsilon(\delp_3), \bbZ)$ is therefore
\beq
\begin{array}{c}
 d_0\to j_0+j_3\,, \\
 d_1\to j_2\,, \\
 d_2\to j_1-j_2+j_3\,, \\
 d_3\to j_0\,. \\
\end{array}
\label{eq:covdP3}
\eeq
Combining the change of variables  \eqref{eq:covdP3} and the log result given by \eqref{eq:dP311log} in \cref{prop:dp311}
 returns \eqref{eq:opendP3}, establishing \eqref{eq:logopen} for $Y(D)=\delp_3(1,1)$. Furthermore, taking the genus zero limit $q\to 1$ and using \cref{thm_log_local,lem:localgw,prop:local-deform} implies \eqref{eq:logopen0}, completing the proof of \cref{thm:logopen} for $Y(D)=\delp_3(D_1^2, D_2^2)$. Use of \cref{prop:openblup,prop:logblup} then concludes the proof of \cref{thm:logopen} for any $Y(D)$ satisfying Property~O with $l=2$.

\subsubsection{$l=3$: holomorphic annuli}

The 3-component Looijenga pair of highest Picard rank satisfying Property~O is $Y(D)=\delp_2(1,0,0)$. Taking $D_1=H-E_1$, $D_2=H-E_2$ we have that $Y$, $D_1$ and $D_2$ are toric, and $\delp_2^{\rm op}(1,0,0)$ is described by the toric CY3 graph in the left part of \cref{fig:annulidp2}. Write $C = \mu_T^{-1}(e_3)$, and for a relative 2-homology class $j\in \hhh^{\rm rel}_2(\delp_2^{\rm op}(1,0,0), \bbZ)$ write $j =  j_1 [\mathcal{D}_1] + j_2 [\mathcal{D}_2] + j_C [C]$ where $[\mathcal{D}_i]$ are integral generators of the first homology of the outer Lagrangians incident to edges adjacent to the vertices $v_i$, $i=1,2$ in \cref{fig:annulidp2}. As in the previous section, we will write $\mathsf{O}_{j_C; j_1,j_2}(\delp_2^{\rm op}(1,0,0))$ (resp. $\cO_{j_1,j_2}(\delp_2^{\rm op}(1,0,0))$) for the generating function $\mathsf{O}_{\beta, \vec \mu}(\delp_2^{\rm op}(1,0,0))$ (resp. $\cO_{\vec\mu}(\delp_2^{\rm op}(1,0,0))$) with $\beta=j_C [C]$ and $\vec \mu= ((j_1),(j_2))$ is a pair of 1-row partitions of length $(j_1,j_2)$. From \eqref{eq:openfreeen}, \eqref{eq:opengf} and \eqref{eq:murnak}, we have
\bea
\cO_{j_1,j_2}(\delp_2^{\rm op}(1,0,0))(Q,\hbar) &=& \frac{\cZ_{j_1,j_2}(\delp_2^{\rm op}(1,0,0))}{Z_{\emptyset, \emptyset}(\delp_2^{\rm op}(1,0,0))} 
-
\frac{\cZ_{(j_1),\emptyset}(\delp_2^{\rm op}(1,0,0)) \cZ_{\emptyset, (j_2)}(\delp_2^{\rm op}(1,0,0))}{\cZ_{\emptyset, \emptyset}(\delp_2^{\rm op}(1,0,0))^2} \nn \\
&=&  \sum_{s_1=0}^{j_1-1}\sum_{s_2=0}^{j_2-1}  \frac{(-1)^{s_1+s_2}}{j_1 j_2}
\bigg[
\frac{\cW_{(j_1-s_1,1^{s_1}),(j_2-s_2,1^{s_2})}(\delp_2^{\rm op}(1,0,0))}{\cW_{\emptyset,\emptyset}(\delp_2^{\rm op}(1,0,0))}
\nn \\ &-&
\frac{\cW_{(j_1-s_1,1^{s_1}),\emptyset}(\delp_2^{\rm op}(1,0,0))
\cW_{\emptyset,(j_2-s_2,1^{s_2})}(\delp_2^{\rm op}(1,0,0))
}{\cW_{\emptyset, \emptyset}(\delp_2^{\rm op}(1,0,0))^2}
\bigg] \,. \nn \\
\label{eq:annuligf}
\eea
\begin{figure}[t]
\centering
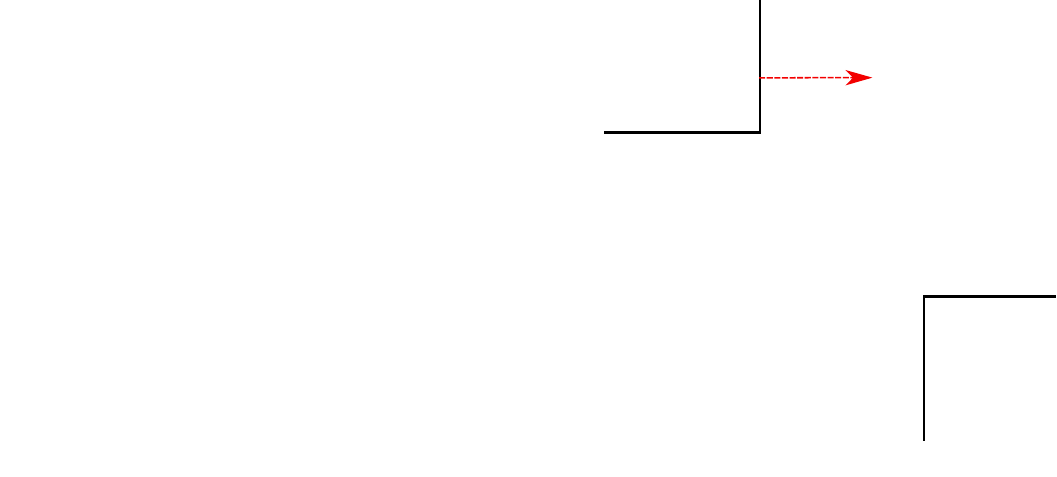
\caption{$\delp_2^{\rm op}(1,0,0)$ from $\delp_2$ and $D_1=H-E_1$, $D_2=H-E_2$.}
\label{fig:annulidp2}
\end{figure}
The framings $\mathsf{f}_1$ and $\mathsf{f}_2$ in \cref{fig:annulidp2} are, respectively, shifted by one unit $f=-1$ from the canonical choice $\mathsf{f_{\rm can}}=\mathsf{p}_{v_1}^{e_3}$, and equal to the canonical framing $\mathsf{f}_2=\mathsf{p}_{v_2}^{e_4}$. Then \eqref{eq:framvert}, \eqref{eq:glueing} and \eqref{eq:vertex} give
\bea
\cW_{\a\b}(\delp_2^{\rm op}(1,0,0))(Q,\hbar) & =&
(-1)^{|\a|}q^{-\kappa(\a)/2} \sum_{\mu,\delta,  \in \cP} s_{\mu^t}(-q^{\rho+\a} Q) s_\a(q^\rho) s_\frac{\mu}{\delta}(q^\rho)s_\frac{\beta}{\delta}(q^\rho) \nn \\
&=  & (-1)^{|\a|} s_{\a^t}(q^\rho) \{\a,\emptyset \}_Q \sum_{\delta \in \cP} s_{\frac{\beta^t}{\delta^t}}(-q^{-\rho}) s_{\delta^t} (-q^{\rho+\a}Q) \nn \\
&=  & (-1)^{|\a|} s_{\a^t}(q^\rho) \{\a,\emptyset \}_Q s_{\beta^t}(-q^{-\rho}, -q^{\rho+\a}Q)\,.
\eea
where we have used \eqref{eq:skewschurt}, \eqref{eq:doubleschur} and \eqref{eq:cauchyprodpss} to perform the summations over partitions. Then, restricting to $\a=(j_1-s_1,1^{s_1})$, $\b=(j_2-s_2,1^{s_2})$,
\bea
& &\cO_{j_1,j_2}(\delp_2^{\rm op}(1,0,0))(Q,\hbar)   
   =    \sum_{s_1=0}^{j_1-1} \frac{(-1)^{j_1+s_1+j_2+1}}{j_1 j_2}
 s_{(s_1+1,1^{j_1-s_1-1})}(q^\rho)
 \prod_{k=0}^{j_1-1}(1- q^k Q q^{-s_1})
  \nn \\
 & \times & 
  \l[p_{(j_2)}(-Q q^{\rho+(j_1-s_1,1^{s_1})} ,-q^{-\rho})-p_{(j_2)}(-Q q^{\rho} ,-q^{-\rho}) \r] \nn \\
  & = &   \sum_{s_1=0}^{j_1-1} \frac{(-1)^{j_1+s_1+j_2+1}}{j_1 j_2}
 s_{(s_1+1,1^{j_1-s_1-1})}(q^\rho)
 \prod_{k=0}^{j_1-1}(1- q^k Q q^{-s_1})
  \nn \\
 & \times & 
  \l[p_{(j_2)}(-Q q^{\rho+(j_1-s_1,1^{s_1})} )-p_{(j_2)}(-Q q^{\rho} ) \r] 
   \nn \\
& = &   \sum_{s_1=0}^{j_1-1} \frac{(-1)^{j_1+s_1+j_2+1}}{j_1 j_2}
 s_{(s_1+1,1^{j_1-s_1-1})}(q^\rho)
 \prod_{k=0}^{j_1-1}(1- q^k Q q^{-s_1})(-Q q^{-s_1-1/2})^{j_2}
  \l[
  q^{j_2j_1}-1
\r]
\nn \\
  & = & \frac{(-1)^{j_1+1} Q^{j_2}[j_1 j_2]_q}{j_1 j_2 [j_2+m]_q} 
  \sum_{m=0}^{j_1} 
\qbinom{j_1}{m}_q 
  \qbinom{j_2+m}{j_1}_q (-Q)^{m}.
\eea 
where in the first row we have used \eqref{eq:schurvspow} and \eqref{eq:murnak}, in the second the fact that for a 1-row partition $\a=(d)$, $p_{(d)}(x_1, \dots, x_n, \dots ; y_1, \dots, y_n, \dots )=p_{(d)}(x_1, \dots, x_n, \dots) + p_{(d)}(y_1, \dots, y_n, \dots )$, and in the third row the fact that the difference of infinite power sums in the term in square brackets telescopes to just two terms; the final calculations are repeated applications of the $q$-binomial theorem. Extracting the $\cO(Q^{j_C})$ coefficient we get
\beq
\mathsf{O}_{j_C; j_1,j_2}(\delp_2^{\rm op}(1,0,0))(\hbar) =  \frac{(-1)^{j_1+1+j_C+j_2} [j_1 j_2]_q}{j_1 j_2 [j_C]_q} 
\qbinom{j_1}{j_C-j_2}_q
  \qbinom{j_C}{j_1}_q \,.
\label{eq:opendP2_1_0_0}
\eeq
From \cref{fig:annulidp2}, the homomorphism of homology groups $\iota: \hhh_2^{\rm rel}(\delp_2^{\rm op}(1,0,0),\bbZ) \to \hhh_2(\delp_2, \bbZ)$ is given by
\beq
\iota[D_1]= [E_1],  \quad  \iota[D_2]=[E_2], \quad
\iota[C]= [H-E_1-E_2],
\eeq
and the resulting map of curve degrees is
\beq
\begin{array}{c}
 d_0\to j_C, \\
 d_1\to j_1, \\
 d_2\to j_2. \\ \\
\end{array}
\label{eq:covdP2ann}
\eeq
Together with the log results given by \eqref{eq:log_dp3_0_0_0}
in \cref{thm:log_dp3_0_0_0} 
and the blow-up formulas of \cref{prop:logblup,prop:openblup}
for the log and open invariants, this proves \cref{thm:logopen} for $l=3$. 

\subsubsection{$l=4$: holomorphic pairs of pants}

According to \cref{ex:pairofpants}, for the only 4-component case $Y(D)= \bbF_0(0,0,0,0)$, we have that $Y^{\rm op}(D)$ is given by the 3-dimensional affine space with Aganagic--Vafa A-branes $L^{(i)}$, $i=1,2,3$ at framing shifted by $-1$, $0$, and $-1$ ending on the three legs of the vertex, as in \cref{fig:f0threeholes}. We will be concerned with counts of 3-holed open Gromov--Witten invariants of $\bbF_0^{\rm op}(0,0,0,0)$, with winding numbers $(j_1, j_1, j_2)$, see Example \ref{ex:pairofpants2}.

The connected generating function, by \eqref{eq:openfreeen} and \eqref{eq:opengf}, is
\bea
& &
\mathsf{O}_{j_1,j_1, j_2}(\bbF_0^{\rm op}(0,0,0,0))(\hbar) = \nn \\
&=& \cZ_{(j_1)(j_1)(j_2)}(\bbF_0^{\rm op}(0,0,0,0))-\cZ_{(j_1)(j_1)\emptyset}(\bbF_0^{\rm op}(0,0,0,0))\cZ_{\emptyset\emptyset(j_2)}(\bbF_0^{\rm op}(0,0,0,0)) \nn \\ 
&-& \cZ_{(j_1)\emptyset(j_2)}(\bbF_0^{\rm op}(0,0,0,0))\cZ_{\emptyset(j_1)\emptyset}(\bbF_0^{\rm op}(0,0,0,0)) - \cZ_{\emptyset(j_1)(j_2)}(\bbF_0^{\rm op}(0,0,0,0))\cZ_{(j_1)\emptyset\emptyset}(\bbF_0^{\rm op}(0,0,0,0)) \nn \\ 
&+& 2 \cZ_{(j_1)\emptyset\emptyset}(\bbF_0^{\rm op}(0,0,0,0)) \cZ_{\emptyset(j_1)\emptyset}(\bbF_0^{\rm op}(0,0,0,0))\cZ_{\emptyset\emptyset(j_2)}(\bbF_0^{\rm op}(0,0,0,0)).
\label{eq:f03hF}
\eea
where, by \eqref{eq:murnak},
\beq
\cZ_{(j_1),(j_1),(j_2)}(\bbF_0^{\rm op}(0,0,0,0)) = \sum_{s_0,s_1,s_2} \frac{(-1)^{s_0+s_1 +s_2}}{j_1^2 j_2} \cW(\bbF_0^{\rm op}(0,0,0,0))_{(j_1-s_0,1^{s_0}),(j_1-s_1,1^{s_1}), (j_2-s_2,1^{s_2})},
\eeq
and, from \eqref{eq:framvert} and \eqref{eq:vertex},
\bea
\cW_{\a\b\gamma}(\bbF_0^{\rm op}(0,0,0,0)) &=& (-1)^{|\a|+|\gamma|}\sum_{\delta} s_{\frac{\a^t}{\delta}}(q^{\rho+\gamma}) s_{\frac{\b}{\delta}}(q^{\rho+\gamma^t}) s_{\gamma^t}(q^\rho).
\label{eq:w03hrep}
\eea
\begin{figure}[t]
\centering
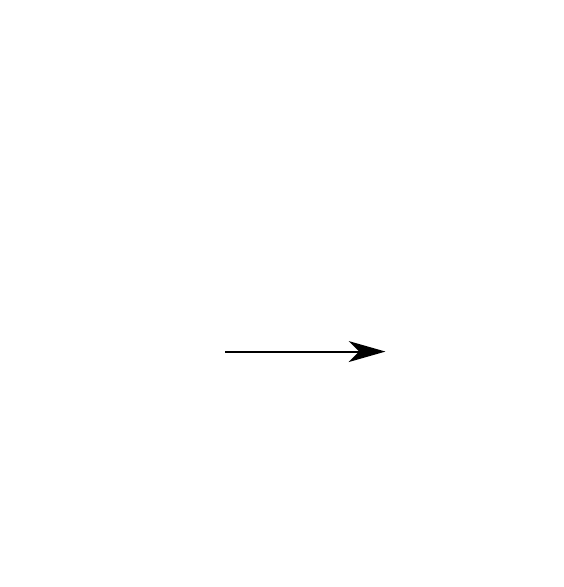
\caption{The toric CY3 graph of $\bbF^{\rm op}_0(0,0,0,0)$.}
\label{fig:f0threeholes}
\end{figure}
Elementary manipulations from \eqref{eq:f03hF}--\eqref{eq:w03hrep} lead to
%
\bea
& & \mathsf{O}_{j_1,j_1, j_2}(\bbF_0^{\rm op}(0,0,0,0))(\hbar) = \nn \\ &=& \sum_{s_0,s_1,s_2} \frac{(-1)^{s_0+s_1 +s_2+j_1+j_2}}{j_1^2 j_2}\bigg[\big(s_{\a^t}(q^{\rho+\gamma})-s_{\a^t}(q^\rho) \big) \big(s_{\b}(q^{\rho+\gamma^t})-s_{\b}(q^\rho) \big) s_{\gamma^t}(q^\rho) \nn \\
&+& \sum_{\delta\neq \emptyset} \l( s_{\frac{\a^t}{\delta}}(q^{\rho+\gamma}) s_{\frac{\b}{\delta}}(q^{\rho+\gamma^t})-s_{\frac{\a^t}{\delta}}(q^{\rho}) s_{\frac{\b}{\delta}}(q^{\rho})  \r) s_{\gamma^t}(q^\rho)\bigg]
\label{eq:f03hF2}
\eea
with $\a=(j_1-s_0,1^{s_0})$, $\b=(j_1-s_1,1^{s_1})$, $\gamma=(j_2-s_2,1^{s_2})$. The first row, after carrying out the sums over $s_0$, $s_1$ and $s_2$ using \eqref{eq:schurvspow} and \eqref{eq:murnak}, is equal to 
\bea
& & 
\sum_{s_2=0}^{j_2-1} \frac{(-1)^{s_2+1+j_2}}{j_1^2 j_2}\big(p_{(j_1)}(q^{\rho+(j_2-s_2,1^{s_2})})-p_{(j_1)}(q^\rho) \big) \big(p_{(j_1)}(q^{\rho+(s_2+1,1^{j_2-s_2-1})})-p_{(j_1)}(q^\rho) \big) \nn \\ 
& \times & 
s_{(s_2+1,1^{j_2-s_2-1})}(q^\rho) \nn \\
&=& 
\sum_{s_2=0}^{j_2-1} \frac{(-1)^{s_2}}{j_1^2 j_2} \big(q^{j_1(j_2-s_2-1/2)}-q^{j_1(-s_2-1/2)} \big) \big(q^{j_1(s_2+1/2)}-q^{j_1(s_2+1/2-j_2)} \big) s_{(j_2-s_2,1^{s_2})}(q^\rho) \nn \\
&=& \frac{1}{j_1^2 j_2} \frac{[j_1 j_2]_q^2}{[j_2]_q},
\eea
while the second row is equal to zero. Indeed, when $\delta=\a^t$, we have $s_{\frac{\beta}{\delta}}(x)=\delta_{\b\a^t}$ (since $|\a|=|\b|=j_1$ in our case, $\a^t \preceq \b$ implies $\a^t=\b$), so the terms appearing in the difference in the second row of \eqref{eq:f03hF2} are either individually zero or cancel out each other. When $\delta \neq \a^t$, we can use \cref{lem:skewhooks} to expand $s_\frac{\a^t}{\delta}(x)$ in terms of ordinary Schur functions $s_{\lambda}(x)$ with $|\lambda|=|\a|-|\delta|$: it is easy to see that in the sum over $s_0$ the contribution labelled by each such Young diagram appears exactly twice and weighted with opposite signs. Therefore,
\beq
\mathsf{O}_{j_1,j_1, j_2}(\bbF_0^{\rm op}(0,0,0,0))(\hbar) = \frac{1}{j_1^2 j_2} \frac{[j_1 j_2]_q^2}{[j_2]_q}.
\label{eq:f03hFfin}
\eeq
By construction from Examples \ref{ex:pairofpants} and \ref{ex:pairofpants2},
\beq
\begin{array}{c}
 d_1\to j_2, \\
 d_2\to j_1, \\ \\
\end{array}
\eeq
and comparing with \eqref{eq:fzerolog} gives \eqref{eq:logopen}, which concludes the proof of \cref{thm:logopen}. \begin{flushright}$\Box$\end{flushright}

\begin{table}[!t]
\begin{minipage}{.45\textwidth}
\centering
\begin{tabular}{|c|c|}
\hline
 $Y(D)$     &  $\Gamma_{Y^{\rm op}(D)}$ \\
 \hline
     \makecell{  $\bbP^2(1,4)$ \\ } & \makecell{ \includegraphics[scale=.8]{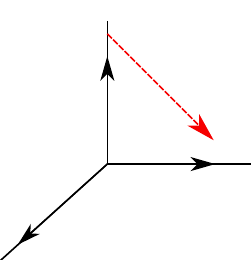} \\ } \\ \hline
     \makecell{ \vspace{-.1cm} \\ $\delp_1(1,3)$ \\ \vspace{-.1cm}} &  \multirow{2}{*}{\includegraphics[scale=.7]{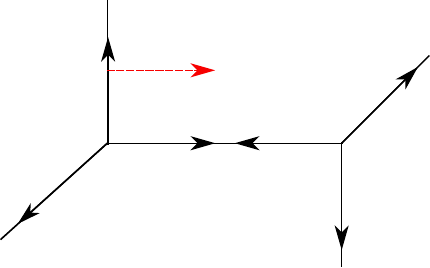}} \\   \cline{1-1}
     \makecell{ \vspace{-.1cm} \\ $\delp_1(0,4)$  \\ \vspace{-.1cm}} & \\ \hline
     \makecell{ \vspace{-.1cm} \\ $\delp_2(1,2)$ \\ \vspace{-.1cm}} & \makecell{\vspace{-.1cm} \\ \multirow{2}{*}{\includegraphics[scale=0.5]{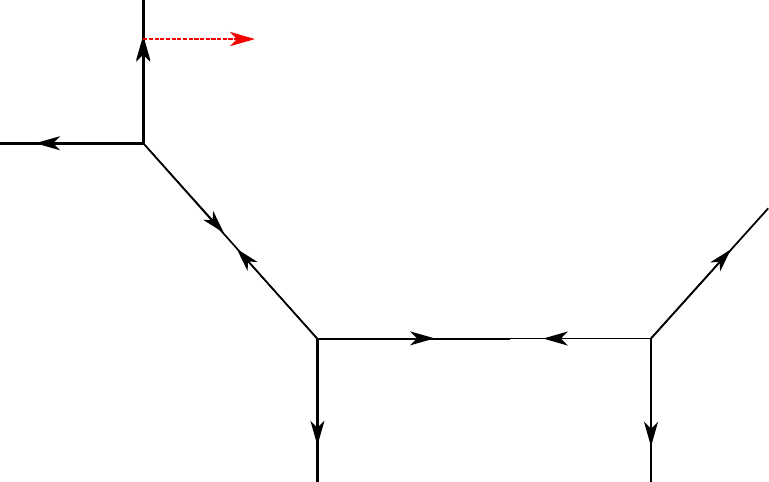}} \\ \vspace{.1cm}} \\   \cline{1-1}
     \makecell{ \vspace{-.1cm} \\ $\delp_2(0,3)$  \\ \vspace{-.1cm}} & \\ \hline
     \makecell{ \vspace{-.1cm} \\ $\delp_3(1,1)$ \\ \vspace{-.1cm}} & \makecell{ \vspace{-.1cm} \\ \multirow{2}{*}{\includegraphics[scale=0.5]{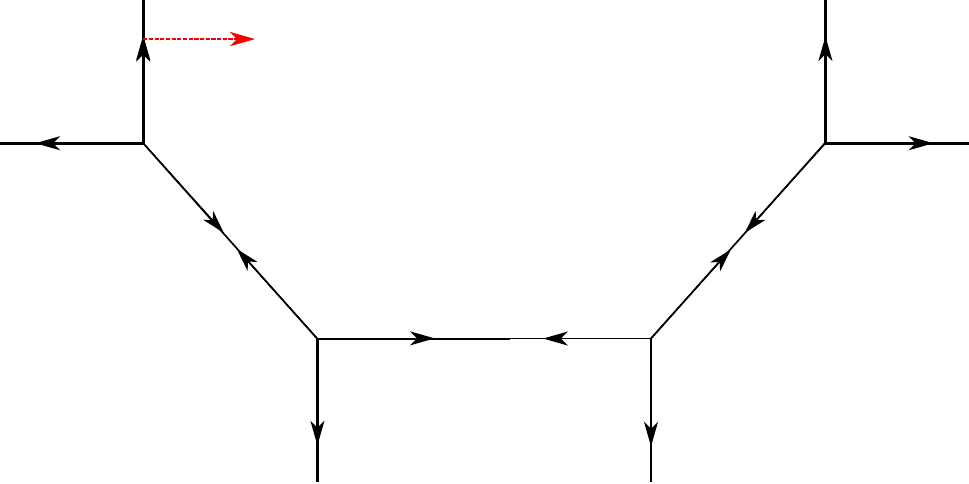}} \\ \vspace{.1cm}} \\   \cline{1-1}
     \makecell{ \vspace{-.1cm} \\ $\delp_3(0,2)$  \\ \vspace{-.1cm}} & \\ \hline
          \makecell{ \vspace{-.1cm} \\ $\bbF_0(2,2)$ \\ \vspace{-.1cm}} & \makecell{ \vspace{-.3cm} \\ \multirow{2}{*}{\includegraphics[scale=0.5]{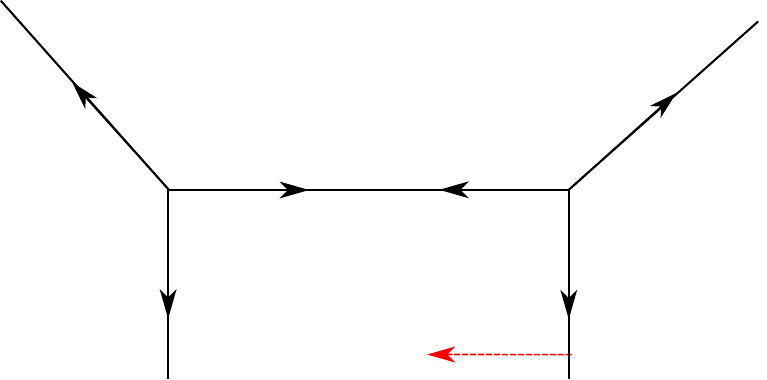}}} \\   \cline{1-1}
     \makecell{ \\ $\bbF_0(0,4)$  \\ \vspace{-.1cm}} & 
     \\   \hline
\end{tabular}
\end{minipage}
\begin{minipage}{.45\textwidth}
\centering
\begin{tabular}{|c|c|}
\hline
 $Y(D)$     &  $\Gamma_{Y^{\rm op}(D)}$ \\
 \hline
     \makecell{\vspace{.01cm} \\ $\bbP^2(1,1,1)$ \\ \vspace{.01cm}} & \makecell{\vspace{.01cm} \\ \includegraphics[scale=.8]{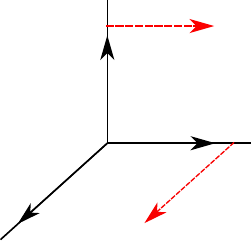} \\ \vspace{.01cm}} \\   \hline
     \makecell{ \vspace{.01cm} \\ $\delp_1(1,1,0)$ \\ \vspace{.01cm}} &  \makecell{\vspace{.01cm} \\ \includegraphics[scale=.8]{annulidp1.pdf} \\ \vspace{.01cm}} \\   \hline
     \makecell{ \vspace{.01cm} \\ $\delp_2(1,0,0)$ \\ \vspace{.3cm}} &  \makecell{\vspace{.01cm}\\ \includegraphics[scale=0.5]{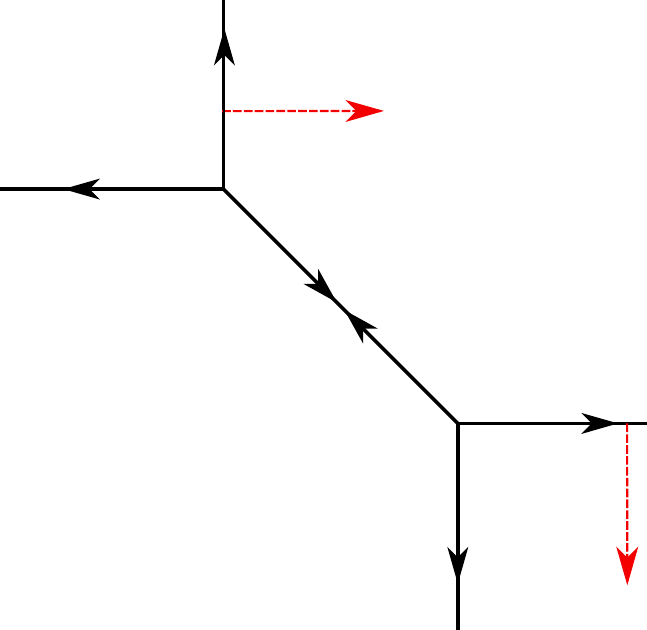}\\\ } \\  \hline
     \makecell{ \vspace{.01cm} \\ $\bbF_0(2,0,0)$ \\ \vspace{.3cm}} &  \makecell{\vspace{.01cm} \\ \includegraphics[scale=.8]{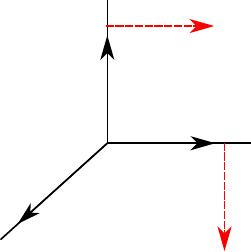} \\ \vspace{.01cm}} \\   \hline
          \makecell{ \vspace{.01cm} \\ $\bbF_0(0,0,0,0)$ \\ \vspace{.01cm}} & \makecell{\vspace{.01cm} \\ \includegraphics[scale=.8]{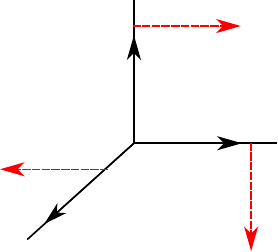} \\ \vspace{.01cm}} \\
     \hline
\end{tabular}
\end{minipage}
\caption{$Y^{\rm op}(D)$ for $l$-component Looijenga pairs satisfying Property~O.}
\label{tab:YopD}
\end{table}

\section{KP and quiver DT invariants}
\label{sec:quivers}

\subsection{Klemm--Pandharipande invariants of CY4-folds}

Let $Z$ be a smooth projective complex Calabi--Yau variety of dimension four and $d \in \hhh_2(Z, \bbZ)$. Since $\mathrm{vdim}\overline{\cM}_{g,n}(Z,d)=1-g+n$, the only non-vanishing genus zero primary Gromov--Witten invariants of $Z$ without divisor insertions are\footnote{By the same formula, there are non-vanishing elliptic unpointed Gromov--Witten invariants for $Z$, which will not concern us in this paper. There are no Gromov--Witten invariants for a CY4 in genus $g>1$.}
\bea
\mathrm{GW}_{0,d;\gamma}(Z) & \coloneqq & \int_{\overline{\cM}_{0,1}(Z,d)} \ev^\star_1{\gamma}, \quad \gamma \in \hhh^4(Z, \bbZ) \,.
\eea 
The same considerations apply to the case of $Z$ the Calabi--Yau total space of a rank-$(4-r)$ concave vector bundle on an $r$-dimensional smooth projective variety. It was proposed by Greene--Morrison--Plesser in \cite[App.~B]{Greene:1993vm} and further elaborated upon by Klemm--Pandharipande in \cite[Sec.~1.1]{Klemm:2007in} that a higher-dimensional version of the Aspinwall--Morrison should conjecturally produce integral invariants $\mathrm{KP}_{0,d}(Z)$, virtually enumerating rational degree-$d$ curves incident to the Poincar\'e dual cycle of $\gamma$:
\beq
\mathrm{GW}_{0,d; \gamma}(Z) = \sum_{k \mid d} \frac{\mathrm{KP}_{0,d/k; \gamma}(Z)}{k^2}\,.
\label{eq:kp}
\eeq
\begin{conj}[Klemm--Pandharipande]
$\mathrm{KP}_{0,d; \gamma}(Z) \in \bbZ$.
\label{conj:KP}
\end{conj}
A symplectic proof of \cref{conj:KP} for projective $Z$, although likely adaptable to the non-compact setting, was given by Ionel--Parker in \cite{MR3739228}. \\

Our main focus will be on $Z$  a non-compact CY4 local surface (i.e. $r=2$). In this case there is a single generator $\gamma=[\mathrm{pt}]$ for the fourth cohomology of $Z$, given by the Poincar\'e dual of the point class on the zero section, and we will henceforth use the simplified notation $\mathrm{KP}_{0,d}(Z) \coloneqq \mathrm{KP}_{0,d; [\mathrm{pt}]}(Z)$

\subsection{Quiver Donaldson--Thomas theory}

Let $\mathsf{Q}$ be a quiver with an ordered set $\mathsf{Q}_0$ of $n$ vertices $v_1, \dots v_n \in \mathsf{Q}_0$ and a set of oriented edges $\mathsf{Q}_1 = \{\a:v_i \to v_j  \}$. We let $\bbN \mathsf{Q}_0$ be the free abelian semi-group generated by $\mathsf{Q}_0$, and for $\mathsf{d}=\sum d_i v_i$, $\mathsf{e}=\sum e_i v_i \in \bbN \mathsf{Q}_0$ we write $E_\mathsf{Q}(\mathsf{d},\mathsf{e})$ for the Euler form
\beq
E_\mathsf{Q}(\mathsf{d},\mathsf{e}) \coloneqq \sum_{i=1}^n d_i e_i -\sum_{\a:v_i \to v_j} d_i e_j\,.
\eeq
We assume in what follows that $\mathsf{Q}$ is symmetric, that is, for every $i$ and $j$, the number of oriented edges from $v_i$ to $v_j$ is equal to the number of oriented edges from 
$v_j$ to $v_i$. The Euler form is then a symmetric bilinear form.
To $C$ a symmetric bilinear pairing on $\bbZ^n$, we associate the generalised $q$-hypergeometric series
\beq
\Phi_C(q; x_1, \dots, x_n) \sum_{\mathsf{d} \in \bbN^n}^\infty \frac{\big(-q^{1/2}\big)^{C(\mathsf{d},\mathsf{d})} \mathsf{x}^\mathsf{d}}{\prod_{i=1}^n (q;q)_{d_i}},
\eeq
where $\mathsf{x}^\mathsf{d} =\prod_{i=1}^n x_i^{d_i}$. The motivic Donaldson--Thomas partition function associated to the cohomological Hall algebra of $\mathsf{Q}$ (without potential) is the generating function \cite{MR2956038}
\beq
P_\mathsf{Q}(q; x_1, \dots, x_{n}) \coloneqq \Phi_{E_\mathsf{Q}}(q; x_1, \dots, x_n) 
\eeq
 and the motivic DT invariants $\mathrm{DT}_{\mathsf{d}; i}(\mathsf{Q})$ of $\mathsf{Q}$ are the formal Taylor coefficients in the expansion of its plethystic logarithm \cites{MR2956038, MR2851153, MR2889742}:
\bea
P_\mathsf{Q}(q; x_1, \dots, x_n) &=& \mathrm{Exp}\l(\frac{1}{[1]_q} \sum_{\mathsf{d} \neq 0} \sum_{i \in \bbZ} \mathrm{DT}_{\mathsf{d}; i}(\mathsf{Q}) \mathsf{x}^\mathsf{d} (-q^{1/2})^{-i} \r) \nn \\
&=&
\exp \l[\sum_{n=1}^\infty \frac{1}{n [n]_{q}} \sum_{\mathsf{d} \neq 0} \sum_{i \in \bbZ} \mathrm{DT}^\mathsf{Q}_{\mathsf{d}; i} \mathsf{x}^{n\mathsf{d}} (-q^{1/2})^{-ni}\r] \nn \\
&=& \prod_{\mathsf{d} \neq 0} \prod_{i \in \bbZ} \prod_{k \geq 0} \l(1-(-1)^i \mathsf{x}^\mathsf{d} q^{-k-(i+1)/2} \r)^{-\mathrm{DT}_{\mathsf{d}; i}(\mathsf{Q})}.
\label{eq:DTmot}
\eea
It will be of particular interest for us to consider a suitable semi-classical limit of \eqref{eq:DTmot}
\bea
y^{(i)}_\mathsf{Q}(x_1, \dots, x_n) & \coloneqq & \lim_{q \to 1} \frac{P_\mathsf{Q}(q; x_1, \dots, q^{1/2}x_i, \dots, x_n)}{P_\mathsf{Q}(q; x_1, \dots, q^{-1/2} x_i, \dots x_n)} \nn \\
&=& \lim_{q \to 1}  \mathrm{Exp}\l( \sum_{\mathsf{d} \neq 0} \frac{1}{[1]_q}\sum_{i \in \bbZ} [d_i]_{q} \mathrm{DT}^\mathsf{Q}_{\mathsf{d}; i} \mathsf{x}^\mathsf{d} (-q^{1/2})^{-i} \r) \nn \\
&=& \prod_{\mathsf{d} \neq 0} \prod_{i \in \bbZ}\l( 1- \mathsf{x}^\mathsf{d}\r)^{-|\mathsf{d}| \mathrm{DT}^{\rm num}_{\mathsf{d}}(\mathsf{Q})},
\label{eq:yQ}
\eea
where
\beq
\mathrm{DT}^{\rm num}_{\mathsf{d}}(\mathsf{Q}) \coloneqq \sum_{i \in \bbZ} (-1)^i \mathrm{DT}_{\mathsf{d},i}(\mathsf{Q})
\label{eq:DTnum}
\eeq
are the numerical DT invariants. From \eqref{eq:yQ}, the numerical invariants can be extracted from the logarithmic primitive of $y_Q^{(i)}(\mathsf{x})$ w.r.t. $x_i$,
\beq
\int \frac{\rd x_i}{x_i} \log y^{(i)}_Q(\mathsf{x}) =: \sum_{\mathsf{d} \neq 0} A_{\mathsf{d}}(\mathsf{Q}) \mathsf{x}^{\mathsf{d}} \,,
\eeq
as
\beq
A_{\mathsf{d}}(\mathsf{Q}) = \sum_{k \mid d} \frac{\mathrm{DT}^{\rm num}_{\mathsf{d/k}}(\mathsf{Q})}{k^2}\,.
\eeq
The generating series $y_\mathsf{Q}(x_1, \dots, x_n) \coloneqq \prod_{i=1}^n y^{(i)}_\mathsf{Q}(x_1, \dots, x_n)$ has an interpretation as a generating function of Euler characteristics of certain non-commutative Hilbert schemes $\mathrm{Hilb}_\mathsf{d}(\mathsf{Q})$ attached to the moduli space of semi-stable representations of the quiver $\mathsf{Q}$ \cites{MR2511752, MR2889742},
\beq
y_\mathsf{Q}(x_1, \dots, x_n) = \sum_{\mathsf{d} \in \bbZ \mathsf{Q}_0} \chi(\mathrm{Hilb}_\mathsf{d}(\mathsf{Q})) \mathsf{x}^\mathsf{d} \in \bbZ[[\mathsf{x}]]\,.
\eeq
In particular, this implies that $(\sum_{i=1}^n d_i) \mathrm{DT}^{\rm num}_{\mathsf{d}}(\mathsf{Q})\in \bbZ$. More is true \cites{MR2956038, MR2889742} by the following
\begin{thm}[Efimov, \cite{MR2956038}]
The numerical Donaldson--Thomas invariants of a symmetric quiver $\mathsf{Q}$ without potential are positive integers, $\mathrm{DT}^{\rm num}_{\mathsf{d}}(\mathsf{Q}) \in \bbN$.
\label{thm:efimov}
\end{thm}

\subsection{KP integrality from DT theory}

The genus zero log-local and log-open correspondences of \cref{thm:logopen} imply that KP invariants of toric local surfaces are, up to a sign and possibly an integral shift, numerical DT invariants of a symmetric quiver. Combined with \cref{thm:efimov} this gives an algebro-geometric proof of \cref{conj:KP} for $Z= \mathrm{Tot}(\cO(-D_1) \oplus \cO(-D_2) \to Y)$.

\begin{thm}
Let $Y(D)$ be a 2-component quasi-tame Looijenga pair. Then there exists a symmetric quiver $\mathsf{Q}(Y(D))$ with $\chi(Y)-1$ vertices and a lattice isomorphism $\kappa: \bbZ(\mathsf{Q}(Y(D)))_0 \stackrel{\sim}{\rightarrow}  \hhh_2(Y,\bbZ) $ such that
\beq
\mathrm{DT}_{d}^{\rm num}(\mathsf{Q}(Y(D))) =
\bigg|\mathrm{KP}_{\kappa(d)}(E_{Y(D)})+ \sum_i \a_i \delta_{d,v_i}\bigg| \,,
\label{eq:kpdt}
\eeq
with $\a_i \in \{-1,0,1\}$. In particular, $\mathrm{KP}_{d}(E_{Y(D)})\in \bbZ$.
\label{thm:kpdt}
\end{thm}

\begin{proof}

The statement is a direct consequence of \cref{thm:logopen} combined with the strips-quivers correspondence of \cite{Panfil:2018faz}, which we briefly review here in our context. 
Since $Y(D)$ is a $2$-component quasi-tame pair, it satisfies Property~O by the discussion of \cref{sec:logopen}. From \cref{lem:skewhooks} and the proof of \cref{thm:logopen} (see in particular \eqref{eq:Wdp3hooks}), we have
%
\beq
\frac{\cW_{(j_0)}(Y^{\rm op}(D))(Q,\hbar)}{\cW_{\emptyset}(Y^{\rm op}(D))(Q,\hbar)} = \frac{(-1)^{f j_0} q^{(f+1/2)\binom{j_0}{2}}}{[j_0]_q!}
\frac{\prod_{i=1}^r (\tilde Q^{(1)}_i;q)_{j_0}}{\prod_{k=1}^s (\tilde Q^{(2)}_{k};q)_{j_0}}
\label{eq:Wsymm}
\eeq
where $f$ is the integral shift of $\mathsf{f}$ from canonical framing,  $(r,s)$ are non-negative integers with $r+s+1=\chi(Y)-1$, and $\tilde Q_i=\prod_{m=1}^{r+s} Q_m^{a_{m,i}}$ with $a_{m,i} \in \{-1,0,1\}$, $i=1, \dots, r+s$. Elementary manipulations and use of the $q$-binomial theorem (see \cite[\S~4.1]{Panfil:2018faz}) show that
\bea
\psi_{Y(D)}(Q, \hbar, z) & \coloneqq &
\sum_{j_0 \geq 0} \frac{\cW_{(j_0)}(Y^{\rm op}(D))(Q,\hbar)}{\cW_{\emptyset}(Y^{\rm op}(D))(Q,\hbar)}  z^{j_0}  \nn \\
&=& \frac{\prod_{i=1}^r (\tilde Q_i;q)_{\infty}}{\prod_{k=1}^s (\tilde Q_{r+k};q)_{\infty}} \Phi_{C(Y(D))}(q^{(r-s-1)/2} z, \tilde Q^{(1)}_1, \dots, \tilde Q^{(1)}_r, q^{1/2} \tilde Q^{(2)}_{1}, \dots, q^{1/2} \tilde Q^{(2)}_{s}), \nn \\
\label{eq:psiYD}
\eea
where
\beq
C(Y(D)) = 
\l( 
\bary{c|ccc|cccc}
& \multicolumn{3}{c}{\overbrace{\rule{2cm}{0pt}}^{r}} & \multicolumn{3}{c}{\overbrace{\rule{2cm}{0pt}}^{s}} & \\
f+1 & 1 & \dots & 1 & 1 & \dots & 1 \\
\cline{1-7}
1 & 0 & \dots & 0 & 0 & \dots & 0 & \rdelim\}{3}{1em}[$r$] \\
\vdots & & \dots & & \vdots & \\
1 & 0 & \dots & 0 & 0 & \dots & 0 \\
\cline{1-7}
1 & 0 & \dots & 0 & 1 & \dots & 0 & \rdelim\}{3}{1em}[$s$]\\
\vdots & & \dots & & \vdots & \\
1 & 0 & \dots & 0 & 0 & \dots & 1 \\
\eary
\r)
\label{eq:adjmat}
\eeq
and moreover, the genus-zero limit of the logarithm of \eqref{eq:psiYD} is the generating function of disk invariants of $Y^{\rm op}(D)$ \cite{Aganagic:2003qj},
\bea
\lim_{\hbar \to 0} \hbar \log \psi_{Y(D)}(Q,\hbar, z) &=&  \lim_{\hbar \to 0} \hbar \cO(Y^{\rm op}(D))(Q,\hbar, \mathsf{x})\big|_{\mathsf{x}_{\vec\mu}=z^{j_0} \delta_{\vec \mu, (j_0)}} \nn \\
&=& \sum_{\beta} O_{0;j_1, \dots, j_{r+s}; (j_0)}(Y^{\rm op}(D)) z^{j_0} \prod_{i=1}^{r+s} Q_i^{j_i}\,.
\label{eq:psig0}
\eea
The matrix $C$ has non-negative off-diagonal entries, and $\Phi_C(q; x_1, \dots, x_{r+s+1})$ cannot therefore be immediately interpreted as a motivic quiver DT partition function. However, writing $\mathsf{Q}(Y(D))$ for the symmetric quiver with adjacency matrix $C(Y(D))$, we have \cite[App.~A]{Panfil:2018faz},
\bea
\Phi_{C(Y(D))}(q; x_1, \dots, x_{r+s+1})
&=& 
\prod_{\mathsf{d} \neq 0} \prod_{j \in \bbZ} \prod_{k \geq 0} \l(1-(-1)^j \mathsf{x}^\mathsf{d} q^{-k-(j+1)/2} \r)^{-\cE^{C(Y(D))}_{\mathsf{d}; j}} \nn \\
&=&
\Phi_{E_{\mathsf{Q}(Y(D))}}(q^{-1}; q^{-1/2} x_1, \dots, q^{-1/2} x_{r+s+1})\,.
\eea
The exponents $\cE^{C(Y(D))}_{\mathsf{d}; j}$ are then equal to the motivic DT invariants of $\mathsf{Q}(Y(D))$ up to sign. Furthermore, the numerical DT invariants also agree with the absolute value of $\cE^{C(Y(D)), \rm num}_{\mathsf{d}} \coloneqq  \sum_j (-1)^j \cE^{C(Y(D))}_{\mathsf{d}; j}$ \cite[App.~A]{Panfil:2018faz},
\beq
\mathrm{DT}^{\rm num}_{\mathsf{d}}(\mathsf{Q}(Y(D)))=\big|\cE^{C(Y(D)), \rm num}_{\mathsf{d}}\big|\,.
\eeq
For $j=(j_0, j_1, \dots, j_{r+s})$, define now the disk BPS invariants of $Y^{\rm op}(D)$ by
\beq 
O_{0, j_1, \dots, j_{r+s}; (j_0)}(Y^{\rm op}(D))
\coloneqq \sum_{k \mid \gcd(j_0, \dots, j_{r+s})} \frac{1}{k^2} \mathfrak{D}_{j/k}(Y^{\rm op}(D))\,.
\label{eq:diskbps}
\eeq
From \eqref{eq:Wsymm} and \eqref{eq:psig0}, we have that $\mathfrak{D}_{\tau(d)}(Y^{\rm op}(D))+\sum_i \a_i \delta_{d,v_i}=\cE^{C(Y(D)), \rm num}_{\mathsf{d}}$, where
\beq
\a_i= \l\{ 
\bary{cc}
0 & i=1, \\
-1 & i=2, \dots, r+1 \\
+1 & i=r+2, \dots, r+s+1
\eary
\r.
\label{eq:kpdtshift}
\eeq
and 
\beq 
\tau(d_1, \dots, d_{r+s+1})= \l(d_1, \sum_{m=1}^{r+s} a_{m,2} d_{m+1}, \dots,  \sum_{m=1}^{r+s} a_{m,r+s} d_{m+1}\r)\,.
\eeq 
But by \eqref{eq:logopen0}, $O_{j}(Y^{\rm op(D)})=N^{\rm loc}_{\iota(j)}(Y(D))$, and therefore $\mathfrak{D}_{j}(Y^{\rm op}(D))=\mathrm{KP}_{\iota(j)}(E_{Y(D)})$, from which the claim follows by setting $\kappa \coloneqq \iota \circ \tau$.

\end{proof}

\begin{table}[t]
    \centering
    \begin{tabular}{|c|c|}
    \hline
        $Y(D)$ & $\mathsf{Q}(Y(D))$  \\
        \hline\hline
        \makecell{\vspace{-.4cm} \\ $\bbP^2(1,4)$ \\}
 & \makecell{\vspace{-.2cm} \\ \includegraphics[]{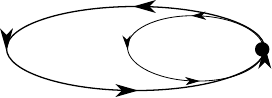}\\} \\ \hline
         $\mathbb{F}_0(2,2)$ &  \multirow{2}{*}{\includegraphics[]{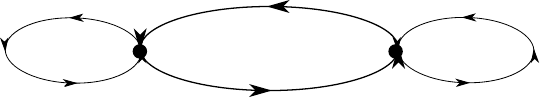}} \\ \cline{1-1}
         $\mathbb{F}_0(0,4)$ & \\ 
        \hline
        $\mathrm{dP}_1(1,3)$ & \multirow{2}{*}{\includegraphics[]{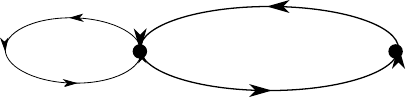}} \\ \cline{1-1}
        $\mathrm{dP}_1(0,4)$ & \\ \hline
        $\mathrm{dP}_2(1,2)$ & \multirow{2}{*}{\includegraphics[]{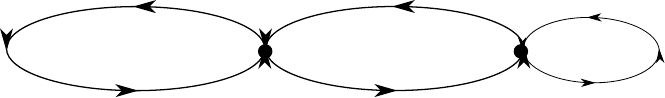}} \\ \cline{1-1}
        $\mathrm{dP}_2(0,3)$ & \\ \hline
        $\mathrm{dP}_3(1,1)$ & \multirow{2}{*}{\includegraphics[]{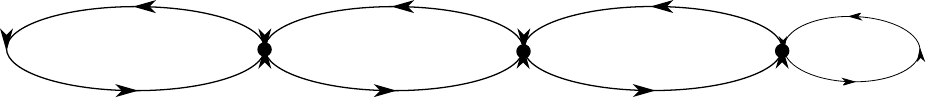}} \\ \cline{1-1}
        $\mathrm{dP}_3(0,2)$ & \\ \hline
    \end{tabular}
    \caption{Quivers for 2-component quasi-tame Looijenga pairs.}
    \label{tab:quivers}
\end{table}

\begin{rmk}
\cref{thm:kpdt}, combined with \cref{thm_log_local}, resembles previous correspondences identifying log~GW invariants to DT invariants of quivers, and in particular \cite{Bou18}, but it differs from them in a number of key respects: the quiver DT invariants here are identified with the (absolute value of the) BPS invariants of the local geometry, and therefore imply a finer integrality property of the log invariants via \eqref{eq:loglocal} and \eqref{eq:kp}. Furthermore, unlike in \cite{Bou18}, the motivic refinement is not expected to reconstruct the open Gromov--Witten count at higher genus, as the higher orders in $\hbar$ of \eqref{eq:psig0} include contributions of open stable maps with more than one boundary component. A separate discussion of the open BPS structure of the higher genus theory is the subject of the next Section.
\end{rmk}

\begin{example} Let $Y(D)=\bbP^2(1,4)$.  
In this case we have $r=s=0$, $f=1$,  and $\mathsf{Q}(\bbP^2(1,4))$ is the 2-loop quiver. Moreover, the identification of dimension vectors with curve degrees is simply the identity, $\kappa=\mathrm{id}$, and the integral shift in \eqref{eq:kpdt} and \eqref{eq:kpdtshift} vanishes, $\a_1=0$. Then, by \cref{thm:kpdt}, the absolute value of the KP invariants of $E_{\bbP^2(1,4)}$ gives the unrefined DT invariants of $\mathsf{Q}(\bbP^2(1,4))$. We can in fact check directly that $\mathrm{KP}_{d}(E_{\bbP^2(1,4)})=(-1)^d\mathrm{DT}^{\rm num}_d(\mathsf{Q}(\bbP^2(1,4)))$: according to \cite[Thm.~3.2]{MR2889742}, we have 
\beq \mathrm{DT}^{\rm num}_d(\mathsf{Q}(\bbP^2(1,4))) =\frac{(-1)^d}{d^2} \sum_{k|d} \mu\l(\frac{d}{k}\r)(-1)^k \binom{2k-1}{k-1}\,\eeq and the result follows from \eqref{eq:NlocP2} and the equality 
\beq \frac{1}{2} \binom{2k}{k}=\frac{1}{2} \frac{(2k)!}{(k!)^2}
= \frac{1}{2}\frac{2k}{k} \frac{(2k-1)!}{k!(k-1)!} =\binom{2k-1}{k-1}\,.\eeq 
\end{example}

\begin{rmk}[Non-quasi-tame pairs]
The condition in \cref{thm:kpdt} that $Y(D)$ is a 2-component quasi-tame pair is likely to be necessary. For example, for $Y(D)$ a non-quasi-tame pair, we do not expect that the result of the finite summation \eqref{eq:dP5} can be further simplified down to a form akin to \eqref{eq:localgw} as a ratio of products of factorials, unlike the case of the hypergeometric summations in the proof of \cref{prop:dp311}. A little experimentation shows that, writing $N_{0,d}^{\rm loc}(\delp_5(0,0))=\frac{m(d)}{n(d)}$ with $\gcd(m(d),n(d))=1$, the numerator $m(d)$ is  divisible by very large primes $\approx 10^7$ for low degrees $d_i \approx 10^1$ with $d_i \neq d_0$, $i>0$. This creates a tension with $m(d)$ being a product of factorials with arguments linear in $d_i$ with coefficients $\approx 10^1$, as those would be divisible by at most the largest prime in the range $\approx 10^1-10^2$. As generating functions of numerical DT invariants are always generalised hypergeometric functions \cite{Panfil:2018sis}, and their coefficients are therefore always products of ratios of factorials in the degrees, the KP/DT correspondence of \cref{thm:kpdt} is unlikely to extend to the non-quasi-tame setting.
\end{rmk}

\begin{rmk}[$l>2$] For $l$-components pairs with $l>2$, a correspondence between quivers and $(l-1)$-holed open GW partition functions has received some preliminary investigation in the context of the links-quivers correspondence \cites{Kucharski:2017ogk,Ekholm:2018eee} where open stable maps are considered with the same colouring by symmetric Young diagrams for all the connected components of the boundary. The general case of stable maps with arbitrary windings which is relevant for our purposes may however fall outside the remit of the open BPS/quiver DT correspondence. In particular, suppose that $\mathsf{Q}$ is a symmetric quiver such that $P_\mathsf{Q}(\a_1 x,\dots, \a_r x, \b_1 y, \dots, \b_2 y)=\sum_{m,n} x^m y^n \cW_{(m),(n)}(X,L_1\cup L_2, \mathsf{f}_1,\mathsf{f}_2)$ with $\a_i$, $\b_i\in \bbC[q]$ and framed toric special Lagrangians $L_1$, $L_2$ in a Calabi--Yau threefold $X$. The simplest instance is $X=\bbC^3$ and $L_1$, $L_2$ are framed toric Lagrangians on different legs: this arises e.g. by considering $\delp_1^{\rm op}(1,1,0)$ and $\bbF_0^{\rm op}(2,0,0)$. It is easy to check that the analogue of the semi-classical limit \eqref{eq:psig0} for the annulus generating function would be
\beq
\lim_{q\to 1} D^q_x D^q_y \log  P_\mathsf{Q}(\a_1 x,\dots, \a_r x, \b_1 y, \dots, \b_s y) = \sum_{j_1,j_2,\beta} x^{j_1} y^{j_2} O_{0; j_1,j_2}(X,L_1\cup L_2,\mathsf{f}_1,\mathsf{f}_2),
\label{eq:limannuli}
\eeq
where $D^q_x$ denotes the $q$-derivative w.r.t $x$. When $X=\bbC^3$, a natural guess in line with the disk case would be to take $\mathsf{Q}$ a quiver with two vertices, with dimension vectors in bijection with winding numbers along the homology circles in $L_1$ and $L_2$. However it is straightforward to verify from \eqref{eq:DTmot} that for $r+s=2$ the l.h.s. of \eqref{eq:limannuli} does not have a limit as $q\to 1$ unless $\mathsf{Q}$ is disconnected, in which case the limit is identically zero, and hence disagrees with the r.h.s.. Although this may not necessarily extend to quivers with higher number of vertices and finely tuned identifications of dimension vectors with winding degrees, it does suggest that a suitable generalisation of the correspondence might be required to encompass the counts of annuli as well.
\end{rmk}

\section{Higher genus BPS invariants}
\label{sec:higein}
For $Y(D)$ a (not necessarily tame) $l$-component Looijenga pair satisfying Property~O, we define
\bea
\Omega_{d}(Y(D))(q) 
& \coloneqq & 
[1]_q^{2} \l(\prod_{i=1}^{l-1} \frac{d\cdot D_i}{[d\cdot D_i]_q}\r)  \sum_{k | d}  \frac{\mu(k)}{k}
\mathsf{O}_{\iota^{-1}(d/k)}(Y^{\rm op}(D))(-\ri k \log q)\, ,
\label{eq:Omegad}
\eea 
and for $Y(D)$ an $l$-component pair, not necessarily satisfying Property O, we write
\beq
\Omega_{d}(Y(D))(q) 
 \coloneqq 
 [1]^2_q
 \left( \prod_{i=1}^{l} \frac{ 1 }{[ d \cdot D_i]_{q}} \right)
 \sum_{k | d}
 (-1)^{ d/k \cdot D + l} \,
[k]_{q}^{l-2} \,  k^{l-2} \, \mu(k) \,
\mathsf{N}_{d/k}^{\rm log}(Y(D))(-\ri k \log q)\, .
\label{eq:Omegad2}
\eeq
The compatibility of \eqref{eq:Omegad} and \eqref{eq:Omegad2} when $Y(D)$ satisfies both tameness and Property O follows from \cref{thm:logopen}. From \cref{tab:classif} and the discussion following \cref{def:propO}, any quasi-tame $l$-component Looijenga pair either satisfies Property~O, or it is tame, or both: in this setting we will take $\Omega_d(Y(D))(q)$ to be either of the applicable definitions \eqref{eq:Omegad} or \eqref{eq:Omegad2}. We further write simply $\Omega_{d}(Y(D))$ for the genus-zero limit $\Omega_{d}(Y(D))(1)$,
%
\bea
\Omega_{d}(Y(D))
& \coloneqq & \frac{1}{\prod_{i=1}^{l} (d \cdot D_i)}
 \sum_{k | d}(-1)^{\sum_{i=1}^{l}  d/k \cdot D_i+1} \frac{\mu(k)}{k^{4-2l}}
N_{0,d/k}^{\rm log}(Y(D)) \nn \\
& = & 
\label{eq:Omega0d}
\sum_{k | d} \frac{\mu(k) }{k^{4-l}}
O_{0,\iota^{-1}(d/k)}(Y^{\rm op}(D))  \\
& = & 
\sum_{k | d} \frac{ \mu(k)}{k^{4-l}}
N_{d/k}^{\rm loc}(Y(D))\,. \nn
\eea 
A priori we can only expect $\Omega_d(Y(D)) \in \bbQ$ and $\Omega_d(Y(D))(q) \in \bbQ(q^{1/2})$. By \eqref{eq:Omegad2} and \eqref{eq:Omega0d}, however, $\Omega_d(Y(D))$ and $\Omega_d(Y(D))(q)$ are amenable to a physical interpretation as Labastida--Mari\~no--Ooguri--Vafa (LMOV) partition functions \cites{Labastida:2000yw, Labastida:2000zp, Ooguri:1999bv,Marino:2001re}. These heuristically count BPS domain walls in an M-theory compactification on $Y^{\rm op}(D)$ (see in particular \cite[Eq.~2.10]{Marino:2001re}):
writing $\Omega_d(Y(D))(q)=\sum_j n_{d,j}(Y(D)) q^j$, the LMOV invariants $n_{d,j}(Y(D))$ would compute the net degeneracy of M2-branes with spin $j$ and magnetic and bulk charge specified by $d$, ending on an M5-brane wrapped around the framed toric Lagrangian $L$ in $Y^{\rm op}(D)=(X,L,\mathsf{f})$. From the vantage point of string theory, \eqref{eq:Omegad2} (resp. \eqref{eq:Omega0d}) are then expected to be integral Laurent polynomials (resp. integers: for $l=2$, since $\Omega_d(Y(D))=\mathrm{KP}_d(E_{Y(D)}) = \mathfrak{D}(Y^{\rm op}(D))$ by \eqref{eq:kp}, \eqref{eq:diskbps}, \eqref{eq:Omega0d} and \eqref{eq:logopen0}, this is implied by \cref{thm:kpdt}). The next Theorem shows that this is indeed the case.
\begin{thm}[The higher genus open BPS property]
Let $Y(D)$ be a quasi-tame Looijenga pair. Then
$\Omega_{d}(Y(D))(q) \in  q^{-\frac{\mathsf{g}_{Y(D)}(d)}{2}} \bbZ[q] $
for an integral quadratic polynomial $\mathsf{g}_{Y(D)}(d)$.
\label{thm:openbps}
\end{thm}
Clearly, from \eqref{eq:GWexp} and \eqref{eq:Omegad}-\eqref{eq:Omegad2}, we have
$\Omega_d(q)=\Omega_d(q^{-1})$, so \cref{thm:openbps} implies in particular
that $\Omega_d(q)$ is a Laurent polynomial truncating at $\cO(q^{\pm\mathsf{g}_{Y(D)}(d)/2})$.\\

To prove \cref{thm:openbps} we shall need the following two Lemmas. Let $\omega_d$ be a primitive $d^{\rm th}$ root of unity.

\begin{lem}[The $q$-Lucas theorem \cite{MR178267}] 
Let $n \geq m $ be non-negative integers. Then
\beq
\qbinom{n}{m}_{\omega_d} = \omega_d^{\frac{m(m-n)}{2}} \binom{\lfloor n/d \rfloor}{\lfloor m/d \rfloor} \qbinom{n-d\lfloor n/d \rfloor}{m-d \lfloor m/d \rfloor}_{\omega_d}. 
\eeq
In particular, if $d \mid m$ and $d \mid n$,  $\qbinom{n}{m}_{\omega_d} = \omega_d^{\frac{m(m-n)}{2}} \binom{ n/d }{m/d}$.
\label{lem:qlucas}
\end{lem}

\begin{proof}
See e.g. \cite[Thm 2.2]{MR1176155} for a proof.
\end{proof}

\begin{lem}
Let $d \mid m \mid n\in \bbZ^+$. Then
$\de_q \qbinom{n}{m}_{q}\big|_{q=\omega_d} = 0$.
\label{lem:binom0}
\end{lem}

\begin{proof}
For every $i<n$ with $d \nmid i$ we have $\qbinom{n}{i}_{\omega_d}=0$,
since then
\beq
\qbinom{n-d\lfloor n/d \rfloor}{i-d \lfloor i/d \rfloor}_{\omega_d}=\qbinom{0}{i~\mathrm{mod}~d}_{\omega_d} = 0
\label{eq:qbinvan}
\eeq
The Cauchy binomial theorem,
\beq
\sum_{m=0}^{n} t^m q^{m(n+1)/2} \qbinom{n}{m}_q = \prod_{i=1}^n (1+ t q^i)\,,
\label{eq:cauchybin}
\eeq
implies that
\beq
q^{m(n+1)/2} \qbinom{n}{m}_q =  e_m(q, \dots, q^n)\,,
\label{eq:binelpol}
\eeq
where $e_j(x_1, \dots, x_n)$ is the $j^{\rm th}$ elementary symmetric polynomials in $n$~variables. 
We differentiate \eqref{eq:binelpol} and evaluate at $q=\omega_d$, where now $d \mid m \mid n$. Write $n=a b d$, $m=b d$ for $a, b \in \bbZ^+$. From \eqref{eq:cauchybin} we find
\bea
\de_q \prod_{i=1}^n (1+ t q^i) &=& \prod_{i=1}^n (1+ t q^i) \l( \sum_{j=1}^{n} \frac{j t q^{j-1}}{1+t q^j} \r) \nn \\
&=& \sum_{i=0}^n t^i e_i(q, \dots, q^n) \cdot t \sum_{j=1}^n \sum_{k=0}^\infty j (-t)^{k} q^{k j+j-1}\,.
\eea
Let us now evaluate at $q=\omega_d$ and take the $\mathcal{O}(t^m)$ coefficient on both sides. We have
\bea
\de_q e_m(q, \dots, q^n)\big|_{q=\omega_d} &=& [t^m]  \sum_{i=0}^n t^i e_i(\omega_d, \dots, \omega_d^n) \cdot t \sum_{j=1}^n \sum_{k=0}^\infty j (-t)^{k} \omega_d^{k j+j-1} \nn \\
 &=& [t^{b d}]  \sum_{i=0}^{a b} t^{d i} \omega_{d}^{i d (n+1)/2} \omega_d^{i d (i d-n)/2} \binom{a b}{i}  \cdot t \sum_{j=1}^n \sum_{k=0}^\infty j (-t)^{k} \omega_d^{k j+j-1} \nn \\
&=& \sum_{i=0}^{b -1} \omega_{d}^{i d(d+1)/2} \binom{a b}{i}   \sum_{j=1}^{a b d} j (-1)^{bd-1-id} \omega_d^{b d j-1-i d j} \nn \\
&=&
 (-1)^{m+1} \frac{n(n+1)}{2 \omega_d} \sum_{i=0}^{b-1} (-1)^{i} \binom{a b}{i}   \nn \\
&=& \frac{(-1)^{b+m} n(n+1)}{2 a\omega_d} \binom{a b}{b},
\label{eq:emrof1}
\eea
where we have used \eqref{eq:qbinvan} and \cref{lem:qlucas}. On the other hand,
\bea
\frac{\de}{\de q} q^{m(n+1)/2} \qbinom{n}{m}_q \bigg|_{q=\omega_d} &=& \frac{m(n+1)}{2 \omega_d} \omega_d^{m(m+1)/2} \binom{ab}{b} + \omega_d^{m(n+1)/2} \de_q  \qbinom{n}{m}_q \bigg|_{q=\omega_d} \nn \\
&=& \frac{m(n+1)}{2 \omega_d} (-1)^{b+m} \binom{ab}{b}+ \omega_d^{m(n+1)/2} \de_q  \qbinom{n}{m}_q \bigg|_{q=\omega_d} \,,
\label{eq:binomrof1} 
\eea
where in tracking down the last sign factor we have been mindful that $(-1)^{b m}=(-1)^{m}$ since $b \mid m$. The claim then follows by equating \eqref{eq:emrof1} to \eqref{eq:binomrof1}.
\end{proof}

\begin{proof}[Proof of \cref{thm:openbps}] We break up the proof of the Theorem by considering each value of $l$ separately.
\begin{description}[leftmargin=*]
\item[$l=2$] it suffices to prove the theorem in the case $Y(D)=\mathrm{dP}_3(1,1)$, since $\Omega_d(\mathrm{dP}_3(1,1))=\Omega_d(\mathrm{dP}_3(0,2))$ from \eqref{eq:Omegad} and the discussion of \cref{sec:disks}, and all other cases are then recovered from the blow-up formulas of \cref{prop:logblup,prop:openblup}. Let $\tilde d \coloneqq \mathrm{gcd}(d_0, d_1, d_2, d_3)$.
  We first plug \eqref{eq:dP311log} into \eqref{eq:Omegad},
\bea
 \Omega_{d}(\mathrm{dP}_3(1,1))(q) &=& [1]_q^2 \sum_{k \mid \tilde d} \mu(k) \frac{(-1)^{(d_1+d_2+d_3)/k}}{[d_0/k]_{q^k} [(d_1+d_2+d_3-d_0)/k]_{q^k}}
\Theta_{d/k}(q^k)
, \nn \\ &= &
  \frac{[1]_q^2}{[d_0]_{q} [d_1+d_2+d_3-d_0]_{q}} \sum_{k \mid \tilde d} \mu(k) (-1)^{(d_1+d_2+d_3)/k} \Theta_{d/k}(q^k),
\label{eq:omegad}
\eea
where
\beq
\Theta_{d}(q) \coloneqq
 \qbinom{d_0}{d_1}_{q}  \qbinom{d_1}{d_0-d_2}_{q}  \qbinom{d_1+d_2+d_3-d_0}{d_1}_{q}  \qbinom{d_1}{d_0-d_3}_{q}.
\eeq
It is immediate to verify that $ \Omega_{d}(\mathrm{dP}_3(1,1))(q)\in q^{-\frac{\mathsf{g}_{\mathrm{dP}_3(1,1)}(d)}{2}} \bbZ[[q]]$, with
\beq
\mathsf{g}_{\mathrm{dP}_3(1,1)}(d)=2 \left(d_1+d_2+d_3-d_0\right) d_0-d_1^2-d_2^2-d_3^2-d_1-d_2-d_3+2
\eeq
since $\qbinom{n}{m}_q \in q^{-m(n-m)/2} \bbZ[q]$, $1/[n]_q \in q^{n/2} \bbZ[[q]]$,
$[m]_q \in q^{-m/2} \bbZ[[q]]$  as formal Laurent series at $q=0$ with truncating principal part for any positive integers $n,m$.
Furthermore, away from $q=0,\infty$, 
$ \Omega_{d}(\mathrm{dP}_3(1,1))(q)\in \bbQ(q^{1/2})$ is a rational function of $q^{1/2}$ with at worst double poles possibly at the zeroes of $[d_0]_{q} [d_1+d_2+d_3-d_0]_{q}$, namely
$q=\omega_{d_0}^j$, $j=1,\dots, d_0-1$, and $q=\omega_{d_1+d_2+d_3-d_0}^j$, $j=1,\dots, d_1+d_2+d_3-d_0-1$. We shall now prove that $ \Omega_{d}(\mathrm{dP}_3(1,1))(q)$ is in fact regular on the unit circle.

First off, upon expanding all $q$-analogues in \eqref{eq:omegad} in cyclotomic polynomials, 
\beq [n]_q = \prod_{d \mid n}\Phi_d(q)\,, \eeq
it is straightforward to check that \cite[Lemma 5.2]{MR3128457}
\beq
\frac{[\mathrm{gcd}(n,m)]_q}{[n+m]_q} \qbinom{n+m}{m}_q \in  q^{\frac{n+m-n m - \mathrm{gcd}(n,m)}{2}} \bbZ[q]\,,
\eeq
which implies that $ \Omega_{d}(\mathrm{dP}_3(1,1))(q)$ is regular on the unit circle outside of $\{\omega_{\tilde d}^j \}_{j=0}^{\tilde d}$, where we recall that $\tilde d \coloneqq \mathrm{gcd}(d_0, d_1, d_2, d_3)$.
Let now $ \widetilde{\Omega}_{d}(\mathrm{dP}_3(1,1))(q) \coloneqq \frac{[d_0]_q [d_1+d_2+d_3-d_0]_q}{[1]_q^2} \Omega_d$ and $\tilde d_i = d_i/\tilde d$. From \cref{lem:qlucas}, we have
\bea
\Theta_{d/k}\big(\omega_{\tilde d}^{k j}\big) &=&
(-1)^{ (d_1+d_2+d_3) j/k}
\binom{\tilde d_1 \epsilon_{k,j} }{\tilde d_2 \epsilon_{k,j}} \binom{\tilde d_1 \epsilon_{k,j}}{ (\tilde d_0- \tilde d_2) \epsilon_{k,j}}
\nn \\ & & \binom{(\tilde d_1+ \tilde d_2+\tilde d_3-\tilde d_0) \epsilon_{k,j}}{\tilde d_1 \epsilon_{k,j}}
\binom{\tilde d_1 \epsilon_{k,j}}{(\tilde d_0-\tilde d_3) \epsilon_{k,j}} \,.
\eea 
where $\epsilon_{k,j}=\mathrm{gcd}(\tilde d/k,j)$. 
Then
\bea
\widetilde{\Omega}_{\tilde d}(\mathrm{dP}_3(1,1))(\omega_{\tilde d}^j) &=& \sum_{k
    \mid \tilde d} \mu\Big(\frac{\tilde d}{k}\Big) (-1)^{ (\tilde d_1+\tilde d_2+\tilde d_3) k (j+1)}
\binom{\tilde d_1 \epsilon_{\tilde d/k,j}}{\tilde d_2 \epsilon_{\tilde d/k,j}}  \binom{\tilde d_1 \epsilon_{\tilde d/k,j}}{(\tilde d_0-\tilde d_2) \epsilon_{\tilde d/k,j}} 
\nn \\ 
& &
\binom{(\tilde d_1+ \tilde d_2+\tilde d_3-\tilde d_0) \epsilon_{\tilde d/k,j}}{\tilde d_1 \epsilon_{\tilde d/k,j}}
\binom{\tilde d_1 \epsilon_{\tilde d/k,j}}{(\tilde d_0-\tilde d_3) \epsilon_{\tilde d/k,j}} \,.
\label{eq:divsumrootun}
\eea
Consider first $\tilde d \neq 1$ and write $\nu_p(n)$, $\mathrm{rad}(n)$ for, respectively, the $p$-adic valuation and the radical of  $n \in \bbZ^+$.  Let $k | \tilde d$ and suppose w.l.o.g. that $\tilde d/k$ has no repeated prime factors, $\tilde d/k=\mathrm{rad}(\tilde d/k)$. 
Then, for $\omega_{\tilde d} \neq 1$, the
  following trichotomy holds:
  \begin{itemize}
\item $\tilde d/k \nmid j$: there exists $p'$ prime with $p' \mid \tilde d/k$, $p'\nmid j$. Let 
      $k' \coloneqq k p'$. Then
$k' \mid \tilde d$, $\mathrm{gcd}(k',j)=\mathrm{gcd}(k,j)$, 
$\mu(\tilde d/k')=-\mu(\tilde d/k)$. Moreover $(-1)^{k' (j+1)}=(-1)^{k (j+1)}$, which is obvious when $p'$ is odd, 
and it also holds when $p'=2$ since in that case $j$ must be odd.
Then the contributions from $k'$ and $k$ to the sum
\eqref{eq:divsumrootun} cancel each
other.
\item 
$\tilde d/k \mid j$ and there exists $p'<\tilde d$  such that $p' \mid d$ and $p' \nmid j$. In this case we have $p' \nmid \tilde d/k$, $p' \mid k$. Let
  $k' \coloneqq k/p'$. Then as before $\mu(\tilde d/k) = -\mu(\tilde d/k')$,
  $(-1)^{k'(j+1)}=(-1)^{k'(j+1)}$ and $\mathrm{gcd}(k',j)=\mathrm{gcd}(k,j)$,
  and the summand corresponding to $k'$ has opposite sign to the one
  corresponding to $k$ in \eqref{eq:divsumrootun}.
\item 
$\tilde d/k \mid j$ and $\tilde d$ has no prime factor $p' \nmid j$. Suppose for simplicity that $\mathrm{rad}(j)/\mathrm{rad}(\tilde d)$ is odd, the even case being essentially identical. Then $\eqref{eq:divsumrootun}$ is unchanged upon replacing $j=:\prod_{p \mid j} p^{\nu_p(j)} $ with $\prod_{p \mid j, p\mid \tilde d} p^{\nu_p(j)}$, so we may assume that $\mathrm{rad}(j)=\mathrm{rad}(\tilde d)$. Let $p'$ be such that $\nu_{p'}(\tilde d) > \nu_{p'}(j)$ and let $k' \coloneqq k/p'$. Then once again the contributions of $k$ and $k'$ to \eqref{eq:divsumrootun} cancel each other.
  \end{itemize}
All in all, the above shows that $\widetilde{\Omega}_d(\mathrm{dP}_3(1,1))$ vanishes at
$\omega_{\tilde d}^j$ for all $\tilde d>1$, $j=1, \dots, \tilde d-1$. But by \cref{lem:binom0} these are all double zeroes, and therefore
$\Omega_d(\mathrm{dP}_3(1,1))$ is regular therein. Moreover, $\Omega_d(\mathrm{dP}_3(1,1))$ is regular by construction at $q=1$, where its value is given by replacing all $q$-expressions in \eqref{eq:omegad} by their classical counterparts. Hence $\Omega_d(\mathrm{dP}_3(1,1))
\in \bbQ[q^{\pm 1/2}]$ is a rational Laurent polynomial; but we also know that
$\Omega(\mathrm{dP}_3(1,1))_d \in q^{-\mathsf{g}_{\mathrm{dP}_3(1,1)}(d)/2} \bbZ[[q]]$ is an integral Laurent series,
which thus truncates at $\cO\big(q^{\mathsf{g}_{\mathrm{dP}_3(1,1)}(d)/2}\big)$. The statement of the theorem follows.\\

\item[$l=3$] as before, we prove the statement for $Y(D)=\mathrm{dP}_3(0,0,0)$ and recover all $3$-component pairs by restriction in the degrees.
Let $\tilde d \coloneqq \mathrm{gcd}(d_0, d_1, d_2, d_3)$ and $\hat d \coloneqq d_0^2-d_0 (d_1+d_2+d_3)+d_1 d_2+d_1 d_3+d_2 d_3$.  From \eqref{eq:log_dp3_0_0_0}
\bea
 \Omega_d(\mathrm{dP}_3(0,0,0))(q) &=& [1]_q^2 \sum_{k \mid \tilde d} \mu(k) k  \frac{(-1)^{(d_0+d_1+d_2)/k+1}  [\frac{\hat d}{k^2}]_{q^k}}{[d_1/k]_q [d_2/k]_q [d_3/k]_{q^k}} 
\Xi_{d/k}(q^k) \,,
\label{eq:omega2d}
\eea
where
\beq
\Xi_{d}(q) \coloneqq
 \qbinom{d_1}{d_0-d_2}_{q}  \qbinom{d_2}{d_0-d_3}_{q}  \qbinom{d_3}{d_0-d_1}_{q} \,.
\eeq
Outside $q=0, \infty$, $\Omega_d(\mathrm{dP}_3(0,0,0))(q)$ has at worst double poles at $q=\omega_d^j$ only; also it is verified directly that $q^{\mathsf{g}_{\mathrm{dP}_3(0,0,0)}(d)/2}\Omega_d$ has a Taylor expansion at $q=0$ with integer coefficients, where
\beq
\mathsf{g}_{\mathrm{dP}_3(0,0,0)}(d)=\mathsf{g}_{\mathrm{dP}_3(1,1)}(d).
\eeq
For $q=1$, the ratios of $q$-numbers in \eqref{eq:omega2d} limits to the corresponding classical counterparts, so $\Omega_d(\mathrm{dP}_3(0,0,0))(1)$ is well-defined. Suppose then $q=\omega_{\tilde d}^j \neq 1$. We have that
\bea
 \frac{[\frac{\hat d}{k^2}]_{q^k}}{[d_1/k]_{q^k} [d_2/k]_{q^k} [d_3/k]_{q^k}}
&=&
\frac{\widehat{d}}{k d_1 d_2 d_3} \l[ \frac{\omega_{\tilde d}^{2 j}}{ \big(q-\omega_{\tilde d}^{j}\big)^2}+\frac{1}{ q-\omega_{\tilde d}^{j}} +\cO(1)\r] \,.
\eea
%
This is nearly $k$-independent, save for the factor of $k$ that cancels the one present in the summand of \eqref{eq:omega2d}. By the same arguments of the previous point, the resulting divisor sum $$ \sum_{k \mid \tilde d} \mu(k) (-1)^{(d_0+d_1+d_2)/k+1} \Xi_{d/k}(q^k)$$ vanishes quadratically at $\omega_{\tilde d}^j$, and therefore $\Omega_d(\mathrm{dP}_3(0,0,0))(q)$ is regular on the unit circle, concluding the proof.\\

\item[$l=4$]
this consists of the single case $Y(D)=\bbF_0(0,0,0,0)$. Let $\tilde d \coloneqq \mathrm{gcd}(d_1, d_2)$. We have, from \eqref{eq:f03hFfin},
that
\bea
 \Omega_d(\bbF_0(0,0,0,0))(q) &=& \frac{[1]_q^2}{[d_1]_q^2 [d_2]_q^2} \sum_{k \mid \tilde d} \mu(k) k^2 [d_1 d_2/k^2]_{q^{k}}^2 \,.
\label{eq:omega3d}
\eea
In this case we have
\beq
\mathsf{g}_{\bbF_0(0,0,0,0)}(d)=2(d_1d_2-d_1- d_2+1)\,.
\eeq
As before, $\Omega_d(\bbF_0(0,0,0,0))(q)$ is a rational function with an integral Taylor-Laurent expansion at $q=0$, order $\mathsf{g}_{\bbF_0(0,0,0,0)}(d)/2$ singularities at $q=0,\infty$ and possibly double poles at $q=\omega_{\tilde d}^j$. Expanding \eqref{eq:omega3d} at $\omega_{\tilde d}^j$ yields
\bea
 \Omega_d(\bbF_0(0,0,0,0))(q) &=& \sum_{k \mid d} \mu (k)\l[
\frac{\omega _{\tilde d}^j \left(\omega _{\tilde d}^j-1\right){}^2}{\left(q-\omega _{\tilde d}^j\right){}^2}+\frac{2 \omega _{\tilde d}^j
   \left(\omega _{\tilde d}^j-1\right)}{q-\omega _{\tilde d}^j}\r]+ \cO(1)
\eea
which vanishes up to $\cO(1)$ since $\sum_{k \mid d} \mu(k)=0$, hence $\Omega_d(\bbF_0(0,0,0,0))(q)\in q^{-\mathsf{g}_{\bbF_0(0,0,0,0)}(d)/2}\bbZ[q]$.
\end{description}
\end{proof}

\section{Orbifolds}\label{sec:orbi}

In \cite{BBvG1}, we proposed in the context of toric pairs that the log-local principle should extend to $Y$ a possibly singular $\bbQ$-factorial projective variety. We expect that this should also hold for nef Looijenga pairs, at least as long as the orbifold singularities are at the intersection of the divisors: the log GW theory is then well-defined since $Y(D)$ is log smooth, and the local GW theory makes sense by viewing $Y$ and $E_{Y(D)}$ as smooth Deligne--Mumford stacks. In particular, introducing singularities gives new infinite lists of examples of nef/quasi-tame/tame Looijenga pairs.

We propose that also \cref{thm:logopen,thm:kpdt,thm:openbps} may extend to the orbifold setting. We present the simplest instance here, and defer a more in-depth discussion, including criteria for the validity of the orbifold version of \cref{thm:logopen,thm:kpdt,thm:openbps}, to \cite{Bousseau:2020ryp}.

\begin{figure}[t]
\caption{$\scatt\bbP_{(1,1,n)}$}
\label{fig:P11nscatt}
\begin{tikzpicture}[smooth, scale=1.2]
\draw[<->] (-2.5,2) to (-1.5,2);
\draw[<->] (-2,1.5) to (-2,2.5);
\node at (-1.67,2.15) {$\scriptstyle{x}$};
\node at (-2.15,2.3) {$\scriptstyle{y}$};
\draw[step=1cm,gray,very thin] (-2.5,-2.5) grid (3.5,2.5);
\draw[thick] (3.5,0) to (-2.5,0);
\draw[thick] (0,-2.5) to (0,0);
\draw[thick] (0,0) to (1.25,2.5);
\node at (-0.4,-1.7) {$D_2$};
\node at (0.9,1.2) {$D_1$};
\node at (-1.5,0.3) {$E$};
\node at (-2.3,0) {$\times$};
\node at (0.55,2.2) {$\sstyle{(1,n+1)}$};
\node at (0.7,0.2) {$\sstyle{1+tx^{-1}}$};
\draw[<->,thick] (1.5,-2.5) to (1.5,0) to (2.75,2.5);
\node at (2.1,-2.2) {$\sstyle{y^{-(n+1)d}}$};
\node at (3.2,2.2) {$\sstyle{x^{d}y^{(n+1)d}}$};
\node at (3.2,0.4) {$\sstyle{\binom{(n+1)d}{d}t^{d}x^{-d}y^{-(n+1)d}}$};
\node at (2,1) {$\bullet$};
\node at (2.2,1.1) {$\sstyle{p}$};
\end{tikzpicture}
\end{figure}

\begin{example}
Let $Y=\bbP_{(1,1,n)}$ be the weighted projective plane with weights $(1,1,n)$, and $D=D_1+D_2$ with 
$D_1$ a toric line passing through the orbifold point and $D_2$ a smooth member of the linear system given by the sum of the two other toric divisors. Since $D_1 \sim \frac{H}{n}$, $D_2 \sim \frac{(n+1)}{n} H$, $H^2 = n$, we have
$D_1^2=\frac{1}{n}$ and $D_2^2=\frac{(n+1)^2}{n}$. Therefore $\bbP_{(1,1,n)}(\frac{1}{n}, \frac{(n+1)^2}{n})$ is a tame orbifold Looijenga pair.

Local Gromov--Witten invariants of $Y(D)$ can be computed by the orbifold quantum Riemann--Roch theorem of \cite{MR2578300}: when restricted to point insertions, it gives \eqref{eq:localgw} specialised to the case at hand, and we get
\beq
N_{0,d}^{\rm loc}\l(\bbP_{(1,1,n)}\Big(\frac{1}{n},
\frac{(n+1)^2}{n} \Big)\r) = \frac{(-1)^{n d}}{(n+1)d^2} \binom{(n+1) d}{d}\,. 
\label{eq:p11nloc}
\eeq 

A toric model and a quantised scattering diagram for $Y(D)$ can be constructed as follows. The fan of $\bbP_{(1,1,n)}$ has 1-skeleton given by $(-1,0)$, $(0,-1)$ and $(1,n)$. Start by adding a ray in the direction $(-1,1)$, and denote $E$ the corresponding divisor. Then the proper transform of $D_{(-1,0)}$ is a $(-1)$-curve, which we contract. The complement of the proper transform of $D$ now has Euler characteristic 0, hence is $(\bbC^*)^2$, and therefore the variety is toric. Applying the $\mathrm{SL}(2,\bbZ)$ transformation
\beq
\begin{pmatrix}
1 & 0 \\
1 & 1 \\
\end{pmatrix}
\eeq
we obtain the toric model depicted in \cref{fig:P11nscatt}, for which the broken line calculation is straightforward. The result is
\beq
N_{0,d}^{\rm log}\l(\bbP_{(1,1,n)}\Big(\frac{1}{n},
\frac{(n+1)^2}{n} \Big)\r) = \binom{(n+1)d}{d},  \label{eq:p11nlog}
\eeq
\beq \mathsf{N}_d^{\rm log}\l(\bbP_{(1,1,n)}\Big(\frac{1}{n},
\frac{(n+1)^2}{n} \Big)\r) = \qbinom{(n+1)d}{d}_q.
\label{eq:p11nlogref}
\eeq
To construct $\bbP_{(1,1,n)}^{\rm op}\big(\frac{1}{n},
\frac{(n+1)^2}{n} \big)$, we delete the line $D_1$. Then $\cO(-D_2)$ is trivial on $\bbP_{(1,1,n)} \setminus D_1 = \bbC^2$, and $\mathrm{Tot}(K_{\bbP_{(1,1,n)} \setminus D_1})=\bbC^3$, with an outer toric Lagrangian at framing shifted by $n$. A topological vertex calculation of higher genus 1-holed open Gromov--Witten invariants as in \cref{sec:disks} shows that
\beq
\mathsf{O}_d\l(\bbP_{(1,1,n)}^{\rm op}\Big(\frac{1}{n},
\frac{(n+1)^2}{n} \Big)\r) = \frac{(-1)^{nd}}{d [(n+1)d]_q}\qbinom{(n+1)d}{d}_q.
\label{eq:p11nop}
\eeq
\cref{eq:p11nloc,eq:p11nlog,eq:p11nlogref,eq:p11nop} together imply that \cref{thm_log_local,thm:logopen} extend to this case as well. The arguments in the proof of \cref{thm:kpdt} also apply verbatim, with $\mathsf{Q}\big(\bbP_{(1,1,n)}(\frac{1}{n},
\frac{(n+1)^2}{n})\big)$ the $(n+1)$-loop quiver. An interesting consequence is that the integrality statement of \cref{conj:KP} appears to persist in the orbifold world too. The proof of the higher genus open BPS property in \cref{thm:openbps} also carries through to this setting with no substantial modification.

\begin{figure}[t]
\includegraphics[scale=1.3]{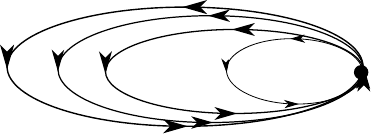}
\caption{The quiver for $Y(D)=\bbP(1,1,3)(\frac{1}{3}, \frac{16}{3})$.}
\end{figure}

\end{example}

\begin{appendix}

\section{Proof of \cref{thm:dP5}}
\label{app:PdP5}

Let $Y$ be the toric surface given by the fan of \cref{fig:PdP5}. It is described by the exact sequence
\beq
\label{eq:PdP5seq}
  \begin{CD}
    0 @>>> \ZZ^6 @>{
      \begin{pmatrix}
	1 &  1 & 0 &  0 & 0  & -2 \\ 
	0  & 0 & 1 &  0 & 0  & 1 \\ 
	1  & 0 & 0 &  1 & 0 &  0 \\ 
	0  & 0 & 0 &  0 & 1 &  0 \\ 
	0  & 1 & 0 &  1 & 0 &  0 \\ 
	1  & 0 & 1 &  0 & 0 &  0 \\ 
	0  & 0 & 0 &  1 & 0 &  0 \\ 
	0  & 0 & 0 &  0 & 1 &  1
      \end{pmatrix}}>>  
    \ZZ^8 @>{
      \begin{pmatrix} \textstyle
        1 & 1 & 0 & -1 & -1 & -1 & 0 & 1  \\
        0 & 1 & 1 & 1 & 0 & -1 &  -1 &  -1 \\
      \end{pmatrix}}>> \ZZ^2 @>>> 0\,
  \end{CD}
\eeq
showing that $Y$ is a GIT quotient $\bbC^{8} \GIT{} (\bbC^\star)^6 = (\bbC^{8} \setminus \{x_i x_j=0 \}_{\stackrel{(i,j) \neq (1,8),}{\tiny j \neq i+1}})/ (\bbC^\star)^6$, with $(\tau_1, \dots, \tau_6) \in (\bbC^\star)^6$ acting as
\beq
\begin{pmatrix}
x_1 \\ 
x_2 \\
x_3 \\
x_4 \\
x_5 \\
x_6 \\
x_7 \\
x_8 \\
\end{pmatrix}
\longrightarrow 
\begin{pmatrix}
\tau_1 \tau_2 \tau_6^{-2} x_1 \\ 
\tau_3 \tau_6 x_2 \\
\tau_1 \tau_4 x_3 \\
\tau_5 x_4 \\
\tau_2 \tau_4 x_5 \\
\tau_1 \tau_3 x_6 \\
\tau_4 x_7 \\
\tau_5 \tau_6 x_8 \\
\end{pmatrix}
\eeq
\begin{figure}[t]
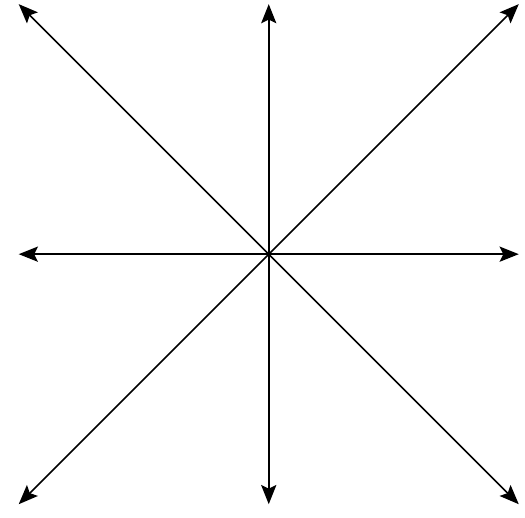
\caption{The fan of $\mathrm{Bl}_{4 {\rm pts}} \bbP^1 \times \bbP^1$}
\label{fig:PdP5}
\end{figure}
There are dominant birational morphisms $Y \stackrel{\pi_1}{\longrightarrow} \bbP^2$, $Y \stackrel{\pi_2}{\longrightarrow} \bbP^1 \times \bbP^1$, obtained by deleting the loci $\{ x_i=0\}_{i\in \{2,4,6,7,8\}}$ and $\{x_{2i}=0\}$ respectively. Therefore $Y\simeq \mathrm{Bl}_{4 {\rm pts}} \bbP^1 \times \bbP^1$, or equivalently, $Y$ is a five-point toric blow-up of $\bbP^2$, and deforms to $\mathrm{dP}_5$ upon taking the points in general position. From \eqref{eq:PdP5seq} and \cref{fig:PdP5}, in terms of the hyperplane $H$ and exceptional classes $E_i \in \mathrm{Pic}(\mathrm{dP}_5)$ the toric divisors $T_i := \{x_i=0\}$ read
\bea
&
T_1 = H - E_1 - E_2 - E_4\,, ~ T_3 = H - E_1 - E_3 - E_5\,, ~
T_5 = E_2 - E_4, ~ T_7 = E_3 - E_5\,, \nn \\
& T_2 = E_1\,,~ T_4 = E_4\,,~ T_6 = H - E_2 - E_3\,,~ T_8 = E_5\,.
\label{eq:divPdP5}
\eea
Under this identification the $-2$-curve classes $T_{2k+1}$ do not belong to $\mathrm{NE}(\mathrm{dP}_5)$ (see the discussion of \cref{sec:eff}); however they do have by construction effective representatives in $A_1(Y)$, since they are prime toric divisors. 

To write the $I$-function, we fix the following set of $\frac{1}{2} \bbZ$-generators of $A_1(Y)$:
\beq
C_i = \left\{
\bary{cl}
T_{2i}, & i=1,\dots, 4~, \\
D_{i+4}, & i=1,2~	,
\eary
\right.
\eeq
where $D_1=H-E_1= T_1+T_3+2T_6$, $D_2=2H-E_2-E_3-E_4-E_5=T_2+T_4+T_5+T_7-T_8$. We will write $\varphi_i$ with $(\varphi_i, C_j)=\delta_{ij}$ for their dual basis in cohomology, and denote curve classes in this basis as $d = \sum_i \frac{\sigma_i \delta_i}{2} C_i$ with $\delta_i\in \bbZ$, and $\sigma_i=-1$ for $1\leq i\leq 4$ and $\sigma_i=1$ otherwise. To write the twisted $I$-function $I^{E_{Y(D)}}$, we need to expand $\theta_a= c_1(\cO(T_a))$ and $\kappa_i = c_1(\cO(D_i))$ in \eqref{eq:Ifunloc}, yielding
\bea
& &
I^{E_{Y(D)}}(y,z) =
\sum_{\delta_i \in \bbZ}
\Bigg[
\frac{ y_1^{-\frac{\delta _1}{2}}
  y_2^{-\frac{\delta _2}{2}} y_3^{-\frac{\delta _3}{2}} y_4^{-\frac{\delta _4}{2}} y_5^{\frac{\delta _5}{2}} y_6^{\frac{\delta
  _6}{2}} (-1)^{\delta _5+\delta _6} }{  \left(1-\frac{2 \varphi_1}{z}\right)_{\delta _1} \left(1-\frac{2 \varphi_2}{z}\right)_{\delta _2} \left(1-\frac{2 \varphi_3}{z}\right)_{\delta
   _3} \left(1-\frac{2 \varphi_4}{z}\right)_{\delta _4}\left(\frac{z+\varphi_1+\varphi_2+\varphi_5}{z}\right)_{\frac{1}{2} \left(-\delta _1-\delta
  _2+\delta _5\right)} }
\nn \\
& &
\frac{\left(2 \varphi_6-\lambda _1\right) \left(2 \varphi_5-\lambda _2\right) \left(\frac{z+2 \varphi_6-\lambda _1}{z}\right)_{\delta _6-1} \left(\frac{z+2 \varphi_5-\lambda _2}{z}\right)_{\delta _5-1}}{ z \left(\frac{z+\varphi_3+\varphi_4+\varphi_5}{z}\right)_{\frac{1}{2} \left(-\delta _3-\delta _4+\delta _5\right)   \left(\frac{z+\varphi_1+\varphi_3+\varphi_6}{z}\right)_{\frac{1}{2} \left(-\delta _1-\delta _3+\delta _6\right)}
   \left(\frac{z+\varphi_2+\varphi_4+\varphi_6}{z}\right)_{\frac{1}{2} \left(-\delta _2-\delta _4+\delta _6\right)}}}
\Bigg],
\label{eq:IfunPdP5}
\eea
where
\beq (a)_n=a(a+1)\dots (a+n-1)\,.
\label{eq:pochh}
\eeq 
is the Pochhammer symbol. By \eqref{eq:mm}, the mirror map is extracted as the formal $\cO(z^0)$ Taylor coefficient around $z=\infty$. We find that the sole contributions to the mirror map arise from multiple covers of our chosen generators $C_i$, that is when $\delta_i=2 \sigma_i n$, $n \in \bbN^+$
\beq
\tilde t^i(y) = \sum_{\delta_i=1}^{\infty} \frac{(2 \delta_i-1)!}{(\delta_i!)^2}y^{\delta_i},
\eeq
which is closed-form inverted as
\beq
y_i(t) = \frac{\exp t^i}{(1-\exp t^i)^2}\,.
\label{eq:PdP5mm}
\eeq
Then\footnote{To obtain the small $J$-function, we should include a string-equation induced shift  by multiplying the $I$-function by an overall factor of $\re^{\lambda_1 \tilde t^5(y) +\lambda_2 \tilde t^6(y)/z}$, in order to guarantee that the small $J$-function satisfies its defining property to be the unique family of Lagrangian cone elements with a Laurent expansion of the form $z + t+O(1/z)$ at $z=\infty$. These would result in a correction of the foregoing discussion for degrees $\delta_i=0$ when $i=1,\dots, 4$. It is justified to ignore this for our purposes:  since $d \cdot D_i =0$ and $\cO(-D_i)$ is not a concave line bundle, the corresponding invariants are non-equivariantly ill-defined; and any sensible non-equivariant definition would satisfy automatically the log-local correspondence of \cref{sec:logloc_gen}, as the corresponding log invariants are trivially zero.}, $$J_{\rm small}^{E_{Y(D)}}=I^{E_{Y(D)}}(y(t),z)$$ and from \eqref{eq:NpsidJfun} and \eqref{eq:IfunPdP5} we find that whenever $d \neq 2 \sigma_i n$, $n \in \bbN^+$
\bea
N^{\rm loc, \psi}_{\delta_1, \dots, \delta_6}(Y(D)) &=& \frac{1}{\lambda_1 \lambda_2}  \Big[z^{-1} \re^{\sum t_i \varphi_i/z}  \mathbf{1}_{H_T(E_{Y(D)})}\Big]  I^{E_{Y(D)}}(y(t),z), \nn \\
&=& \big[\re^{\sum_i \delta_i  t_i} \big]  \sum_{\delta'_i}^\infty  \mathsf{S}^{[0]}_{\delta'_1, \dots, \delta'_6} \prod_{i=1}^6 y_i(t)^{\frac{\sigma_i\delta _i'}{2}},
\eea
where
\beq
\mathsf{S}^{[0]}_{\delta'_1, \dots, \delta'_6} :=
\frac{(-1)^{\delta _5'+\delta _6'} \left(\delta _5'-1\right)! \left(\delta _6'-1\right)!} {\delta _1'! \delta _2'! \delta _3'! \delta _4'! \left(\frac{1}{2} \left(\delta _5'-\delta _1'-\delta _2'\right)\right)! \left(\frac{1}{2} \left(\delta _5'-\delta _3'-\delta _4'\right)\right)! \left(\frac{1}{2} \left(\delta _6'-\delta _1'-\delta
   _3'\right)\right)! \left(\frac{1}{2} \left(\delta _6'-\delta _2'-\delta _4'\right)\right)!}\,.
\eeq
The arguments of the factorials in the denominator constrain the range of summation to extend over $\delta_i \neq 0$ alone; in particular the r.h.s. is a Taylor series in $(y_1^{-1/2},y_2^{-1/2},y_3^{-1/2},y_4^{-1/2},y_5^{1/2},y_6^{1/2})$, convergent in a ball centred at $y_i^{\sigma_i}=0$. We first perform the summation over $\delta_6'$ to obtain
\bea
\sum_{\delta'_6=0}^{\infty}\mathsf{S}^{[0]}_{\delta'_1, \dots, \delta'_6} y_6^{\delta'_6/2} &=&
\frac{(-1)^{\delta _2'+\delta _4'+\delta _5'} \left(\delta _2'+\delta _4'-1\right)! \left(\delta _5'-1\right)! \left(\frac{\re^{t_6}}{(\re^{t_6}+1)^2}\right)^{\frac{1}{2} \left(\delta _2'+\delta _4'\right)}  }{\delta _1'! \delta _2'! \delta _3'! \delta _4'! \left(\frac{1}{2} \left(-\delta _1'+\delta _2'-\delta _3'+\delta _4'\right)\right)! \left(\frac{1}{2} \left(-\delta _1'-\delta _2'+\delta _5'\right)\right)!
   \left(\frac{1}{2} \left(-\delta _3'-\delta _4'+\delta _5'\right)\right)!} \nn \\
& \times & \, _2F_1\left(\frac{1}{2} \left(\delta _2'+\delta _4'\right),\frac{1}{2} \left(\delta _2'+\delta _4'+1\right);\frac{1}{2} \left(-\delta _1'+\delta _2'-\delta _3'+\delta
   _4'+2\right);\frac{4 \re^{t_6}}{(\re^{t_6}+1)^2}\right), \nn \\
\eea
where
\beq
     {}_{p}F_{r}(a_1,\dots,a_p;b_1,\dots,b_r;z) \coloneqq 
    \sum_{k \geq 0} \frac{z^k}{k!} \frac{\prod_{j=1}^p (a_j)_k}{\prod_{j=1}^r(b_j)_k} 
    \label{eq:pFr}
\eeq 
is the generalised hypergeometric function. Applying Kummer's quadratic transformation,
\beq
\, _2F_1(a,b;a-b+1;z)=(z+1)^{-a} \, _2F_1\left(\frac{a}{2},\frac{a+1}{2};a-b+1;\frac{4 z}{(z+1)^2}\right)
\eeq
we obtain
\beq
N^{\rm loc, \psi}_{\delta_1, \dots, \delta_6}(Y(D))= \big[\re^{\sum_{i=1}^5 \delta_i  t_i} \big]  \sum_{\delta'_i}^\infty  \mathsf{S}^{[1]}_{\delta'_1, \dots, \delta'_5, \delta_6}\prod_{i=1}^5 y_i(t)^{\frac{\sigma_i\delta _i'}{2}},
\eeq
where
\bea
\mathsf{S}^{[1]}_{\delta'_1, \dots, \delta'_5, \delta_6} &:=&
\frac{(-1)^{\delta _2'+\delta _4'+\delta _5'} \left(\delta _5'-1\right)! \left(\frac{1}{2} \left(\delta _1'+\delta _3'\right)+\delta _6-1\right)! \left(\frac{1}{2}
   \left(\delta _2'+\delta _4'\right)+\delta _6-1\right)!}{\delta _1'! \delta _2'! \delta _3'! \delta _4'! \left(\frac{1}{2} \left(\delta _1'+\delta _2'+\delta
   _3'+\delta _4'-2\right)\right)! \left(\frac{1}{2} \left(-\delta _1'-\delta _2'+\delta _5'\right)\right)!} \nn \\ & \times & \frac{1}{\left(\frac{1}{2} \left(-\delta _3'-\delta _4'+\delta
   _5'\right)\right)! \left(-\frac{\delta _1'}{2}-\frac{\delta _3'}{2}+\delta _6\right)! \left(-\frac{\delta _2'}{2}-\frac{\delta _4'}{2}+\delta _6\right)!}\,.
\eea
Performing the same sequence of operations on the sum over $\delta_5'$ yields
\beq
N^{\rm loc, \psi}_{\delta_1, \dots, \delta_6}(Y(D))= \big[\re^{\sum_{i=1}^4 \delta_i  t_i} \big]  \sum_{\delta'_i}^\infty  \mathsf{S}^{[2]}_{\delta'_1, \dots, \delta'_4, \delta_5, \delta_6}\prod_{i=1}^4 y_i(t)^{\frac{\sigma_i\delta _i'}{2}},
\eeq
where
\bea
\mathsf{S}^{[2]}_{\delta'_1, \dots, \delta'_4, \delta_5, \delta_6} &:=&
\frac{(-1)^{\delta _1'+\delta _4'} \left(\frac{1}{2} \left(\delta _1'+\delta _2'\right)+\delta _5-1\right)! \left(\frac{1}{2} \left(\delta _3'+\delta
   _4'\right)+\delta _5-1\right)! \left(\frac{1}{2} \left(\delta _1'+\delta _3'\right)+\delta _6-1\right)! }{\delta _1'! \delta _2'! \delta _3'! \delta _4'! \left(\left(\frac{1}{2} \left(\delta _1'+\delta _2'+\delta _3'+\delta
   _4'-2\right)\right)!\right){}^2 \left(-\frac{\delta _1'}{2}-\frac{\delta _2'}{2}+\delta _5\right)! } \nn \\
& \times & 
\frac{\left(\frac{1}{2} \left(\delta _2'+\delta
   _4'\right)+\delta _6-1\right)!}{
\left(-\frac{\delta _3'}{2}-\frac{\delta _4'}{2}+\delta
   _5\right)! \left(-\frac{\delta _1'}{2}-\frac{\delta _3'}{2}+\delta _6\right)! \left(-\frac{\delta _2'}{2}-\frac{\delta _4'}{2}+\delta _6\right)!}\,.
\eea
The final step is to now plug in the mirror maps \eqref{eq:PdP5mm} for $i=1,...,4$. This gives
\beq
N^{\rm loc, \psi}_{\delta_1, \dots, \delta_6}(Y(D))= \sum_{j_1, \dots, j_4=0}^\infty  \mathsf{S}^{[3]}_{\delta'_1+2j_1, \dots, \delta'_4+2j_4, j_1, \dots, j_4, \delta_5, \delta_6}\,,
\eeq
where
\bea
\mathsf{S}^{[3]}_{\delta'_1, \dots, \delta'_4, j_1, \dots, j_4, \delta_5, \delta_6} &:=& \mathsf{S}^{[2]}_{\delta'_1, \dots, \delta'_4, \delta_5, \delta_6} \prod_{i=1}^4 \binom{\delta'_i}{j_i}\,.
\eea
The change-of-basis $\{C_1, \dots, C_6\} \rightarrow \{H-E_1-\dots -E_5, E_1, \dots E_5\}$ in \eqref{eq:divPdP5} and the corresponding change-of-variables in the curve degrees parameters $\{\delta_1, \dots, \delta_6\} \to \{d_0, \dots, d_5\}$ finally leads to \eqref{eq:dP5}.

\section{Infinite scattering}

\label{app:scatt}


We compute the invariants of
\cref{conj:dp302} for the geometries $\delp_1(0,4)$ and $\bbF_0(0,4)$. This application of our correspondences predicts new relations for $q$-hypergeometric sums in \cref{conj:dp1binom}. We provide calculations by picture and leave the details to the reader.

\label{sec:infscatt}

Denote by $E$ the exceptional divisor obtained by blowing up a point on $D_1$ in $\ptwo(1,4)$. We write a curve class $d\in\hhh_2(\delp_1(0,4),\bbZ)$ as $d=d_0(H-E)+d_1E$. 
If $d_0=0$ or $d_1=0$, then the moduli space of stable log maps is empty and 
$\mathsf{N}^{\rm log}_d(\delp_1(0,4))(\hbar)=0$. 
If $d_1>d_0$, then there are no irreducible curve classes and $\mathsf{N}^{\rm log}_d(\delp_1(0,4))(\hbar)=0$. 
The toric model of $\delp_1(0,4)$ is obtained from the toric model of $\ptwo(1,4)$ by adding a focus-focus singularity in the direction of $D_1$. The opposite primitive vectors in the $F_2$ and $D_1$ directions are $\gamma_1=(1,0)$ and $\gamma_2=(-1,-2)$. Since the absolute value of their determinant is 2 and not $1$, there is infinite scattering, which is described in Section \ref{sec:theta}. By choosing our broken lines to be sufficiently into the $x$-direction, we can restrict to walls that lie on the halfspace $x>0$. Then these walls have slope $(n+1)\gamma_1+n\gamma_2=(1,-2n)$, for $n\geq0$. The wallcrossing functions attached to them are $1+t^{n+1}t_1^nx^{-1}y^{2n}$. The broken line computation is summarised in \cref{fig:delp1scatt}.

\begin{thm} \label{thm:log_dp1_0_4}
Let $d_0 > d_1\geq 1$ and $d=d_0(H-E)+d_1E$. Then 
$\mathsf{N}^{\rm log}_d(\delp_1(0,4))(\hbar)$ equals
\beq
\sum_{m=1}^{d_1} \quad \sum_{\substack{ k_0+\sum_{j=1}^{m}k_j=d_1 \\ \sum_{j=1}^m n_jk_j=d_0-d_1 \\ k_1,\dots,k_m >0, \, k_0\geq 0 \\ n_1>n_2>\cdots>n_{m}>0 }} \quad
\qbinom{2d_0}{k_1}_q \qbinom{2d_0-2(n_1-n_2)k_1}{k_2}_q \cdots \qbinom{2d_0-2\sum_{j=1}^{m-1}(n_j-n_m)k_j}{k_m}_q \qbinom{2d_1}{k_{0}}_q \,.
\eeq
\end{thm}



For the case of $\fzero(0,4)$, let $D_1$ be a line of bidegree $(1,0)$ and let $D_2$ be a smooth divisor of bidegree $(1,2)$. Let $d$ be a curve class of bidegree $(d_1,d_2)$. We have 
$d\cdot D_1=d_2$ and $d\cdot D_2=2d_1+d_2$. Denote by $\mathrm{pt}_1$, resp. $\mathrm{pt}_2$, their intersection points and by $L_1$, resp. $L_2$, the lines of bidegree $(0,1)$ passing through $\mathrm{pt}_1$, resp. $\mathrm{pt}_2$. We blow up $\mathrm{pt}_1$ and $\mathrm{pt}_2$ leading to exceptional divisors $F_1$ and $F_2$ and blow down the strict transforms of $L_1$ and $L_2$. The result is the Hirzebruch surface $\ftwo$ with a focus-focus singularity on each of the fibrewise toric divisors, as in Figure \ref{fig:fzero0}.

\begin{figure}[t]
\begin{tikzpicture}[smooth, scale=1.2]
\draw[step=1cm,gray,very thin] (-2.5,-2.5) grid (2.5,2.5);
\draw[thick] (0,0) to (-2.5,0);
\draw[thick] (0,-2.5) to (0,2.5);
\draw[thick] (0,0) to (1.25,2.5);
\node at (0.3,-1.7) {$D_2$};
\node at (1.3,1.7) {$F_1$};
\node at (-0.3,1.7) {$D_1$};
\node at (-1.5,-0.3) {$F_2$};
\node at (-2.3,0) {$\times$};
\node at (0.75,1.5) {$\times$};
\end{tikzpicture}
\caption{$\fzero(0,4)$}
\label{fig:fzero0}
\end{figure}

Let $d$ be a curve class of bidegree $(d_1,d_2)$. 
The opposite primitive vectors in the $F_2$ and $F_1$ directions are $\gamma_1=(1,0)$ and $\gamma_2=(-1,-2)$. The absolute value of their determinant is 2, so there is infinite scattering as described in Section \ref{sec:theta}. We choose $p$ to be in the lower left quadrant with coordinate $(a,b)$ for $-1\ll a<0$ and $b\ll 0$. This depends on the degree and ensures that the broken lines are vertical at $p$. In particular, we can restrict to walls that lie on the halfspace $x<0$. Then these walls have slope $(n-1)\gamma_1+n\gamma_2=(-1,-2n)$, for $n\geq1$. The wall-crossing functions attached to them are $1+t^{n-1} t_1^n x y^{2n}$. The broken line calculation is summarised in \cref{fig:delp1scatt}.

\begin{thm} \label{thm:log_f0_0_4}
For $d_1\geq1$, the generating function $\mathsf{N}^{\rm log}_{(d_1,d_2)}(\fzero(0,4))(\hbar)$ equals
\begin{IEEEeqnarray*}{rC}
\sum_{m=1}^{\left \lfloor{  \frac{1}{2}(\sqrt{1+8d_1}-1)  }\right \rfloor }
\sum_{\substack{ d_1 = \sum_{j=1}^m n_j k_j \\ k_m,\dots,k_1 >0 \\ n_m>\cdots>n_{1}>0 }}
\, & \qbinom{d_2+2d_1}{k_m}_q \, \cdots  \,
\qbinom{d_2 + 2n_i\sum_{j=i}^{m}k_j + 2\sum_{j=1}^{i-1} n_j k_j}{k_{i}}_q 
\, \cdots \\
 \cdots \, &
\qbinom{d_2+2n_2 \sum_{j=2}^{m}k_j + 2n_1k_1}{k_2}_q 
\qbinom{d_2+2n_1\sum_{j=1}^{m}k_j}{k_1}_q  
\qbinom{d_2}{\sum_{j=1}^{m}k_j}_q \,.
\end{IEEEeqnarray*}
\end{thm}

\cref{conj:log_defo} predicts that the multi-variate $q$-hypergeometric sums of \cref{thm:log_dp1_0_4,thm:log_f0_0_4} dramatically simplify to remarkably compact $q$-binomial expressions. This is expressed by the following new conjectural $q$-binomial identities.

\begin{conj}\label{conj:dp1binom}
The $q$-hypergeometric sums of \cref{thm:log_dp1_0_4,thm:log_f0_0_4} are equal to
\bea
\mathsf{N}^{\rm log}_d(\delp_1(0,4))(\hbar) &=&\frac{[2d_0]_q}{[d_0]_q}\,\qbinom{d_0}{d_1}_q \qbinom{d_0+d_1-1}{d_0}_q \,, \\
%
    \mathsf{N}^{\rm log}_d(\fzero(0,4))(\hbar)
    &=& \frac{[2d_1+d_2]_q}{[d_2]_q} 
    \qbinom{d_1+d_2-1}{d_1}_q^2 \,,
\eea
where $q=\re^{\ri \hbar}$.

\end{conj}

\begin{figure}[t]
\begin{minipage}{0.48\textwidth}
\flushleft
\begin{tikzpicture}[smooth, scale=1.2]
\draw[step=1cm,gray,very thin] (-2.5,-4.5) grid (3.5,2.5);
\draw[<->] (-2.5,2) to (-1.5,2);
\draw[<->] (-2,1.5) to (-2,2.5);
\node at (-1.67,2.15) {$\scriptstyle{x}$};
\node at (-2.15,2.3) {$\scriptstyle{y}$};
\draw(-2.5,0) to (3.5,0);
\draw (0,-4.5) to (0,2.5);
\draw (0,0) to (1.25,2.5);
\node at (-0.25,-1.7) {$D_2$};
\node at (0.95,2.35) {$D_1$};
\node at (-0.25,1.7) {$F_1$};
\node at (-1.5,0.2) {$F_2$};
\node at (-2.3,0) {$\times$};
\node at (0.75,1.5) {$\times$};
\node at (-0.5,0.2) {$\sstyle{1+tx^{-1}}$};
\node[rotate=63.43] at (0.35,1) {$\sstyle{1+t_1xy^2}$};
\node at (0.4,-2.5) {$\cdots$};
\draw (0,0) to (2.25,-4.5);
\draw (0,0) to (1.125,-4.5);
\draw[<->,thick] (0.75,-4.5) to (0.75,-3) to (0.9,-1.8) to (1.5,0) to (2.75,2.5);
\node at (0.4,-4.2) {$\sstyle{y^{-2d_0}}$};
\node at (3.1,2.2) {$\sstyle{x^{d_1}y^{2d_1}}$};
%
\node at (2,1) {$\bullet$};
\node at (2.2,1.1) {$\sstyle{p}$};
\end{tikzpicture}
\end{minipage}
\begin{minipage}{0.48\textwidth}
\flushright
\begin{tikzpicture}[smooth, scale=1.2]
\draw[step=1cm,gray,very thin] (-3.5,-4.5) grid (2.5,2.5);
\draw[<->] (-3.5,2) to (-2.5,2);
\draw[<->] (-3,1.5) to (-3,2.5);
\node at (-2.67,2.15) {$\scriptstyle{x}$};
\node at (-3.15,2.3) {$\scriptstyle{y}$};
\draw(-3.5,0) to (2.5,0);
\draw (0,-4.5) to (0,2.5);
\draw (0,0) to (1.25,2.5);
\node at (0.3,-1.7) {$D_2$};
\node at (1.3,1.7) {$F_1$};
\node at (-0.3,1.7) {$D_1$};
\node at (-2.5,-0.3) {$F_2$};
%
\node at (-3.3,0) {$\times$};
\node at (0.75,1.5) {$\times$};
\node at (-2.5,0.2) {$\sstyle{1+tx^{-1}}$};
\node[rotate=63.43] at (0.35,1) {$\sstyle{1+t_1xy^2}$};
\node at (-0.3,-2.5) {$\cdots$};
\draw (0,0) to (-2.25,-4.5);
\draw (0,0) to (-1.125,-4.5);
\draw[<->,thick] (-0.75,-4.5) to (-0.75,-3) to (-0.9,-1.8) to (-1.5,0) to (-1.5,2.5);
\node at (-0.15,-4.25) {$\sstyle{y^{-2d_1-d_2}}$};
\node at (-1.75,2.2) {$\sstyle{y^{d_2}}$};
%
\node at (-0.75,-4) {$\bullet$};
\node at (-0.6,-3.9) {$\sstyle{p}$};
\end{tikzpicture}
\end{minipage}
\caption{Scattering diagrams of $\delp_1(0,4)$ (left) and $\fzero(0,4)$ (right).}
\label{fig:delp1scatt}
\end{figure}

A proof of the identities of \cref{conj:dp1binom} was communicated to us by C.~Krattenthaler \cite{CKpriv}. Note that the genus $0$ log-local correspondence of Theorem \ref{thm_log_local} and the deformation invariance of local Gromov--Witten invariants give an entirely geometric proof of their classical limit at $q=1$.

\section{Proof of \cref{thm_key}}
\label{sec:constr_trop}
Recall the notation of Section \ref{sec:logloc_gen} and let $h \colon \Gamma \rightarrow \Delta^h$ be a rigid decorated parametrised tropical curve with $N^{{\rm loc}, h}_{0,d} (Y(D)) \neq 0$. 
Our goal is to prove that $h=\bar{h}$.
This will be done by a series of Lemmas constraining further and further the possible shape of $h$.

\begin{lem} \label{lem_central_vertex}
There exists at least one vertex $V$ of 
$\Gamma$ with $h(V)=v_Y$.
\end{lem}

\begin{proof}
We are considering stable log maps to $\PP(\cV_0)$ with $l-1>0$ marked points mapping to the interior of $\PP(\cV_{0,Y})$. So irreducible components containing these marked points map to $\PP(\cV_{0,Y})$, and the corresponding vertices of $\Gamma$ are mapped to $v_Y$ by $h$. 
\end{proof}

We choose a flow on $\Gamma$ such that unbounded edges are ingoing, such that every vertex has at most one outgoing edge, and such that the sink $V_0$ satisfies $h(V_0)=v_Y$. 
Such flow exists by Lemma \ref{lem_central_vertex}. Following the flow, the maps $\eta_V$
define a cohomology class $\alpha_E \in \hhh^{*}(X_E)$ for every edge $E$ of $\Gamma$. 
The degeneration formula can be rewritten as 
\beq N^{{\rm loc}, h}_{0,d} (Y(D))
=\eta_{V_0} \left( \prod_{E \in \cE_{in}(V_0)} \alpha_E \right) \,.\eeq

For every  $V$ vertex of $\Gamma$ with $h(V) \in (\partial \Delta)_j$,
denote by $D_{j,V}^{\partial}$ the divisor of $Y_V$ which is the component in 
$Y_V$ of the intersection with $\cY_{0}^h$ of the closure of 
$D_j \times (\A^1-\{0\})$ in $\cY^h$.

\begin{lem} \label{lem_toric_homological_balancing}
Let $V$ be a vertex of $\Gamma$ with $h(V) \in (\partial \Delta)_j$
for some $1 \leq j \leq l$. Then we have 
$d_V \cdot D_{j,V}^{\partial} > 0$ 
if and only if there is an edge $E$
of $\Gamma$ incident to $V$ such that $h(E) \not\subset (\partial \Delta)_j$.
\end{lem} 

\begin{proof}
First assume that $h(V) \neq v_j$. Then, $Y_V$ can be described as a
toric blow-up of $\PP^1 \times \PP^1$, where all the added rays 
are contained in the lower half-plane of the fan, and where the vertical ray corresponds to 
$D_{j,V}^{\partial}$, see Figure \ref{figure_fan_6}. The lower part of the fan gives a local picture of 
$\Delta^h$ near $h(V)$. By definition of the $\Delta^h$, every edge $E$ 
of $\Gamma$ incident to $V$ is mapped by $h$ to one of the rays in the lower part of the fan. We have 
$h(E) \not\subset (\partial \Delta)_j$ if and only if $E$ is contained in 
on of the rays in the strict lower part of the fan. The result then follows from toric homological balancing.

If $h(V) =v_j$, the argument is similar. Recall that we have $D_j \simeq \PP^1$.
The key point is that $D_j$ is nef and so 
$D_j^2 \geq 0$. Therefore, $Y_V$ can be described as a toric blow-up of the Hirzebruch surface $\F_{D_j^2}$, where all the added rays are contained in the lower half-part of the fan, and where the vertical ray, with self-intersection $D_j^2$, corresponds to $D_{j,V}^{\partial}$, see Figure 
\ref{figure_fan_7}. The lower part of the fan gives a local picture of $\Delta^h$
near $h(V)$.  By definition of the $\Delta^h$, every edge $E$ 
of $\Gamma$ incident to $V$ is mapped by $h$ to one of the rays in the lower part of the fan. We have $h(E) \not\subset (\partial \Delta)_j$ if and only if $h(E)$ is contained in 
on of the rays in the strict lower part of the fan. As $D_j^2 \geq 0$, the lower part of the fan is convex and so the result follows from toric homological balancing.

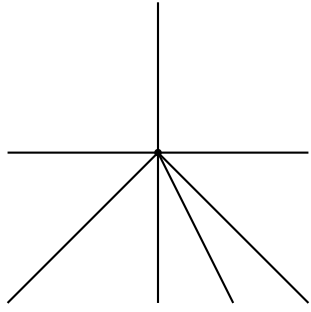
\begin{figure}
\caption{Toric description of $Y_V$ for $V \in (\partial \Delta)_j - \{v_j\}$}
\begin{center}
\setlength{\unitlength}{1cm}
\begin{picture}(6,5)
\thicklines
\put(3,3){\circle*{0.1}}
\put(3,3){\line(1,0){2}}
\put(3,3){\line(-1,0){2}}
\put(3,3){\line(0,1){2}}
\put(3,3){\line(0,-1){2}}
\put(3,3){\line(1,-1){2}}
\put(3,3){\line(-1,-1){2}}
\put(3,3){\line(1,-2){1}}
\end{picture} 
\end{center}
\label{figure_fan_6}
\end{figure}

\begin{figure}
\caption{Toric description of $Y_V$ for $V=v_j$}
\begin{center}
\setlength{\unitlength}{1cm}
\begin{picture}(6,5)
\thicklines
\put(3,3){\circle*{0.1}}
\put(3,3){\line(-1,0){2}}
\put(3,3){\line(0,1){2}}
\put(3,3){\line(0,-1){2}}
\put(3,3){\line(1,-1){2}}
\put(3,3){\line(-1,-1){2}}
\put(3,3){\line(1,-2){1}}
\end{picture} 
\end{center}
\end{figure}
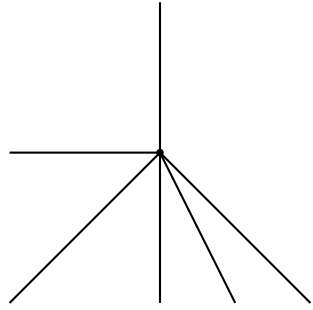
\label{figure_fan_7}
\end{proof}

For every edge $E$ of $\Gamma$, we denote by $H_{j,E} \in \hhh^2(X_E)$
the first Chern class of the tautological line bundle 
$\cO_{\PP(\ccL_j|_{Y_E} \oplus \cO_{Y_E})}(1)$.
We have $H_{j,E}^2=-c_1(\ccL_j|_{Y_E}) H_{j,E}$. 
If $h(E) \not\subset (\partial \Delta)_j$, then $\ccL_j|_{Y_E}=\cO_{Y_E}$
and so $H_{j,E}^2=0$. If $h(E) \subset (\partial \Delta)_j$, then 
$\ccL_j|_{Y_E}=\cO(-1)$, and so 
$H_{j,E}^2=(\pi_E^{*} \pt_E) H_{j,E}$.

\begin{lem}\label{lem_boundary}
Let $V$ be a vertex of $\Gamma$
with $h(V) \in (\partial \Delta)_j$. 
Assume that there exists an ingoing edge $E$
incident to $V$ such that $\alpha_E$ is
a non-zero multiple of $H_{j,E}$. 
Then, the image by $h$ of the outgoing edge $E_V$
incident to $V$ is contained in $(\partial \Delta)_j$
and $\alpha_{E_V}$ is a non-zero multiple of $H_{j,E_V}$.
\end{lem}

\begin{proof}
If $h(E_V) \not\subset (\partial \Delta)_j$,
then, by Lemma \ref{lem_toric_homological_balancing}, we have 
$d_V \cdot D_{j,V}^{\partial}>0$ and so $\alpha_{E_V}$
is proportional to $H_{j,E_V}^2=0$. Therefore, $\alpha_{E_V}=0$, in contradiction
with the assumption $N^{{\rm loc},h}_d(Y(D)) \neq 0$, and so this does not happen.

Therefore, we can assume that $h(E_V) \subset (\partial \Delta)_j$. 
If $d \cdot D_{j,V}^{\partial}>0$, then $\alpha_{E_V}$
is a multiple of $H_{j,E_V}^2 =(\pi_{E_V}^{*} \pt_{E_V})H_{j,E_V}$.
If $d \cdot D_{j,V}^{\partial}=0$, then $\alpha_{E_V}$
is a multiple of $H_{j,E_V}$. 
\end{proof}

\begin{lem} \label{lem_repulsive_1}
Let $V$ be a vertex of $\Gamma$ with an incident ingoing edge $E$
such that $\alpha_E$ is proportional to $H_{j,E}$. Then, $h(V) \notin (\partial \Delta)_j$.
\end{lem}

\begin{proof}
Else, by iterative application of Lemma \ref{lem_boundary}, all the descendants of $V$
are mapped by $h$ to $(\partial  \Delta)_j$, in contradiction with the fact that the sink
$V_0$ of $\Gamma$ is mapped by $h$ to $v_Y$.
\end{proof}

\begin{lem}\label{lem_repulsive_2}
Let $V$ be a vertex of $\Gamma$ such that $V \in (\partial \Delta)_j$ and such that there 
exists an ingoing edge $E$ incident to $V$
with $h(E) \not\subset (\partial \Delta)_j$.
Then, denoting by $E_V$ the outgoing edge incident to $V$, $\alpha_{E_V}$
is a non-zero multiple of $H_{j,E_V}$.
\end{lem}

\begin{proof}
By Lemma \ref{lem_toric_homological_balancing}, we have 
$d_V \cdot D_{j,V}^{\partial} > 0$, and so the result follows.
\end{proof}

\begin{lem} \label{lem_repulsive_3}
Let $V$ be a vertex of $\Gamma$ such that $V \in (\partial \Delta)_j$
and such that there exists an incident to $V$ ingoing edge $E$ with $h(E) \not\subset (\partial \Delta)_j$.
Then, denoting by $E_V$ the outgoing edge incident to $V$, we have $h(E_V) \not\subset (\partial \Delta)_j$.
\end{lem}

\begin{proof}
The result follows from the combination of Lemma \ref{lem_repulsive_1} and Lemma \ref{lem_repulsive_2}.
\end{proof}

We say that a vertex $V$ of $\Gamma$ is a source if every bounded edge incident to $V$ is outgoing. 
As we are assuming that every vertex of $\Gamma$ has at most one outgoing edges, a source has a unique 
bounded incident edge.

\begin{lem} \label{lem_source_1}
Let $V$ be a source of $\Gamma$. Then, either $h(V)=v_Y$ or $h(V) \in \partial \Delta$.
\end{lem}

\begin{proof}
For $V$ such that $h(V) \in \Delta - \partial \Delta -\{v_Y\}$, the toric 
balancing condition holds at $h(V)$. This balancing condition cannot hold 
if there is a unique bounded edge incident to $V$.
\end{proof}

\begin{lem} \label{lem_source_2}
Let $V$ be a vertex of $\Gamma$ such that $V$ is a source and 
$h(V) \in \partial \Delta$. Then, there exists $1 \leq j \leq l$
such that $\alpha_{E_V}$ is a non-zero multiple of $H_{j,E_V}$.
\end{lem}

\begin{proof}
We know that $h(V) \in (\partial \Delta)_j$ for at least one $j$.

Assume first that $h(V) \neq v_p$ for every $p \in D_j \cap D_{j'}$, that is $h(V) \in (\partial \Delta)_j$
for a  unique $j$. As $V$ is a source, there is a single edge  incident to $V$. By homological toric balancing (see Figures 6 and 7), this 
is possible only if $h(E_V)$ is contained in the ray opposite to the ray corresponding to $D_{j,V}^{\partial}$, and in particular
we then  have $d_V \cdot D_{j,V}^{\partial}>0$.

It remains to treat the case where $h(V)=v_p$ for some 
$p \in D_j \cap D_{j'}$. In this case, we have $h(V)=v_p 
\in (\partial \Delta)_ij\cap (\partial \Delta)_{j'}$.
By homological toric balancing, we necessarily have
$d_V \cdot D_{k,V}^{\partial}>0$ for some $k \in \{j,j'\}$.
\end{proof}

\begin{lem} \label{lem_source_4}
Let $V$ be a vertex of $\Gamma$ such that $V$ is a source, $h(V) \in (\partial \Delta)_j$
for some $1 \leq j \leq l$, and $h(V) \neq v_p$ for every $p \in D_j \cap D_{j'}$. 
Then $d_{V}$ is a multiple of the class of a $\PP^1$-fibre of $Y_V$ and $E_V \not\subset (\partial \Delta)_j$.
\end{lem}

\begin{proof}
Similar to the proof of Lemma \ref{lem_source_2}.
\end{proof}

\begin{lem} \label{lem_H_propagation}
Let $V$ be a vertex of $\Gamma$ with $V \neq V_0$ and an incident ingoing edge $E$
with $\alpha_E$ a non-zero multiple of $H_{j,E}$. Then, denoting by $E_V$
the outgoing edge incident to $V$, $\alpha_{E_V}$ is a non-zero multiple of $H_{j,E_V}$.
\end{lem}

\begin{proof}
If $h(V) \in (\partial \Delta)_j$, then the result follows from Lemma 
\ref{lem_boundary}.
If $h(V) \notin (\partial \Delta)_j$, then the result is clear as the line bundle
$\ccL_j|_{Y_V}$ is trivial.
\end{proof}

\begin{lem} \label{lem_existence}
For every $1 \leq j \leq l$, there exists an edge $E$ of $\Gamma$ with $\alpha_E$ 
a non-zero multiple of $H_{j,E}$.
\end{lem}

\begin{proof}
As $d \cdot D_j>0$, there exists a vertex $V$ of $\Gamma$ with 
$d_V \cdot D_{j,V}^{\partial}>0$.
Denoting by $E_V$ the outgoing edge incident to $V$, $\alpha_{E_V}$ is a non-zero 
multiple of $H_{j,E_V}$.
\end{proof}

For every $1 \leq j \leq l$, we denote by 
$V_j(h)$ the set of vertices $V$ of $\Gamma$
with $h(V) \in (\partial \Delta)_j$.
As $d \cdot D_j>0$, there exists a vertex $V$ of $\Gamma$ with 
$d_V \cdot D_{j,V}^{\partial}>0$, 
and so with $h(V) \in (\partial \Delta)_j$, and in particular 
$V_j(h)$ is non-empty.

\begin{lem} \label{lem_unique}
For every $1 \leq l \leq l$, there exists exactly one $V \in V_j(h)$
such that, denoting by $S(V)$ the successor of $V$, we have $S(V) \notin V_j(h)$.
We denote this vertex by $V_j$.
\end{lem}

\begin{proof}
As the sink $V_0$ satisfies $h(V_0)=v_Y$, there exists at least one 
$V \in V_j(h)$ such that $S(V) \notin V_j(h)$.

Remark that if $V \in V_j(h)$ is such that 
$S(V) \notin V_j(h)$, then by Lemma \ref{lem_toric_homological_balancing}, we have 
$d_V \cdot D_{j,V}^{\partial}>0$ and so $\alpha_{E_V}$
is a non-zero multiple of $H_{j,E_V}$. Therefore, by Lemma \ref{lem_H_propagation},
all descendant edges $E$ of $V$ have $\alpha_E$ equal to non-zero multiple of $H_{j,E}$. 
Also, by Lemma 
\ref{lem_repulsive_3}, we have $h(E) \not\subset (\partial \Delta)_j$ for every edge $E$
descendant from $V$.

Assume that we had $V_1$ and $V_2$ in $V_j(h)$ with $V_1 \neq V_2$,
$S(V_1) \notin V_j(h)$, and $S(V_2) \notin V_j(h)$. Then, the flow descendant from
$V_1$ and $V_2$ meet somewhere, either at a vertex $V \neq V_0$ with 
$h(E_V) \not\subset (\partial \Delta)_j$, or at $V_0$. In either case, we deduce from 
$H_{j,E_V}^2=0$ and $H_{j,V_0}^2=0$ that $N^{{\rm loc},h}_d(Y(D))=0$, contradiction.

\end{proof}

\begin{lem} \label{lem_anc}
Let $1 \leq j \leq l$. Every $V \in V_j(h)$
is an ancestor of $V_j$.
\end{lem}

\begin{proof}
As the sink $V_0$ is such that $h(V_0)=v_Y$, the flow descendant of $V$ has to go out of 
$(\partial \Delta)_j$, and this can only happen via $V_j$ by Lemma 
\ref{lem_unique}.
\end{proof}








\begin{lem} \label{lem_H_everywhere}
Let $E$ be a bounded edge of $\Gamma$ such that 
$\alpha_E$ is not a non-zero multiple of any $H_{j,E}$. Then we have $E=E_V$ where $V$ is a source of $\Gamma$ with $h(V) =v_Y$.
\end{lem}

\begin{proof}
By Lemma \ref{lem_source_1}, for a source $V$ of $\Gamma$, we have either 
$h(V)=v_Y$ or $h(V) \in \partial \Delta$. If one of the source ancestor $V$ of $E$ had 
$h(V) \in \partial \Delta$, we would have by combination of Lemma \ref{lem_source_2}
and Lemma \ref{lem_H_propagation} that $\alpha_E$ is a non-zero multiple of 
$H_{j,E}$ for some $1 \leq j \leq l$. Therefore, for every source $V$
ancestor of $E$, we have $h(V)=v_Y$. 

Assume by contradiction that there are at least two distinct sources ancestor of $E$. 
Then, there exists a vertex $V$ ancestor of $E$ where at least two distinct source edges meet.
As the source edges are emitted by sources mapped 
to $v_Y$ by $h$, they can only meet if their images by 
$h$ are contained in a common half-line in $\Delta$
with origin $v_Y$.
If $h(V) \in (\partial \Delta)_j$ for some $j$, then $\alpha_{E_V}$, and so 
$\alpha_E$ by Lemma \ref{lem_H_propagation}, would have been a non-zero multiple of 
$H_{j,E}$ by Lemma \ref{lem_H_propagation}. Therefore, $h(V) \in \Delta - \partial \Delta$.
On the other hand, we have $h(V) \neq v_Y$. Therefore, the toric balancing condition applies at 
$h(V)$ and $h(E_V)$ is parallel to the direction of the ingoing edges. Moving 
$h(V)$ along the common direction of all the edges incident to $V$ 
produces a contradiction with the assumed rigidity of $h$.

Therefore, $E$ admits a unique ancestor source $V$. So any other vertex of $\Gamma$
along the flow from $V$ to $E$ would have to be a 2-valent vertex, 
in contradiction with the rigidity of $h$. We conclude that 
$E=E_V$.
\end{proof}



Assume that $l=2$. We choose the flow such that $V_0$ is the vertex $V$ of $\Gamma$
incident to the ($l=2$!) unbounded edge of $\Gamma$.

\begin{lem} \label{lem_key_l_2}
The set of bounded edges of $\Gamma$ incident to 
$V_0$ consists of two elements $E_1$ and $E_2$ with 
$\alpha_{E_1}=\lambda_1 H_{1,E_1}$ and $\alpha_{E_2}= \lambda_2 H_{2,E_2}$, where $\lambda_1, \lambda_2 \in \Q-\{0\}$.
\end{lem}

\begin{proof}
As $h(V_0)=v_Y$, we can apply Lemma 
\ref{lem_H_everywhere} to $V_0$. Therefore, for every bounded 
edge $E$ incident to $V_0$, there exists $1 \leq j \leq 2$
such that $\alpha_E$ is a non-zero multiple of $H_{j,E}$.
By combination of Lemma \ref{lem_existence} and Lemma \ref{lem_H_propagation}, 
for every $1 \leq j \leq 2$, there exists at least one bounded edge $E$ incident to $V_0$ with $\alpha_E$ a non-zero multiple of $H_{j,E}$.
As $H_1^2=H_2^2=0$ on $\PP(\cV_{0,V_0})$, for every $1 \leq j \leq 2$, there is at most one bounded edge incident to $V_0$ with 
$\alpha_E$ a non-zero multiple of $H_{j,E}$.

Therefore, we have two cases. Either the set of bounded edges incident to $V_0$ consists of one edge $E$ with $\alpha_E$ a non-zero multiple of 
$H_{1,E}H_{2,E}$, or the set of bounded edges
incident to $V_0$ consists of two edges $E_1$ and $E_2$ with $\alpha_{E_1}$ a non-zero multiple of 
$H_{1,E}$ but not of $H_{2,E}$, and 
$\alpha_{E_2}$ a non-zero multiple of 
$H_{2,E}$ but not $H_{1,E}$.

Let us show that the first case does not arise. If the set of bounded edges incident to $V_0$ consists of a single element, then the moduli space $M_{V_0}$ has virtual dimension $2$. 
Indeed, the virtual dimension of $M_{V_0}$ is $0+2$, where 
$0$ is the virtual dimension for rational curves in the log Calabi--Yau surface $Y$ intersecting the boundary divisor $D$ in a single point, and $2$ comes from the two extra trivial directions
$\cO_Y^{\oplus 2}$. But we need to integrate over $[M_{V_0}]^{\virt}$ the pullbacks of the class $H_{1,E}H_{2,E}$ (coming from the bounded edge $E$ incident to $V_0$) and the pullback of 
$\pi_{V_0}^{*}\pt_Y$ (coming from the unbounded edge incident to 
$V_0$. Therefore, the integrand is a class of degree at least $3>2$, and so this case does not arise if $N^{{\rm loc},h}_d(Y(D)) \neq 0$.

Thus, we are in the second case, where the set of bounded edges 
incident to $V_0$ consists of two edges $E_1$ and $E_2$ with $\alpha_{E_1}$ a non-zero multiple of 
$H_{1,E}$ but not of $H_{2,E}$, and 
$\alpha_{E_2}$ a non-zero multiple of 
$H_{2,E}$ but not $H_{1,E}$. In particular, 
the moduli space $M_{V_0}$ has virtual dimension $3$.
Indeed, the virtual dimension of $M_{V_0}$ is $1+2$, where 
$1$ is the virtual dimension for rational curves in the log Calabi--Yau surface $Y$ intersecting the boundary divisor $D$ in two points, and $2$ comes from the two extra trivial directions
$\cO_Y^{\oplus 2}$.
As we need to integrate over $[M_{V_0}]^{\virt}$ the pullbacks of the classes 
$\alpha_{E_1}$, $\alpha_{E_2}$, and $\pi_{V_0}^{*} \pt_Y$, with 
$\deg \alpha_{E_1} \geq 1$, and $\deg \alpha_{E_2} \geq 1$, 
the condition $N^{{\rm loc},h}_d(Y(D)) \neq 0$ implies that  
$ \deg \alpha_{E_1}=\deg \alpha_{E_2}=1$
and so the classes $\alpha_{E_1}$ and $\alpha_{E_2}$
are scalar-multiple of $H_{1,E_1}$ and $H_{2,E_2}$
respectively.
\end{proof}

For every $1 \leq j \leq 2$, let $\Gamma_j$ be the subset of
$\Gamma$ described by the flow from $V_j$ to $V_0$.

\begin{lem} \label{lem_gamma_intersection}
We have $\Gamma_1 \cap \Gamma_2=\{V_0\}$.
\end{lem}

\begin{proof}
Assume by contradiction that there exists a vertex $V$ of $\Gamma$ with $V \neq V_0$ and $V \in \Gamma_1 \cap \Gamma_2$.
By Lemma \ref{lem_H_propagation}, $\alpha_{E_V}$
is a non-zero multiple of $H_{1,E_V}H_{2,E_V}$. By iterative application of Lemma \ref{lem_H_propagation} along the flow from $V$ to $V_0$, we would deduce that there exists a bounded edge $E$
incident to $V_0$ with 
$\alpha_{E}$ a non-zero multiple of 
$H_{1, E}H_{2,E}$, which is not possible by Lemma 
\ref{lem_key_l_2}.
\end{proof}

\begin{lem} \label{lem_int}
Let $1 \leq j \leq 2$. For every vertex $V$ of 
$\Gamma_j$ with $V \neq V_j$, we have
$h(V) \in \Delta - \partial \Delta$.
\end{lem}

\begin{proof}
Up to exchanging $1$ and $2$ in the following argument, we can 
assume that $j=1$. By definition of $V_1$, we have 
$V \notin (\partial \Delta)_1$ for every vertex $V$ of 
$\Gamma_1$ with $V  \neq V_1$. 

Assume by contradiction that there exists a vertex $V$ of $\Gamma_1$ with $V \in (\partial \Delta)_2$.
Then $V$ is an ancestor of $V_2$ by 
Lemma \ref{lem_anc} and so $V_2 \in \Gamma_1$.
Therefore, we have $V_2 \in \Gamma_1 \cap \Gamma_2$
and this contradicts Lemma
\ref{lem_gamma_intersection}.
\end{proof}

\begin{lem} \label{lem_away}
Let $1 \leq j \leq 2$. Let $V$ be a vertex of $\Gamma_j$ with $V \neq V_j$
and $V \neq V_0$.
Let $E$ be an ingoing edge incident to $V$. Then, either $E$ is a descendant of $V_j$,
or $E=E_{V'}$ for $V'$ a source with $h(V)=v_Y$.
\end{lem}

\begin{proof}
Up to exchanging $1$ and $2$ in the following argument, we can 
assume that $j=1$.
Remark first that $V \in \Delta -\partial \Delta$ by 
Lemma \ref{lem_int}.
Assume that $E$ is not a descendant of $V_1$. Then 
$\alpha_E$ is not a non-zero multiple of $H_{1,E}$. On the other hand, by Lemma 
\ref{lem_gamma_intersection}, $E$ is not a descendant of $V_2$, and so
$\alpha_E$ is not a non-zero multiple of $H_{2,E}$.
Therefore, the result follows from Lemma \ref{lem_H_everywhere}.
\end{proof}

We say that an edge $E$ of $\Gamma$
is \emph{radial} if 
$h(E) \not\subset \partial \Delta$ and the direction of 
$h(E)$ passes through $v_Y$.

\begin{lem} \label{lem_nobody}
For every $1 \leq j \leq 2$, we have $h(\Gamma_j)=[V_j,V_0]$.
\end{lem}

\begin{proof}
If $E$ were a non-radial edge of $\Gamma_j$, then by Lemma
\ref{lem_away},
no descendant edge of $E$ can be radial, which is a contradiction because the final edge of 
$\Gamma_j$ entering $V_0$ is radial.
\end{proof}

\begin{lem} \label{lem_source_3}
Let $V$ be a source with $h(V) \in \partial \Delta$. Then, $V \in \{V_1, V_2\}$.
\end{lem}

\begin{proof}
If there exists a unique $j$ such that $h(V) \in (\partial \Delta)_j$, then $h(E_V) \not\subset 
(\partial \Delta)_j$ by Lemma \ref{lem_source_4},
and so $V=V_j$ by Lemma \ref{lem_unique}.

If not, then $h(V)=v_p$ for some $p \in D_j \cap D_{j'}$.
We have either $h(E_V) \subset (\partial \Delta)_j$ or 
$h(E_V) \subset (\partial \Delta)_{j'}$: else, we would have 
$V=V_j=V_{j'}$, in contradiction with Lemma \ref{lem_gamma_intersection}.
Therefore, up to relabeling $j$ and $j'$,
we can assume that 
$h(E_V) \subset (\partial \Delta)_j$. By toric homological balancing, it follows that $\alpha_{E_V}$ is a non-zero multiple of 
$H_{j',E_V}$. The flow descendant from $V$ have to go out from 
$(\partial \Delta)_j$, necessarily at $V_j$ by Lemma \ref{lem_unique}, in contradiction with the fact that
$\alpha_{E_{V_j}}$ is not a non-zero multiple of $H_{j',E_{V_j}}$ by Lemma \ref{lem_key_l_2} 
(and we used Lemma \ref{lem_H_propagation}).
\end{proof}

\begin{lem} \label{lem_vertex}
For every $1 \leq j \leq 2$, we have $V_j(h)=\{V_j \}$.
\end{lem}

\begin{proof}
Assume by contradiction that there exists $V \in V_j(h)$ with 
$V \neq V_j$. By Lemma \ref{lem_anc}, $V$ is an ancestor of $V_j$, and so the flow from 
$V$ to $V_j$ is entirely contained in $(\partial \Delta)_j$. Let $V$ be one of the oldest vertices with these properties. By Lemma \ref{lem_source_3}, $V$ is not a source. Therefore, there exists an edge $E$
ingoing incident to $V$ with $h(E) \not\subset (\partial \Delta)_j$ and so $V=V_j$ by Lemma
\ref{lem_repulsive_3}, contradiction.
\end{proof}

\begin{lem} \label{lem_fibre}
For every $1 \leq j \leq 2$, the class $d_{V_j}$ of $V_j$
is a multiple of the class of a $\PP^1$-fibre of $Y_{V_j}$. Moreover, we have $h(V_j)=v_j$.
\end{lem}

\begin{proof}
Up to exchanging $1$ and $2$ in the following argument, we can assume that $j=1$.
Let $E$ be an edge ingoing incident to $V_1$. Then $\alpha_E$ is not a non-zero multiple of 
$H_{1,E}$ by Lemma 
\ref{lem_boundary} and is also not a non-zero multiple of $H_{2,E}$ by combination of 
Lemma \ref{lem_key_l_2} and Lemma \ref{lem_H_propagation}.
Therefore, $E$ is radial by Lemma \ref{lem_H_everywhere}. 
On the other hand, $E_{V_1}$ is radial by Lemma \ref{lem_nobody}. Thus, all edges through $V_j$ are radial.
By toric homological balancing, this is only possible if $d_{V_1}$
is a multiple of the class of a $\PP^1$-fibre of $Y_{V_1}$.
\end{proof}

\begin{lem} \label{lem_unique_edge}
For every $1 \leq j \leq 2$, there exists a unique edge incident to $V_j$.
\end{lem}

\begin{proof}
This follows from the dimension argument of Lemma 5.4 of \cite{vGGR}.
\end{proof}

It follows from the combination of Lemmas \ref{lem_vertex},  \ref{lem_fibre}, 
and  \ref{lem_unique_edge} that $h=\bar{h}$, and this concludes the proof 
of \cref{thm_key}.
\begin{flushright}$\Box$\end{flushright}

\section{Symmetric functions}
\label{sec:prelim}

\subsubsection{Partitions and representations of $S_n$}

A partition $\lambda \vdash d$ of a non-negative integer $d\in \bbN$ is a monotone non-increasing sequence $\lambda:=\{\lambda_i\}_{i=1}^r$, $\lambda_1\geq \lambda_2\geq  \dots \lambda_r\geq0$ such that $\sum_{i=1}^r \lambda_i =d$; when $d=0$ we write $\lambda=\emptyset$ for the empty partition. We will often use the short-hand notation
\beq
\{ \lambda_1^{n_1}, \dots, \lambda_k^{n_k} \} \coloneqq \{ \overbrace{\lambda_1, \dots, \lambda_1}^{n_1~\mathrm{times}}, \dots, \overbrace{\lambda_k, \dots, \lambda_k}^{n_k~\mathrm{times}} \}
\eeq 
for partitions with repeated entries. 


With notation as in the beginning of \cref{sec:opengw}, a partition $\lambda$ is bijectively associated to:
\bit
\item a Young diagram $Y_\lambda$ with $m_j(\lambda)$ rows of boxes of length $j$; there is a natural involution in the space of partitions, $\lambda \to \lambda^t$, given by transposition of the corresponding Young diagram;
\item a conjugacy class $C_\lambda \in \mathrm{Conj}(S_{|\lambda|})$ of the symmetric group $S_{|\lambda|}$ with automorphism group of order $|\mathrm{Aut}_{C_\lambda}|=|\lambda|! z_\lambda$, with $$z_\lambda \coloneqq \prod_j m_j(\lambda)! j^{m_j(\lambda)};$$ 
\item an irreducible representation $\rho_\lambda \in \mathrm{Rep}(S_d)$; for $\eta\in\mathrm{Conj}(S_d)$, we write $\chi_\lambda(\eta)$ for the irreducible character $\mathrm{Tr}_{\rho_\lambda}(\eta)$;
\item by Schur--Weyl duality, an irreducible representation $R_\lambda \in \mathrm{Rep}(\mathrm{GL_n(\bbC)})$ for $n\geq \ell_\lambda$.
\eit
We will be concerned with two linear bases of the ring of integral symmetric polynomials in $n$ variables, $\Lambda_n:= \bbZ[x_1, \dots, x_n]^{S_n}$, labelled by partitions with $\ell(\lambda)\leq n$. Write $x:=(x_1, \dots, x_n)^{S_n} \in \bbC^n/S_n$ for an orbit $x$ of the adjoint action of $\mathrm{GL}_n(\bbC)$ (equivalently, the Weyl group action on $\bbC^n$), and $g_x$ for any element of the orbit. We write
\beq
p_\lambda(x) \coloneqq \prod_i \mathrm{Tr}_{\bbC^n} g_x^{m_i(\lambda)}, \qquad s_\lambda(x) \coloneqq \mathrm{Tr}_{R_\lambda}(g_x)\,,
\eeq
for, respectively, the symmetric power function and the Schur function determined by $\lambda$; we have $\Lambda_n = \mathrm{span}_\bbZ \{p_\lambda\}_{\{\lambda \in \cP, \ell_\lambda \leq n\}}=\mathrm{span}_\bbZ \{s_\lambda\}_{\{\lambda \in \cP, \ell{\lambda}\leq n\}}$. These two bases are related as
\beq
s_\mu(x)=\sum_{|\lambda|=|\mu|} \frac{\chi_\mu(\lambda)}{z_\lambda} p_\lambda(x)\,, \qquad p_\mu(x)=\sum_{|\lambda|=|\mu|} \chi_\mu(\lambda) s_\lambda(x)\,.
\label{eq:schurvspow}
\eeq
For $\lambda, \mu$ a pair of partitions, the skew Schur polynomials $s_{\lambda/\mu}(x)$ are defined by 
\beq
s_{\frac{\lambda}{\mu}}(x)= \sum_{\nu \in \cP} \mathsf{LR}^{\lambda}_{\mu\nu} s_\nu(x)\,,
\eeq 
where $\mathsf{LR}^{\lambda}_{\mu\nu}$ are the Littlewood--Richardson coefficients  $R_\mu \otimes R_\nu =: \bigoplus_{\lambda \vdash (|\mu|+|\nu|)} \mathsf{LR}^\lambda_{\mu\nu} R_\lambda$.\\

Let $\rho: \Lambda_n \to \Lambda_{n+1}$ be the monomorphism of rings defined by $\rho(p_{(i)}(x_1, \dots, x_n)) = p_{(i)}(x_1, \dots, x_{n+1})$. We define the ring of symmetric functions $\Lambda := \varinjlim \Lambda_n$ as the direct limit under these inclusions, and denote by the same symbols $p_\lambda$, $s_\lambda$, $s_{\lambda/\mu}$ the symmetric functions obtained as the images of the power sums, Schur polynomials, and skew Schur polynomials under the direct limit. In the next sections it will be of importance to formally expand the infinite product $\prod_{i,j} (1- x_i y_j) \in \Lambda \otimes_\bbZ \Lambda$ around $(x,y)=(0,0)$, and it is a classical result in the theory of symmetric functions out that this expansion can be cast in multiple ways in terms of an average over partitions of bilinear expressions of linear generators of $\Lambda$. In particular, we have the Cauchy identities
\beq
\sum_{\lambda \in \cP} s_\lambda(x) s_\lambda(y) = \prod_{i,j} (1-x_i y_j)^{-1}\,, \quad \sum_{\lambda \in \cP} s_\lambda(x) s_{\lambda^t}(y) = \prod_{i,j} (1+x_i y_j)\,.
\label{eq:cauchyprod}
\eeq
A skew generalisation of these \cite[\S I.5]{MR1354144} is 
\bea
\sum_{\lambda \in \cP} s_{\frac{\lambda}{\mu}}(x) s_{\frac{\lambda}{\nu}}(y) &=& \prod_{i,j} (1-x_i y_j)^{-1} \sum_{\eta \in \cP} s_{\frac{\nu}{\eta}}(x) s_{\frac{\mu}{\eta}}(y)\,, \nn \\
\sum_{\lambda \in \cP} s_{\frac{\lambda^t}{\mu}}(x) s_{\frac{\lambda}{\nu}}(y) &=& \prod_{i,j} (1+x_i y_j) \sum_{\eta \in \cP} s_{\frac{\nu^t}{\eta}}(x) s_{\frac{\mu^t}{\eta^t}}(y)\,.
\label{eq:cauchyprodskew}
\eea
Another noteworthy sum we will need is \cite[\S I.5]{MR1354144}
\beq
\sum_{\delta \in \cP} s_{\frac{\lambda}{\delta}}(x) s_{\frac{\delta}{\nu}}(y) = s_{\frac{\lambda}{\mu}}(x,y) \,,
\label{eq:doubleschur}
\eeq
where $s_{\frac{\lambda}{\mu}}(x,y)$ denotes the skew Schur function in the variables $(x_1, x_2, \dots, x_i, \dots, y_1, y_2, \dots, y_i, \dots)$.

\subsubsection{Shifted symmetric functions and the principal stable specialisation}
From these ingredients and $\mu \in \cP$, we define a class of Laurent series of a single variable $q^{1/2}$ obtained by the {\it principal stable specialisation} 
\beq
\bary{cccc}
\mathfrak{q} : & \Lambda & \longrightarrow & \bbQ[[q^{-1/2}]] \\
& f(x_1, \dots, x_i, \dots) & \longrightarrow &  f(x_1=q^{-1+1/2}, \dots, x_n=q^{-i+1/2}, \dots)\,.
\eary
\eeq
As is customary in the topological vertex literature, and since $-i+1/2$ is the component of the Weyl vector $\rho$ of $A_n$ with respect to the fundamental weight $\omega_{n-i}$, we use the short-hand notation $f(q^\rho):=f(x_i=q^{-i+1/2})$. For $f$ a power sum or Schur function, $f(q^\rho)$ converges to a rational function of $q^{1/2}$. In particular,
\beq
p_{(d_1, \dots, d_n)}(q^\rho) = \prod_{i=1}^n \frac{1}{[d_i]_q}\,,
\eeq
and for Schur functions, Stanley proved the product formula \cite{MR325407}
\beq
s_\lambda(q^\rho) = \frac{q^{\kappa(\lambda)/4}}{\prod_{(i,j)\in \lambda}[h(i,j)]}\,,
\label{eq:stanleygen}
\eeq
where $h(i,j)$ is the number of squares directly below or to the right of a cell $(i,j)$ (counting $(i,j)$ once) in the Young diagram of $\lambda$. For example, when $\lambda=(i-j,1^j)$ is a hook Young diagram with $i$ boxes and $j+1$ rows, this gives
\beq 
s_{(i-j,1^j)}(q^\rho) = \frac{q^{\frac{1}{2}\l(\binom{i}{2}-i j\r)}}{[i]_q [i-j-1]_q! [j]_q!}\,.
\label{eq:stanley}
\eeq 
More generally, for $\mu \in \cP$, we will consider the shifted power, Schur, and skew Schur functions
\bea
p_\lambda(q^{\rho+\mu}) & \coloneqq & p_\lambda(x_i=q^{-i+\mu_i+1/2})\,, \nn \\
s_\lambda(q^{\rho+\mu}) & \coloneqq & s_\lambda(x_i=q^{-i+\mu_i+1/2})\,, \nn \\ s_{\lambda/\delta}(q^{\rho+\mu}) & \coloneqq & \sum_{\nu \in \cP} \mathsf{LR}^{\lambda}_{\delta\nu} s_\nu(q^{\rho+\mu})\,.
\label{eq:shiftskewschur}
\eea
The following identities follow easily from \eqref{eq:stanley}, \eqref{eq:shiftskewschur} and the fact that Littlewood--Richardson coefficients are invariant under simultaneous transposition of their arguments:
\bea
s_{\lambda}(q^\rho) &=& q^{\kappa(\rho)/2} s_{\lambda^t}(q^\rho) \,,  \label{eq:schurt} \\
s_\frac{\lambda}{\mu}(q^{\rho+\a}) &=&  
s_\frac{\lambda^t}{\mu^t}(-q^{-\rho-\a^t}) \,. \label{eq:skewschurt} 
\eea
Following \cite{Iqbal:2004ne}, we introduce the following notation for the Cauchy infinite products \eqref{eq:cauchyprod} in the principal stable specialisation:
\bea
\big\{ \a, \b \big\}_Q & \coloneqq & \prod_{i,j\geq 1} \big(1-Q q^{-i-j+1+\a_i+\b_i}\big) \nn \\
&=& \sum_{\lambda \in \cP} s_\lambda\big( q^{\rho+\a} \big) s_{\lambda^t}\big( -Q q^{\rho+\beta}\big) \nn \\
&=& \l[\sum_{\lambda \in \cP} s_\lambda \big( q^{\rho+\a} \big) s_{\lambda}\big(Q q^{\rho+\beta}\big)\r]^{-1} \,.
\label{eq:cauchyprodpss}
\eea
Finally, we will need to specialise expressions involving skew Schur functions and Cauchy products to the case of hook Young diagrams. These can be given closed-form  $q$-factorial expressions, as follows.
\begin{lem}
We have
\beq
 \mathsf{LR}^{(i-r,1^r)}_{\b, \gamma} = 
\l\{
\begin{array}{cc}
\delta_{i,j+k}\l(\delta_{r,s+t}+\delta_{r,s+t+1}\r) & \b=(j-s,1^s),\, \gamma = (k-t,1^t), \\
0 & \mathrm{else.}
\end{array}
\r.
\label{eq:LRhooks}
\eeq
Moreover,
%
\beq
s_{\frac{(i-j,1^j)}{\gamma}}(q^\rho) =
\l\{
\bary{cc}
\frac{q^{\frac{1}{4} (i-k-1) (i-2 j-k+2 l)}}{[i-k+l-j]_q! [j-l]_q!}, & \gamma = (k-l,1^l), \\
0, & \mathrm{else},
\eary
\r.
\label{eq:skewhooks}
\eeq
and
\beq
\frac{\{ (i-j, 1^j), \emptyset\}_Q}{\{ \emptyset, \emptyset\}_Q} = \prod_{k=0}^{i-1}(1-q^k Q q^{-j}) = (Q q^{-j};q)_i \,.
\label{eq:hookpairing}
\eeq
\label{lem:skewhooks}
\end{lem}

The content of the Lemma follows from a straightforward application of the Littlewood--Richardson rule in the case of hook partitions $(i-r, 1^r)$. The product formula\footnote{Unlike the Schur case of \eqref{eq:stanleygen}, closed $q$-formulas for principally-specialised skew-Schur functions are generally difficult to find, and \eqref{eq:skewhooks} is not listed in the most recent literature about them \cites{MR4023568,MR4017975}, although it can be seen to follow easily from existing results, see e.g. \cite[Theorem~1.4]{MR3718070}.} for the hook skew-Schur functions \eqref{eq:skewhooks} follows then immediately from \eqref{eq:stanley}. Finally, \eqref{eq:hookpairing} follows from a straightforward calculation from  \eqref{eq:cauchyprodpss}; see \cite[\S 3.4]{Iqbal:2004ne} for details.



\end{appendix}


\bibliography{miabiblio}

\end{document}